\documentclass[11pt, reqno]{amsart}

\usepackage{amssymb}
\usepackage{mathrsfs}
\usepackage{amsmath}
\usepackage{amsthm}
\usepackage{amsfonts}
\usepackage{latexsym}

\usepackage{color}
\usepackage{txfonts}
\usepackage{enumerate}

\usepackage[colorlinks=true,
linkcolor=blue,
citecolor=red,
urlcolor=magenta,
backref=page
]{hyperref}

\usepackage[english]{babel}

\usepackage{txfonts}

\usepackage{anysize}

\allowdisplaybreaks

\newtheorem{theorem}{Theorem}[section]
\newtheorem{proposition}[theorem]{Proposition}
\newtheorem{corollary}[theorem]{Corollary}
\newtheorem{lemma}[theorem]{Lemma}

\theoremstyle{definition}
\newtheorem{definition}[theorem]{Definition}
\newtheorem{example}[theorem]{Example}

\theoremstyle{remark}
\newtheorem{remark}[theorem]{Remark}

\numberwithin{equation}{section}

\newcommand{\A}{\mathscr A}

\newcommand{\C}{\mathbb C}

\newcommand{\D}{\mathscr D}

\newcommand{\N}{\mathbb N}

\newcommand{\R}{\mathbb R}

\newcommand{\Z}{\mathbb Z}

\newcommand{\bddlin}{\mathscr L}

\newcommand{\fx}{\mathfrak e}
\newcommand{\fy}{\mathfrak u}
\newcommand{\fz}{\mathfrak v}
\newcommand{\fX}{\mathfrak X}
\newcommand{\fY}{\mathfrak Y}
\newcommand{\fZ}{\mathfrak Z}

\newcommand{\eps}{\varepsilon}

\newcommand{\BMO}{\operatorname{BMO}}

\newcommand{\esssup}{\operatornamewithlimits{ess\,sup}}
\newcommand{\essinf}{\operatornamewithlimits{ess\,inf}}

\newcommand{\one}{\mathbf{1}}
\newcommand{\supp}{\operatorname{supp}}
\newcommand{\loc}{\mathrm{loc}}

\newcommand{\ave}[1]{\langle #1\rangle}
\newcommand{\aveL}{\textit{\L}}

\newcommand{\pair}[2]{\langle #1, #2 \rangle}

\newcommand{\Bpair}[2]{\Big\langle #1, #2 \Big\rangle}

\newcommand{\abs}[1]{|#1|}
\newcommand{\Babs}[1]{\Big|#1\Big|}
\newcommand{\Norm}[2]{\|#1\|_{#2}}
\newcommand{\bNorm}[2]{\big\|#1\big\|_{#2}}
\newcommand{\BNorm}[2]{\Big\|#1\Big\|_{#2}}

\newcommand{\so}{\operatorname{so}}
\newcommand{\HT}{\mathcal{H}}
\newcommand{\floor}[1]{\lfloor #1 \rfloor}
\newcommand{\ceil}[1]{\lceil #1 \rceil}

\begin{document}


\title[Operator-valued $A_p$ weights in Banach spaces]
{Real-variable theory of function spaces with\\ operator-valued $A_p$ weights in Banach spaces}

\author[T. P. Hyt\"onen]{Tuomas P.\ Hyt\"onen}
\address{Department of Mathematics and Systems Analysis,
Aalto University, P.O. Box 11100 (Otakaari~1), \mbox{FI-00076} Aalto, Finland}
\curraddr{}
\email{tuomas.hytonen@aalto.fi}

\author[Y. Li]{Yinqin Li}
\address{Laboratory of Mathematics and Complex Systems (Ministry of Education of China),
School of Mathematical Sciences, Beijing Normal University, Beijing 100875, The People's Republic of China}
\curraddr{}
\email{yinqli@mail.bnu.edu.cn}

\author[D. Yang]{Dachun Yang}
\address{Laboratory of Mathematics and Complex Systems (Ministry of Education of China),
School of Mathematical Sciences, Institute for Advanced Study,
Beijing Normal University, Beijing 100875, The People's Republic of China}
\curraddr{}
\email{dcyang@bnu.edu.cn}

\author[W. Yuan]{Wen Yuan}
\address{Laboratory of Mathematics and Complex Systems (Ministry of Education of China),
School of Mathematical Sciences, Beijing Normal University, Beijing 100875, The People's Republic of China}
\curraddr{}
\email{wenyuan@bnu.edu.cn}

\thanks{This project is partially supported by the National
Natural Science Foundation of China (Grant Nos. 12431006, 12371093, and 124B2004),
the Beijing Natural Science Foundation (Grant No. 1262011),
the Fundamental Research Funds for the Central Universities (Grant No. 2253200028),
and the Research Council of Finland (Grant Nos. 346314, 364208, and 371637).}





\begin{abstract}
While the theory of matrix-weighted function spaces is well established, the majority of previous results in the infinite-dimensional operator-valued setting deal with ``no go'' theorems, showing the impossibility of some prospective generalizations. However, we show that a complete real-variable theory of Besov and Triebel--Lizorkin spaces with operator-valued Muckenhoupt $A_p$ weights can still be developed, once correctly formulated. This covers operator-weighted extensions of results like the $\varphi$-transform characterization in terms of discrete sequence spaces, the boundedness of almost diagonal operators, and applications to the $T(1)$ theorem and trace/extension theorems. A key tool is a version of the reverse H\"older inequality, which is weak enough to follow from the operator-valued $A_p$ condition (unlike a variant that had to be imposed as an additional assumption in some previous works), yet strong enough to be used much like its classical counterpart.

In contrast to the established scalar and matrix-weighted theories, our approach cannot build on operator-weighted $L^p$ results, as these fail in a very definite way. We also strengthen the existing ``no go'' statements in Hilbert spaces, showing (among other counterexamples) that every infinite-dimensional Banach space has an operator-valued $A_p$ weight $V$ for which the Hilbert transform is unbounded on $L^p(V)$. This is nontrivial, since the lack of Hilbert space structure also complicates the construction of $A_p$ weights. Building on results from Banach space theory, we achieve this unboundedness by combining two distinct methods in two different classes of spaces (so-called $K$-convex ones and those that are {\em not} UMD), whose union covers all Banach spaces.
\end{abstract}

\maketitle

\tableofcontents

\section{Introduction}

Since its discovery in the 1970's, Muckenhoupt's theory of $A_p$ weights has become one of the central topics in harmonic analysis with diverse applications. Originally dealing with scalar-valued weight functions, the extension of this theory to matrix $A_p$ weights is also well developed since the 1990's.

The theme of this paper is operator-valued $A_p$ weights, a further extension of the matrix $A_p$ theory to possibly infinite-dimensional target spaces. Much less is known about this area, and much of what is known concerns ``no go'' theorems pointing to the failure of natural extensions of familiar results from finite dimensions, even in the supposedly simplest case of $p=2$ and a Hilbert space target, not to mention other values of $p$ and more general Banach spaces, where relatively little has been said so far.

Nevertheless, at least a definition of $A_p$ weights is readily available in great generality. While much of the recent literature on matrix $A_p$ weights has been based on an equivalent (in finite dimensions) condition given by Roudenko \cite[Lemma 1.3]{Rou:03}, the original definition by Nazarov--Treil \cite[p.~21]{NT:hunt} and Volberg \cite[p.~450]{Volberg:97}, stated in terms of dual norms, makes perfect sense for an arbitrary Banach space $\fX$. In a way, this generality even clarifies the picture, since one can no longer identify $\fX$ with its dual $\fX^*$, making it necessary to distinguish between norms on the original space and on the dual space. Some aspects of this general framework were already developed by Lauzon \cite{Lauzon}, but many of his results are only proved under additional postulates concerning versions of the reverse H\"older inequality that are not implied by the $A_p$ condition in general; for $p=2$ and Hilbert targets, such theory under reverse H\"older postulates was first developed by Pott \cite{Pott:07}.

The aim of this article is to show that the natural $A_p$ condition for operator-valued weights, acting on a general Banach space,
\begin{enumerate}[\rm(i)]
  \item\label{it:Ap} retains several key features of the classical $A_p$ condition;
  \item\label{it:Lp} is necessary but {\em insufficient} for a weighted theory in $L^p$;
  \item\label{it:BTL} is nevertheless sufficient for developing substantial parts of the theory of function spaces of Besov and Triebel--Lizorkin type.
\end{enumerate}
We will elaborate on each of these points shortly, but let us first provide some more context.

In both the scalar-weighted and the matrix-weighted cases, the resolution of the basic problems in $L^p$ (like the boundedness of the maximal operator and the Hilbert transform) came first, and served as stepping stones for the weighted theory of other function spaces. For scalar weights, the said $L^p$ results were obtained in \cite{Muck:72} and \cite{HMW}, respectively, and for matrix weights, in \cite{NT:hunt,TV:angle,Volberg:97} and \cite{CG:01,Gold:03}, in the reverse order. The scalar-weighted theory of Besov and Triebel--Lizorkin spaces seems to have started with \cite{Bui:82}, the matrix-weighted theory of Besov spaces with \cite{Rou:03}, and a systematic theory of matrix-weighted Triebel--Lizorkin spaces with \cite{FR:21}.

Against this background, the fact that we obtain a reasonable theory of Besov and Triebel--Lizorkin spaces with operator weights, even in the lack of a corresponding $L^p$ theory, might seem surprising. However, a natural point of analogy is found in the (unweighted) theory of singular integrals of Banach space valued functions. The problem of the $L^p$ boundedness of basic singular operators is understood in this situation, and known to be characterised by the UMD property of the target space $\fX$; this is a celebrated result of Burkholder \cite{Burk:83} and Bourgain \cite{Bou:83}; see also \cite[Theorem 5.1.1]{HNVW1} for a detailed treatment. However, the boundedness on (unweighted) Besov and Triebel--Lizorkin spaces is available much more generally, without any restrictions on $\fX$ (see \cite[Chapter 14]{HNVW3}). What we obtain for operator-weighted spaces is essentially the same, with the only difference that we do not have the characterisation of the weights for the $L^p$ boundedness. But even in the unweighted theory, the known characterisation of spaces by the UMD property does not play any role in the development of the theory of Besov and Triebel--Lizorkin spaces with values in arbitrary spaces.

The rest of this Introduction is divided into three sections, where we elaborate on each of the points \eqref{it:Ap}, \eqref{it:Lp}, and \eqref{it:BTL} above, indicating our main results and providing some further background. The rest of the paper is then divided into four parts, where we develop the theory on each of these points in full detail. The content and the order of both Sections \ref{ss:Ap}, \ref{ss:Lp}, and \ref{ss:BTL} of this Introduction, and Parts \ref{part:Ap}, \ref{part:Lp}, and \ref{part:BTL} of the rest of the paper, is the same as that of points \eqref{it:Ap}, \eqref{it:Lp}, and \eqref{it:BTL} above. The final Part \ref{part:discretizable} continues further with point \eqref{it:BTL} within a new axiomatic framework.

\subsection{Key properties of $A_p$ weights}\label{ss:Ap}

Concerning point \eqref{it:Ap} above, some key lemmas for scalar and
matrix weights that we are able to extend to our general setting are
\begin{enumerate}[\rm(a)]
  \item reverse H\"older inequalities (Proposition \ref{prop:RHI});
  \item\label{it:QvsR} quantitative doubling properties (Lemma~\ref{lem:QvsR}).
\end{enumerate}

The role of reverse H\"older inequalities for both scalar and matrix weights is well established since \cite{CF:74} and \cite{CG:01,Gold:03}, respectively.
For operator-valued weights, as already mentioned, Pott \cite{Pott:07} (for $p=2$ and Hilbert targets) and Lauzon \cite{Lauzon} (more generally) showed that many results can still be extended, {\em provided that versions of reverse H\"older inequalities are assumed}, but these are genuine additional assumptions on the weights, no longer consequences of the $A_p$ condition as they were in finite dimensions.

We wish to make the claim that, in the operator-weighted setting, the notion of a reverse H\"older inequality should be (only slightly) revised from the matrix-valued version of \cite{CG:01,Gold:03}. This revised formulation has the benefit of being implied by the operator-valued $A_p$ condition (as we show in Proposition \ref{prop:RHI}) while still strong enough to serve the role of the classical reverse H\"older inequality in several applications, notably in the development of the theory of Besov and Triebel--Lizorkin spaces with operator weights further below.

The role of various doubling properties is also well established in analysis. In point \eqref{it:QvsR} above, we have especially in mind types of estimates that were introduced into the analysis of matrix-weighted Besov spaces by Roudenko \cite[Definition 1.3]{Rou:04}, and that have played a significant role in related developments ever since; cf.\ \cite[Definition 2.1]{FR:21} and \cite[Section 2.1]{BHYY:I}. In the matrix-weighted theory, these estimates take the form
\begin{equation}\label{eq:AQAR}
  \Norm{A_Q A_R^{-1}}{}\leq c(Q,R),
\end{equation}
where $A_Q$ and $A_R$ are {\em reducing operators} of the matrix weight, and $c(Q,R)$ has moderate growth in terms of the size ratio and the relative distance of the cubes $Q$ and $R$. A reducing operator is a matrix inducing an equivalent Euclidean norm for a given norm on $\R^m$ or $\C^m$; it is produced by the John ellipsoid theorem \cite{John:48}, a strictly finite-dimensional tool.

While we have no analogue for reducing operators in infinite-dimensional study, we can still find useful substitutes for bounds of the form \eqref{eq:AQAR} with the left-hand side replaced by the ratio of the operator-weight-induced norms over the cubes $Q$ and $R$. Indeed, such a result is proved in Lemma \ref{lem:QvsR}, and then successfully applied in the theory of operator-weighted Besov and Triebel--Lizorkin spaces in Part \ref{part:BTL} of the paper.

\subsection{$L^p$ results}\label{ss:Lp}

Concerning operator-weighted $L^p$ spaces, we are going to show that the operator-valued $A_p$ condition remains
\begin{enumerate}[\rm(a)]
  \item necessary and sufficient for the weighted $L^p$ boundedness of basic averaging operators (Corollary \ref{cor:ave}; previously proved for $p=2$ and Hilbert targets by Aleman--Constantin \cite[Lemma 2.2]{AC:12});
  \item\label{it:sparse} sufficient for the weighted $L^p$ boundedness of certain sparse operators (Corollary \ref{cor:sparse}; previously proved for $p=2$ and Hilbert targets by Limani--Pott \cite[Theorem 2.1]{LP:22}); but
  \item\label{it:Hunbd} {\em insufficient} for the weighted $L^p$ boundedness of the Hilbert transform $\HT$ (Theorem \ref{thm:Hunbd}; previously proved for $p=2$ and Hilbert targets by Gillespie--Pott--Treil--Volberg \cite[Theorem 1.3]{GPTV:00}, with a sharper version in \cite[Theorem 1.1]{GPTV:04});
  \item\label{it:commuUnbd} as a tool for \eqref{it:Hunbd} but with independent interest, we show that certain operator-valued bounded mean oscillation (BMO) conditions are similarly insufficient for the boundedness of the commutator $[B,\HT]$ on the {\em unweighted} $L^p(\R;\fX)$ (Theorem \ref{thm:commuUnbd}; previously proved for $p=2$ and Hilbert targets by Nazarov--Pisier--Treil--Volberg \cite[Theorem 1.2]{NPTV:02} and Gillespie--Pott--Treil--Volberg \cite[Theorem 1.2]{GPTV:00}, \cite[Theorem 1.3]{GPTV:04}; see also \cite{GPTV:01} for related results on martingale transforms.).
\end{enumerate}

Related to point \eqref{it:commuUnbd}, we should also mention the classical works of Sarason \cite{Sarason:67} and Page \cite{Page:70} on the factorisation of Hardy spaces of operator-valued functions and the boundedness of Hankel operators with operator-valued symbols, both of which are known to be essentially equivalent to the boundedness of the commutators $[B,\HT]$ with operator-valued $B$. In contrast to the negative results of point \eqref{it:commuUnbd},  \cite{Sarason:67} and \cite{Page:70} achieve characterisations in infinite-dimensional Hilbert spaces, but these are expressed in terms of quotient norms whose relation to the real-variable theory that we develop remains unclear.

The {\em sparse operators} appearing in \eqref{it:sparse} have become instrumental in the scalar-weighted theory since the work of Lerner \cite{Lerner:simple}, and also in the matrix-weighted theory since the work of Nazarov--Petermichl--Treil--Volberg \cite{NPTV:17}, thanks to the fact they dominate classical operators in a way that is compatible with both scalar and matrix weighted estimates. In contrast to this, the positive result about the operator-weighted boundedness of sparse operators in \eqref{it:sparse}, together with the negative result for the Hilbert transform in \eqref{it:Hunbd}, shows that Calder\'on-Zygmund operators are no longer dominated by sparse operators in the infinite-dimensional setting, making sparse operators less relevant objects in this theory than in finite dimensions. For $p=2$ and Hilbert targets, this was already pointed out by Limani--Pott \cite{LP:22}.

In some sense, points \eqref{it:Hunbd} and \eqref{it:commuUnbd} are perhaps expected; after all, if things go wrong even in a Hilbert space, as already shown in \cite{GPTV:00,GPTV:04,NPTV:02}, how could they possibly behave better in more complicated situations. However, the ``good'' properties of a Hilbert space also ensure the existence of a plenty of operators (such as bounded projections onto any closed subspace, certainly not available in a general Banach space), therefore also a plenty of operator-valued $A_p$ weights (resp.\ BMO functions), making a condition over all $A_p$ weights (resp.\ BMO functions) more difficult to be satisfied than in a setting with few $A_p$ weights (resp.\ BMO functions) only. In another extreme, there are Banach spaces with ``very few operators'', in the sense that every bounded linear operator has the form $\lambda\cdot I+K$ of a scalar-multiple of the identity plus a compact operator $K$; the first such example is a celebrated construction of Argyros and Haydon \cite{AH:11}. In a space with this scalar-plus-compact property, one might imagine that operator $A_p$ weights (resp.\ BMO functions) are ``not far'' from matrix $A_p$ weights (resp.\ BMO functions), making it plausible that the positive results of the finite-dimensional theory might still be pushed to hold. This suggests that achieving the negative results in \eqref{it:Hunbd} and \eqref{it:commuUnbd} for any given Banach space $\fX$ requires some effort.

Our constructions of the counterexamples in both \eqref{it:Hunbd} and \eqref{it:commuUnbd} split into two cases using different properties of the Banach space under consideration. The first case is when $\fX$ is {\em not} a UMD space, in which case the Hilbert transform is unbounded on the unweighted $L^p(\R;\fX)$. The other case is when $\fX$ is $K$-convex, in which case the counterexamples are obtained by carefully embedding a sequence of the finite dimensional Hilbertian counterexamples of \cite{NPTV:02} into $\fX$ in a quantitatively controlled manner, using an elaboration of Dvoretzky's theorem in $K$-convex spaces due to Figiel and Tomczak-Jaegermann \cite{FTJ:79} as a key tool. Since every UMD space is $K$-convex, but not conversely, these two cases cover all Banach spaces, even with some overlap.

\subsection{Besov and Triebel--Lizorkin spaces}\label{ss:BTL}

In contrast to the predominantly ``negative'' results about the spaces $L^p(V)$, we will show that Besov and Triebel--Lizorkin spaces $\dot{A}^s_{p,q}$
with operator $A_p$ weights can be meaningfully defined so that they enjoy various good properties.
To be precise, we introduce pointwise operator-weighted space $\dot{A}^s_{p,q}(V)$
and average operator weighted space $\dot{A}^s_{p,q}(\rho)$
(see Definition \ref{def:space}) and obtain the following:
\begin{enumerate}[\rm(a)]
  \item the equivalence
  between $\dot{A}^s_{p,q}(V)$ and $\dot{A}^s_{p,q}(\rho)$
  for Besov spaces in the full parameter range
  (Theorem \ref{thm:equiB}; extending some results in \cite{Rou:03} and \cite{FR:04})
  and for Triebel--Lizorkin spaces under some upper bound on $q$
  (Theorem \ref{thm:equiB} and Lemma \ref{lem:dualNorm},
  extending \cite[Theorem 3.5]{FR:21} and Proposition \ref{prop:FV<Frho,q<p}, partially extending \cite[Corollary 3.9]{FR:21});
  surprisingly, we find that in infinite-dimensional case,
  the aforementioned equivalence for Triebel--Lizorkin spaces
  does not hold when $q=\infty$ (see Proposition \ref{prop:f}).
  Although the optimal range ensuring this equivalence is still unclear,
  we achieve a satisfactory theory of both the
  average-weighted space $\dot A^s_{p,q}(\rho)$ in the full parameter range
  and the pointwise-weighted space $\dot A^s_{p,q}(V)$ under the about upper bound on $q$, covering:
  \item the equivalence of the operator-weighted space with the
  corresponding discrete sequence space $\dot a^s_{p,q}$ via the $\varphi$ transform
  (Theorems \ref{thm:phi} and \ref{thm:phiV}; extending matrix-weighted analogues \cite[Theorem 1.4]{Rou:03} and \cite[Corollary 3.2 and Theorem 3.4]{FR:04} in the Besov case with $p\in(1,\infty)$ resp.\ $p\in(0,1]$ and \cite[Theorem 2.3]{FR:21} in the Triebel--Lizorkin case);
  \item\label{it:operBd} the boundedness of various operators, including the Hilbert transform (Corollary \ref{cor:Hilbert}) and almost diagonal operators (Theorems
  \ref{prop:AD} and \ref{bdd-AD-ple1};
  extending a matrix-weighted version in \cite[Theorem 2.6]{FR:21}); in particular,
  for the pointwise operator-weighted spaces, we show that our
  assumptions on almost diagonal operators are optimal in some sense
  (see Remark \ref{rem:ADV});
  \item\label{it:sandwich} a ``sandwich'' embedding
  $\dot F^0_{p,1}(\rho)\subset L^p(V)\subset F^0_{p,\infty}(\rho)$
  of $L^p(V)$ between two Triebel--Lizorkin spaces with the same smoothness and integrability (Propositions \ref{prop:Lp<F} and \ref{prop:F<Lp}; this is a substitute of the equality $\dot F^0_{p,2}(V)=L^p(V)$ for matrix $A_p$ weights \cite[Section 4]{FR:21}, and an extension of a similar sandwich relation in the unweighted but Banach space valued case \cite[Proposition 14.6.13]{HNVW3}).
\end{enumerate}

Once we have laid down these foundations of the theory of operator-weighted Besov and Triebel--Lizorkin spaces in Part \ref{part:BTL}, it turns out that extensions of many more results from the matrix-weighted versions of these spaces can be developed in an axiomatic framework, taking only a few of the key properties established in Part \ref{part:BTL} as postulates of an abstract class of {\em discretizable spaces} that we study in Part \ref{part:discretizable}. The results obtained in this abstract setting include molecular and wavelet characterizations of these spaces and their applications to the $T(1)$ theorem for Calder\'on--Zygmund operators.
We also apply these results to the boundedness of trace and extension operators in operator-weighted Besov and Triebel--Lizorkin spaces. The results of Part \ref{part:discretizable} extend several earlier results in the matrix-weighted setting from \cite{FR:04,FR:08,FR:21,Rou:03}.

In the classical theory, Triebel--Lizorkin spaces arise as an {\em extension} of the scale of $L^p(\R^n)=\dot F^0_{p,2}(\R^n)$ for $p\in(1,\infty)$, and this relation persists even for matrix-$A_p$-weighted versions, as shown in \cite[Section 4]{FR:21}. In view of \eqref{it:operBd} and \eqref{it:sandwich} above, while recalling \eqref{it:Hunbd} of Section \ref{ss:Lp}, the operator-weighted Triebel--Lizorkin spaces should rather be seen as a {\em substitute} for $L^p(V)$, being relatively close to $L^p(V)$ according to \eqref{it:sandwich}, yet sufficiently different so that good properties that fail in $L^p(V)$ remain true in $\dot F^s_{p,q}(\rho)$. This phenomenon is already present in the unweighted but Banach space -valued versions of these spaces (cf.\ \cite[p.~294]{HNVW3}); and we extend it to the general operator $A_p$ weighted setting.

Matrix-weighted Besov spaces were extensively studied in \cite{FR:04,FR:08,Rou:03,Rou:04} and a systematic study of matrix-weighted Triebel--Lizorkin spaces was started in \cite{FR:21}. This has been followed by numerous extensions by many authors \cite{BCHYY,BHYY:I,BHYY:II,BHYY:III,BHYY:SCM,byyz26,BX:study,BX:Iran,LYY:24,WYZ:24}.
However, all these works deal with finite-dimensional spaces, where operators are just finite matrices.

Despite the various ``no go'' theorems around the Hilbert transform in operator-weighted $L^p$ spaces mentioned above, an important indication for us, that positive results could still be achieved in other spaces, was a work of Aleman and Constantin \cite{AC:12}, where they characterised the boundedness of the Bergman projection on operator-weighted $L^2$ spaces in terms of an operator-valued B\'ekoll\'e--Bonami $B_2$ condition---a well-known relative of the Muckenhoupt $A_2$ condition in the Bergman context. While their precise setting is different, it also has important similarities with the theory of Besov and Triebel--Lizorkin spaces. Notably, the arguments of \cite{AC:12} make critical use of the mean value property of analytic functions, which bears some similarity with versions of the Calder\'on reproducing formula commonly applied in the context of Besov and Triebel--Lizorkin spaces; indeed, the components of the Littlewood--Paley decomposition are analytic functions! Building on this analogy, we are able to extend a lot of the recent matrix-weighted theory of Besov and Triebel--Lizorkin spaces to the operator-weighted context, achieving the types of results listed above.

While the settings are not directly comparable, our results go arguably beyond mere analogues of those of \cite{AC:12} in the following sense: In \cite{AC:12}, only $L^2$ spaces of Hilbert space-valued functions and the corresponding $B_2$ weights are considered, while we treat the full scale of Besov and Triebel--Lizorkin spaces $\dot A^s_{p,q}$ of Banach space-valued distributions and the corresponding Muckenhoupt classes $A_p$ for all $p\in(0,\infty)$. (Recall that, already in the matrix-weighted setting, the single scalar-valued $A_1$ condition splits into a scale of weight classes $A_p$ with $p\in(0,1]$, first introduced by Frazier and Roudenko \cite{FR:04}.) This, in turn, leads to the following qualitative difference: In the Hilbertian $L^2$ case, the reducing operator can be computed explicitly by
\begin{equation}\label{eq:redopL2}
\begin{split}
  \fint_Q\Norm{W^{\frac 12}(x)\fx}{}^2\, dx
  &=\fint_Q(W^{\frac12}(x)\fx,W^{\frac12}(x)\fx)\, dx \\
  &=\fint_Q(W(x)\fx,\fx) \,dx =(\ave{W}_Q\fx,\fx)
  =\Norm{\ave{W}_Q^{\frac12}\fx}{}^2,
\end{split}
\end{equation}
and these identities remain perfectly valid in infinite-dimensional Hilbert spaces. Thus, even in the lack of John's ellipsoid theorem, the methods around reducing operators, familiar from the matrix-weighted theory, are still available for operator weights in $L^2$ with Hilbert targets, but not if {\em either} $L^2$ is replaced by $L^p$ or the Hilbert target by a Banach target. In our investigation, we make {\em both} these replacements, so that \eqref{eq:redopL2} completely fails, and only the expression on the left, or its $L^p$ analogue, remains for us to work with. Thus, while \cite{AC:12} made the important opening move in positive results for the boundedness of singular operators on operator-weighted spaces under Muckenhoupt-style assumptions without side conditions, our investigation takes it a step further still from the established framework of the matrix-weighted theory.

\part{Basic properties of operator weights}\label{part:Ap}

In this part, we develop the relevant background for operator-valued weights to be used in the study of operator-weighted $L^p$ spaces in Part \ref{part:Lp} and operator-weighted Besov and Triebel--Lizorkin spaces in Part \ref{part:BTL}.

\section{$L^p$ spaces of operator-valued functions}

Let $\Omega$
be a measure space and $\fX$ a Banach space.
A function $f:\Omega\to\fX$ is said to be {\em simple}
if $f=\sum_{i=1}^n \fx_i\one_{A_i}$ with $\fx_i\in\fX$ and each 
$A_i\subseteq\Omega$ is measurable, and {\em strongly measurable} if it is the pointwise limit of simple functions. Then, in particular, $\Norm{f(\cdot)}{\fX}$ is a pointwise limit of nonnegative simple functions, and hence measurable as a scalar-valued function.

 For all $p\in[0,\infty]$, we define
\begin{equation*}
  L^p(\fX):=
  L^p(\Omega;\fX):=
  \begin{cases}\{f:\Omega\to\fX\text{ strongly measurable}\}, & p=0, \\
    \{f\in L^0(\Omega;\fX);\ \ \Norm{f}{L^p(\Omega;\fX)}<\infty\}, & p\in(0,\infty], \end{cases}
\end{equation*}
where we may use the shorter notation $L^p(\fX)$ if $\Omega$ is understood from the context and where
\begin{equation*}
  \Norm{f}{L^p(\Omega;\fX)}:=\begin{cases}\displaystyle
  \Big(\int_\Omega\Norm{f(\omega)}{\fX}^p d\omega\Big)^{\frac1p}, & p\in(0,\infty), \\
  \esssup_{\omega\in\Omega}\Norm{f(\omega)}{\fX}, & p=\infty\end{cases}
\end{equation*}
is a norm for $p\in[1,\infty]$, a $p$-norm ($\Norm{\fy+\fz}{}^p\leq\Norm{\fy}{}^p+\Norm{\fz}{}^p$) for $p\in(0,1]$, and hence a quasi-norm ($\Norm{\fy+\fz}{}\leq c(\Norm{\fy}{}+\Norm{\fz}{})$) for all $p\in(0,\infty]$. The space $L^p(\Omega;\fX)$ are complete for all $p\in(0,\infty]$.
For $p\in[1,\infty]$, a detailed treatment of these spaces can be found in \cite[Chapter 1]{HNVW1}.

Let $\bddlin(\fX)$ be the space of bounded linear operators on $\fX$. Since $\bddlin(\fX)$ is another Banach space, the previous notion could be applied to $F:\Omega\to\bddlin(\fX)$, but this would often result in an unnecessarily stringent requirement. The following notion is more flexible: A function $F:\Omega\to\bddlin(\fX)$ is said to be {\em strongly measurable in the strong-operator sense} if $F(\cdot)\fx:\Omega\to\fX$ is strongly measurable for each $\fx\in\fX$. Then
\begin{equation*}
  L^p_{\so}(\Omega;\bddlin(\fX)):=\Big\{F:\Omega\to\bddlin(\fX); F(\cdot)\fx\in L^p(\Omega;\fX)\ \forall\ \fx\in\fX\Big\}.
\end{equation*}

If $f:\Omega\to\fX$ and $F:\Omega\to\bddlin(\fX)$, one can define the pointwise product $Ff:\Omega\to\fX$ in the obvious way by $(Ff)(\omega):=F(\omega)f(\omega)$. The above notions of measurability for $\fX$-valued and $\bddlin(\fX)$-valued functions are compatible in the following sense:

\begin{lemma}\label{lem:prod}
If $f\in L^0(\Omega;\fX)$ and $F\in L^0_{\operatorname{so}}(\Omega;\bddlin(\fX))$, then $Ff\in L^0(\Omega;\fX)$.
\end{lemma}

\begin{proof}
If $f=\one_E \fx$ for some measurable set $E$ and vector $\fx\in\fX$, then $Ff=\one_E F\fx\in L^0(\Omega;\fX)$ by the defining condition of $V\in L^0_{\operatorname{so}}(\Omega;\bddlin(\fX))$. The case of simple $f$ follows by linearity. For general $f\in L^0(\Omega;\fX)$, the defining condition of strong measurability implies that $f$ is the pointwise limit $f=\lim_{k\to\infty}f_k$ of some simple $f_k$. Then $Ff_k\in L^0(\Omega;\fX)$ for each $k$, and $Ff=\lim_{k\to\infty}Ff_k\in L^0(\Omega;\fX)$ as the pointwise limit of functions in this space.
\end{proof}

The following lemma shows that $L^p_{\operatorname{so}}(\Omega;\bddlin(\fX))$ also comes with a natural quasi-norm; however, it is not complete in general (see \cite[bottom of p.~64]{HNVW1}).

\begin{lemma}\label{lem:LpVso}
Let $(\fX,\Norm{\ }{})$ be a Banach space, $\Omega$ be a measure space, $p\in(0,\infty]$, and $F\in L^p_{\so}(\Omega;\bddlin(\fX))$.
Then $\fx\mapsto F(\cdot)\fx$ is continuous from $\fX$ to $L^p(\Omega;\fX)$, and hence
\begin{equation*}
  \Norm{F}{L^p_{\so}(\Omega;\bddlin(\fX))}:=\sup_{\Norm{\fx}{\fX}\leq 1}\Norm{F(\cdot)\fx}{L^p(\Omega;\fX)}
\end{equation*}
is finite.
\end{lemma}

\begin{proof}
We note that $L^p(\Omega;\fX)$ is a complete quasi-normed space; in particular, it is an F-space in the sense of \cite[1.8(e)]{Rudin:FA}, i.e., a topological vector space whose topology is induced by a complete metric that is invariant under vector addition: $d(\fx,\fy)=d(\fx+\fz,\fy+\fz)$ for all $\fx,\fy,\fz\in\fX$.

The map $\fx\mapsto F(\cdot)\fx$ from $\fX$ to $L^p(\Omega;\fX)$ is obviously well defined and linear. We check that it is closed. For this, suppose that $\fx_n\to\fx$ in $\fX$ and $F(\cdot)\fx_n\to f$ in $L^p(\Omega;\fX)$. Passing to a subsequence, it follows that $F(\cdot)\fx_n\to f$ almost everywhere, but also $F(\cdot)\fx_n\to F(\cdot)\fx$, since each $F(\omega)$ is continuous on $\fX$. By the uniqueness of the limit in $\fX$, it follows that $f=F(\cdot)\fx$ almost everywhere, and hence $\fx\mapsto F(\cdot)\fx$ is closed, as claimed. Now the Closed Graph Theorem for F-spaces, \cite[2.15]{Rudin:FA}, implies that the said map is continuous.

Suppose for contradiction that $\Norm{F}{L^p_{\so}(\Omega;\bddlin(\fX))}=\infty$. Then we can find some $\fx_n\in\fX$ of norm one such that $\Norm{F(\cdot)\fx_n}{L^p(\Omega;\fX)}>n$. But then $\frac1n\fx_n\to 0$ in $\fX$, while $\Norm{F(\cdot)\frac1n\fx_n}{L^p(\Omega;\fX)}>1$, which contradicts the fact that these norms should converge to $0$ by the proven continuity.
\end{proof}

\begin{lemma}\label{lem:equifx}
Under the assumptions of Lemma \ref{lem:LpVso}, suppose in addition that $F(\omega)$ is invertible at almost every $\omega\in\Omega$. Then the quasi-norm
\begin{equation*}
  \rho_{L^p(\Omega,F)}(\fx):=\Norm{F(\cdot)\fx}{L^p(\Omega;\fX)}
\end{equation*}
is qualitatively equivalent to $\Norm{\fx}{}$, i.e., the positive equivalence constants are allowed to depend on $p,\Omega,F$ in an unspecified way.
\end{lemma}

\begin{proof}
We already know from Lemma \ref{lem:LpVso} that $\rho_{L^p(\Omega,F)}(\fx)\lesssim_{p,\Omega,F}\Norm{\fx}{}$, so it remains to prove the converse.

We claim that the image $F\fX:=\{F(\cdot)\fx\in L^p(\Omega;\fX):\fx\in\fX\}$ is a closed subspace of $L^p(\Omega;\fX)$. Indeed, suppose that $F(\cdot)\fx_n\to f$ in $L^p(\Omega;\fX)$. Then a subsequence converges in $\fX$ at almost every $\omega\in\Omega$. Removing another null set and applying $F(\omega)^{-1}$ on both sides, we conclude that $\fx_n\to F(\omega)^{-1}f(\omega)$ at almost every $\omega\in\Omega$. Since $\fx_n$ is independent of $\omega$, this shows that $F(\omega)^{-1}f(\omega)\equiv\fx$ at almost every $\omega\in\Omega$. Hence $f=F(\cdot)\fx\in F\fX$, which proves the closedness of $F\fX$.

As a closed subspace of the complete space $L^p(\Omega;\fX)$, the space $F\fX$ itself is complete. The map $\fx\mapsto F(\cdot)\fx$ is obviously surjective from $\fX$ to $F\fX$. By the pointwise invertibility of $F$, it is also injective. Now the Open Mapping Theorem for F-spaces, \cite[2.12]{Rudin:FA}, shows that $F^{-1}:F\fX\to \fX:F(\cdot)\fx\mapsto\fx$ is also continuous, which means that
\begin{equation*}
  \Norm{\fx}{}\lesssim_{p,\Omega,F} \rho_{L^p(\Omega,F)}(\fx).\qedhere
\end{equation*}
\end{proof}

For the sake of easy reference, we record the following well-known fact about the inverse and the adjoint:

\begin{remark}
Let $\fX,\fY$ be Banach spaces and $A\in\bddlin(\fX,\fY)$ have inverse $A^{-1}\in\bddlin(\fY,\fX)$. Then one easily checks that the adjoint of the inverse, $(A^{-1})^*\in\bddlin(\fX^*,\fY^*)$, is the inverse of the adjoint $A^*\in\bddlin(\fY^*,\fX^*)$. Thus, without ambiguity, we can simplify notation to write simply
\begin{equation*}
  A^{-*}:=(A^{-1})^*=(A^*)^{-1}\in\bddlin(\fX^*,\fY^*).
\end{equation*}
\end{remark}

\section{Operator-valued weights}

Let us start with a brief background for the setting that we are going to adopt.

A scalar-valued weight $0<w\in L^1_{\loc}(\R^n)$ may be equivalently seen as a modifier of the underlying measure, or of the function in the weighted $L^p$ space, via the trivial correspondence
\begin{equation}\label{eq:wv}
  \Big(\int_{\R^n}\abs{f(x)}^p w(x)\,dx\Big)^{\frac1p}
  =\Big(\int_{\R^n}\abs{v(x)f(x)}^p \, dx\Big)^{\frac1p},
\end{equation}
where
\begin{equation*}
    v:=w^{\frac1p}\in L^p_{\loc}(\R^n).
\end{equation*}

If $f:\R^n\to\C^m$ and we would like to weigh its value by a matrix-valued weight, only the right-hand side of \eqref{eq:wv} gives a meaningful definition. Due to the tradition of viewing $w$ in \eqref{eq:wv} as ``the weight'', much of the matrix-weighted theory has been developed under the assumption that $W\in L^1_{\loc}(\R^n;\bddlin(\C^m))$ is an almost-everywhere positive-definite-valued function, and by writing formulas in terms of the pointwise $p$-th root $V(x):=W(x)^{\frac1p}$, which is defined via the functional calculus of self-adjoint operators. This still makes perfect sense for operator weights on an infinite-dimensional Hilbert space as in \cite{AC:12,Pott:07}. However, already in the matrix-weighted theory, one can see that it is {\em de facto} $V$, and not $W$, that is the primary object of study, and the repeated appearance of $p$-th roots in the formulas is mostly a notational nuisance only caused by the insistence of writing expression in terms of $W$ instead of the more natural $V$.
This point becomes even more prominent when we pass to a general Banach space $\fX$.

Let us then motivate a notion of $\bddlin(\fX)$-valued weights $V$ and weighted spaces $L^p(V)$ of $\fX$-valued functions, modelled after the right-hand side of \eqref{eq:wv}. It seems reasonable to require that
\begin{enumerate}
  \item\label{it:LpVsubL0} $L^p(V)\subseteq L^0(\fX)$;
  \item\label{it:LpVsupSimple} $L^p(V)\supseteq\{f\in L^0(\fX)\text{ simple and boundedly supported}\}$;
  \item\label{it:LpVtest} given $f\in L^0(\fX)$, we have $f\in L^p(V)$ if and only if $Vf\in L^p(\fX)$;
  \item\label{it:Lp2LpV} given $g\in L^p(\fX)$, we have $V^{-1}g\in L^p(V)$.
\end{enumerate}
Let $Q$ be a cube and $\fx\in\fX$. Since $\one_Q\fx$ is simple and boundedly supported, it follows that
$\one_Q\fx\in L^p(V)$ by \eqref{it:LpVsupSimple}, hence $\one_Q V\fx \subseteq L^p(\fX)$ by \eqref{it:LpVtest}.
Being true for all cubes $Q$ and $\fx\in\fX$, this shows that $V\in L^p_{\operatorname{so},\loc}(\bddlin(\fX))$.

On the other hand, we also have $\one_Q\fx\in L^p(\fX)$.
Hence $\one_QV^{-1}\fx\in L^p(V)$ by \eqref{it:Lp2LpV}, thus $\one_QV^{-1}\fx\in L^0(\fX)$ by \eqref{it:LpVsubL0}.
Being true for all cubes $Q$ and $\fx\in\fX$, this shows that $V^{-1}\in L^0_{\operatorname{so}}(\bddlin(\fX))$.

Based on this, we give:

\begin{definition}\label{def:weight}
Let $\fX$ be a Banach space and $V:\R^n\to\bddlin(\fX)$.
Then $V$ is called a {\em weight} if
\begin{enumerate}[\rm(i)]
  \item\label{it:VL0} $V\in L^0_{\operatorname{so}}(\R^n;\bddlin(\fX))$,
  \item\label{it:Vinv} $V(x)$ is invertible at a.e.\ $x\in\R^n$, and
  \item\label{it:invL0} $V^{-1}\in L^0_{\operatorname{so}}(\R^n;\bddlin(\fX))$.
\end{enumerate}
\end{definition}

The following result clarifies the role of condition \eqref{it:invL0} in Definition \ref{def:weight}:

\begin{proposition}
Let $\fX$ be a Banach space and $V:\R^n\to\bddlin(\fX)$ be a function that satisfies \eqref{it:VL0} and \eqref{it:Vinv} of Definition \ref{def:weight}.
\begin{enumerate}[\rm(i)]
  \item\label{it:sep} If $\fX$ is separable, then $V$ also satisfies \eqref{it:invL0} of Definition \ref{def:weight}.
  \item\label{it:nonsep} Without separability, the previous conclusion fails, in general.
\end{enumerate}
\end{proposition}

\begin{proof}
\eqref{it:sep}: For a separable Banach space $\fX$, we recall (see \cite[Corollary 1.1.10]{HNVW1}) that a function $f:\R^n\to\fX$ is {\em strongly measurable} (a pointwise limit of simple functions) if and only if it is {\em measurable} (i.e., the inverse image $f^{-1}(B)$ is a measurable subset of $\R^n$ for every $B$ in $\operatorname{Bor}(\fX)$, the \emph{Borel $\sigma$-algebra} generated by the open subsets of $\fX$). The argument that follows is more conveniently expressed via the latter notion of inverse images.

Let $\fX$ and $V$ be as in the assumptions. We consider the map
\begin{equation*}
  \Phi:\R^n\times\fX\to\fX,(x,\fx)\mapsto V(x)\fx.
\end{equation*}
This is a so-called {\em Carath\'eodory function} \cite[Definition 14.74]{AB:book}:
\begin{enumerate}[\rm(i)]
  \item $\Phi(\cdot,\fx)=V(\cdot)\fx:\R^n\to\fX$ is measurable for every $\fx\in\fX$. Indeed, this is the assumption that $V\in L^0_{\operatorname{so}}(\R^n;\fX)$, combined with the above remark about measurability and strong measurability.
  \item $\Phi(x,\cdot)=V(x):\fX\to\fX$ is continuous. This is clear from $V(x)\in\bddlin(\fX)$.
\end{enumerate}
By \cite[Lemma 14.75]{AB:book}, Carath\'eodory functions are {\em jointly measurable}, i.e., for every closed $F\subseteq\fX$, the pre-image $\Phi^{-1}(F)$ satisfies
\begin{equation*}
  \Phi^{-1}(F):=\{(x,\fx)\in\R^n\times\fX: \Phi(x,\fx)\in F\}
  \in \operatorname{Leb}(\R^n)\times\operatorname{Bor}(\fX),
\end{equation*}
where the right-hand side is the product $\sigma$-algebra. In particular, we can take $F=\{\fy\}$ with $\fy\in\fX$. Then also
\begin{equation}\label{eq:setS}
   S:=\{(x,\fx)\in\R^n\times B: \Phi(x,\fx)\in \{\fy\}\}
   =\Phi^{-1}(\{\fy\})\cap(\R^n\times B)
\end{equation}
belongs to $\operatorname{Leb}(\R^n)\times\operatorname{Bor}(\fX)$ for every $\fy\in\fX$ and $B\in\operatorname{Bor}(\fX)$.

Since $\fX$ is (in particular) a complete separable metric space, \cite[Theorem 12.3.4]{KT:book} guarantees that, for every $S\in \operatorname{Leb}(\R^n)\times\operatorname{Bor}(\fX)$, its $\R^n$-projection
\begin{equation*}
   \pi(S):=\{x\in\R^n:\exists\fx\in\fX,(x,\fx)\in S\}
\end{equation*}
belongs to the completion (called the \emph{Lebesgue extension} in \cite{KT:book}) of the $\sigma$-algebra $\operatorname{Leb}(\R^n)$. Since $\operatorname{Leb}(\R^n)$ is already complete, it follows that
\begin{equation*}
  \pi(S)\in \operatorname{Leb}(\R^n).
\end{equation*}
For the concrete $S$ in \eqref{eq:setS}, we observe that
\begin{equation*}
  \Phi(x,\fx)\in\{\fy\}\quad\Leftrightarrow \quad
  V(x)\fx=\fy  \quad\Leftrightarrow \quad
  \fx=V(x)^{-1}\fy,
\end{equation*}
hence
\begin{equation*}
  \exists\fx\in B:\Phi(x,\fx)\in\{\fy\}\quad\Leftrightarrow \quad
  \exists\fx\in B:V(x)^{-1}\fy=\fx\quad\Leftrightarrow \quad
  V(x)^{-1}\fy\in B,
\end{equation*}
and thus
\begin{equation*}
\begin{split}
  \pi(S) &=\{x\in\R^n: \exists\fx\in B,\Phi(x,\fx)\in\{\fy\}\} \\
   &=\{x\in\R^n: V(x)^{-1}\fy\in B\} = (V(\cdot)^{-1}\fy)^{-1}(B)
\end{split}
\end{equation*}
is the inverse image of $B\in\operatorname{Bor}(\fX)$ under $V(\cdot)^{-1}\fy:\R^n\to\fX$.
Since these sets are measurable for all $B\in\operatorname{Bor}(\fX)$, it means that the function $V(\cdot)^{-1}\fy:\R^n\to\fX$ is measurable and hence strongly measurable. Valid for every $\fy\in\fX$, it means that $V(\cdot)^{-1}\in L^0_{\operatorname{so}}(\R^n;\fX)$, which proves assertion \eqref{it:sep}.

\eqref{it:nonsep}:
Consider the non-separable Hilbert space $\fX=\ell^2([0,\infty))$ with the uncountable complete orthonormal system $(\fy_t)_{t\in[0,\infty)}$. For each $t\in(0,1)$, we define $V(t)\in\bddlin(\fX)$ via its action on this orthonormal system as follows:
\begin{equation*}
  V(t):\begin{cases} \fy_{t+n+1}\mapsto \fy_{t+n}, & n\in\N:=\{0,1,2,\ldots\}, \\
    \fy_t\mapsto \fy_0, \\
    \fy_n\mapsto \fy_{n+1}, & n\in\N
  \end{cases}
\end{equation*}
and $V(t):\fy_s\mapsto \fy_s$ for all other $s\in[0,\infty)$. Since $V(t)$ is a permutation of an orthonormal system, it is clearly invertible. For a function on $\R$, we can simply take $V(t)=I$ for $t\in\R\setminus(0,1)$.

To check that $V\in L^0_{\operatorname{so}}(\R;\bddlin(\fX))$, let $\fx=\sum_{s\in[0,\infty)}x_s \fy_s\in\fX$. Note that at most countably many $x_s$ are non-zero. Then
\begin{equation*}
\begin{split}
  V(t)\fx &=\sum_{n=0}^\infty x_{t+n+1}\fy_{t+n}+x_t \fy_0+\sum_{n=0}^\infty x_n\fy_{n+1}+\sum_{s\in[0,\infty)\setminus(\N\cup(t+\N))}x_s\fy_s \\
  &=\sum_{n=0}^\infty (x_{t+n+1}-x_{t+n})\fy_{t+n}+x_t \fy_0+\sum_{n=0}^\infty x_n(\fy_{n+1}-\fy_n)  +\fx.
\end{split}
\end{equation*}
The key observation is that, for all $n\in\N$, we have $x_{t+n}=0$ for a.e.\ $t\in(0,1)$, by the fact that at most countably many $x_s$ are non-zero. Hence, indeed, several terms in the expansion of $V(t)\fx$ vanish almost everywhere, and we arrive at the simpler formula
\begin{equation*}
  V(t)\fx = \sum_{n=0}^\infty x_n(\fy_{n+1}-\fy_n)  +\fx\quad \text{for a.e.}\quad t\in(0,1).
\end{equation*}
Thus, being a constant almost everywhere, $V(t)\fx$ is clearly measurable.

On the other hand, the inverse $V(t)^{-1}$ satisfies $V(t)^{-1}\fy_0=\fy_t$. Thus the range of $V(\cdot)^{-1}\fy_0$ contains the uncountable orthonormal system $(\fy_t)_{t\in(0,1)}$, which is clearly non-separable. Every strongly measurable function has a separable range, so $V(\cdot)^{-1}\fy_0\notin L^0(\fX)$, and hence $V(\cdot)^{-1}\notin L^0_{\operatorname{so}}(\bddlin(\fX))$.
\end{proof}

\begin{definition}
A weight is called a {\em $p$-weight}, where $p\in(0,\infty]$, if
\begin{equation*}
  V\in L^p_{\loc,\so}(\R^n;\bddlin(\fX))
  :=
  \left.\begin{cases} F\in L^0_{\operatorname{so}}(\R^n;\bddlin(\fX)); &\one_Q F\in L^p_{\so}(\R^n;\bddlin(\fX)) \\
  &\text{ for all cubes }Q\subset\R^n\end{cases}\right\}
\end{equation*}
\end{definition}

\begin{definition}\label{def:LpV}
Let $\fX$ be a Banach space, $p\in(0,\infty]$, and $V:\R^n\to\bddlin(\fX)$ be a $p$-weight. Then
\begin{equation}\label{eq:LpV}
  L^p(V):=\{f\in L^0(\R^n;\fX): Vf\in L^p(\R^n;\fX)\}
\end{equation}
with norm
\begin{equation*}
  \Norm{f}{L^p(V)}:=
  \Norm{f}{L^p(V;\fX)}:=\Norm{Vf}{L^p(\R^n;\fX)}.
\end{equation*}
\end{definition}

\begin{remark}
Lemma \ref{lem:prod} provides the following elaboration of the meaning of the defining condition \eqref{eq:LpV}: Given $f\in L^0(\R^n;\fX)$, we automatically know that $Vf\in L^0(\R^n;\fX)$. Hence, the second condition that $Vf\in L^p(\R^n;\fX)$ does not impose any additional measurability restrictions on $f$; it only amounts to checking a size condition on the measurable function $Vf$.
\end{remark}

\begin{remark}
Specialising back to a Hilbert space $\fX=H$, we note that
\begin{equation*}
  \Norm{A\fx}{H}^2=(A\fx,A\fx)=(A^*A\fx,\fx)=((A^*A)^{\frac12}\fx,(A^*A)^{\frac12}\fx)
  =\Norm{(A^*A)^{\frac12}\fx}{H}^2,
\end{equation*}
where $(A^*A)^{\frac12}$ is a non-negative operator. Hence the general weighted norm
\begin{equation*}
   \Norm{f}{L^p(V;H)}=\Norm{f}{L^p((V^*V)^{\frac12};H)}
\end{equation*}
always reduces to a norm with respect to a non-negative weight; moreover, the condition that $V$ is invertible almost everywhere implies that $(V^*V)^{\frac12}$ is (strictly) positive-definite almost everywhere (as typically assumed in the Hilbert space theory).
\end{remark}

The following result, collecting some basic facts about the $L^p(V)$ spaces just introduced, is relatively routine, but illustrates the role of the different conditions that we imposed.

\begin{proposition}
Let $\fX$ be a Banach space, $p\in(0,\infty]$, and $V:\R^n\to\bddlin(\fX)$ be a $p$-weight. Then
\begin{enumerate}[\rm(i)]
  \item\label{it:LpVisomLp} $L^p(V)$ is isometrically isomorphic to $L^p(\R^n;\fX)$, where the isomorphisms are given by pointwise multiplication by $V:L^p(V)\to L^p(\R^n;\fX)$ and its pointwise inverse $V^{-1}:L^p(\R^n;\fX)\to L^p(V)$.
  \item\label{it:LpVBanach} $L^p(V)$ is a Banach space if $p\geq 1$, or a quasi-Banach space if $p<1$.
  \item\label{it:LpVdense} $\{f\in L^0(\fX)\text{ simple, boundedly supported}\}\subseteq L^p(V)$; it is dense if $p<\infty$.
\end{enumerate}
\end{proposition}

\begin{proof}
Let $f\in L^p(V)$. It follows directly from Definition \ref{def:LpV} of $L^p(V)$ that $Vf\in L^p(\fX)$ and $\Norm{Vf}{L^p(\fX)}=\Norm{f}{L^p(V)}$. Hence $V:L^p(V)\to L^p(\fX)$ is an isometry.

Conversely, let $g\in L^p(\fX)\subseteq L^0(\fX)$. Since $V^{-1}\in L^0_{\operatorname{so}}(\bddlin(\fX))$ by Definition \ref{def:weight}\eqref{it:invL0} of a weight, it follows from Lemma \ref{lem:prod} that $f:=V^{-1}g\in L^0(\fX)$. This function satisfies $Vf=VV^{-1}g=g\in L^p(\fX)$ by our assumption on $g$, and hence $f\in L^p(V)$ by Definition \ref{def:LpV} of $L^p(V)$. Moreover, $\Norm{f}{L^p(V)}:=\Norm{Vf}{L^p(\fX)}=\Norm{g}{L^p(\fX)}$, which shows that $V^{-1}:L^p(\fX)\to L^p(V)$ is also an isometry.

It is evident that both $VV^{-1}$ and $V^{-1}V$ are identities, showing that the said isometries are isomorphisms between the spaces in question. This completes the proof of \eqref{it:LpVisomLp}. Then \eqref{it:LpVBanach} follows from the corresponding well-known properties of the unweighted $L^p(\fX)$ spaces via the established isomorphism.

\eqref{it:LpVdense}: Let $f=\one_E\fx$, where $E$ is measurable and boundedly supported and $\fx\in\fX$. Then $Vf=\one_EV\fx\in L^p(\fX)$ by the assumption that $V$ is a $p$-weight. The claimed containment ``$\subseteq$'' then follows by linearity.

For density, consider some $f\in L^p(V)\subseteq L^0(\fX)$. As a strongly measurable function, it has separable range, so indeed $f\in L^0(\fY)$ for some separable subspace $\fY\subseteq\fX$. Let $(\fy_k)_{k=1}^\infty$ be a dense sequence in $\fY$. Then each $V\fy_k\in L^0(\fX)$ also has a separable range, and hence all of these countably many functions take values in some common separable subspace $\fZ\subseteq\fX$. By choosing a larger $\fZ$ if necessary, we may assume without loss of generality that $\fY\subseteq\fZ\subseteq\fX$. For all $\fy\in\fY$, the function $V\fy$ is the pointwise limit of some $V\fy_k$, so they also take values in $\fZ$. By linearity, the same is true for $Vh$ for all simple $h\in L^0(\fY)$ and, by another limit,
\begin{equation}\label{eq:VhL0Z}
  Vh\in L^0(\fZ)\qquad\forall\ h\in L^0(\fY).
\end{equation}

Since $\fZ$ is separable, and we can choose a countable dense sequence $(\fz_k)_{k=1}^\infty$ in the unit sphere of $\fZ$. For all $A\in\bddlin(\fX)$, it then follows that
\begin{equation*}
  \Norm{A|_{\fZ}}{}
  :=\sup\{ \Norm{A\fz}{\fZ}:\fz\in\fZ,\Norm{\fz}{\fX}\leq 1\}
  =\sup_{k\geq 1}\Norm{A\fz_k}{\fZ}.
\end{equation*}
Thus, both
\begin{equation*}
  \Norm{V(\cdot)|_{\fZ}}{}=\sup_{k\geq 1}\Norm{V(\cdot)\fz_k}{\fX},\qquad
  \Norm{V(\cdot)^{-1}|_{\fZ}}{}=\sup_{k\geq 1}\Norm{V(\cdot)^{-1}\fz_k}{\fX}
\end{equation*}
are measurable real-valued functions. In particular, they are finite almost everywhere.
Hence the sets
\begin{equation*}
  E_\lambda:=\{x\in\R^n: \abs{x}\leq\lambda,\Norm{V(x)|_{\fZ}}{}\leq\lambda,\Norm{V(x)^{-1}|_{\fZ}}{}\leq\lambda\}
\end{equation*}
are measurable, and $\one_{E_\lambda}(x)\uparrow 1$ almost everywhere as $\lambda\to\infty$.

From dominated convergence, it follows that $\one_{E_\lambda}f\to f$ in $L^p(V)$; indeed, this is equivalent to $\one_{E_\lambda}Vf\to Vf$ in $L^p(\fX)$, which follows from the usual dominated convergence in $L^p(\fX)$. Given $\eps>0$, we choose $\lambda$ large enough so that $g:=\one_{E_\lambda} f$ satisfies
\begin{equation}\label{eq:approx1}
   \Norm{f-g}{L^p(V)}<\eps.
\end{equation}

For all functions of the form $h=\one_{E_\lambda}h\in L^0(\fY)$, we observe that
\begin{equation*}
  \Norm{h}{L^p(V)}=
  \Norm{Vh}{L^p(\fX)}
  \leq\lambda\Norm{h}{L^p(\fX)},
\end{equation*}
since $h(x)\in\fY\subseteq\fZ$ and $\Norm{V(x)|_{\fZ}}{}\leq\lambda$ for $x\in E_\lambda$, and that
\begin{equation*}
  \Norm{h}{L^p(\fX)}
  =\Norm{V^{-1}Vh}{L^p(\fX)}
  \leq\lambda\Norm{Vh}{L^p(\fX)}
  =\lambda\Norm{h}{L^p(V)},
\end{equation*}
since $(Vh)(x)\in\fZ$ for $h\in L^0(\fY)$ by \eqref{eq:VhL0Z} and since $\Norm{V(x)^{-1}|_{\fZ}}{}\leq\lambda$ for $x\in E_\lambda$.

Summarising the last two estimates, we have
\begin{equation}\label{eq:LpVsimLp}
  \lambda^{-1}\Norm{h}{L^p(V)}\leq\Norm{h}{L^p(\fX)}\leq \lambda\Norm{h}{L^p(V)}\qquad
  \forall\ h=\one_{E_\lambda}h\in L^0(\fY).
\end{equation}
Applying this estimate to $g:=\one_{E_\lambda}f\in L^p(V)\cap L^0(\fY)$ in place of $h$, we see that $g\in L^p(\fY)$. By the density of simple functions in this unweighted space we can find a simple $s\in L^p(\fY)$ with $\Norm{g-s}{L^p(\fY)}<\eps/\lambda$. Replacing $s$ by $\one_{E_\lambda}s$ is necessary, we can assume that $s=\one_{E_\lambda}s$.

Then we apply \eqref{eq:LpVsimLp} to $h=g-s$ to see that
\begin{equation}\label{eq:approx2}
  \Norm{g-s}{L^p(V)}\leq\lambda\Norm{g-s}{L^p}<\lambda\frac{\eps}{\lambda}=\eps.
\end{equation}
Combining \eqref{eq:approx1} and \eqref{eq:approx2}, for any given $f\in L^p(V)$ and $\eps>0$, we have found a simple, boundedly supported $s\in L^0(\fX)$ such that $\Norm{f-s}{L^p(V)}\lesssim \eps$, which proves the claimed density.
\end{proof}

\section{Operator-valued $A_p$ weights}

Having discussed operator-valued weights at large, we next investigate more specific assumptions to be imposed on these weights in analogy with the Muckenhoupt $A_p$ conditions that are well studied for scalar and matrix-valued weights. Our setting is modelled after that of \cite{NT:hunt,Volberg:97}, which is phrased in terms of certain integral (quasi-)norms induced by the weights. We start with some generalities concerning these objects:

\begin{definition}
Let $\Omega$ be a measure space, $\fX$ be a Banach space, $u\in(0,\infty]$, and $V\in L_{\so}^u(\Omega;\bddlin(\fX))$. Then, for $\fx\in\fX$, we define the \emph{quasi-norm}
\begin{equation*}
  \rho_{L^u(\Omega,V)}(\fx):=\Norm{V(\cdot)\fx}{L^u(\Omega;\fX)}.
\end{equation*}
\end{definition}

If, in addition, $V$ is invertible almost everywhere and
$$V^{-*}\in L_{\so}^v(\Omega;\bddlin(\fX^*)),$$ we can apply
the same definition to $\fx^*\in\fX^*$ to obtain a \emph{quasi-norm}
\begin{equation*}
  \rho_{L^v(\Omega,V^{-*})}(\fx^*):=\Norm{V(\cdot)^{-*}\fx^*}{L^v(\Omega;\fX^*)}.
\end{equation*}

\begin{lemma}\label{lem:uvHolder}
If $\Omega$ is a probability space, then for all $u,v\in(0,\infty]$, we have
\begin{equation}\label{eq:uvHolder}
  \abs{\pair{\fx^*}{\fx}}
   \leq\rho_{L^{u}(\Omega,V)}(\fx)\rho_{L^{v}(\Omega,V^{-*})}(\fx^*).
\end{equation}
\end{lemma}

\begin{proof}
Since $\rho_{L^{u}(\Omega,V)}(\fx)\leq\rho_{L^{\infty}(\Omega,V)}(\fx)$ by H\"older's inequality and similarly for the second term on the right, it suffices to consider $u,v\in(0,\infty)$.
At all points of invertibility of $V$, we have
\begin{equation*}
  \abs{\pair{\fx^*}{\fx}}
  =\abs{\pair{V(\omega)^{-*}\fx^*}{V(\omega)\fx}}.
\end{equation*}
Raising this to power $s\in(0,\infty)$, integrating over $\omega\in\Omega$ (minus a possible null set, where $V$ fails to be invertible, but this has no effect on the result), and applying H\"older's inequality with exponent $q\in(1,\infty)$, it follows that
\begin{equation}\label{eq:sqHolder2}
   \abs{\pair{\fx^*}{\fx}}
   \leq\rho_{L^{sq}(\Omega,V)}(\fx)\rho_{L^{sq'}(\Omega,V^{-*})}(\fx^*).
\end{equation}
Given $u,v\in(0,\infty)$, take $s:=\frac{uv}{u+v}\in(0,\infty)$ and $q:=\frac{u+v}{v}\in(1,\infty)$ so that $q'=\frac{u+v}{u}$. Then $sq=u$ and $sq'=v$, so that \eqref{eq:sqHolder2} reduces to the claimed~\eqref{eq:uvHolder}.
\end{proof}

For any quasi-norm $\rho$ on $\fX$, we define the \emph{dual norm} on $\fX^*$ by
\begin{equation}\label{eq:dualDef}
  \rho^*(\fx^*):=\sup_{\fx\in\fX\setminus\{0\}}\frac{\abs{\pair{\fx^*}{\fx}}}{\rho(\fx)}.
\end{equation}

\begin{remark}
The ``dual norm'' $\rho^*$ is indeed a norm, even if the original $\rho$ is just a quasi-norm. This is immediate from the defining formula.
\end{remark}

Likewise, for a quasi-norm on $\fX^*$, its dual norm is defined on $\fX^{**}$, but can in particular be restricted to $\fX\subseteq\fX^{**}$ via the canonical embedding.

Directly from \eqref{eq:uvHolder} and \eqref{eq:dualDef}, it follows:

\begin{corollary}\label{cor:easyAp}
For all $u,v\in(0,\infty]$,
\begin{equation*}
\begin{split}
    \rho_{L^u(\Omega,V)}^*(\fx^*) &\leq\rho_{L^{v}(\Omega,V^{-*})}(\fx^*),\qquad\fx^*\in\fX^*, \\
    \rho_{L^{v}(\Omega,V^{-*})}^*(\fx) &\leq\rho_{L^u(\Omega,V)}(\fx),\qquad\fx\in\fX.
\end{split}
\end{equation*}
\end{corollary}

When $\Omega=Q\subset\R^n$ is a cube equipped with the normalised measure $dx/\abs{Q}$, we denote $L^p(\Omega,V)=\aveL^p(Q,V)$. In this situation, it is of interest to be able to reverse the inequalities of Corollary \ref{cor:easyAp}, which leads to:

\begin{definition}\label{def:Ap}
Let $p\in(0,\infty)$ and $\fX$ be a Banach space. Let $V:\R^n\to\bddlin(\fX)$.
We say that
\begin{equation*}
   V\in\A^p(\R^n;\bddlin(\fX)),
\end{equation*}
or that $V$ satisfies the $\A_p$ condition if $V$ is a $p$-weight and the following condition holds with some finite constant $C$ and for all cubes $Q\subset\R^n$:
\begin{enumerate}[\rm(i)]
  \item\label{it:Ap>1} if $p\in(1,\infty)$, then also $V^{-*}$ is a $p'$-weight and
\begin{equation*}
  \rho_{\aveL^{p'}(Q,V^{-*})}(\fx^*)\leq C\cdot\rho_{\aveL^{p}(Q,V)}^*(\fx^*)\qquad\text{for all }\fx^*\in\fX^*,
\end{equation*}
  \item\label{it:Ap<1} if $p\in(0,1]$, then
\begin{equation*}
  \esssup_{y\in Q}\rho_{\aveL^p(Q,V)}(V(y)^{-1}\fx)\leq C\Norm{\fx}{}\quad\text{for all }\fx\in\fX.
\end{equation*}
\end{enumerate}
In either case, we denote the smallest such $C$ by $[V]_{\A_p}$.
\end{definition}

\begin{remark}\label{rem:defAp}
For $p\in(1,\infty)$, Definition \ref{def:Ap}\eqref{it:Ap>1} agrees with the definition of the matrix $A_p$ condition in both \cite[Eq.~$(\mathcal A_{p,q})$ on p.~21]{NT:hunt} and \cite[Definition on p.~450]{Volberg:97}. For $p\in(0,1]$, Definition \ref{def:Ap}\eqref{it:Ap<1} is a natural extension of \cite[Eq.~(1.2)]{FR:04}. Literally, their defining condition reads as
\begin{equation}\label{eq:FRorig}
  \esssup_{y\in Q}\Big(\fint_Q\Norm{V(x)V(y)^{-1}}{\bddlin(\fX)}^p \,dx\Big)^{\frac1p}\lesssim 1,
\end{equation}
where $\fX=\C^m$. However, in this finite-dimensional case, we have the norm equivalence $$\Norm{A}{\bddlin(\C^m)}\approx\sum_{i=1}^m\Norm{A\fy_k}{\C^m},$$ where $(\fy_k)_{k=1}^m$ is a fixed orthonormal basis, from which it easily follows that \eqref{eq:FRorig} is equivalent to
\begin{equation}\label{eq:FRequiv}
    \esssup_{y\in Q}\Big(\fint_Q\Norm{V(x)V(y)^{-1}\fx}{\fX}^p \,dx\Big)^{\frac1p}\lesssim\Norm{\fx}{\fX}
\end{equation}
for all $\fx\in\fX=\C^m$. And this is precisely Definition \ref{def:Ap}\eqref{it:Ap<1}. Thus, our definition is obtained by writing a condition equivalent to \cite[Eq.~(1.2)]{FR:04} in $\C^m$, and then extending this equivalent definition to general $\fX$. Note that the difference of \eqref{eq:FRorig} and \eqref{eq:FRequiv} is similar to various other situations encountered in harmonic analysis of operator-valued functions, where useful conditions often involve integral norms of the operator-valued function evaluated at a fixed vector $\fx$ as in \eqref{eq:FRequiv}, instead of operator norms as in \eqref{eq:FRorig}; cf.\ \cite[Definition 11.2.1 and the Notes on page 71]{HNVW3}.
\end{remark}

We next present some equivalent forms and consequences of the $\A_p$ conditions. For the case $p\leq 1$, the following lemma will be useful:

\begin{lemma}\label{lem:GH}
Let $\rho$ be a $p$-norm on $\fX$ with $p\in(0,\infty)$ and let $G,H\in L^0_{\operatorname{so}}(\R^n;\bddlin(\fX))$. Let $Q\subseteq\R^n$ be measurable.
If for every $\fx\in\fX$, we have
\begin{equation}\label{eq:GHe}
  \rho(G(y)\fx)\leq C\Norm{H(y)\fx}{}\quad\text{for a.e. }y\in Q,
\end{equation}
then for every $f\in L^0(\R^n;\fX)$, we have
\begin{equation}\label{eq:GHf}
  \rho(G(y)f(y))\leq C\Norm{H(y)f(y)}{}\quad\text{for a.e. }y\in Q.
\end{equation}
\end{lemma}

\begin{proof}
A strongly measurable function has a separable range by \cite[Theorem 1.1.6]{HNVW1}. Hence, we can pick a sequence $(\fx_n)_{n=1}^\infty$ that is dense in $\{f(y):y\in Q\}$. By the assumed condition \eqref{eq:GHe}, we have
\begin{equation*}
  \rho(G(y)\fx_n)\leq C\Norm{H(y)\fx_n}{}\quad\text{for all $n$ and all }y\in Q\setminus N_n,
\end{equation*}
where $N_n$ has measure zero. In particular, we have
\begin{equation*}
  \rho(G(y)\fx_n)\leq C\Norm{H(y)\fx_n}{}\quad\text{for all $n$ and all }y\in Q\setminus N,
\end{equation*}
where $N:=\bigcup_{n=1}^\infty N_n$ also has measure zero. For each $y\in Q\setminus N$, we can choose a subsequence of $\fx_n$ such that $\fx_n\to f(y)$. Since both $\fx\mapsto\rho(G(y)\fx)$ and $\fx\mapsto\Norm{H(y)\fx}{}$ are continuous, it follows that
\begin{equation*}
  \rho(G(y)f(y))\leq C\Norm{H(y)f(y)}{}\quad\text{for all }y\in Q\setminus N.
\end{equation*}
This is exactly the claim \eqref{eq:GHf}
\end{proof}

\begin{proposition}\label{prop:Ap<1}
Let $p\in(0,1]$ and $V:\R^n\to\bddlin(\fX)$ be a $p$-weight. Then the following conditions, each to hold for some finite constant $C$ and all cubes $Q\subset\R^n$, are equivalent:
\begin{enumerate}[\rm(i)]
  \item\label{it:Ap<1sup} $V\in\A_p$ in the sense of Definition \ref{def:Ap}\eqref{it:Ap<1}, i.e., for all $\fx\in \fX$, we have
\begin{equation*}
   \esssup_{y\in Q}\rho_{\aveL^p(Q,V)}(V(y)^{-1}\fx)\leq C\Norm{\fx}{\fX}.
\end{equation*}
  \item\label{it:Ap<1inf} For all $\fx\in \fX$, we have
\begin{equation*}
   \rho_{\aveL^p(Q,V)}(\fx)\leq C\essinf_{y\in Q}\Norm{V(y)\fx}{\fX}.
\end{equation*}
  \item\label{it:Ap<1,f(y)} For all $f\in L^0(\R^n;\fX)$, we have
\begin{equation*}
   \rho_{\aveL^p(Q,V)}(f(y))\leq C\Norm{V(y)f(y)}{\fX}\quad\text{for a.e. } y\in Q.
\end{equation*}
\end{enumerate}
The minimal $C$ in all these conditions are equal to each other and to $[V]_{\A_p}$.
\end{proposition}

\begin{proof}
We apply Lemma \ref{lem:GH} with the $p$-norm $\rho=\rho_{\aveL^p(Q,V)}$ and different choices of $G$ and $H$:

\eqref{it:Ap<1inf} $\Rightarrow$ \eqref{it:Ap<1,f(y)}: The condition in \eqref{it:Ap<1inf} is \eqref{eq:GHe} with $G=I$ and $H=V$. The condition in \eqref{it:Ap<1,f(y)} is \eqref{eq:GHf} with these same $G$ and $H$. Hence the implication follows from Lemma \ref{lem:GH} with these choices.

\eqref{it:Ap<1,f(y)} $\Rightarrow$ \eqref{it:Ap<1inf}: This is immediate by considering all constant functions $f\equiv\fx$.

\eqref{it:Ap<1sup} $\Rightarrow$ \eqref{it:Ap<1,f(y)}: The condition in \eqref{it:Ap<1sup} is \eqref{eq:GHe} with $G=V^{-1}$ (which belongs to $L^0_{\operatorname{so}}(\R^n;\bddlin(\fX))$, since $V$ is assumed to be a weight) and $H=I$, the identity. On the other hand, the condition in \eqref{it:Ap<1,f(y)} is \eqref{eq:GHf} with these same $G$ and $H$, and $Vf\in L^0(\R^n;\fX)$ (by Lemma \ref{lem:prod}) in place of $f$. Hence the implication follows from Lemma \ref{lem:GH} with these choices.

\eqref{it:Ap<1,f(y)} $\Rightarrow$ \eqref{it:Ap<1sup}: This follows by considering all functions of the form $f=V(\cdot)^{-1}\fx$ (which belong to $L^0(\R^n;\fX)$, since $V$ is assumed to be a weight).
\end{proof}

As explained in Remark \ref{rem:defAp}, motivated by the earlier definitions for matrix weights, we have stated Definition \ref{def:Ap} in terms of norms of $\fx^*\in\fX^*$ for $p\in(1,\infty)$, and (quasi)norms of $\fx\in\fX$ for $p\in(0,1]$. However, we next show that the case $p\in(1,\infty)$ can also be equivalently expressed in terms of norms of $\fx\in\fX$. The following result is an extension of \cite[Proposition 1.2(2)]{Volberg:97}, where the case $\fX=\C^m\simeq\fX^*$ is treated. The extension is not entirely straightforward, owing some subtleties of the duality of $L^p$ spaces of Banach space-valued functions.

\begin{proposition}\label{prop:Ap}
Let $p\in(1,\infty)$, and let $V:\R^n\to\bddlin(\fX)$ be a $p$-weight such that $V^{-*}$ is a $p'$-weight. Then the following conditions, each to hold with some finite constant $C$ and all cubes $Q\subset\R^n$, are equivalent:
\begin{enumerate}[\rm(i)]
  \item\label{it:Ap>1orig} $V\in\A_p$ in the sense of Definition \ref{def:Ap}\eqref{it:Ap>1}, i.e., for all $\fx^*\in\fX^*$,
\begin{equation*}
  \rho_{\aveL^{p'}(Q,V^{-*})}(\fx^*)\leq C\cdot\rho_{\aveL^{p}(Q,V)}^*(\fx^*).
\end{equation*}
  \item\label{it:Ap>1new}  For all $\fx\in\fX$,
\begin{equation*}
   \rho_{\aveL^p(Q,V)}(\fx)\leq C\cdot\rho_{\aveL^{p'}(Q,V^{-*})}^*(\fx).
\end{equation*}
\end{enumerate}
The best constants in both conditions coincide with $[V]_{\A_p}$.
\end{proposition}

\begin{proof}
\eqref{it:Ap>1orig} $\Rightarrow$ \eqref{it:Ap>1new}:
By the elementary duality of vector-valued $L^p$ spaces,
\begin{equation*}
\begin{split}
  \rho_{\aveL^p(Q,V)}(\fx)
  &=\Norm{V(\cdot)\fx}{\aveL^p(Q;\fX)} \\
  &=\sup\Big\{\Babs{\fint_Q \pair{V(\cdot)\fx}{g}}; \Norm{g}{\aveL^{p'}(Q;\fX^*)}\leq 1\Big\}.
\end{split}
\end{equation*}
For any fixed $g$ as above, the mapping $\fx\mapsto
\fint_Q \pair{V(\cdot)\fx}{g}$
defines an element of $\fX^*$, which we denote by $\fx^*$.
Thus
\begin{equation*}
\begin{split}
  \Babs{\fint_Q \pair{V(\cdot)\fx}{g}}
  =\abs{\pair{\fx}{\fx^*}}
  &\leq \rho_{\aveL^{p'}(Q,V^{-*})}^*(\fx) \cdot\rho_{\aveL^{p'}(Q,V^{-*})}(\fx^*) \\
  &\leq\rho_{\aveL^{p'}(Q,V^{-*})}^*(\fx) \cdot [V]_{\A_p}\rho_{\aveL^{p}(Q,V)}^*(\fx^*)
\end{split}
\end{equation*}
by the assumed $\A_p$ property.
Applying the definition of dual norms and the specific choice of $\fx^*$,
we obtain, for any $\fy\in\fX$,
\begin{align*}
  \abs{\pair{\fy}{\fx^*}}
  =\Babs{\fint_Q\pair{V(\cdot)\fy}{g}}
  \leq\Norm{V(\cdot)\fy}{\aveL^p(Q;\fX)}\Norm{g}{\aveL^{p'}(Q;\fX^*)}\le
  \rho_{\aveL^p(Q,V)}(\fy),
\end{align*}
and hence $\rho_{\aveL^{p}(Q,V)}^*(\fx^*)\leq 1$,
which concludes the proof of \eqref{it:Ap>1new}.

\eqref{it:Ap>1new} $\Rightarrow$ \eqref{it:Ap>1orig}:
This argument follows the same broad outline, but some details are slightly trickier.
By duality, again,
\begin{equation}\label{eq:trickyDual}
\begin{split}
  \rho_{\aveL^{p'}(Q,V^{-*})}(\fx^*)
  &=\Norm{V(\cdot)^{-*}\fx^*}{\aveL^{p'}(Q;\fX^*)} \\
  &=\sup\Big\{\Babs{\fint_Q \pair{V(\cdot)^{-*}\fx^*}{f}}; \Norm{f}{\aveL^{p}(Q;\fX)}\leq 1\Big\},
\end{split}
\end{equation}
but we wish to show that the supremum can be restricted to a somewhat smaller set.

Recall that $V^{-1}\in L^0_{\operatorname{so}}(\R^n;\bddlin(\fX))$ (as part of the definition that $V$ is a weight). Hence, for $f\in L^p(Q;\fX)\subseteq L^0(Q;\fX)$, we have $V^{-1}f\in L^0(Q;\fX)$ by Lemma \ref{lem:prod}. In particular, the sets $E_k:=\{x\in Q:\Norm{V(x)f(x)}{\fX}\leq k\}$ are measurable and exhaust $Q$ as $k\to\infty$. Since $\pair{V(\cdot)^{-*}\fx^*}{f}\in L^1(Q)$ (as the pointwise pairing of $V(\cdot)^{-*}\fx^*\in L^{p'}(Q;\fX^*)$ and $f\in L^p(Q;\fX)$), dominated convergence guarantees that
\begin{equation}\label{eq:trickyLim}
  \fint_Q \pair{V(\cdot)^{-*}\fx^*}{f}
  =\lim_{k\to\infty}\fint_Q \pair{V(\cdot)^{-*}\fx^*}{\one_{E_k}f}.
\end{equation}
Like the original $f$, the truncated functions $f_k:=\one_{E_k}f$ satisfy $\Norm{f_k}{\aveL^{p}(Q;\fX)}\leq 1$, but also, by construction, the pointwise bound $\Norm{V^{-1}f_k}{\fX}\leq k$. We already observed that measurability $V^{-1}f\in L^0(\R^n;\fX)$, and hence $V^{-1}f_k\in L^0(\R^n;\fX)$. Combined with the pointwise bound, it follows that $V^{-1}f_k\in L^\infty(Q;\fX)\subseteq L^1(Q;\fX)$.

Thanks to the limit \eqref{eq:trickyLim}, the supremum in \eqref{eq:trickyDual} is still reached if we only consider the truncated $f_k$ in place of $f$. Denoting these simply by $f$ again, we see that
\begin{equation}\label{eq:trickyDual2}
\begin{split}
  &\rho_{\aveL^{p'}(Q,V^{-*})}(\fx^*) \\
&\quad=\sup\Big\{\Babs{\fint_Q \pair{V(\cdot)^{-*}\fx^*}{f}}; \Norm{f}{\aveL^{p}(Q;\fX)}\leq 1, V^{-1}f\in L^1(Q;\fX)\Big\}.
\end{split}
\end{equation}

For $f$ as in \eqref{eq:trickyDual2}, we can write
\begin{equation}\label{eq:trickyDual3}
  \fint_Q \pair{V(\cdot)^{-*}\fx^*}{f}
  =\fint_Q \pair{\fx^*}{V^{-1}f}=\Bpair{\fx^*}{\fint_Q V^{-1}f}=:\pair{\fx^*}{\fx},
\end{equation}
where the reduction from \eqref{eq:trickyDual} to \eqref{eq:trickyDual2} was needed to guarantee the existence of the integral $\fx:=\fint_Q V^{-1}f\in\fX$. (Arguing in a simpler way as in the first part of the proof, we would have only obtained an element $\fx^{**}$ of the second dual $\fX^{**}$, which would not be sufficient for what follows, since our assumption \eqref{it:Ap>1new} only involved vectors $\fx\in\fX$.)

Continuing from \eqref{eq:trickyDual3} and using assumption \eqref{it:Ap>1new}, we obtain
\begin{equation*}
\begin{split}
  \Babs{\fint_Q \pair{V(\cdot)^{-*}\fx^*}{f}}
  =\abs{\pair{\fx^*}{\fx}}
  &\leq\rho_{\aveL^p(Q,V)}^*(\fx^*)\rho_{\aveL^p(Q,V)}(\fx) \\
  &\leq \rho_{\aveL^p(Q,V)}^*(\fx^*)\cdot C\rho_{\aveL^{p'}(Q,V^{-*})}^*(\fx).
\end{split}
\end{equation*}
Applying the definition of the dual norm to estimate the last factor, for $\fy^*\in\fX^*$, we have
\begin{equation*}
\begin{split}
  \abs{\pair{\fy^*}{\fx}}
  &=\Babs{\fint_Q\pair{\fy^*}{V^{-1}f}}
  =\Babs{\fint_Q\pair{V(\cdot)^{-*}\fy^*}{f}} \\
  &\leq\Norm{V(\cdot)^{-*}\fy^*}{\aveL^{p'}(Q;\fX^*)}\Norm{f}{\aveL^p(Q;\fX)}
  \leq\rho_{\aveL^{p'}(Q,V^{-*})}(\fy^*)\cdot 1.
\end{split}
\end{equation*}
Hence
\begin{equation*}
  \rho_{\aveL^{p'}(Q,V^{-*})}^*(\fx)=\sup_{\fy^*\in\fX^*\setminus\{0\}}
  \frac{\abs{\pair{\fy^*}{\fx}}}{\rho_{\aveL^{p'}(Q,V^{-*})}(\fy^*)}\leq 1,
\end{equation*}
thus
\begin{equation*}
  \Babs{\fint_Q \pair{V(\cdot)^{-*}\fx^*}{f}}
  \leq C\rho_{\aveL^p(Q,V)}^*(\fx^*),
\end{equation*}
and substitution into \eqref{eq:trickyDual2} shows that
\begin{equation*}
  \rho_{\aveL^{p'}(Q,V^{-*})}(\fx^*)\leq C\rho_{\aveL^p(Q,V)}^*(\fx^*).
\end{equation*}
This is condition \eqref{it:Ap>1orig} that we wanted to prove.
\end{proof}

\begin{corollary}\label{cor:Ap}
Let $p\in(0,\infty)$ and $V$ be a $p$-weights. If $p\in(1,\infty)$, we also assume that $V^{-*}$ is a $p'$-weight. Then $V\in\A_p$ if and only if there is a finite positive constant $C$ such that, for all cubes $Q\subset\R^n$ and $\fx\in\fX$,
\begin{equation*}
  \rho_{\aveL^p(Q,V)}(\fx)
  \leq C\begin{cases} \rho^*_{\aveL^{p'}(Q,V^{-*})}(\fx), & p\in(1,\infty), \\
    \displaystyle\essinf_{y\in Q}\Norm{V(y)\fx}{\fX}, & p\in(0,1].\end{cases}
\end{equation*}
The best constant $C$ is equal to $[V]_{\A_p}$.
\end{corollary}

\begin{proof}
For $p\in(1,\infty)$, this is Proposition \ref{prop:Ap}. For $p\in(0,1]$, this is contained in Proposition \ref{prop:Ap<1}.
\end{proof}

\begin{corollary}\label{cor:Ap1}
If $p,r\in(0,\infty)$ and $V\in\mathscr{A}_p$, then,
for any $Q\in\mathscr{D}$ and $\fx\in\fX$,
\begin{align*}
\rho_{\aveL^p(Q,V)}(\fx)\lesssim
\rho_{\aveL^r(Q,V)}(\fx).
\end{align*}
\end{corollary}

\begin{proof}
If $p\in(0,1]$, then, by Corollary \ref{cor:Ap}, we immediately
find
\begin{align*}
\rho_{\aveL^p(Q,V)}(\fx)
\lesssim\essinf_{y\in Q}\Norm{V(y)\fx}{\fX}
\le\rho_{\aveL^r(Q,V)}(\fx).
\end{align*}
On the other hand, if $p\in(1,\infty)$, then, combining
Corollaries \ref{cor:Ap} and \ref{cor:easyAp}, we obtain
\begin{align*}
\rho_{\aveL^p(Q,V)}(\fx)
\lesssim\rho^*_{\aveL^{p'}(Q,V^{-*})}(\fx)
\lesssim\rho_{\aveL^r(Q,V)}(\fx).
\end{align*}
Hence, we finish the proof of Corollary \ref{cor:Ap1}.
\end{proof}

\section{Reverse H\"older properties}

In this section, we explore versions of reverse H\"older inequalities satisfied by operator $\A_p$ weights. This may at first seem to contradict \cite{Lauzon,Pott:07}, where counterexamples are given to show that the $\A_2$ condition does {\em not} imply a reverse H\"older inequality for operator-valued weights in infinite dimensions. The example sketched in \cite[paragraph after Corollary 4.7]{Lauzon} is very simple; \cite[Theorem 5.1]{Pott:07} provides a more elaborate example to show that even the weighted $L^2$ boundedness of the Hilbert transform, which is known to be strictly stronger than the $\A_2$ condition in infinite dimensions, still does not imply a reverse H\"older inequality. However, both \cite{Lauzon,Pott:07} interpret the notion of a reverse H\"older inequality in the sense introduced in \cite[Eq.~(2.1)]{CG:01} for matrix weights; notwithstanding its established merits in the finite-dimensional theory, we wish to argue that this is no longer the ``right'' definition in infinite dimensions. The support for this argument comes, on the one hand, from the counterexamples of \cite{Lauzon,Pott:07} just mentioned and, on the other hand, from our proposal of a substitute notion and the results that we are able to obtain by this alternate version.

\begin{proposition}\label{prop:RHI}
If $V\in\A_p$, then there are $\eps,\eta>0$ such that
\begin{equation*}
\rho_{\aveL^{p+\eps}(Q,V)}(\fx)\lesssim\rho_{\aveL^p(Q,V)}(\fx)
\end{equation*}
and
\begin{equation*}
\rho_{\aveL^{p'+\eta}(Q,V^{-*})}(\fx^*)\lesssim\rho_{\aveL^{p'}(Q,V^{-*})}(\fx^*)
\end{equation*}
for all cubes $Q\subset\R^n$ and all $\fx\in\fX$ and $\fx^*\in\fX^*$.
\end{proposition}

\begin{proof}
We note that the second inequality is trivial for $p'=\infty=p'+\eta$, so it is enough to consider $p\in(1,\infty)$ in the part of the proof involving norms on $\fX^*$.

From Lemma \ref{lem:uvHolder} and the definition of the dual norm, it follows that
\begin{equation}\label{eq:uvHolderDual}
\begin{split}
   \rho_{\aveL^{v}(Q,V^{-*})}^*(\fx)
   &\leq\rho_{\aveL^u(Q,V)}(\fx), \\
   \rho_{\aveL^{u}(Q,V)}^*(\fx^*)
   &\leq\rho_{\aveL^v(Q,V^{-*})}(\fx^*)
\end{split}
\end{equation}
for all $u,v\in(0,\infty]$ and all cubes $Q\subset\R^n$.
Then, from $V\in\A_p$ we deduce the following:
\begin{equation}\label{eq:preRHI}
\begin{split}
  \rho_{\aveL^p(Q,V)}(\fx)
  &\leq[V]_{\A_p}\left.\begin{cases} \rho_{\aveL^{p'}(Q,V^{-*})}^*(\fx), & \text{if $p>1$} \\
    \displaystyle\essinf_{y\in Q}\Norm{V(y)\fx}{\fX}, & \text{if $p\leq 1$}\end{cases}\right\}
    \quad\text{by Corollary \ref{cor:Ap}} \\
  &\leq [V]_{\A_p}\cdot \rho_{\aveL^{u}(Q,V)}(\fx) \\
\end{split}
\end{equation}
by \eqref{eq:uvHolderDual} in the first case and a trivial estimate in the second one.
Similarly, for $\fx^*\in\fX^*$, where we only consider $p\in(1,\infty)$, it follows that
\begin{equation*}
\begin{split}
  \rho_{\aveL^{p'}(Q,V^{-*})}(\fx^*)
  &\leq[V]_{\A_p}\cdot\rho_{\aveL^{p}(Q,V)}^*(\fx^*)\quad\text{by Definition \ref{def:Ap}} \\
  &\leq[V]_{\A_p}\cdot\rho_{\aveL^v(Q,V^{-*})}(\fx^*)\quad\text{by \eqref{eq:uvHolderDual}}. \\
\end{split}
\end{equation*}

Let us now fix some $u\in(0,p)$ and $v\in(0,p')$. Spelling out the previous inequalities, they say that, for all cubes $Q\subset\R^n$,
\begin{equation*}
\begin{split}
  \Big(\fint_Q\Norm{V(y)\fx}{}^p dy\Big)^{\frac1p}
  &\leq  [V]_{\A_p}\Big(\fint_Q\Norm{V(y)\fx}{}^u dy\Big)^{\frac1u}, \\
  \Big(\fint_Q\Norm{V^{-*}(y)\fx^*}{}^{p'} dy\Big)^{\frac{1}{p'}}
  &\leq  [V]_{\A_p}\Big(\fint_Q\Norm{V^{-*}(y)\fx^*}{}^v dy\Big)^{\frac1v}.
\end{split}
\end{equation*}
These are classical reverse H\"older inequalities for the functions $\Norm{V(\cdot)\fx}{}^u$ and $\Norm{V^{-*}(\cdot)\fx^*}{}^v$ with exponents $p/u>1$ and $p'/v$, respectively. By the classical Gehring lemma \cite[Lemma~3]{Gehring}, the same functions satisfy the reverse H\"older inequality with slightly larger exponents, which we may write in the form $(p+\eps)/u$ and $(p'+\eta)/v$ with $\eps,\eta>0$. This means that
\begin{equation*}
\begin{split}
  \Big(\fint_Q\Norm{V(y)\fx}{}^{p+\eps} dy\Big)^{\frac{1}{p+\eps}}
  &\lesssim\Big(\fint_Q\Norm{V(y)\fx}{}^{p} dy\Big)^{\frac{1}{p}}, \\
  \Big(\fint_Q\Norm{V^{-*}(y)\fx^*}{}^{p'+\eta} dy\Big)^{\frac{1}{p'+\eta}}
  &\lesssim\Big(\fint_Q\Norm{V^{-*}(y)\fx}{}^{p'} dy\Big)^{\frac{1}{p'}},
\end{split}
\end{equation*}
which are precisely the claimed inequalities.
\end{proof}

\section{Comparing quasi-norms at different cubes}

Starting with \cite{Rou:04}, a significant role in the analysis of matrix-weighted Besov and Triebel--Lizorkin spaces has been played by estimates for expressions of the type
\begin{equation*}
  \Norm{A_Q A_R^{-1}}{},
\end{equation*}
where $A_Q,A_R$ are the reducing operators related to the John ellipsoids of the (quasi)norms $\rho_{\aveL^p(Q,V)}$ and $\rho_{\aveL^p(R,V)}$. In our framework, where the reducing operators are not available, these will be replaced by analogous control of the ratios
\begin{equation*}
  \frac{\rho_{\aveL^p(Q,V)}(\fx)}{\rho_{\aveL^p(R,V)}(\fx)},
\end{equation*}
uniformly over $\fx\in\fX\setminus\{0\}$. This has some general similarity with \cite[e.g. Eq. (4.1)]{Lauzon}, but our considerations will almost immediately take a different directions from those of \cite{Lauzon}.

\begin{lemma}\label{lem:elemDb}
For cubes $Q\subseteq S$, we have
\begin{equation}\label{eq:elemDb1}
   \Big(\frac{\abs{Q}}{\abs{S}}\Big)^{\frac1p}
   \leq   \frac{\rho_{\aveL^p(S,V)}(\fx)}{\rho_{\aveL^p(Q,V)}(\fx)}
   \leq   [V]_{\A_p}\Big(\frac{\abs{S}}{\abs{Q}}\Big)^{\frac{1}{p'}},
\end{equation}
where the first estimate does not require that $V\in\A_p$. If $V\in\A_p$ and $\eps,\eta>0$ are as in Proposition \ref{prop:RHI}, we also have
\begin{equation}\label{eq:elemDb2}
   \Big(\frac{\abs{Q}}{\abs{S}}\Big)^{\frac{1}{p+\eps}}
   \lesssim   \frac{\rho_{\aveL^p(S,V)}(\fx)}{\rho_{\aveL^p(Q,V)}(\fx)}
   \lesssim   \Big(\frac{\abs{S}}{\abs{Q}}\Big)^{\frac{1}{p'+\eta}},
\end{equation}
where the implied positive constants may depend on $[V]_{\A_p}$.
\end{lemma}

\begin{proof}
We prove both \eqref{eq:elemDb1} and \eqref{eq:elemDb2} at once, writing out the details for \eqref{eq:elemDb2}. For \eqref{eq:elemDb1}, it suffices to take $\eps=\eta=0$ and observe that some steps are then avoided.

The first inequality in \eqref{eq:elemDb2} follows from the elementary calculation
\begin{equation}\label{eq:QvsSelem}
\begin{split}
  \rho_{\aveL^p(Q,V)}(\fx)
  &=\Big(\fint_Q\Norm{V(\cdot)\fx}{}^p\Big)^{\frac1p}
  \le\Big(\fint_Q\Norm{V(\cdot)\fx}{}^{p+\eps}\Big)^{\frac{1}{p+\eps}} \\
  &\leq\Big(\frac{\abs{S}}{\abs{Q}} \fint_S\Norm{V(\cdot)\fx}{}^{p+\eps}\Big)^{\frac{1}{p+\eps}}
  =\Big(\frac{\abs{S}}{\abs{Q}}\Big)^{\frac{1}{p+\eps}}\rho_{\aveL^{p+\eps}(S,V)}(\fx) \\
  &\lesssim\Big(\frac{\abs{S}}{\abs{Q}}\Big)^{\frac{1}{p+\eps}}\rho_{\aveL^{p}(S,V)}(\fx)
\end{split}
\end{equation}
by Proposition \ref{prop:RHI} in the last step; this step is redundant if $\eps=0$.

Turning to the second inequality in \eqref{eq:elemDb2}, let first $p\in(0,1]$. Then Corollary \ref{cor:Ap} followed by trivial estimates shows that
\begin{equation*}
\begin{split}
  \rho_{\aveL^p(S,V)}(\fx)
  &\leq[V]_{\A_p}\essinf_{y\in S}\Norm{V(y)\fx}{} \\
  &\leq[V]_{\A_p}\essinf_{y\in Q}\Norm{V(y)\fx}{}
  \leq[V]_{\A_p}\rho_{\aveL^p(Q,V)}(\fx),
\end{split}
\end{equation*}
which proves the second inequality in \eqref{eq:elemDb2} with $p'=p'+\eta=\infty$.

Let then $p\in(1,\infty)$. Similarly to \eqref{eq:QvsSelem}, but with $V^{-*}$ and $p'$ in place of $V$ and $p$, we obtain
\begin{equation*}
  \rho_{\aveL^{p'}(Q,V^{-*})}(\fx^*)
  \lesssim\Big(\frac{\abs{S}}{\abs{Q}}\Big)^{\frac{1}{p'+\eta}}
  \rho_{\aveL^{p'}(S,V^{-*})}(\fx^*),
\end{equation*}
where the implied positive constant is $1$ if $\eta=0$.

Dualising, it follows that
\begin{equation}\label{eq:SvsRdual}
\begin{split}
  \rho_{\aveL^{p'}(S,V^{-*})}^*(\fx)
  &=\sup_{\fx^*\in\fX^*\setminus\{0\}} \frac{\abs{\pair{\fx^*}{\fx}}}
  {\rho_{\aveL^{p'}(S,V^{-*})}(\fx^*)} \\
  &\lesssim\Big(\frac{\abs{S}}{\abs{Q}}\Big)^{\frac{1}{p'+\eta}}
  \sup_{\fx^*\in\fX^*\setminus\{0\}} \frac{\abs{\pair{\fx^*}{\fx}}}{\rho_{\aveL^{p'}(Q,V^{-*})}(\fx^*)} \\
  &=\Big(\frac{\abs{S}}{\abs{Q}}\Big)^{\frac{1}{p'+\eta}}\rho_{\aveL^{p'}(Q,V^{-*})}^*(\fx).
\end{split}
\end{equation}

Hence, if $V\in\A_p$ and $p\in(1,\infty)$, then
\begin{equation*}
\begin{split}
  \rho_{\aveL^p(S,V)}(\fx)
  &\leq [V]_{A_p}\rho_{\aveL^{p'}(S,V^{-*})}^*(\fx)\qquad \text{by Corollary \ref{cor:Ap}} \\
  &\lesssim [V]_{A_p}\Big(\frac{\abs{S}}{\abs{Q}}\Big)^{\frac{1}{p'+\eta}}\rho_{\aveL^{p'}(Q,V^{-*})}^*(\fx)
  \qquad\text{by \eqref{eq:SvsRdual}} \\
  &\leq [V]_{A_p}\Big(\frac{\abs{S}}{\abs{Q}}\Big)^{\frac{1}{p'+\eta}}\rho_{\aveL^{p}(Q,V)}(\fx)\qquad
  \text{by \eqref{eq:uvHolderDual}}.
\end{split}
\end{equation*}
Recalling that we already handled the case $p\in(0,1]$, we have completed the proof of Lemma \ref{lem:elemDb}.
\end{proof}

The upper bounds in Lemma \ref{lem:elemDb} motivate:

\begin{definition}\label{def:dbDim}
For $p\in(0,\infty)$, we say that $V$ is a {\em $(p,\beta)$-doubling weight}
if $V$ is a $p$-weight and for all cubes $Q\subset S$ there holds
\begin{equation*}
   \frac{\rho_{\aveL^p(S,V)}(\fx)}{\rho_{\aveL^p(Q,V)}(\fx)}\lesssim \Big(\frac{\ell(S) }{ \ell(Q) }\Big)^{\frac{\beta-n}{p}}.
\end{equation*}
Here $\beta$ is called the \emph{doubling dimension} of $V$.
\end{definition}

In the language of Definition \ref{def:dbDim}, part of Lemma \ref{lem:elemDb} may be stated~as:

\begin{corollary}\label{cor:ApDb}
Every $V\in\A_p$ is $(p,\beta)$-doubling with $\beta=n\max(1,p)$. If $p>1$, then it is also $(p,\beta)$-doubling for some $\beta\in[n,np)$.
\end{corollary}

\begin{remark}
We have chosen the potentially strange-looking normalisation of the exponent for consistency with the notion of the doubling exponent introduced in \cite[Definition 1.5]{Rou:03} and used in several subsequent papers. Note that an equivalent condition is
\begin{equation}\label{eq:dbRou}
  \int_S\Norm{V(y)\fx}{\fX}^p dy\lesssim\Big(\frac{\ell(S) }{ \ell(Q) }\Big)^{\beta} \int_Q\Norm{V(y)\fx}{\fX}^p dy,
\end{equation}
showing the connection to the doubling property of the measures $$\Norm{V(y)\fx}{\fX}^p dy,$$ uniformly over $\fx\in\fX$.

Definition \ref{def:dbDim} still differs slightly from \cite[Definition 1.5]{Rou:03} as follows: Defining (as usual) the doubling constant $c$ as the best constant in the inequality
\begin{equation*}
  \int_{2B}\Norm{V(y)\fx}{\fX}^p dy\leq c\int_B\Norm{V(y)\fx}{\fX}^p dy,
\end{equation*}
over all balls $B$ and $\fx\in\fX$,  \cite[Definition 1.5]{Rou:03} {\em defines} the doubling exponent as $\hat\beta:=\log_2 c$. It is elementary to see that this specific value gives an {\em upper bound} for our $\beta$, i.e., $\beta\leq\hat\beta$. However, in general the inequality can be strict.

To see this, it suffices to take a scalar weight $V=v$ such that, for every $k\in\Z^n$, we have $v|_{[0,1)^n+k}\equiv 1$ or $N$ if $\sum_{i=1}^n k_i$ is even or odd, respectively. Then it is easy to see that $\int_S v^p\lesssim N^p\cdot \frac{\abs{S}}{\abs{Q}}\int_Q v^p$, so that only the implied positive constant, but not the exponent in \eqref{eq:dbRou} is affected by $N$. On the other hand, the doubling constant satisfies $c\approx N^p$. Hence $\beta=n$ while $\hat\beta\approx p\log N$ can be arbitrarily large.

On the other hand, it is clear that it is the rate of growth described by the exponent $\beta$ in \eqref{eq:dbRou} that matters, not the multiplicative constant in this inequality.
\end{remark}

\begin{remark}
The constant $\beta$ in Definition \ref{def:dbDim} satisfies $\beta\in[n,\infty)$;
see, for instance, \cite[Remark 2.20]{BHYY:I}.
\end{remark}

\begin{lemma}\label{lem:QvsRdb}
Let $V$ be a $(p,\beta)$-doubling weight.
Then for all $\fx\in\fX$ and all cubes $Q,R\subset\R^n$
we have
\begin{equation*}
\frac{\rho_{\aveL^p(Q,V)}(\fx)}{\rho_{\aveL^p(R,V)}(\fx)}
   \lesssim B_{\frac{n}{p},\frac{\beta-n}{p},\frac{\beta}{p}}(Q,R),
\end{equation*}
where
\begin{equation}\label{eq:BabD}
  B_{a,b,c}(Q,R)
   :=\Big(1+\frac{\ell(R)}{\ell(Q)}\Big)^{a}\Big(1+\frac{\ell(Q)}{\ell(R)}\Big)^{b}
   \Big(1+\frac{\abs{x_Q-x_R}}{\max\{\ell(Q),\ell(R)\}}\Big)^c.
\end{equation}
In particular, for all cubes $Q,R$ of equal size,
\begin{equation*}
   \frac{\rho_{L^p(Q,V)}(\fx)}{\rho_{L^p(R,V)}(\fx)}
   \lesssim
   \Big(1+\frac{\abs{x_Q-x_R}}{\ell(Q)}\Big)^{\frac{\beta}{p}}.
\end{equation*}
\end{lemma}

\begin{proof}
Let $S$ be the minimal cube that contains both $Q$ and $R$. Then
\begin{equation}\label{eq:QvsRstart}
\begin{split}
    \frac{\rho_{\aveL^p(Q,V)}(\fx)}{\rho_{\aveL^p(R,V)}(\fx)}
    =\frac{\rho_{\aveL^p(Q,V)}(\fx)}{\rho_{\aveL^p(S,V)}(\fx)}
    \frac{\rho_{\aveL^p(S,V)}(\fx)}{\rho_{\aveL^p(R,V)}(\fx)}
   \lesssim\Big(\frac{\ell(S)}{\ell(Q)}\Big)^{\frac np}
   \Big(\frac{\ell(S)}{\ell(R)}\Big)^{\frac{\beta-n}{p}}
\end{split}
\end{equation}
by the first estimate of Lemma \ref{lem:elemDb} (which does not need $\A_p$) for the first factor and Definition \ref{def:dbDim} for the second.

Then we observe that
\begin{equation*}
  \Big(\frac{\ell(S)}{\ell(Q)}\Big)^{\frac np}\Big(\frac{\ell(S)}{\ell(R)}\Big)^{\frac{\beta-n}{p}}
  \approx\Big(1+\frac{\ell(R)}{\ell(Q)}\Big)^{\frac np}\Big(1+\frac{\ell(Q)}{\ell(R)}\Big)^{\frac{\beta-n}{p}}\Big(\frac{\ell(S)}{\ell(Q)\vee \ell(R)}\Big)^{\frac{\beta}{p}},
\end{equation*}
and
\begin{equation*}
  \frac{\ell(S)}{\ell(Q)\vee \ell(R)}\approx 1+\frac{\abs{x_Q-x_R}}{\ell(Q)\vee\ell(R)}.
\end{equation*}
This finishes the proof of Lemma \ref{lem:QvsRdb}.
\end{proof}

\begin{lemma}\label{lem:QvsR}
Let $V\in\A_p$, and let $\eps,\eta\geq 0$ be numbers for which $V$ satisfies the conclusions of Proposition \ref{prop:RHI}.
Then for all $\fx\in\fX$ and all cubes $Q,R\subset\R^n$ we have
\begin{equation}\label{eq:QvsR}
\frac{\rho_{\aveL^p(Q,V)}(\fx)}{\rho_{\aveL^p(R,V)}(\fx)}
   \lesssim B_{a,b,c}(Q,R),
\end{equation}
where $B_{a,b,c}$ is defined in \eqref{eq:BabD} and
\begin{equation}\label{eq:abc}
  (a,b,c)=(\frac{n}{p+\eps},\frac{n}{p'+\eta},\frac{n}{p+\eps}+\frac{n}{p'+\eta}).
\end{equation}
\end{lemma}

\begin{proof}
The proof is like that of Lemma \ref{lem:QvsRdb}, but replacing the initial estimate \eqref{eq:QvsRstart} by
\begin{equation*}
  \frac{\rho_{\aveL^p(Q,V)}(\fx)}{\rho_{\aveL^p(S,V)}(\fx)}
  \frac{\rho_{\aveL^p(S,V)}(\fx)}{\rho_{\aveL^p(R,V)}(\fx)}
   \lesssim\Big(\frac{\ell(S)}{\ell(Q)}\Big)^{\frac{n}{p+\eps}}
   \Big(\frac{\ell(S)}{\ell(R)}\Big)^{\frac{n}{p'+\eta}},
\end{equation*}
where the estimates of both factors are those in Lemma \ref{lem:elemDb} under the assumption that $V\in\A_p$.
\end{proof}

\begin{remark}
For many purposes, any growth of the type \eqref{eq:QvsR} is enough.
Lemma \ref{lem:QvsR} shows that, for $V\in\A_p$, we can always take some $a<\frac{n}{p}$ and $c<\frac{n}{\min(1,p)}$. Moreover, $b<\frac{n}{p'}$, unless $p'=\infty$, in which case we have the equality $b=\frac{n}{p'}=0$. We refer the reader to \cite{BHYY:I} for several examples of weights, showing that, except for the restrictions just indicated, all possible growth profiles can appear already for scalar and matrix valued weights.
\end{remark}

\begin{corollary}\label{cor:QvsR}
Let $V\in\A_p$, and let $\eps,\eta\geq 0$ be numbers for which $V$ satisfies the conclusions of Proposition \ref{prop:RHI}.
Then for all $\fx^*\in\fX^*$ and all cubes $Q,R\subset\R^n$ with minimal containing cube $S$, we have
\begin{equation*}
   \frac{\rho_{\aveL^{p'}(Q,V^{-*})}(\fx^*)}{\rho_{\aveL^{p'}(R,V^{-*})}(\fx^*)}
   \lesssim \Big(\frac{\abs{S}}{\abs{Q}}\Big)^{\frac{1}{p'+\eta}}\Big(\frac{\abs{S}}{\abs{R}}\Big)^{\frac{1}{p+\eps}}
   \approx B_{b,a,c}(Q,R),
\end{equation*}
where the last expression is defined as in \eqref{eq:BabD}, with $a$, $b$, and $c$ as in \eqref{eq:abc} (but note the swap of two parameters in the bound above compared to Lemma \ref{lem:QvsR}). In particular, for all cubes $Q,R$ of equal size,
\begin{equation*}
   \frac{\rho_{\aveL^{p'}(Q,V^{-*})}(\fx^*)}{\rho_{\aveL^{p'}(R,V^{-*})}(\fx^*)}
   \lesssim
   \Big(1+\frac{\abs{x_Q-x_R}}{\ell(Q)}\Big)^c
   \leq\Big(1+\frac{\abs{x_Q-x_R}}{\ell(Q)}\Big)^n.
\end{equation*}
Here all the implicit positive constants are independent of $Q$, $R$, and $\fx^*$.
\end{corollary}

\begin{proof}
The proof is entirely analogous to that of Lemma \ref{lem:QvsR}, replacing each step by its analogue on the dual side.
\end{proof}

Let us also record the simple:

\begin{lemma}
Let $\fX$ be a Banach space, and let $\rho_Q$ be quasi-norms indexed by cubes $Q\in\R^n$ that satisfy
\begin{equation*}
  \frac{\rho_Q(\fx)}{\rho_R(\fx)}\lesssim B_{a,b,c}(Q,R)\qquad\forall\ \fx\in\fX.
\end{equation*}
Then the dual norms $\rho_Q^*$ on $\fX^*$ satisfy
\begin{equation*}
  \frac{\rho_Q^*(\fx^*)}{\rho_R^*(\fx^*)}\lesssim B_{b,a,c}(Q,R)\qquad\forall\ \fx^*\in\fX^*.
\end{equation*}
\end{lemma}

\begin{proof}
By assumption, we have
\begin{equation*}
  \frac{1}{\rho_R(\fx)}\lesssim\frac{1}{\rho_Q(\fx)}B_{a,b,c}(Q,R).
\end{equation*}
By the definition of the dual norm, it follows that
\begin{equation*}
\begin{split}
  \rho_R^*(\fx^*) &=\sup_{\fx\in\fX\setminus\{0\}}\frac{\abs{\pair{\fx^*}{\fx}}}{\rho_R(\fx)} \\
  &\lesssim \sup_{\fx\in\fX\setminus\{0\}}\frac{\abs{\pair{\fx^*}{\fx}}}{\rho_Q(\fx)}B_{a,b,c}(Q,R)
  =\rho_Q^*(\fx^*)B_{a,b,c}(Q,R).
\end{split}
\end{equation*}
The claim follows by swapping the names of $Q$ and $R$ and noting that $B_{a,b,c}(R,Q)=B_{b,a,c}(Q,R)$ by definition \eqref{eq:BabD}.
\end{proof}

\section{Stability of weights under some operations}

In this section, we record two lemmas of relatively elementary nature that will be helpful in streamlining some considerations further below. Since these results have a somewhat technical nature and we only need them for $p\in(1,\infty)$, we restrict our considerations to this range for brevity, leaving the extension to $p\in(0,1]$ for the interested reader.

\begin{lemma}\label{lem:ApIsom}
Let $p\in(1,\infty)$, let $V\in\A_p(\R^n;\fY)$, and let $\iota:\fX\to\fY$ be an isomorphism. Then $\iota^{-1}V\iota\in\A_p(\R^n;\fX)$ and
\begin{equation*}
  [\iota^{-1}V\iota]_{\A_p(\R^n;\fX)}\leq\Norm{\iota^{-1}}{}\Norm{\iota}{}[V]_{\A_p(\R^n;\fY)}.
\end{equation*}
\end{lemma}

\begin{proof}
Noting that $(\iota^{-1})^*=(\iota^*)^{-1}=:\iota^{-*}$, we see that $\iota^*:\fY^*\to\fX^*$ is also an isomorphism. For $\fx^*\in\fX^*$, we have
\begin{equation*}
\begin{split}
  \rho_{\aveL^{p'}(Q,(\iota^{-1}V\iota)^{-*})}(\fx^*)
  &=\rho_{\aveL^{p'}(Q, \iota^{*}V^{-*}\iota^{-*} )}(\fx^*)\leq\Norm{\iota^*}{} \rho_{\aveL^{p'}(Q,V^{-*})}(\iota^{-*}\fx^*)
  \leq\Norm{\iota}{} [V]_{\A_p}\rho_{\aveL^{p}(Q,V)}^*(\iota^{-*}\fx^*).
\end{split}
\end{equation*}
For $\fy\in\fY$,
\begin{equation*}
\begin{split}
  \abs{\pair{\iota^{-*}\fx^*}{\fy}}
  &=\abs{\pair{\fx^*}{\iota^{-1}\fy}} \\
  &\leq\rho_{\aveL^p(Q,\iota^{-1}V\iota)}^*(\fx^*)\rho_{\aveL^p(Q,\iota^{-1}V\iota)}(\iota^{-1}\fy) \\
  &\leq\rho_{\aveL^p(Q,\iota^{-1}V\iota)}^*(\fx^*)\Norm{\iota^{-1}}{}\rho_{\aveL^p(Q,V)}(\fy).
\end{split}
\end{equation*}
Hence
\begin{equation*}
  \rho_{\aveL^p(Q,V)}^*(\iota^{-*}\fx^*)
  \leq\Norm{\iota^{-1}}{}\rho^*_{\aveL^p(Q,\iota^{-1}V\iota)}(\fx^*),
\end{equation*}
and thus
\begin{equation*}
    \rho_{\aveL^{p'}(Q,(\iota^{-1}V\iota)^{-*})}(\fx^*)
    \leq\Norm{\iota}{} [V]_{\A_p}\Norm{\iota^{-1}}{}\rho^*_{\aveL^p(Q,\iota^{-1}V\iota)}(\fx^*),
\end{equation*}
which is equivalent to the claim of the lemma.
\end{proof}

\begin{lemma}\label{lem:ApDirSum}
Let $p\in(1,\infty)$, let $\fX=\fX_1\oplus\fX_2$ and $V\in L^p_{\loc,\so}(\R^n;\fX)$ with
$$V^{-*}\in L^{p'}_{\loc,\so}(\R^n;\fX^*)$$ be a weight of the form $V=V_1\oplus V_2$ with respect to the above splitting of $\fX$. Then
\begin{equation*}
  [V]_{\A_p(\R^n;\fX)}\approx\sum_{i=1}^2[V_i]_{\A_p(\R^n;\fX_i)}
\end{equation*}
with the positive equivalence constants independent of $V$, $V_1$, and $V_2$.
\end{lemma}

\begin{proof}
We also have $\fX^*=\fX_1^*\oplus\fX_2^*$ and $V^{-*}=V_1^{-*}\oplus V_2^{-*}$. Thus
\begin{equation*}
  \rho_{\aveL^{p'}(Q,V^{-*})}(\fx^*)
  \approx\sum_{i=1}^2\rho_{\aveL^{p'}(Q,V_i^{-*})}(\fx_i^*)
\end{equation*}
and
\begin{equation*}
  \rho_{\aveL^{p}(Q,V)}(\fx)
  \approx\sum_{i=1}^2\rho_{\aveL^{p}(Q,V_i)}(\fx_i).
\end{equation*}
For the dual norm, we argue that
\begin{equation}\label{eq:ApDirSum}
\begin{split}
   \rho_{\aveL^{p}(Q,V)}^*(\fx^*)
   &\approx\sup_{\fx=(\fx_1,\fx_2)\in\fX\setminus\{0\}}\frac{\abs{\sum_{i=1}^2\pair{\fx_i}{\fx_i^*}}}{\sum_{i=1}^2\rho_{\aveL^{p}(Q,V_i)}(\fx_i)} \\
   &\approx\sum_{i=1}^2\sup_{\fx_i\in\fX_i\setminus\{0\}}\frac{\abs{\pair{\fx_i}{\fx_i^*}}}{\rho_{\aveL^{p}(Q,V_i)}(\fx_i)}
   =\sum_{i=1}^2\rho_{\aveL^p(Q,V_i)}^*(\fx_i);
\end{split}
\end{equation}
indeed, in the last ``$\approx$'', the ``$\gtrsim$'' follows by taking either one of $\fx_i=0$, while ``$\leq$'' is the general bound, for arbitrary positive functions $\lambda_i,\mu_i$,
\begin{equation*}
  \sum_{i=1}^2 \lambda_i(\fx_i)\leq\sum_{i=1}^2 C_i \mu_i(\fx_i)
  \leq(C_1+C_2)\sum_{i=1}^2\mu_i(\fx_i),\quad C_i:=\sup\frac{\lambda_i(\fx_i)}{\mu_i(\fx_i)},
\end{equation*}
specialised to $\lambda_i=\fx_i^*$ and $\mu_i=\rho_{\aveL^p(Q,V_i)}$.
Thus
\begin{equation*}
\begin{split}
  \sup_{\fx^*\in\fX^*\setminus\{0\}}
  \frac{\rho_{\aveL^{p'}(Q,V^{-*})}(\fx^*)}{\rho_{\aveL^p(Q,V)}^*(\fx^*)}
  &\approx\sup_{(\fx_1^*,\fx_2^*)\in\fX^*\setminus\{0\}}
  \frac{\sum_{i=1}^2\rho_{\aveL^{p'}(Q,V_i^{-*})}(\fx_i^*)}
  {\sum_{i=1}^2\rho_{\aveL^p(Q,V_i)}^*(\fx_i^*)} \\
  &\approx\sum_{i=1}^2\sup_{\fx_i^*\in\fX_i^*\setminus\{0\}}
  \frac{\rho_{\aveL^{p'}(Q,V_i^{-*})}(\fx_i^*)}{\rho_{\aveL^p(Q,V_i)}^*(\fx_i^*)},
\end{split}
\end{equation*}
where the last ``$\approx$'' follows in the same way as the last ``$\approx$'' in \eqref{eq:ApDirSum}. The lemma follows by taking the supremum over all cubes $Q$.
\end{proof}

\part{Operator-weighted $L^p(V)$ spaces}\label{part:Lp}

\section{Averaging and sparse operators}

The following result is an analogue of \cite[Lemma 2.1]{NT:hunt} and \cite[Lemma 2.2]{AC:12}. The former deals with $p\in(1,\infty)$ and $\fX=\C^d$, the latter with $p=2$ and a general Hilbert space $\fX$. The former, like our present study, is also written in the formalism of the norms $\rho$, while the latter makes use of the explicit form of the reducing operators available for $p=2$. Nevertheless, our argument below, passing through certain quotient norms, is more closely modelled after the one in \cite{AC:12}.

\begin{proposition}\label{prop:ave}
Let $p\in[1,\infty)$ and $\fX$ be a Banach space.
Let $V$ be an $\bddlin(\fX)$-valued weight, and let $a,b\in L^0$ be scalar-valued functions such that $aV\in L^p_{\so}(\bddlin(\fX))$ and $bV^{-*}\in L^{p'}_{\so}(\bddlin(\fX^*))$. Let
\begin{equation*}
  N:=\sup_{\fx\in\fX\setminus\{0\}}\frac{\rho_{L^p(aV)}(\fx)}{\rho_{L^{p'}(bV^{-*})}^*(\fx)},
\end{equation*}
and consider the operator
\begin{equation}\label{eq:Pf}
  Pf(x):=a(x)\int b(y)f(y)dy,
\end{equation}
initially defined on
\begin{equation*}
  L^p_b(V;\fX):=\{f\in L^p(V;\fX): bf\in L^1(\fX)\}.
\end{equation*}
Then
\begin{enumerate}[\rm(i)]
  \item\label{it:LpbDense} $L^p_b(V;\fX)\subset L^p(V;\fX)$ is dense.
  \item\label{it:P<N} If $N<\infty$, then $P$ extends to a bounded linear operator on $L^p(V;\fX)$ with
  $$\Norm{P}{L^p(V;\fX)\to L^p(V;\fX)}  \leq N.$$
  \item\label{it:N<P} If $\int ab=1$ and $P$ extends to a bounded linear operator on $L^p(V;\fX)$, then
  $$N=\Norm{P}{L^p(V;\fX)\to L^p(V;\fX)}.$$
\end{enumerate}
In particular,
if $V$ is a $p$-weight and $V^{-*}$ is a $p'$-weight, then
\begin{equation}\label{eq:ave}
  \Norm{f\mapsto\one_Q\ave{f}_Q}{L^p(V;\fX)\to L^p(V;\fX)}
  =\sup_{\fx\in\fX\setminus\{0\}}\frac{\rho_{\aveL^p(Q,V)}(\fx)}{\rho_{\aveL^{p'}(Q,V^{-*})}^*(\fx)}.
\end{equation}
\end{proposition}

\begin{proof}
The last claim \eqref{eq:ave} is recognised as the special case $a=b=\abs{Q}^{-\frac12}\one_Q$, so it remains to prove the other claims.

\eqref{it:LpbDense}: Given $f\in L^p(V;\fX)$, we have $bf\in L^0(\fX)$. For $k\in\N\setminus\{0\}$, let
\begin{equation*}
  f_k:=\one_{B(0,k)}\one_{\{\Norm{bf}{\fX}\leq k\}} f.
\end{equation*}
Then $bf_k$ is measurable and $\Norm{bf}{\fX}\leq k\one_{B(0,k)}$, hence $bf_k\in L^1(\fX)$. Since $f_k\to f$ pointwise and $\Norm{Vf_k}{\fX}\leq\Norm{Vf}{\fX}$, we have $f_k\to f$ in $L^p(V;\fX)$ by dominated convergence. This proves the density claimed in \eqref{it:LpbDense}.

\eqref{it:P<N}: By \eqref{it:LpbDense}, it is enough to prove an estimate for $f\in L^p_b(V;\fX)$, where $Pf$ is given by the explicit formula \eqref{eq:Pf}. Let $e_f:=\int bf\in \fX$. Then
\begin{equation*}
\begin{split}
  \Norm{Pf}{L^p(V;\fX)}
  &=\Norm{aV\fx_f}{L^p(\fX)}
  =\rho_{L^p(aV)}(\fx_f)
  \leq N\rho_{L^{p'}(bV^{-*})}^*(\fx_f) \\
  &=N\sup_{\fx^*\in\fX^*\setminus\{0\}}\frac{\abs{\pair{\fx_f}{\fx^*}}}{\rho_{L^{p'}(bV^{-*})}(\fx^*)},
\end{split}
\end{equation*}
where
\begin{equation*}
\begin{split}
  \abs{\pair{\fx_f}{\fx^*}}
  &=\Babs{\Bpair{\int bf }{\fx^*}}
  =\Babs{\int \pair{Vf}{bV^{-*}\fx^*}} \\
  &\leq\Norm{Vf}{L^p(\fX)}\Norm{bV^{-*}\fx^*}{L^{p'}(\fX^*)}
  =\Norm{f}{L^p(V;\fX)}\cdot\rho_{L^{p'}(bV^{-*})}(\fx^*).
\end{split}
\end{equation*}
Substituting back, we obtain $\Norm{Pf}{L^p(V;\fX)}\leq N\Norm{f}{L^p(V;\fX)}$, and thus \eqref{it:P<N}.

\eqref{it:N<P}:
From \eqref{eq:Pf}, it is clear that $P$ maps $L^p_b(V;\fX)$ into
$\{a\fx:\fx\in\fX\}\subseteq L^p(V;\fX),$
where the inclusion follows from the assumption that $aV\in L^p_{\operatorname{so}}(\bddlin(\fX))$.
We check that this is a closed subspace: Indeed, if $f_n=a\fx_n\to f$ in $L^p(V;\fX)$, then $Vf_n=aV\fx_n\to Vf$ in $L^p(\fX)$; hence  a subsequence converges almost everywhere. Since $V^{-1}$ exists almost everywhere, it follows that $a\fx_n\to f$ almost everywhere. Thus, a.e.\ on $\{a=0\}$, we have $f=0$, while $\fx_n\to a^{-1}f$ a.e.\ on $\{a\neq 0\}$. Since $\fx_n$ are constant functions, their a.e.\ limit must also be a constant, say $\fx$. Hence $a^{-1}f=\fx$ a.e.\ on $a\neq 0$ and $f=0$ a.e.\ on $\{a=0\}$. These can be combined to the statement that $f=a\fx$ almost everywhere, which confirms the closedness of $\{a\fx:\fx\in\fX\}$ in $L^p(V;\fX)$.

If $P:L^p_b(V;\fX)\to\{a\fx:\fx\in\fX\}$ extends to a bounded linear operator on $L^p(V;\fX)$, the density proved in \eqref{it:LpbDense} shows that $P:L^p(V;\fX)\to\{a\fx:\fx\in\fX\}$. Hence, for every $f\in L^p(V;\fX)$, there is $\fx_f\in\fX$ such that $Pf=a\fx_f$. Under the assumption that $\int ab=1$, it follows in particular that $a$ cannot vanish almost everywhere, and hence such $\fx_f$ is unique.

We then observe that
\begin{equation}\label{eq:AC1}
\begin{split}
  \Norm{Pf}{L^p(V;\fX)\to L^p(V;\fX)}
  &=\sup_{f\in L^p(V;\fX)\setminus\{0\}}\frac{\Norm{aV\fx_f}{L^p(\fX)}}{\Norm{f}{L^p(V;\fX)}} \\
  &=\sup_{\fx\in\fX\setminus\{0\}}
  \sup_{\genfrac{}{}{0pt}{}{f\in L^p(V;\fX)} {\fx_f=\fx}}\frac{\rho_{L^p(aV)}(\fx)}{\Norm{f}{L^p(V;\fX)}}.
\end{split}
\end{equation}

Given any $f\in L^p(V;\fX)$ with $\fx_f=\fx$, we consider $f':=f-a\fx$. For the second term, we have $b(a\fx)\in L^1(\fX)$ since $\int ab=1$, and $V(a\fx)\in L^p(\fX)$ since $aV\in L^p_{\operatorname{so}}(\bddlin(\fX))$. Thus $a\fx\in L^p_b(V;\fX)$, and $P(a\fx)$ can be computed by the explicit formula \eqref{eq:Pf}. It follows from the assumption $\int ab=1$ and $\fx_f=\fx$ that
\begin{equation*}
  P(a\fx)=a\int ba\fx=a\fx,\qquad
  Pf'=Pf-P(a\fx)=a\fx_f-a\fx=0.
\end{equation*}
Hence an arbitrary $f\in L^p(V;\fX)$ with $\fx_f=fx$ can be written as $f=a\fx+f'$, where $f'\in\operatorname{Ker}P$. Thus the inner supremum on the right of \eqref{eq:AC1} can be written as
\begin{equation}\label{eq:AC2}
\begin{split}
  \sup_{\genfrac{}{}{0pt}{}{f\in L^p(V;\fX)} {\fx_f=\fx}} \big( \Norm{f}{L^p(V;\fX)}\big)^{-1}
 &=\Big(\inf_{\genfrac{}{}{0pt}{}{f'\in L^p(V;\fX)} {P f'=0}}\Norm{a\fx+f'}{L^p(V;\fX)}\Big)^{-1} \\
 &=\big(\Norm{a\fx}{L^p(V;\fX)/ \operatorname{Ker}P }\big)^{-1}.
\end{split}
\end{equation}
Moreover,
\begin{equation}\label{eq:axQuot}
  \Norm{a\fx}{L^p(V;\fX)/ \operatorname{Ker}P }
  =\sup\Big\{\abs{\Lambda(a\fx)}:\Norm{\Lambda}{(L^p(V;\fX)/\operatorname{Ker}P)^*}\leq 1\Big\},
\end{equation}
where, by general functional analysis (see e.g.\ \cite[Proposition B.1.4]{HNVW1}), the dual of the quotient is the annihilator
\begin{equation*}
\begin{split}
  (L^p(V;\fX)/\operatorname{Ker}P)^*
  =&\,(\operatorname{Ker}P)^\perp \\
  :=&\,\{\Lambda\in (L^p(V;\fX))^*:\Lambda(f)=0\ \forall\ f\in\operatorname{Ker}P\}.
\end{split}
\end{equation*}
We claim that
\begin{equation}\label{eq:kerPerp}
  (\operatorname{Ker}P)^\perp
  =\{b\fx^*:\fx^*\in\fX^*\}\subset L^{p'}(V^{-*};\fX^*),
\end{equation}
where $L^{p'}(V^{-*};\fX^*)$ is  isometrically identified with a subspace of the dual $(L^p(V;\fX))^*$ with respect to the duality $$\pair{f}{g}=\int\pair{f(y)}{g(y)}dy=\int\pair{V(y)f(y)}{V(y)^{-*}g(y)}dy.$$

Assuming \eqref{eq:kerPerp} for the moment, we can continue from \eqref{eq:axQuot} with
\begin{equation*}
\begin{split}
  \Norm{a\fx}{L^p(V;\fX)/ \operatorname{Ker}P }
  &=\sup\Big\{ \Babs{\int\pair{a\fx}{b\fx^*}}:  \Norm{b\fx^*}{L^{p'}(V^{-*};\fX^*)}\leq 1\Big\} \\
  &=\sup\Big\{ \abs{\pair{\fx}{\fx^*}}:  \rho_{L^{p'}(bV^{-*})}(\fx^*)\leq 1\Big\}
  =\rho_{L^{p'}(bV^{-*})}^*(\fx).
\end{split}
\end{equation*}
Substituting this into \eqref{eq:AC2}, and then \eqref{eq:AC2} into \eqref{eq:AC1}, we obtain the identity claimed in \eqref{it:N<P}.

It only remains to verify claim \eqref{eq:kerPerp}.
We first observe that, for each $\Lambda\in (L^p(V;\fX))^*$, we can define $\fx^*_\Lambda\in \fX^*$ by setting $\pair{\fx}{\fx_\Lambda^*}:=\Lambda(a\fx)$
for $\fx\in\fX$, and hence
\begin{equation*}
  \Norm{\fx_\Lambda^*}{\fX^*}\leq\sup_{\Norm{\fx}{\fX}\leq 1}\Norm{\Lambda}{}\Norm{a\fx}{L^p(V;\fX)}=\Norm{\Lambda}{}\Norm{aV}{L^p_{\so}(\bddlin(\fX))}.
\end{equation*}

For arbitrary $f\in L^p(V;\fX)$ we already checked that $f':=f-a\fx_f\in\operatorname{Ker} P$. Hence every $\Lambda\in(\operatorname{Ker}P)^\perp\subset (L^p(V;\fX))^*$ satisfies
\begin{equation*}
  0=\Lambda(f')=\Lambda(f)-\Lambda(a\fx_f).
\end{equation*}
In particular, for $f\in L^p_b(V;\fX)$ (for which $\fx_f=\int bf$), it follows that
\begin{equation*}
  \Lambda(f)
  =\Lambda(a\fx_f)
  =\pair{\fx_f}{\fx_\Lambda^*}
  =\Bpair{\int bf}{\fx_\Lambda^*}
  =\int \pair{f}{b\fx_\Lambda^*}.
\end{equation*}
Note that both left and right-hand sides of the previous identity define continuous linear functionals on $f\in L^p(V;\fX)$. Since they agree on the dense (by \eqref{it:LpbDense}) subspace $L^p_b(V;\fX)$, they must agree on all $L^p(V;\fX)$. This proves ``$\subseteq$'' in place of ``$=$'' in \eqref{eq:kerPerp}.

To check ``$\supseteq$'', let $\fx^*\in\fX^*$ and $f\in\operatorname{Ker}P\subseteq L^p(V;\fX)$. Thus, there is a sequence $f_n\in L^p_b(V;\fX)$ with $f\to f_n$ in $L^p(V;\fX)$ and $Pf_n\to Pf=0$ in $L^p(V;\fX)$, hence, for a subsequence, also almost everywhere. Since $Pf_n=a\fx_{f_n}$ and $a$ does not vanish a.e., it follows that $\fx_{f_n}\to 0$ in $\fX$, where $\fx_{f_n}=\int bf_n$. Thus also $\int \pair{bf_n}{\fx^*}=\int \pair{f_n}{b\fx^*}\to 0$. But $b\fx^*\in L^{p'}(V^{-*};\fX^*)$ and $f_n\to f$ in $L^p(V;\fX)$, hence $\int \pair{f}{b\fx^*}=\lim_{n\to\infty}\int \pair{f_n}{b\fx^*}=0$. This confirms  ``$\subseteq$'' in place of ``$=$'' in \eqref{eq:kerPerp}.

Thus, we have completed the verification of \eqref{eq:kerPerp}, and thereby of the proposition.
\end{proof}

\begin{remark}
A curious detail of the proof above is that we manage to identify $(\operatorname{Ker}P)^\perp$, {\em a priori} a subspace of {\em functionals} $(L^p(V;\fX))^*$, with an explicit subspace of {\em functions} $L^{p'}(V^{-*},\fX^*)$, without the need to know whether the full spaces $(L^p(V;\fX))^*$ and  $L^{p'}(V^{-*},\fX^*)$ coincide. Indeed, in the generality of arbitrary Banach spaces that we consider, this coincidence in general fails, as it would require the so-called Radon--Nikod\'ym property of $\fX^*$ (see e.g.\ \cite[Section 1.3.b]{HNVW1}).
\end{remark}

\begin{corollary}\label{cor:ave}
Let $\fX$ be a Banach space, $p\in(1,\infty)$, and $V$ be an $\bddlin(\fX)$-valued $p$-weight such that $V^{-*}$ is a $p'$-weight. Then the averaging operators $f\mapsto\one_Q\ave{f}_Q$ are bounded on $L^p(V;\fX)$, uniformly over all cubes $Q\subset\R^n$, if and only if $V\in\A_p$.
\end{corollary}

\begin{proof}
By \eqref{eq:ave} of Proposition \ref{prop:ave}, we find that
\begin{equation*}
    \Norm{f\mapsto\one_Q\ave{f}_Q}{L^p(V;\fX)\to L^p(V;\fX)}\leq C
\end{equation*}
for all cubes $Q\subset\R^n$ if and only if
\begin{equation*}
    \rho_{\aveL^p(Q,V)}(\fx)\leq C\rho_{\aveL^{p'}(Q,V^{-*})}^*(\fx)
\end{equation*}
for all cubes $Q\subset\R^n$ and all $\fx\in\fX$. By Corollary \ref{cor:Ap}, this latter condition is equivalent to $V\in\A_p$.
\end{proof}

As a further corollary, we obtain the boundedness of a class of sparse operators, previously considered by \cite{LP:22} in the context of $\A_2$-weighted $L^2$ spaces of Hilbert space -valued functions, in the more general setting of $\A_p$-weighted $L^p$ spaces of Banach space-valued functions. We should note, however, that our Corollary \ref{cor:sparse} below is not strictly an extension of \cite[Theorem 2.1]{LP:22}, since these authors obtain a somewhat sharper control, requiring the $\A_2$ property only for the cubes that appear in the sparse collection. It might seem at first that this is the case in our argument as well, but the fact that we apply the reverse H\"older inequality (Proposition \ref{prop:RHI}) makes implicit use of the $\A_p$ condition for the subcubes of these cubes as well.

\begin{corollary}\label{cor:sparse}
Let $p\in(1,\infty)$, let $V\in\A_p(\R^n;\bddlin(\fX))$, let $\mathscr S\subset\mathscr D$ be a sparse collection of dyadic cubes of $\R^n$, and for each $Q\in\mathscr S$, let $a_Q$ and $b_Q$ be scalar functions supported in $Q$ and bounded in absolute value by $1$. Then
\begin{equation*}
  Tf:=\sum_{Q\in\mathscr S}a_Q\fint_Q b_Qf
\end{equation*}
defines a bounded operator on $L^p(V;\fX)$.
\end{corollary}

\begin{proof} 
Let us begin by observing that each individual term in the series of $Tf$ defines a bounded operator on $L^p(V)$ by Proposition \ref{prop:ave}. Hence, so does any finite sum of these terms. Our plan is to first prove that such finite sums satisfy uniform bounds that are independent of the number of terms, and then use this to deduce the convergence of the infinite series to a bounded operator.

\subsubsection*{A priori bounds for finite sums}

Let first $\mathscr S$ be finite, so that $Tf\in L^p(V)$ is well defined. We estimate its norm by dualising with $g\in L^{p'}(V^{-*})$; cf.~\cite[Proposition 1.3.1.]{HNVW1}. Hence, we are led to study
\begin{equation*}
    \pair{Tf}{g}
    =\sum_{Q\in\mathscr S}\abs{Q}\Bpair{\fint_Q b_Qf}{\fint_Q a_Q g},
\end{equation*}
and we begin with the estimate
\begin{equation}\label{eq:sparse-dual}
  \abs{\pair{Tf}{g}}
  \leq\sum_{Q\in\mathscr S}\abs{Q}\rho_{\aveL^v(Q,V)}\Big(\fint_Q b_Q f\Big)\rho_{\aveL^v(Q,V)}^*\Big(\fint_Q a_Q g\Big).
\end{equation}
The last term is
\begin{equation*}
  \rho_{\aveL^v(Q,V)}^*\Big(\fint_Q a_Q g\Big)
  =\sup_{\fx\in\fX\setminus\{0\}}\frac{1}{\rho_{\aveL^v(Q,V)}(\fx)}\Babs{\Bpair{\fx}{\fint_Q a_Q g}},
\end{equation*}
where
\begin{equation*}
  \Babs{\Bpair{\fx}{\fint_Q a_Q g}}
  =\Babs{\fint_Q a_Q\pair{V\fx}{V^{-*}g}}
  \leq\Norm{V\fx}{\aveL^v(Q)}\Norm{V^{-*}g}{\aveL^{v'}(Q)},
\end{equation*}
and hence
\begin{equation}\label{eq:rho-g}
  \rho_{\aveL^v(Q,V)}^*\Big(\fint_Q a_Q g\Big)\leq\Norm{V^{-*}g}{\aveL^{v'}(Q)}.
\end{equation}

Let now $v=p+\eps$. By Proposition \ref{prop:RHI} and \eqref{eq:preRHI}, we have
\begin{equation*}
  \rho_{\aveL^v(Q,V)}\Big(\fint_Q b_Q f\Big)\lesssim \rho_{\aveL^u(Q,V)}\Big(\fint_Q b_Q f\Big)
\end{equation*}
for any $u\in(0,\infty)$.
On the other hand, by Proposition \ref{prop:ave},
\begin{equation*}
\begin{split}
  \rho_{\aveL^u(Q,V)}\Big(\fint_Q b_Q f\Big)
  &=\BNorm{\frac{\one_Q}{\abs{Q}^{\frac1u}} \int \frac{b_Q}{\abs{Q}^{\frac{1}{u'}}} \frac{f}{\abs{Q}^{\frac 1u}}   }{L^u(V)} \\
  &\leq\sup_{\fx\in\fX\setminus\{0\}}\frac{\rho_{L^u(\one_Q\abs{Q}^{-\frac1u}V)}(\fx)}{\rho_{L^{u'}(b_Q\abs{Q}^{-\frac{1}{u'}}V^{-*})}^*(\fx)}
    \Norm{\one_Q\abs{Q}^{-\frac1u}f}{L^u(V)} \\
  &=\sup_{\fx\in\fX\setminus\{0\}}\frac{\rho_{\aveL^u(Q,V)}(\fx)}{\rho_{\aveL^{u'}(Q,b_QV^{-*})}^*(\fx)}
    \Norm{f}{\aveL^u(Q,V)}. \\
\end{split}
\end{equation*}

Taking $u=(p'+\eps)'$, thus $u'=p'+\eps$, it follows that
\begin{equation*}
\begin{split}
    \rho_{\aveL^{u'}(Q,b_QV^{-*})}(\fx^*)
    &\leq\rho_{\aveL^{u'}(Q,V^{-*})}(\fx^*)\qquad\text{since }\abs{b_Q}\leq\one_Q, \\
    &\lesssim\rho_{\aveL^{p'}(Q,V^{-*})}(\fx^*)\qquad\text{by Proposition \ref{prop:RHI}},
\end{split}
\end{equation*}
and hence $\rho_{\aveL^{u'}(Q,b_QV^{-*})}^*(\fx)\gtrsim \rho_{\aveL^{p'}(Q,V^{-*})}^*(\fx)$. Thus
\begin{equation*}
  \frac{\rho_{\aveL^u(Q,V)}(\fx)}{\rho_{\aveL^{u'}(Q,b_QV^{-*})}^*(\fx)}
  \lesssim\frac{\rho_{\aveL^p(Q,V)}(\fx)}{\rho_{\aveL^{p'}(Q,V^{-*})}^*(\fx)}\lesssim 1
\end{equation*}
by the $\A_p$ condition in the last step. Thus, we have checked that
\begin{equation}\label{eq:rho-f}
  \rho_{\aveL^v(Q,V)}\Big(\fint_Q b_Q f\Big)\lesssim\Norm{f}{\aveL^u(Q,V)},\qquad
  \begin{cases} v = p+\eps, \\ u=(p'+\eps)'.\end{cases}
\end{equation}
Substituting \eqref{eq:rho-g} and \eqref{eq:rho-f} into \eqref{eq:sparse-dual}, and using the sparseness assumption $\abs{Q}\lesssim\abs{E(Q)}$ for some disjoint subsets $E(Q)\subseteq Q$, we obtain
\begin{equation}\label{eq:sparse-final}
\begin{split}
  \abs{\pair{Tf}{g}}
  &\lesssim\sum_{Q\in\mathscr S}\abs{E(Q)}\Norm{f}{\aveL^u(Q,V)}\Norm{g}{\aveL^{v'}(Q,V^{-*})} \\
  &\leq\int_{\R^n}M_u(Vf)M_{v'}(V^{-*}g) \\
  &\leq\Norm{M_u(Vf)}{L^p}\Norm{M_{v'}(V^{-*}g)}{L^{p'}} \\
  &\lesssim\Norm{Vf}{L^p} \Norm{V^{-*}g}{L^{p'}} =\Norm{f}{L^p(V)} \Norm{g}{L^{p'}(V^{-*})}
\end{split}
\end{equation}
by the maximal theorem, since $p>u$ and $p'>v'$. This shows that $\Norm{Tf}{L^p(V)}\lesssim\Norm{f}{L^p(V)}$ by duality.

\subsubsection*{Convergence of the infinite series}

We then turn to the general case of infinite $\mathscr S$. We will show that the series defining $Tf$ converges unconditionally in $L^p(V)$, i.e., that all $$T_k f:=\sum_{Q\in\mathscr S_k}a_Q\fint_Q b_Q f$$
with $k\in\N\setminus\{0\}$
form a Cauchy sequence in $L^p(V)$ whenever $\mathscr S_k\subseteq\mathscr S$ are finite subsets with $\mathscr S_k\uparrow\mathscr S$. Thus, we need to show that $T_j f-T_k f=\sum_{Q\in\mathscr S_j\setminus \mathscr S_k}a_Q\fint_Q b_Q f$ converges to zero in $L^p(V)$ as $j>k\to\infty$.

For all $M,N\in\N$, there are only finitely many dyadic cubes $Q\in\mathscr S$ such that $Q\subseteq[-2^M,2^M]$ and $2^{-N}\leq\ell(Q)\leq 2^M$. Hence, for all $k\geq k(M,N)$, all such cubes will be contained in $\mathscr S_k$ and hence not contained in $\mathscr S_j\setminus\mathscr S_k$ for $j>k\geq k(M,N)$. In other words, $\mathscr S_j\setminus \mathscr S_k$ only contains cubes that are
\begin{equation}\label{eq:cube-cases}
  \begin{cases} \text{large:} & \ell(Q)>2^M,\quad\text{or} \\
    \text{far:} & Q\subseteq[-2^M,2^M]^c,\quad\text{or} \\
    \text{small:} & Q\subseteq[-2^M,2^M]\quad\text{and}\quad \ell(Q)<2^{-N}. \\
  \end{cases}
\end{equation}

We now apply \eqref{eq:sparse-final} with $\mathscr S_j\setminus\mathscr S_k$ in place of $\mathscr S$, observing that we can replace the maximal function by a slightly smaller object, where we only take the supremum over the cubes that actually appear in the sum:
\begin{equation}\label{eq:TjCauchy}
\begin{split}
  \abs{\pair{T_j f-T_k f}{g}}
  &\lesssim\sum_{Q\in\mathscr S_j\setminus\mathscr S_k}\abs{E(Q)}\Norm{f}{\aveL^u(Q,V)}\Norm{g}{\aveL^{v'}(Q,V^{-*})} \\
  &\leq\int_{\R^n}M_u^{\mathscr S_j\setminus\mathscr S_k}(Vf)M_{v'}(V^{-*}g) \\
  &\lesssim\Norm{M_u^{\mathscr S_j\setminus\mathscr S_k}(Vf)}{L^p} \Norm{g}{L^{p'}(V^{-*})},
\end{split}
\end{equation}
where (writing simply $\phi:=Vf\in L^p$)
\begin{equation*}
\begin{split}
   M_u^{\mathscr S_j\setminus\mathscr S_k}\phi(x)
   &:=\sup_{x\in Q\in\mathscr S_j\setminus\mathscr S_k}\Norm{\phi}{\aveL^u(Q)}
   \leq M_{u,M}^{\operatorname{large}}\phi(x)
   +M_{u,M}^{\operatorname{far}}\phi(x)
   +M_{u,M,N,\mathscr S}^{\operatorname{small}}\phi(x)
\end{split}
\end{equation*}
and the last three maximal operators are defined by taking the supremum over the respective collections of cubes defined in \eqref{eq:cube-cases}. It is evident that each of them is pointwise dominated by $M_u\phi\in L^p$; hence, by dominated convergence, to show their $L^p$ convergence to zero, it is enough to show pointwise convergence.

Concerning $M_{u,M}^{\operatorname{large}}\phi$, we note that
\begin{equation*}
  \Norm{\phi}{\aveL^u(Q)}\leq  \Norm{\phi}{\aveL^p(Q)}\leq \abs{Q}^{-\frac1p}\Norm{\phi}{L^p(\R^n)}
  \leq 2^{-\frac{nM}{p}}\Norm{\phi}{L^p(\R^n)}
\end{equation*}
for $\ell(Q)>2^M$, and hence
\begin{equation*}
  M_{u,M}^{\operatorname{large}}\phi(x)\leq 2^{-\frac{nM}{p}}\Norm{\phi}{L^p(\R^n)}\to 0
\end{equation*}
as $M\to\infty$. For the ``far'' case, it is even simpler to see that
\begin{equation*}
  M_{u,M}^{\operatorname{far}}\phi(x)\leq \one_{[-2^M,2^M]^c}(x)M_u\phi(x)\to 0.
\end{equation*}

Finally, we turn to the case of small cubes. This requires a little more care and in particular the fact that we are dealing with subsets of a fixed collection $\mathscr S$ of sparse cubes only. (Note that the full maximal function over all cubes, when restricted to small cubes, converges to $\abs{\phi(x)}$ instead of $0$ by Lebesgue's differentiation theorem!)

The maximal function $M_{u,M,N,\mathscr S}^{\operatorname{small}}\phi$ is supported in the union $U_{M,N}$ of cubes $Q\in\mathscr S$ with $Q\subseteq[-2^M,2^M]$ and $\ell(Q)<2^{-N}$. The measure of this union is dominated by
\begin{equation*}
  \sum_{\genfrac{}{}{0pt}{}{Q\in\mathscr S: \ell(Q)<2^{-N}}{Q\subseteq[-2^M,2^M]}}
  \abs{Q}
  \lesssim  \sum_{\genfrac{}{}{0pt}{}{Q\in\mathscr S} {Q\subseteq[-2^M,2^M]}}
  \abs{E(Q)}\leq\abs{[-2^M,2^M]}=2^{n(M+1)},
\end{equation*}
and hence by the tail (with respect to the condition $\ell(Q)<2^{-N}$) of a convergent series. Thus, $\abs{U_{M,N}}\to 0$ as $N\to\infty$. Since clearly $U_{M,N}\supseteq U_{M,N+1}$, it follows that $U_{M,N}\downarrow U_{M,\infty}$ as $N\to\infty$, where $\abs{U_{M,\infty}}=0$. Hence
\begin{equation*}
  M_{u,M,N,\mathscr S}^{\operatorname{small}}\phi(x)\leq \one_{U_{M,N}}(x)M_u\phi(x)\to 0
\end{equation*}
almost everywhere as $N\to\infty$.

To summarise, given $\eps>0$, we can first choose $M$ so large that
\begin{equation*}
  \Norm{M_{u,M}^{\operatorname{large}}\phi}{L^p}<\frac{\eps}{3},
  \qquad\Norm{M_{u,M}^{\operatorname{large}}\phi}{L^p}<\frac{\eps}{3},
\end{equation*}
and then $N$ so large that
\begin{equation*}
  \Norm{M_{u,M,N,\mathscr S}^{\operatorname{small}}\phi}{L^p}<\frac{\eps}{3}
\end{equation*}
to see that
\begin{equation*}
  \Norm{M_u^{\mathscr S_j\setminus \mathscr S_k}\phi}{L^p}<\eps
\end{equation*}
for all $j>k\geq k(M,N)$. In view of \eqref{eq:TjCauchy}, we find that $\{T_k f\}_{k\in\N\setminus\{0\}}$ is a Cauchy sequence in $L^p(V)$, which we wanted to prove.
\end{proof}

\section{Unbounded commutators}

The aim of this section is to prove the following statement:

\begin{theorem}\label{thm:commuUnbd}
Let $\fX$ be an infinite-dimensional Banach space and let $p\in(1,\infty)$. Then there exists a function $B\in\BMO_{\so}(\R;\bddlin(\fX))$ with $B^*\in\BMO_{\so}(\R;\bddlin(\fX^*))$ such that its commutator $[B,\HT]$ with the Hilbert transform
\begin{equation*}
  \HT f(x):=\lim_{\eps\to 0}\frac{1}{\pi}\Big(\int_{-\infty}^{x-\eps}+\int_{x+\eps}^\infty\Big)\frac{f(y)dy}{x-y}
\end{equation*}
is not bounded on $L^p(\R;\fX)$.
\end{theorem}

While not making direct reference to operator-weighted spaces, this result is closely related, and this connection will be elaborated in the following section, where we return to the setting of $L^p(V)$. In the special case that $p=2$ and $\fX$ is a Hilbert space, Theorem \ref{thm:commuUnbd} is contained in \cite[Theorem 1.2]{GPTV:00}, \cite[Theorem 1.2]{NPTV:02}, or \cite[Theorem 1.3]{GPTV:04}. Formally, the said results state a dimensional growth of the commutator norms in finite-dimensional Hilbert spaces, but it is easy to obtain the infinite-dimensional version from this by noting that an infinite-dimensional Hilbert space can be written as an $\ell^2$ direct sum containing all the finite-dimensional versions.

To construct a function $B$ as in Theorem \ref{thm:commuUnbd}, we will use two different methods depending on the Banach space $\fX$. The first method works if $\fX$ is {\em not} a UMD space \cite[Definition 4.2.1]{HNVW1}, while the second method requires $\fX$ to be $K$-convex \cite[Definition 4.3.9]{HNVW1}. Since every UMD space is $K$-convex \cite[Proposition 4.3.10]{HNVW1}, the two cases cover all Banach spaces, even with some overlap, since not every $K$-convex space has UMD.

\subsection{Scalar commutators and the UMD property}

For Banach spaces that fail the UMD property, Theorem \ref{thm:commuUnbd} will be a consequence of the following:

\begin{proposition}\label{prop:commuUMD}
Let $\fX$ be an infinite-dimensional Banach space and let $p\in(1,\infty)$. Then the following conditions are equivalent:
\begin{enumerate}[\rm(i)]
  \item\label{it:UMD} $\fX$ is a UMD space;
  \item\label{it:BMO} $[b,\HT]$ is bounded on $L^p(\R;\fX)$ for all scalar-valued $b\in\BMO(\R)$;
  \item\label{it:Linfty} $[b,\HT]$ is bounded on $L^p(\R;\fX)$ for all scalar-valued $b\in L^\infty(\R)$.
\end{enumerate}
\end{proposition}

The statement is perhaps expected, in view of the well-known characterisation of the UMD condition by the boundedness of the Hilbert transform on $L^p(\R;\fX)$. While this will be used in the proof below, observe that some work still needs to be done, since the (un)boundedness of $\HT$ is not the same as the (un)boundedness of a commutator $[b,\HT]$, and their equivalence will in any case require some conditions on the function $b$. Proposition \ref{prop:commuUMD} continues the tradition of characterising the UMD condition by the boundedness of various operators, like the Hilbert transform \cite{Bou:83,Burk:83}, imaginary powers $(-\triangle)^{is}$ of the Laplacian \cite{Guerre:91}, or other Fourier multipliers with sufficient structure \cite{GMS}. (The mentioned results can also be found in \cite[Theorem 5.1.1]{HNVW1}, \cite[Corollary 10.5.2]{HNVW2}, and \cite[Corollary 13.3.10]{HNVW3}, respectively.) In contrast to all these examples, the commutators $[b,\HT]$ in Proposition \ref{prop:commuUMD} are {\em not} translation-invariant operators.

\begin{proof}[Proof of Proposition \ref{prop:commuUMD}]
\eqref{it:UMD} $\Rightarrow$ \eqref{it:BMO}: This can be proved by repeating the classical Cauchy integral trick of \cite{CRW:76}, where the scalar-valuedness of $b$ is used, while the vector-valuedness of $f$ poses no problem. An easy identity followed by the Cauchy integral formula shows that
\begin{equation*}
  [b,\HT]f =\frac{d}{dz}(e^{zb}\HT(e^{-zb}f))\Big|_{z=0}
  =\frac{1}{2\pi i}\oint_{\abs{z}=\eps} e^{zb}\HT(e^{-zb}f)\,\frac{dz}{z^2}.
\end{equation*}
The assumption \eqref{it:UMD} implies that $\HT$ is bounded on $L^p(v;\fX)$ for all scalar-valued weights $v\in\A_p(\R)$ (see \cite[Corollary 11.3.27]{HNVW3}). If $\eps>0$ is small enough (depending on $\Norm{b}{\BMO(\R)}$), it is well known and easy to check that $e^{\Re(zb)}\in\A_p(\R)$, uniformly in $\abs{z}=\eps$, as already observed in \cite{CRW:76}. Combining these observations, it follows that
\begin{equation*}
\begin{split}
  \Norm{[b,\HT]f}{L^p(\R;\fX)}
  &\leq\frac{1}{2\pi}\oint_{\abs{z}=\eps} \Norm{e^{\Re(zb)}\HT(e^{-zb}f)}{L^p(\R;\fX)}\,\frac{\abs{dz}}{\abs{z}^2} \\
  &\lesssim\oint_{\abs{z}=\eps} \Norm{e^{\Re(zb)}(e^{-zb}f)}{L^p(\R;\fX)}\,\frac{\abs{dz}}{\abs{z}^2} \\
  &=\oint_{\abs{z}=\eps} \Norm{f}{L^p(\R;\fX)}\,\frac{\abs{dz}}{\abs{z}^2} \\
  &\lesssim\frac{1}{\eps}\Norm{f}{L^p(\R;\fX)}\approx\Norm{b}{\BMO(\R)}\Norm{f}{L^p(\R;\fX)}.
\end{split}
\end{equation*}

\eqref{it:BMO} $\Rightarrow$ \eqref{it:Linfty} is obvious. Let us also note that \eqref{it:UMD} $\Rightarrow$ \eqref{it:Linfty} can be seen directly without the need for the Cauchy integral trick or the weighted theory, using only the unweighted boundedness of $\HT$ under the UMD condition:
\begin{equation*}
\begin{split}
  \Norm{[b,\HT]f}{L^p(\R;\fX)}
  &\leq\Norm{b\HT f}{L^p(\R;\fX)}+\Norm{\HT(bf)}{L^p(\R;\fX)}
  \lesssim\Norm{b}{L^\infty(\R)}\Norm{f}{L^p(\R;\fX)}.
\end{split}
\end{equation*}

\eqref{it:Linfty} $\Rightarrow$ \eqref{it:UMD}:
Let us fix a function $b$ such that
\begin{equation}\label{eq:bChoice}
  b(x)=\begin{cases} 1 & \text{if }x\in[m-\frac14,m+\frac14]\text{ for an even }m\in\Z, \\
  -1 & \text{if }x\in[m-\frac14,m+\frac14]\text{ for an odd }m\in\Z,\end{cases}
\end{equation}
where $\Z$ denotes the set of all integers.
Let us fix a bump $\varphi$ of integral one supported in $[-\frac14,\frac14]$, and consider
\begin{equation*}
  f(y)=\sum_{\genfrac{}{}{0pt}{}{j\in\Z} {\text{odd}}}a_j\varphi(y-j),\quad
  g(x)=\sum_{\genfrac{}{}{0pt}{}{k\in\Z} {\text{even}}}c_k\varphi(x-k),
\end{equation*}
so that $\Norm{f}{L^p(\R;\fX)}\approx\Norm{\{a_j\}}{\ell^p(\Z_{\textup{odd}};\fX)}$ and $\Norm{g}{L^{p'}(\R;\fX)}\approx\Norm{\{c_j\}}{\ell^{p'}(\Z_{\textup{even}};\fX^*)}$.

Then $bf=-f$ and $bg=g$, hence
\begin{equation}\label{eq:bHvsH}
\begin{split}
    \pair{[b,\HT]f}{g}
    &=\pair{\HT f}{bg}-\pair{\HT(bf)}{g} \\
    &=2\pair{\HT f}{g}
    =\sum_{\genfrac{}{}{0pt}{}
    {j\in\Z_{\text{odd}}} {k\in\Z_{\text{even}}}}2d_{k-j}\pair{a_j}{c_k},     \\
\end{split}
\end{equation}
where the factor
\begin{equation*}
\begin{split}
  d_{k-j}:=\iint\frac{1}{x-y}\varphi(y-j)\varphi(x-k)\,dx \,dy
  =\iint\frac{\varphi(v)\varphi(u)}{(k+u)-(j+v)}\,du \,dv
\end{split}
\end{equation*}
differs from $(k-j)^{-1}=(k-j)^{-1}\iint\varphi(u)\varphi(v)dudv$ by
\begin{equation*}
  \Babs{\iint \frac{v-u}{(k-j+u-v)(k-j)}\varphi(u)\varphi(v)dudv} \lesssim\frac{1}{(k-j)^2}.
\end{equation*}

By \eqref{eq:bHvsH} and assumption \eqref{it:Linfty}, the mapping
\begin{equation*}
  \{a_j\}_{j\in\Z_{\text{odd}}}\mapsto\Big\{\sum_{j\in\Z_{\text{odd}}} d_{k-j}a_j\Big\}_{k\in\Z_{\text{even}}}
\end{equation*}
is bounded from $\ell^p(\Z_{\text{odd}};\fX)$ to $\ell^p(\Z_{\text{even}};\fX)$. Since the same is clearly true with $(k-j)^{-2}$ in place of $d_{k-j}$, it is also true with $(k-j)^{-1}$ in place of $d_{k-j}$, by the above estimate on their difference.

By a change of variables, it follows that the discrete convolution by $\{(k+\frac12)^{-1}\}_{k\in\Z}$ is bounded on $\ell^p(\Z;\fX)$. The difference of this kernel from $\{\one_{\Z\setminus\{0\}}k^{-1}\}_{k\in\Z}$ is dominated by $\{(1+k^2)^{-1}\}_{k\in\Z}\in\ell^1(\Z)$, which clearly induces a bounded convolution operator on $\ell^p(\Z;\fX)$. Hence, the same is true of $\{\one_{\Z\setminus\{0\}}k^{-1}\}_{k\in\Z}$. But the convolution with this sequence is the discrete Hilbert transform, whose boundedness on $\ell^p(\Z;\fX)$ is equivalent to the UMD property of $\fX$ by \cite[Theorem 2.8]{BGM:86}. This completes the proof.
\end{proof}

\begin{remark}
The proof of \eqref{it:Linfty} $\Rightarrow$ \eqref{it:UMD} shows that \eqref{it:Linfty} could be further weakened by requiring only that $[b,\HT]$ is bounded for a single function $b$ satisfying \eqref{eq:bChoice}; in particular $b$ may be taken to be $C^\infty$ and periodic, or alternatively piecewise constant and periodic. On the other hand, it {\em cannot} be taken significantly simpler than this, as shown by the following:
\end{remark}

\begin{proposition}
Suppose that $b:\R\to\C$ satisfies the following properties for finitely many points $a_1<a_2<\ldots<a_n$ of $\R$:
\begin{enumerate}[\rm(i)]
  \item $b\in C^1((a_i,a_{i+1}))$ with one-sided limits at the end-points, for each $i=1,\ldots,n-1$;
  \item $b$ has (possibly different) constant values on each of the unbounded intervals $(-\infty,a_0)$ and $(a_n,\infty)$.
\end{enumerate}
Then $[b,\HT]$ is bounded on $L^p(\R;\fX)$ for every $p\in(1,\infty)$ and every Banach space $\fX$.
\end{proposition}

\begin{proof}
Let $I_0:=(-\infty,a_0)$, $I_i:=(a_i,a_{i+1})$ for $i\in\{1,\ldots,n-1\}$, and $I_n:=(a_n,\infty)$. Moreover, let $c_i$ be the constant value of $b$ on $I_i$ for $i\in\{0,n\}$. Then
\begin{equation*}
\begin{split}
      [b,\HT]f &=\sum_{i,j=0}^n\one_{I_i}[b,\HT](\one_{I_j}f).
\end{split}
\end{equation*}

For the terms with $i=j\in\{0,n\}$, we note that
\begin{equation*}
  \one_{I_i}[b,\HT](\one_{I_i}f)
  =\one_{I_i}[c_i,\HT](\one_{I_i}f)=0.
\end{equation*}
For $i=j\in\{1,\ldots,n-1\}$, we note that $[b,\HT]$ is an integral operator with kernel $(b(x)-b(y))/(x-y)$; for $x,y\in I_i$, this is uniformly bounded by $\Norm{\one_{I_i}b'}{\infty}<\infty$. Hence
\begin{equation*}
\begin{split}
  \Norm{\one_{I_i}[b,\HT](\one_{I_i} f)}{p}
  &\leq\abs{I_i}^{\frac1p}\Norm{[b,\HT](\one_{I_i} f)}{\infty} \\
  &\leq\abs{I_i}^{\frac1p}\Norm{\one_{I_i} b'}{\infty}\Norm{\one_{I_i} f}{1}
    \leq\abs{I_i}\Norm{\one_{I_i} b'}{\infty}\Norm{f}{p}.
\end{split}
\end{equation*}

For the remaining terms $i\neq j$, we note that the restriction of the Hilbert transform $f\mapsto\one_{I_i}\HT(\one_{I_j} f)$ on complementary intervals $I_i\cap I_j=\varnothing$ has a kernel that does not change sign, and hence it extends boundedly from $L^p(\R)$ to $L^p(\R;\fX)$ for an arbitrary Banach space $\fX$ (see \cite[Theorem 2.1.3]{HNVW1}). Thus,
\begin{equation*}
   \Norm{\one_{I_i}b\HT(\one_{I_j}f)}{p}
   \leq\Norm{b}{\infty}\Norm{\one_{I_i}\HT(\one_{I_j}f)}{p}
   \lesssim\Norm{b}{\infty}\Norm{f}{p}
\end{equation*}
and the estimate with $\HT b$ in place of $b\HT$ is similar. Also note that
$$\Norm{b}{\infty}=\max\{\abs{c_0},\abs{c_n},\Norm{\one_{I_i} b}{\infty}:i=1,\ldots,n-1\}$$
is finite for the functions of the type that we consider.
\end{proof}

\subsection{Operator commutators and $K$-convexity}

We now turn to the more delicate part in the proof of Theorem \ref{thm:commuUnbd}. While Proposition \ref{prop:commuUMD} provides us with a required function $B=b$ for non-UMD spaces, it also shows that a scalar-valued function will not suffice in the case of UMD spaces, and hence a genuinely operator-valued construction will be necessary. While additional steps will be needed, it turns out that we can make use of the examples already available in finite-dimensional Hilbert spaces. For concreteness, we recall the following precise statement:

\begin{theorem}[{\cite[Theorem 1.2]{NPTV:02}}]\label{thm:NPTV}
Let $E$ be a $d$-dimensional Hilbert space. Then there exists a function $B=B^*$ with
\begin{equation*}
  \Norm{B}{\BMO_{\so}(\R;\bddlin(E))}\leq 1,\qquad
  \Norm{[B,\HT]}{\bddlin(L^2(\R;E))}\gtrsim \log d,
\end{equation*}
where $\HT$ is the Hilbert transform and the implied constant is absolute.
\end{theorem}

\begin{corollary}\label{cor:NPTV}
In the situation of Theorem \ref{thm:NPTV}, we also have
\begin{equation*}
  \Norm{[B,\HT]}{\bddlin(L^p(\R;E))}\gtrsim \log d\qquad\forall\ p\in(1,\infty).
\end{equation*}
\end{corollary}

\begin{proof}
The adjoint of $[B,\HT]$ with respect to the standard duality of $L^2(\R;E)$ is
\begin{equation*}
  [B,\HT]^*=[\HT^*,B^*]=[-\HT,B]=[B,\HT];
\end{equation*}
thus, this operator is self-adjoint. Since $(L^p(\R;E))^*=L^{p'}(\R;E)$, we find that
\begin{equation*}
  \Norm{[B,\HT]}{\bddlin(L^p(\R;E))}
  =\Norm{[B,\HT]^*}{\bddlin((L^p(\R;E)^*)}
  =\Norm{[B,\HT]}{\bddlin(L^{p'}(\R;E))}.
\end{equation*}
Since $\frac12=\frac12(\frac{1}{p}+\frac{1}{p'})$, it follows from the Riesz--Thorin interpolation theorem that
\begin{equation*}
\begin{split}
  \Norm{[B,\HT]}{\bddlin(L^2(\R;E))}
  &\leq\Norm{[B,\HT]}{\bddlin(L^p(\R;E))}^{\frac12}\Norm{[B,\HT]}{\bddlin(L^{p'}(\R;E))}^{\frac12} \\
  &=\Norm{[B,\HT]}{\bddlin(L^p(\R;E))}.
\end{split}
\end{equation*}
The claim of the corollary is now immediate from Theorem \ref{thm:NPTV}.
\end{proof}

To proceed, we will need to embed the Hilbert space examples into a given Banach space $\fX$ in a quantitatively controlled manner. This will be accomplished with the help of the following result, which depends on the notion of the {\em $K$-convexity} (see \cite[Definition 4.3.9]{HNVW1} or \cite[Section 7.4.a]{HNVW2}) of a Banach space.

\begin{theorem}[\cite{FTJ:79}]\label{thm:FTJ}
Let $\fX$ be a $K$-convex normed space. Then there are positive constants $c,C$ depending only on $\fX$ with the following properties.
Every finite-dimensional subspace $F\subset\fX$ with $\dim F=n$ has a further subspace $G\subset F$ such that $\dim G\geq c n^c$ and
\begin{enumerate}[\rm(i)]
  \item\label{it:G=H} there is a Hilbert space $H$ and an isomorphism $\iota:G\to H$ with
\begin{equation*}
     \Norm{\iota}{}\Norm{\iota^{-1}}{}\leq 2;
\end{equation*}
  \item\label{it:P<C} there is a projection $P$ of $\fX$ onto $G$ with $\Norm{P}{}\leq C$.
\end{enumerate}
\end{theorem}

\begin{proof}
While this is implicitly contained in \cite{FTJ:79}, we indicate the required details to obtain the above formulation from the results explicitly stated in \cite{FTJ:79}.
First, \cite[Corollary 8.3]{FTJ:79} contains a similar statement except that the $K$-convexity is replaced by the assumption that both $\fX$ and $\fX^*$ have some finite {\em cotype} $q,q^*<\infty$ (see \cite[Definition 7.1.1]{HNVW2}) and conclusion \eqref{it:P<C} is replaced by $\Norm{P}{}\leq 27\mathscr K_n(\fX)$ for a quantity $\mathscr K_n(\fX)$ defined at the beginning of \cite[Section 8]{FTJ:79}.

Now, suppose that $\fX$ is $K$-convex, as assumed in Theorem \ref{thm:FTJ}. Then $\fX$ has some non-trivial type $p>1$ and some finite cotype $q<\infty$ by \cite[Proposition 7.4.12]{HNVW2}, and then also $\fX^*$ has finite cotype $q^*=p'<\infty$ by \cite[Proposition 7.1.13]{HNVW2}. Thus $\fX$ satisfies the assumptions, and hence the conclusions, of \cite[Corollary 8.3]{FTJ:79}. Moreover, \cite[Proposition 9.1]{FTJ:79} shows that the $K$-convexity is equivalent to $\sup_{n\in\N}\mathscr K_n(\fX)<\infty$. Thus, the conclusions of \cite[Corollary 8.3]{FTJ:79} take exactly the form stated in Theorem \ref{thm:FTJ}, with $C=27\sup_{n\in\N}\mathscr K_n(\fX)<\infty$ depending only on $\fX$.
\end{proof}

\begin{corollary}\label{cor:FTJ}
Let $\fX$ be a $K$-convex normed space with $\dim\fX=\infty$.
Then it has subspaces $G$ of any given finite dimension with properties \eqref{it:G=H} and \eqref{it:P<C}.
\end{corollary}

\begin{proof}
Since we can take any $n$ in Theorem \ref{thm:FTJ}, the said theorem guarantees that we can find $G$ with $\dim G\geq cn^c$ {\em larger than} any given number. It remains to see that we can find $G$ with $\dim G$ {\em equal to} any given number $m$. To see this, let $G\subset\fX$ and $H$ be the spaces provided by Theorem \ref{thm:FTJ} with $\dim G=\dim H>m$. Let $H'\subset H$ be a Hilbert space with $\dim H'=m$, and let $Q$ be the orthogonal projection of $H$ onto $H'$. If $\iota:G\to H$ is the isomorphism provided by Theorem \ref{thm:FTJ}, then $G':=\iota^{-1}H'\subset\iota^{-1}H=G\subset\fX$ is a subspace with $\dim G'=\dim H'=m$, and the restriction $\iota':=\iota |_{G'}$ is an isomorphism of $\iota':G'\to H'$ with
\begin{equation*}
  \Norm{\iota'}{}\Norm{(\iota')^{-1}}{}\leq\Norm{\iota}{}\Norm{\iota^{-1}}{}\leq 2.
\end{equation*}
Moreover, $P':=\iota^{-1}\circ Q\circ\iota\circ P:\fX\to G'$ is a projection of norm
\begin{equation*}
  \Norm{P'}{}\leq\Norm{\iota^{-1}}{}\Norm{Q}{}\Norm{\iota}{}\Norm{P}{}
  \leq \Norm{\iota^{-1}}{}\cdot 1\cdot \Norm{\iota}{} \cdot C \leq 2C.
\end{equation*}
Hence $G'\subset\fX$ with the given dimension $\dim G'=m$ satisfies \eqref{it:G=H} and \eqref{it:P<C} of Theorem \ref{thm:FTJ} with $2C$ in place of $C$, but this is just another constant that depends only on $\fX$.
\end{proof}

In the following lemma, we embed a sequence of finite-dimensional examples into a given $K$-convex space $\fX$, taking $\fX_1=\fX_2=\fX$ in the statement. We prove the result slightly more generally with two spaces $\fX_j$, since this requires almost no additional effort but will be convenient for a subsequent discussion.

\begin{lemma}\label{lem:badBk}
For $j=1,2$, let $\fX_j$ be infinite-dimensional $K$-convex Banach spaces. Then there exist functions $B_k\in\BMO_{\so}(\R;\bddlin(\fX_1,\fX_2))$ with $B_k^*\in\BMO_{\so}(\R;\bddlin(\fX_2^*,\fX_1^*))$ such that
\begin{equation*}
  \Norm{B_k}{\BMO_{\so}(\R;\bddlin(\fX_1,\fX_2))}\leq 1,\qquad   \Norm{B_k^*}{\BMO_{\so}(\R;\bddlin(\fX_2^*,\fX_1^*))}\leq 1,
\end{equation*}
and
\begin{equation*}
  \Norm{[B_k,\HT]}{\bddlin(L^p(\R;\fX_1),L^p(\R;\fX_2))}>k.
\end{equation*}
\end{lemma}

\begin{proof}
By Corollary \ref{cor:FTJ}, both spaces $\fX_j$ have subspaces $G_j$ with the properties \eqref{it:G=H} and \eqref{it:P<C} of Theorem \ref{thm:FTJ}, and $\dim G_1=\dim G_2=N$ can be taken equal and as large as we like. Let $\iota_j$, $H_j$, $P_j$, and $C_j$ have the same meaning as in the said properties with $G_j$ in place of $G$. By redefining the norm on $H_j$ with a multiplicative constant, we may assume, in place of $\Norm{\iota_j}{}\Norm{\iota_j^{-1}}{}\leq 2$, that $\Norm{\iota_j}{}=1$ and $\Norm{\iota_j^{-1}}{}\leq 2$. Since two Hilbert spaces of the same dimension are isometrically isomorphic, we may further take $H:=H_1=H_2$. Let $C:=\max(C_1,C_2)$.

 By Corollary \ref{cor:NPTV}, there is a function $B\in\BMO_{\so}(\R;\bddlin(H))$ with $B^*(=B)\in\BMO_{\so}(\R;\bddlin(H))$ of norm at most one in these spaces such that
 \begin{equation*}
  \Norm{[B,\HT]}{\bddlin(L^p(\R;H))}\gtrsim\log\dim H=\log\dim G_j=\log N.
\end{equation*}
For all $x\in\R$, we define
\begin{equation*}
  \tilde B(x) :=J_2\circ\iota^{-1}_2\circ B(x)\circ\iota_1\circ P_1 \in\bddlin(\fX_1,\fX_2),
\end{equation*}
where $J_2:G_2\to\fX_2$ is the injection $\fx\mapsto\fx$.
The adjoint is
\begin{equation*}
  \tilde B(x)^* =P_1^*\circ \iota_1^{*}\circ B(x)^*\circ\iota_2^{-*}\circ J_2^* \in\bddlin(\fX_2^*,\fX_1^*).
\end{equation*}

For $\fx\in\fX_1$,
\begin{equation*}
\begin{split}
  \Norm{\tilde B(\cdot)\fx}{\BMO(\R;\fX_2)}
  &\leq\Norm{J_2}{}\Norm{\iota_2^{-1}}{}\Norm{B(\cdot)(\iota_1(P_1\fx))}{\BMO(\R;H)} \\
  &\leq 1\cdot 2\cdot\Norm{B}{\BMO_{\so}(\R;\bddlin(H))}\Norm{\iota_1(P_1\fx))}{H},
\end{split}
\end{equation*}
where $\Norm{\iota_1(P_1\fx))}{H}\leq\Norm{\iota_1}{}\Norm{P_1}{}\Norm{\fx}{\fX_1}\leq C\Norm{\fx}{\fX_1}$.
Thus
\begin{equation}\label{eq:Bso}
  \Norm{\tilde B}{\BMO_{\so}(\R;\bddlin(\fX_1,\fX_2))}
  \leq 2C\Norm{B}{\BMO_{\so}(\R;\bddlin(H))}.
\end{equation}

Similarly, for $\fx^*\in\fX_2^*$,
\begin{equation*}
\begin{split}
  \Norm{\tilde B(\cdot)^*\fx^*}{\BMO(\R;\fX_1^*)}
  &\leq\Norm{P_1^*}{}\Norm{\iota_1^{*}}{}\Norm{\tilde B(\cdot)^*(\iota_2^{-*}(J_2^*\fx^*))}{\BMO(\R;H)} \\
  &\leq\Norm{P_1}{}\Norm{\iota_1}{}\Norm{\tilde B^*}{\BMO_{\so}(\R;\bddlin(H))}\Norm{\iota_2^{-*}(J_2^*\fx^*))}{H},
\end{split}
\end{equation*}
where $\Norm{P_1}{}\Norm{\iota_1}{}\leq C\cdot 1$, and
\begin{equation*}
  \Norm{\iota_2^{-*}(J_2^*\fx^*))}{H}\leq\Norm{\iota_2^{-*}}{}\Norm{J_2^*}{}\Norm{\fx^*}{\fX_2^*}
  =\Norm{\iota_2^{-1}}{}\Norm{J_2}{}\Norm{\fx^*}{\fX_2^*}
  \leq 2\Norm{\fx^*}{\fX^*}.
\end{equation*}
Thus
\begin{equation}\label{eq:Bstar}
  \Norm{\tilde B^*}{\BMO_{\so}(\R;\bddlin(\fX_1^*))}
  \leq 2C\Norm{B^*}{\BMO_{\so}(\R;\bddlin(H))}.
\end{equation}

Moreover,
\begin{equation*}
\begin{split}
  \Norm{[B,\HT]}{\bddlin(L^p(\R;H))}
  &=\Norm{\iota_2\circ\iota_2^{-1}[B,\HT]\iota_1\circ\iota_1^{-1} }{\bddlin(L^p(\R;H))} \\
  &\leq\Norm{\iota_2}{} \Norm{\iota_2^{-1} [B,\HT] \iota }{\bddlin(L^p(\R;G_1),L^p(\R;G_2))} \Norm{\iota_1^{-1}}{} \\
  &\leq 1\cdot \Norm{J_2\iota_2^{-1} [B,\HT] \iota_1 P_1}{\bddlin(L^p(\R;\fX_1),L^p(\R;\fX_2))} \cdot 2,
\end{split}
\end{equation*}
where the last estimate follows from the fact that $\iota_2^{-1} [B,\HT] \iota_1$ is the restriction of $J_2\iota^{-1}_2 [B,\HT] \iota_1 P_1$ to $L^p(\R;G_1)$. Since the operators $\iota_j$, $J_j$, and $P_j$ commute with $\HT$, it follows that
\begin{equation*}
  J_2\iota_2^{-1} [B,\HT] \iota_1 P_1
  =[J_2\iota_2^{-1} B\iota_1 P_1,\HT] =[\tilde B,\HT].
\end{equation*}
Thus, we obtain
\begin{equation}\label{eq:BHT}
  \log N\lesssim\Norm{[B,\HT]}{L^p(\R;H)}\leq 2\Norm{[\tilde B,\HT]}{\bddlin(L^p(\R;\fX_1),L^p(\R;\fX_2))}.
\end{equation}
From \eqref{eq:Bso}, \eqref{eq:Bstar}, and \eqref{eq:BHT}, it follows that $B_k:=\frac{1}{2C}\tilde B$ has the required properties, as soon as $N=\dim G_1=\dim G_2$ is chosen large enough, as we may by Corollary \ref{cor:FTJ}.
\end{proof}

We next wish to combine together the different functions $B_k$. For this, we will use the following construction. Given a sequence of integrable functions $\{b_k\}_{k\in\Z}$, where each $b_k$ is defined (at least) on $[4k,4k+1]$, we define a new function $b$ on $\R$ as follows, where $k\in\Z$ and $j\geq 0$ are arbitrary and $I_j:=[2^{-j-1},2^{-j})$:
\begin{equation}\label{eq:bCombo}
  b(x):=\begin{cases} b_k(x), & x\in 4k+[0,1], \\
  \ave{b_k}_{4k+I_j}, & x\in 4k-I_j, \\
  \ave{b_k}_{4k+1-I_j}, & x\in 4k+1+I_j, \\
  0, & x\in 4k+[2,3]. \end{cases}
\end{equation}

\begin{lemma}
Equation \eqref{eq:bCombo} defines $b$ uniquely at every $x\in\R$.
\end{lemma}

\begin{proof}
Clearly, the first case of \eqref{eq:bCombo} defines $b$ on $[4k,4k+1]$, and the last one on $[4k+2,4k+3]$.
Noting that $\bigcup_{j\geq 0}I_j=(0,1)$, the second case defines $b$ on $(4k-1,4k)$, the third one on $(4k+1,4k+2)$.
As $k\in\Z$ varies, this uniquely defines $b$ at every point $x\in\R$.
\end{proof}

Note that the values of $b$ on $(4k-1,4k)$ and $(4k+1,4k+2)$ are Whitney-type extensions of $b_k$ over the boundary points $4k$ and $4k+1$, respectively. While definition \eqref{eq:bCombo} can be made for any integrable functions $b_k$, it is particularly useful under the additional assumption that $\ave{b_k}_{[4k,4k+1]}=0$, in which case the definition of $b$ on $[4k+2,4k+3]$ is not as arbitrary as it might otherwise seem.

\begin{lemma}\label{lem:bCombo}
For any $k\in\Z$, let $J_k:=[4k,4k+1]$ and let $b_k\in\BMO(J_k;\fX)$ satisfy $\ave{b_k}_{J_k}=0$ and $\Norm{b_k}{\BMO(J_k;\fX)}\leq 1$. Let $b$ be a function on $\R$ defined by \eqref{eq:bCombo}. Then $b\in\BMO(\R;\fX)$ and $\Norm{b}{\BMO(\R;\fX)}\lesssim 1$
with the implicit positive constant independent of $b$.
\end{lemma}

\begin{proof}
We need to estimate $\inf_{c\in\fX}\fint_I\Norm{b(x)-c}{}\,dx$ for all finite intervals $I\subset\R$. We split this into several cases.

Case $I\subseteq[4k,4k+1]$ for some $k\in\Z$ is clear, since this is just the condition that $\Norm{b_k}{\BMO(I_k;\fX)}\leq 1$.

Case $I\subseteq 4k-[2^{-j-1},2^{-j}]$ for some $k\in\Z$ and $j\geq-1$ is also clear, as $b_k$ is essentially constant on $I$. Note that we intentionally include the case $j=-1$ here, i.e., $I\subseteq [4k-2,4k-1]=[4(k-1)+2,4(k-1)+3]$, in which case $b=0$ on $I$.

Case: $I\subseteq (4k-2,4k)$ meets exactly two intervals $4k-I_j$, where $j\geq-1$ (again, intentionally starting from this value). Let $k=0$ (without loss of generality), and $j=i,i+1$. To simplify writing, in the special case $j=-1$, we abuse notation and define $\ave{b_0}_{I_{-1}}:=0$. (Recall that $I_{-1}=[1,2)$, but $b_0$ is only defined on $[0,1]$, so this abuse should not cause any ambiguity.) With this convention, let $c=\ave{b_0}_{I_i}$. Then
\begin{equation*}
\begin{split}
  \fint_I\Norm{b-c}{}
  &=\frac{1}{\abs{I}}\sum_{j=i}^{i+1}\abs{I\cap(-I_j)}\Norm{\ave{b_0}_{I_j}-c}{} \\
  &=\frac{\abs{I\cap(-I_{j+1})}}{\abs{I}}\Norm{\ave{b_0}_{I_{i+1}}-\ave{b_0}_{I_i}}{}
  \lesssim 1\cdot\Norm{b_0}{\BMO(J_0;\fX)}\leq 1
\end{split}
\end{equation*}
by a well-known estimate for averages over close-by, comparable intervals.

Case: $I\subseteq (4k-2,4k)$ meets more than two intervals $4k-I_j$ with $j\geq -1$. Let again $k=0$, and let $i$ be the minimal index such that $I$ meets $-I_i$. Then
\begin{equation*}
 [2^{-i-2},2^{-i-1}) =  -I_{i+1}\subseteq I\subseteq \bigcup_{j=i}^\infty(-I_j)=-(0,2^{-i})=:-\hat I_i.
\end{equation*}
Thus $\abs{I}\geq\abs{I_{i+1}}=\frac14\abs{\hat I_i}$.

If $i=-1$, we again abuse notation and define $\ave{b_0}_{\hat I_{-1}}:=0$, keeping also the earlier abuse $\ave{b_0}_{I_{-1}}:=0$.
Then, with $c=\ave{b_0}_{\hat I_i}$, it follows that
\begin{equation*}
  \fint_I\Norm{b-c}{}
  \leq 4\fint_{-\hat I_i}\Norm{b-c}{}
  =\frac{4}{\abs{\hat I_i}}\sum_{j=i}^\infty\abs{I_j}\Norm{\ave{b_0}_{I_j}-c}{}=:A_i.
\end{equation*}
By the conventions that we made, if $i=-1$, then $\Norm{\ave{b_0}_{I_{-1}}-c}{}=\Norm{0-0}{}=0$, and hence the term $j=-1$ does not contribute. Moreover, we have $\abs{\hat I_{-1}}\geq\abs{\hat I_0}$. Hence $A_{-1}\leq A_0$, and so it suffices to estimate $A_i$ with $i\geq 0$. In this case, all $\ave{b_0}_{I_j}$ have their usual meaning and $\hat I_i\subseteq[0,1]$; thus, we can estimate
\begin{equation*}
  A_i\leq\frac{4}{\abs{\hat I_i}}\sum_{j=i}^\infty\abs{I_j}\fint_{I_j}\Norm{b_0-c}{}
  =4\fint_{\hat I_i}\Norm{b_0-c}{}\leq 4\Norm{b_0}{\BMO(J_0;\fX)}\leq 4.
\end{equation*}
Thus we have verified the case under consideration.

Case $I\subseteq [4k-2,4k]$ is now covered by the previous three cases.

Case $I\subseteq [4k+1,4k+3]$ is symmetric.

Case: $I\subseteq[4k-1,4k+1]$. We may assume that $I$ meets both $(4k-1,4k)$ and $[4k,4k+1]$, since otherwise this is already covered. Let $k=0$ (without loss of generality). Let $\tilde I:=I\cup(-I)$ so that $I\subseteq\tilde I$ and $\abs{\tilde I}\leq\abs{I}+\abs{-I}=2\abs{I}$. Let $i$ be the smallest index such that $\tilde I$ meets $I_i$. Then
\begin{equation*}
   \tilde I_{i+1}=[-2^{-i-1},2^{-i-1}]\subseteq\tilde I\subseteq[-2^{-i},2^{-i}]=:\tilde I_i,
\end{equation*}
and hence $\abs{\tilde I_i}=2\abs{\tilde I_{i+1}}\leq 2\abs{\tilde I}\leq 4\abs{I}$. Hence, with $c=\ave{b_0}_{[0,2^{-i}]}$,
\begin{equation*}
\begin{split}
  \fint_I\Norm{b-c}{}
  &\leq 4\fint_{\tilde I_i}\Norm{b-c}{} \\
  &=\frac{4}{\abs{\tilde I_i}}\Big(\int_0^{2^{-i}}\Norm{b_0-c}{}+\sum_{j=i}^\infty\abs{I_j}\Norm{\ave{b_0}_{I_j}-c}{}\Big) \\
  &\leq\frac{4}{\abs{\tilde I_i}}\Big(\int_0^{2^{-i}}\Norm{b_0-c}{}+\sum_{j=i}^\infty\abs{I_j}\fint_{I_j}\Norm{b_0-c}{}\Big) \\
  &=\frac{8}{\abs{\tilde I_i}}\int_0^{2^{-i}}\Norm{b_0-c}{}
  =4\fint_0^{2^{-i}}\Norm{b_0-c}{} \\
  &\leq 4\Norm{b_0}{\BMO(J_0;\fX)}\leq 4.
\end{split}
\end{equation*}

Case $I\subseteq[4k,4k+2]$ is symmetric to the previous one.

Since each $h\in\Z$ has one of the forms $4k-2$, $4k-1$, $4k$, and $4k+1$, we have  now covered all cases, where $I$ is contained in an interval of the form $[h,h+2]$ for some $h\in\Z$. Clearly, every interval of length $\abs{I}\leq 1$ is of this type. (If $h\leq\inf I<h+1$, then $I\subset[h,h+2]$.) Thus, it only remains to consider:

Case: $\abs{I}>1$. Let $m:=\floor{\inf I}$, $n:=\ceil{\sup I}$, and $J:=[m,n]$. Then $\abs{J}\leq\abs{I}+2\leq 3\abs{I}$. With $c=0$, we have
\begin{equation*}
  \fint_I\Norm{b-c}{}
  \leq 3\fint_J\Norm{b}{}
  =\frac{3}{\abs{J}}\sum_{i=m}^{n-1}\int_i^{i+1}\Norm{b}{}.
\end{equation*}
Recall that each $b_k$ has $\ave{b_k}_{[4k,4k+1]}=0$ by assumption, and then $b$ has $\ave{b}_{[i,i+1]}=0$ for all $i\in\Z$ by its construction \eqref{eq:bCombo}. Hence
\begin{equation*}
  \int_i^{i+1}\Norm{b}{}
  =\fint_{[i,i+1]}\Norm{b-\ave{b}_{[i,i+1]}}{}
  \leq 2\inf_{c\in\fX}\fint_{[i,i+1]}\Norm{b-c}{}
  \lesssim 1
\end{equation*}
by the cases of short intervals that we have already handled. Substituting back, it follows that
\begin{equation*}
  \fint_I\Norm{b-c}{}
  \lesssim\frac{1}{\abs{J}}\sum_{i=m}^{n-1}1=\frac{1}{\abs{J}}(n-m)=1,
\end{equation*}
recalling that $J=[m,n]$.

We have now completed the proof of the lemma.
\end{proof}

We are now ready to present the function that proves Theorem \ref{thm:commuUnbd} for any $K$-convex space $\fX$.

\begin{proposition}\label{prop:badB}
For $j=1,2$, let $\fX_j$ be infinite-dimensional $K$-convex Banach spaces. Then there exists a function
\begin{equation*}
   B\in\BMO_{\so}(\R;\bddlin(\fX_1,\fX_2))
\end{equation*}
with $B^*\in\BMO_{\so}(\R;\bddlin(\fX_2^*,\fX_1^*))$ such that $[B,\HT]$ is not bounded from $L^p(\R;\fX_1)$ to $L^p(\R;\fX_2)$.
\end{proposition}

\begin{proof}
Let $B_k\in\BMO_{\so}(\R;\bddlin(\fX_1,\fX_2))$ be the functions provided by Lemma \ref{lem:badBk}. By the estimate $\Norm{[B_k,\HT]}{\bddlin(L^p(\R;\fX_1),L^p(\R;\fX_2))}>k$, we can find compactly supported $f_k\in L^p(\R;\fX_1)$ and $g_k\in L^{p'}(\R;\fX_2^*)$ such that
\begin{equation*}
  \Norm{f_k}{L^p(\R;\fX_1)}=\Norm{g_k}{L^{p'}(\R;\fX_2^*)}=1
\end{equation*}
and
\begin{equation*}
  \abs{\pair{[B_k,\HT]f_k}{g_k}}>k.
\end{equation*}
By the translation and the dilation invariances of the Hilbert transform, it follows that
\begin{equation*}
  \pair{[B_k,\HT]f_k}{g_k}
  =\pair{[\tilde B_k,\HT]\tilde f_k}{\tilde g_k},
\end{equation*}
where
\begin{equation*}
  \tilde f_k(x):=r_k^{-\frac1p}f_k\Big(\frac{x-x_k}{r_k}\Big),\quad
  \tilde g_k(x):=r_k^{-\frac{1}{p'}}g_k\Big(\frac{x-x_k}{r_k}\Big),
\end{equation*}
\begin{equation*}
  \tilde B_k(x):=B_k\Big(\frac{x-x_k}{r_k}\Big),
\end{equation*}
and the $x_k$ and $r_k$ are arbitrary. These satisfy $\Norm{\tilde f_k}{p}=\Norm{f_k}{p}$ and $\Norm{\tilde g_k}{p'}=\Norm{g_k}{p'}$, and also $\tilde B_k$ and $\tilde B_k^*$ have the same $\BMO_{\so}$ norms as $B_k$ and $B_k^*$, respectively.

Let $x_k:=4k+\frac12$. Since $f_k$ and $g_k$ are compactly supported, a sufficiently small choice of $r_k$ guarantees that $\tilde f_k$ and $\tilde g_k$ are supported in $[4k,4k+1]$. After fixing these $x_k$ and $r_k$, we will finally replace $\tilde B_k$ by $\tilde B_k-\ave{\tilde B_k}_{[4k,4k+1]}$, noting that this affects neither the $\BMO_{\so}$ norms nor the commutator $[\tilde B_k,\HT]$. Thus, without loss of generality, we may assume that
\begin{equation}\label{eq:intBkWlog}
   \ave{\tilde B_k}_{[4k,4k+1]}=0.
\end{equation}

We then define a function $B$ by applying the algorithm \eqref{eq:bCombo} to the functions $\tilde B_k|_{[4k,4k+1]}$, where we may define simply $\tilde B_k:=0$ for $k\leq 0$. Since the algorithm \eqref{eq:bCombo} is linear, it is clear that, for each $\fx\in\fX_1$, the function $B(\cdot)\fx$ is the result of the same algorithm applied to the functions $\tilde B_k(\cdot)\fx |_{[4k,4k+1]}$, which satisfy
\begin{equation*}
  \Norm{\tilde B_k(\cdot)\fx}{\BMO([4k,4k+1];\fX_2)}
  \leq\Norm{\tilde B_k(\cdot)\fx}{\BMO_{\so}(\R;\bddlin(\fX_2))}\Norm{\fx}{\fX_1}\leq 1
\end{equation*}
for $\Norm{\fx}{\fX_1}\leq 1$. Hence Lemma \ref{lem:bCombo} (here we need \eqref{eq:intBkWlog}) guarantees that $ \Norm{B(\cdot)\fx}{\BMO(\R;\fX_2)}\lesssim 1$
for the same $\fx$, and thus $\Norm{B}{\BMO_{\so}(\R;\bddlin(\fX_1,\fX_2))}\lesssim 1$. Similarly, we also obtain
\begin{equation*}
  \Norm{B^*}{\BMO_{\so}(\R;\bddlin(\fX_2^*,\fX_1^*))}\lesssim 1.
\end{equation*}

It remains to check the unboundedness of $[B,\HT]$. Testing with $\tilde f_k,\tilde g_k$ from above, we find that
\begin{equation*}
\begin{split}
  \Norm{[B,\HT]}{\bddlin(L^p(\R;\fX_1),L^p(\R;\fX_2))}
  &\geq\abs{\pair{[B,\HT]\tilde f_k}{\tilde g_k}} \\
  &=\abs{\pair{[\tilde B_k,\HT]\tilde f_k}{\tilde g_k}}>k,
\end{split}
\end{equation*}
where the identity follows from the fact that $\tilde f_k$ and $\tilde g_k$ are supported in $[4k,4k+1]$, where $B=B_k$ by construction. Since the previous estimate is valid for every $k\in\N$, we see that $[B,\HT]$ cannot be bounded.
\end{proof}

\begin{proof}[Proof of Theorem \ref{thm:commuUnbd}]
If $\fX$ is non-UMD (resp.\ $K$-convex), then the required $B$ is produced by Proposition \ref{prop:commuUMD} (resp. \ref{prop:badB}, taking $\fX_1=\fX_2=\fX$). Since every UMD space is $K$-convex (see \cite[Proposition 4.3.10]{HNVW1}), these two cases cover all Banach spaces.
\end{proof}

\section{Unbounded Hilbert transforms}

The main result of this section is the following theorem, which shows the insufficiency of the $\A_p$ condition to guarantee the boundedness of the Hilbert transform on the operator-weighed $L^p(V;\fX)$ space, whenever $p\in(1,\infty)$ and $\fX$ is an arbitrary Banach space. In the special case that $p=2$ and $\fX$ is an infinite-dimensional Hilbert space, this is contained in \cite[Theorem 1.3]{GPTV:00} or \cite[Theorem 1.1]{GPTV:04}. Formally, these results show the existence of a sequence of finite-dimensional weights, whose $\A_2$ constants remain bounded while the norms of the weighted Hilbert transforms grow without limit as the dimension increases to infinity. However, since an infinite-dimensional Hilbert space can be written as an $\ell^2$ direct sum containing all the finite-dimensional versions, the infinite-dimensional counterexample readily follows.

\begin{theorem}\label{thm:Hunbd}
Let $p\in(1,\infty)$ and $\fX$ be an arbitrary Banach space.
Then there exists $V\in\A_p(\R;\bddlin(\fX))$ such that the Hilbert transform is unbounded on $L^p(V;\fX)$.
\end{theorem}

As in Theorem \ref{thm:commuUnbd}, we will consider two cases for the Banach space $\fX$. However, this time the case that $\fX$ is not UMD is entirely trivial: in this case, the Hilbert transform is unbounded on $L^p(\R;\fX)$, and hence we can simply take $V\equiv I$, the identity operator on $\fX$.

In the case that $\fX$ is $K$-convex, the proof of Theorem \ref{thm:Hunbd} will make use of the constructions in the proof of Theorem \ref{thm:commuUnbd} through a close link of the two types of (un)boundedness results that goes back to \cite{GPTV:01}, and will be reviewed shortly. However, while the essence of this link remains in general Banach spaces, there is a devil in the details, which prevents a soft deduction of Theorem \ref{thm:Hunbd} from Theorem \ref{thm:commuUnbd} that one can do in a Hilbert space (and more generally, but not always), and seems to necessitate revisiting also several details of the proof.

The following lemma gives the basic construction of $\A_p$ weights from BMO functions. The case that $p=2$ and $\fX_1=\fX_2$ is a Hilbert space goes back to \cite[Lemma 2.1]{GPTV:01}. To streamline the statement of the lemma, we introduce the $L^p$ seminorm modulo constant functions,
\begin{equation*}
  \Norm{b}{\aveL^p_0(Q;\fX)}:=\Big(\fint_Q\Norm{b(x)-\ave{b}_Q}{\fX}^p \,dx\Big)^{\frac1p},
\end{equation*}
and its strong-operator version
\begin{equation*}
  \Norm{B}{\aveL^p_{0,\so}(Q,\bddlin(\fX_1,\fX_2))}:=\sup_{\Norm{\fx_1}{\fX_1}\leq 1}\Norm{B(\cdot)\fx_1}{\aveL^p_0(Q;\fX_2)}.
\end{equation*}

\begin{lemma}\label{lem:Ap<BMO}
Let $p\in(1,\infty)$. For $j=1,2$, let $\fX_j$ be Banach spaces and
\begin{equation*}
   B\in L^p_{\so}(Q;\bddlin(\fX_1,\fX_2))
\end{equation*}
be a function whose pointwise adjoint satisfies $B^*\in L^{p'}_{\so}(Q;\bddlin(\fX_2^*,\fX_1^*))$. Let $\fX:=\fX_1\oplus\fX_2$, and consider the $\bddlin(\fX)$-valued weight
\begin{equation}\label{eq:VfromB}
  V(x)=\begin{pmatrix} I & 0 \\ B(x) & I \end{pmatrix}.
\end{equation}
Then, for all $\fx^*\in\fX^*$ and $\fx\in\fX$, the following bounds are valid:
\begin{equation}\label{eq:Ap<BMOspec}
  \rho_{\aveL^{p'}(Q,V^{-*})}(\fx^*)\leq (1+\Norm{B^*}{\aveL^{p'}_{\so}(Q;\bddlin(\fX_2^*,\fX_1^*))})\Norm{\fx^*}{\fX^*},
\end{equation}
\begin{equation}\label{eq:Ap<BMOspec2}
   \rho_{\aveL^p(Q,V)}(\fx)\leq (1+\Norm{B}{\aveL^p_{\so}(Q,\bddlin(\fX_1,\fX_2))})\Norm{\fx}{\fX},
\end{equation}
and
\begin{equation}\label{eq:Ap<BMO}
  \frac{\rho_{\aveL^{p'}(Q,V^{-*})}(\fx^*)}{\rho_{\aveL^p(Q,V)}^*(\fx^*)}
  \leq 1  + \Norm{B}{\aveL^p_{0,\so}(Q;\bddlin(\fX_1,\fX_2))}
  + \Norm{B^*}{\aveL^{p'}_{0,\so}(Q;\bddlin(\fX_2^*,\fX_1^*))}.
\end{equation}
In particular, if
\begin{equation*}
   B\in\BMO_{\so}(\R^n;\bddlin(\fX_1,\fX_2)), \quad B^*\in\BMO_{\so}(\R^n;\bddlin(\fX_2^*,\fX_1^*))
\end{equation*}
then $V\in\A_p(\R^n;\bddlin(\fX))$ and
\begin{equation*}
  [V]_{\A_p}\lesssim 1+\Norm{B}{\BMO_{\so}(\R^n;\bddlin(\fX_1,\fX_2))}
  +\Norm{B^*}{\BMO_{\so}(\R^n;\bddlin(\fX_2^*,\fX_1^*))}.
\end{equation*}
\end{lemma}

Note that \eqref{eq:Ap<BMOspec} and \eqref{eq:Ap<BMOspec2} have $\aveL^p$ norms, while \eqref{eq:Ap<BMO} has $\aveL^p_0$.

\begin{proof}
Note that the last assertion concerning $\A_p$ and $\BMO$ is immediate from \eqref{eq:Ap<BMO}, the relevant definitions, and the John--Nirenberg inequality. Thus we concentrate on proving \eqref{eq:Ap<BMOspec}, \eqref{eq:Ap<BMOspec2}, and \eqref{eq:Ap<BMO}.

For the matrix weight $V$ as defined, it is immediate to compute
\begin{equation*}
  V^{-1}=\begin{pmatrix} I & 0 \\ -B & I \end{pmatrix},\qquad
  V^{-*}=\begin{pmatrix} I & -B^* \\ 0 & I \end{pmatrix}.
\end{equation*}
The assumptions on $B$ and $B^*$ imply in particular the qualitative conditions
\begin{equation*}
  V\in L^p_{\so}(Q;\bddlin(\fX)),\qquad
  V^{-*}\in L^{p'}_{\so}(Q;\bddlin(\fX^*)),
\end{equation*}
so that the left-hand side of \eqref{eq:Ap<BMO} is well defined, as well as
\begin{equation}\label{eq:Vinv}
    V^{-1}\in L^{p}_{\so}(Q;\bddlin(\fX)),
\end{equation}
which will be used shortly.

The bound \eqref{eq:Ap<BMOspec} follows simply by taking the $\aveL^{p'}(Q)$ norms of
\begin{equation*}
 \Norm{V^{-*}(x)\fx^*}{\fX^*}\leq\Norm{\fx^*}{\fX^*}+\Norm{B^*(x)\fx_2^*}{\fX_1^*},
\end{equation*}
and noting that $\Norm{\fx_2^*}{\fX_2^*}\leq\Norm{\fx^*}{\fX^*}$.
Similarly, \eqref{eq:Ap<BMOspec2} is obtained by taking the $\aveL^{p}(Q)$ norms
\begin{equation*}
  \Norm{V(x)\fx}{\fX}\leq\Norm{\fx}{\fX}+\Norm{B(x)\fx_1}{\fX_2}.
\end{equation*}
Thus \eqref{eq:Ap<BMOspec} and \eqref{eq:Ap<BMOspec2} are proved, and we turn to the proof of \eqref{eq:Ap<BMO}.

The numerator of \eqref{eq:Ap<BMO} can be expressed as
\begin{equation*}
  \rho_{\aveL^{p'}(Q,V^{-*})}(\fx^*)
  =\Big(\fint_Q\Norm{V^{-*}(\cdot)\fx^*}{}^{p'}\Big)^{\frac{1}{p'}}
  =\sup_{\genfrac{}{}{0pt}{}
  {\Norm{f}{\aveL^p(Q;\fX)}\leq 1}
  {f\text{ simple}}}\Babs{\fint_Q\pair{V^{-*}(\cdot)\fx^*}{f}}.
\end{equation*}
Thanks to \eqref{eq:Vinv}, the integral on the right-hand side may be written as
\begin{equation*}
  \fint_Q\pair{V^{-*}(\cdot)\fx^*}{f}
  =\fint_Q\pair{\fx^*}{V^{-1}f}
  =\Bpair{\fx^*}{\fint_Q V^{-1}f}
  =:\pair{\fx^*}{\fx_f}.
\end{equation*}
Hence, by the definition of the dual norm,
\begin{equation*}
  \Babs{\fint_Q\pair{V^{-*}(\cdot)\fx^*}{f}}
  \leq\rho_{\aveL^p(Q,V)}^*(\fx^*)\rho_{\aveL^p(Q,V)}(\fx_f),
\end{equation*}
where
\begin{equation*}
\begin{split}
  \rho_{\aveL^p(Q,V)}(\fx_f)
  &=\Big(\fint_Q\Norm{V(x)\fx_f}{}^p \,dx\Big)^{\frac1p} \\
  &=\Big(\fint_Q\Norm{V(x)\fint_Q V(y)^{-1}f(y)dy}{}^p \,dx\Big)^{\frac1p} \\
  &=\sup_{\Norm{g}{\aveL^{p'}(Q;\fX^*)}\leq 1}\Babs{\fint_Q\fint_Q \pair{V(x)V(y)^{-1}f(y)}{g(x)} \,dy\, dx}.
\end{split}
\end{equation*}
For $V$ as in the assumptions, we have
\begin{equation*}
\begin{split}
  V(x)V(y)^{-1} &=\begin{pmatrix} I & 0 \\ B(x)  & I \end{pmatrix}
  \begin{pmatrix} I & 0 \\ -B(y)  & I \end{pmatrix} = \begin{pmatrix} I & 0 \\ B(x)-B(y) & I \end{pmatrix} \\
 & = \begin{pmatrix} I & 0 \\ 0 & I \end{pmatrix} + \begin{pmatrix} 0 & 0 \\ B(x)-\ave{B}_Q  & 0 \end{pmatrix}
  - \begin{pmatrix} 0 & 0 \\ B(y)-\ave{B}_Q  & 0 \end{pmatrix}.
\end{split}
\end{equation*}
Hence
\begin{equation*}
\begin{split}
  \fint_Q &\fint_Q \pair{V(x)V(y)^{-1}f(y)}{g(x)} \,dy\, dx \\
  &=\fint_Q\fint_Q \pair{f(y)}{g(x)} \,dy \,dx
  +\fint_Q\fint_Q \pair{(B(x)-\ave{B}_Q)f_1(y)}{g_2(x)} \,dy \,dx \\
  &\qquad-\fint_Q\fint_Q \pair{f_1(y)}{(B^*(y)-\ave{B^*}_Q)g_2(x)}\, dy \,dx\\
  &=:{\rm I}+{\rm II}-{\rm III},
\end{split}
\end{equation*}
where $f=(f_1, f_2)$ and $g=(g_1,g_2)$. We can estimate
\begin{equation*}
\begin{split}
  \abs{\rm II} &\leq\fint_Q\Big(\fint_Q\Norm{[B(x)-\ave{B}_Q]f_1(y)}{\fX_2}^p \,dx\Big)^{\frac1p}\Norm{g_2}{\aveL^{p'}(Q;\fX_2^*)}\, dy \\
  &\leq\fint_Q\Norm{B}{\aveL_{0,\so}^p(Q;\bddlin(\fX_1,\fX_2))}
  \Norm{f_1(y)}{\fX_1}\Norm{g_2}{\aveL^{p'}(Q;\fX_2^*)}\, dy \\
  &=\Norm{B}{\aveL_{0,\so}^p(Q;\bddlin(\fX_1,\fX_2))}
  \Norm{f_1}{\aveL^1(Q;\fX_1)}\Norm{g_2}{\aveL^{p'}(Q;\fX_2^*)}.
\end{split}
\end{equation*}
Similarly,
\begin{equation*}
  \abs{\mathrm{III}}\leq\Norm{B^*}{\aveL_{0,\so}^{p'}(\R^n;\bddlin(\fX_2^*,\fX_1^*))}
  \Norm{f_1}{\aveL^p(Q;\fX_1)}\Norm{g_2}{\aveL^1(Q;\fX_2^*)}
\end{equation*}
and of course
\begin{equation*}
  \abs{\mathrm{I}}\leq\Norm{f}{\aveL^1(Q;\fX)}\Norm{g}{\aveL^1(Q;\fX^*)}.
\end{equation*}

Substituting back and using H\"older's inequality, we find that
\begin{equation*}
\begin{split}
  \rho_{\aveL^p(Q,V)}(\fx_f)
  &\leq(1+\Norm{B}{\aveL_{0,\so}^p(\R^n;\bddlin(\fX_1,\fX_2))}
  +\Norm{B^*}{\aveL_{0,\so}^{p'}(\R^n;\bddlin(\fX_2^*\fX_1^*))})
  \Norm{f}{ \aveL^p(Q;\fX)} \Norm{g}{\aveL^{p'}(Q;\fX^*)},
\end{split}
\end{equation*}
where the last factor is bounded by $1$ by the conditions on $f$ and $g$. Substituting further back, this proves the claimed bound \eqref{eq:Ap<BMO}.
\end{proof}

\begin{remark}\label{rem:naiveFails}
Before proceeding with the proof of Theorem \ref{thm:Hunbd} for a general $K$-convex Banach space $\fX$,
let us indicate an easier argument under an additional assumption that the Banach space $\fX$ is {\em decomposable}, meaning that it admits a representation $\fX=\fX_1\oplus\fX_2$ as a direct sum of two {\em infinite-dimensional} Banach spaces $\{\fX_j\}_{j=1}^2$. Let $B\in\BMO_{\so}(\R;\bddlin(\fX_1,\fX_2))$ with $B^*\in\BMO_{\so}(\R;\bddlin(\fX_2^*,\fX_1^*))$ be a function, provided by Proposition \ref{prop:badB}, such that $[B,\HT]$ is not bounded from $L^p(\R;\fX_1)$ to $L^p(\R;\fX_2)$. Let $V:\R\to\bddlin(\fX)$ be defined by \eqref{eq:VfromB}; thus $V\in\A_p(\R;\bddlin(\fX))$ by Lemma \ref{lem:Ap<BMO}.

Note that $\HT$ is bounded on $L^p(V;\fX)$ if and only if $V\HT V^{-1}$ is bounded on $L^p(\R;\fX)$. Relative to the direct sum splitting $\fX=\fX_1\oplus\fX_2$ and the induced splitting $L^p(\R;\fX)=L^p(\R;\fX_1)\oplus L^p(\R;\fX_2)$, this latter operator is given by
\begin{equation*}
   V\HT V^{-1}=\begin{pmatrix} I & 0 \\ B & I \end{pmatrix}\HT\begin{pmatrix} I & 0 \\ -B & I \end{pmatrix}
   =\begin{pmatrix} \HT & 0 \\ [B,\HT] & \HT \end{pmatrix}.
\end{equation*}
Thus, a necessary condition for the boundedness of $\HT:L^p(V;\fX)\to L^p(V;\fX)$ is the boundedness of $[B,\HT]:L^p(\R;\fX_1)\to L^p(\R;\fX_2)$, which was specifically excluded by Proposition \ref{prop:badB}. Hence, the proof of Theorem \ref{thm:Hunbd} can be completed in this soft way in the special case of {\em decomposable} Banach spaces.

While this property holds for most classical Banach spaces, it does not hold in general. Indeed, there even exist {\em hereditarily indecomposable} (H.I.) Banach spaces $\fX$, meaning that not only $\fX$ itself but all of its infinite-dimensional closed subspaces are indecomposable (i.e., not decomposable). The first such example is due to Gowers and Maurey \cite{GM:93}, and others have been found since then. Notably, Ferenczi \cite{Fere:97} has constructed an H.I.\ space $\fX$ that is {\em uniformly convex} and hence in particular $K$-convex. (See the end of this Remark for this implication.) Thus, the above construction of $V$ is not possible in arbitrary $K$-convex spaces, and not even in the smaller class of uniformly convex spaces.

To complete the proof of Theorem \ref{thm:Hunbd}, we would actually only need this construction in the still smaller class of UMD spaces. Whether every UMD space is decomposable appears to be open. If this turns out to be the case, then a simpler proof of Theorem \ref{thm:Hunbd} will be available, and we completed it already. However, in the lack of knowing such a result, we adopt a different approach below, which does not require decomposability.

The idea is the following: Instead of using the finite-dimensional functions $B_k$ to build an infinite-dimensional $B$ and then using this to build the weight $V$ by Lemma \ref{lem:Ap<BMO} (as we did above, assuming decomposability), we use each finite-dimensional function $B_k$ to build a finite-dimensional weight $V_k$ by Lemma \ref{lem:Ap<BMO}, and then we use the finite-dimensional weights $V_k$ to build an infinite-dimensional weight $V$ in a similar way as the functions $B_k$ were used to build $B$. This turns out to work, as we will see below.

[The implication from uniform convexity to $K$-convexity, which we mentioned above, is well known, but seems difficult to quote as a single result. It can be obtained, e.g., from the following chain: if $\fX$ is uniformly convex, then $\fX^*$ is {\em uniformly smooth} (by \cite[Proposition 10.21]{Pisier:book}), thus $\fX^*$ admits an equivalent $p$-smooth norm for some $p>1$ (see \cite[Definition 10.18 and 10.30, and Theorem 10.25]{Pisier:book}), and therefore $\fX^*$ has type $p>1$ (by \cite[Proposition 10.39]{Pisier:book}); hence $\fX^*$ is $K$-convex (see \cite[Theorem 7.4.23]{HNVW2}) and thus so is $\fX$ (by \cite[Proposition 7.4.5]{HNVW2}).]
\end{remark}

Having explained in Remark \ref{rem:naiveFails} why a soft construction of the weight is not available in general, we now turn to the details of the alternative approach that was also sketched above.

\begin{proposition}\label{prop:badV}
Let $\fX$ be an infinite-dimensional $K$-convex Banach space. Then there exists $V\in\A_p(\R;\bddlin(\fX))$ such that the Hilbert transform is not bounded on $L^p(V;\fX)$.
\end{proposition}

\begin{proof}
The construction of $V$ involves a variation of the construction of $B$ in Proposition \ref{prop:badB}. By Corollary \ref{cor:FTJ}, for every $k\geq 1$, the space $\fX$ has a subspace $G_k$ with the properties \eqref{it:G=H} and \eqref{it:P<C} of Theorem \ref{thm:FTJ}, and $\dim G_k=2N_k$ can be taken {\em even} and as large as we like. Let $\iota_k$, $P_k$, $H_k$,
and $C$ have the same meaning as in Theorem \ref{thm:FTJ} with $G_k$ in place of $G$; note that $C$ depends only on $\fX$, and is hence independent of $k$. Since $\dim H_k=\dim G_k=2N$, we can write it as a direct sum $H_k=E_k\otimes E_k$ of two copies of the same $N_k$-dimensional Hilbert space $E_k$.

\subsubsection*{Auxiliary $\BMO$ functions}

By Corollary \ref{cor:NPTV}, there is a function $$B_k\in\BMO_{\so}(\R;\bddlin(E_k))$$ with $B^*(=B)\in\BMO_{\so}(\R;\bddlin(E_k))$ of norm at most one in these spaces such that
\begin{equation*}
  \Norm{[B_k,\mathcal H]}{\bddlin(L^p(\R;E_k))}\gtrsim\log\dim E_k=\log N_k.
\end{equation*}
Replacing $B_k$ by its translated and dilated version with the same estimates, as in the proof of Proposition \ref{prop:badB}, we obtain the existence of $f_k\in L^p(\R;E_k)$ and $g_k\in L^{p'}(\R;E_k)$ of norm $1$, both supported in $J_k:=[4k,4k+1]$, such that
\begin{equation}\label{eq:BkH>Nk}
  \abs{\pair{[B_k,\mathcal H]f_k}{g_k}}\gtrsim \log N_k.
\end{equation}
Moreover, subtracting a constant, which does not affect the commutator nor the $\BMO$ norms, we may assume that $\ave{B_k}_{J_k}=0$.

For each $h\geq 1$, we now apply the algorithm \eqref{eq:bCombo} to the sequence $\{b_k=\delta_{h,k} B_h\}_{k\in\Z}$ to produce a function $b=:\tilde B_h$, which agrees with $B_h$ on $J_h$, consists of Whitney averages of $B_h|_{J_j}$ on the two adjacent unit intervals, and vanishes everywhere else. Applying Lemma \ref{lem:bCombo} to each $b_k(\cdot)\fx$ and $b_k(\cdot)^*\fx$, with $\Norm{\fx}{E}\leq 1$, in place of $b_k$, we deduce that
\begin{equation}\label{eq:tildeBkBMO}
  \Norm{\tilde B_k}{\BMO_{\so}(\R;\bddlin(E_k))}\lesssim1,\quad   \Norm{\tilde B_k^*}{\BMO_{\so}(\R;\bddlin(E_k))}\lesssim 1,
\end{equation}
while \eqref{eq:BkH>Nk} still holds with $\tilde B_k$ in place of $B_k$, since the supports of $f_k$ and $g_k$ ensure that the left-hand side only sees the values of $B_k$ on $J_k$, where it agrees with $\tilde B_k$.

\subsubsection*{Initial weighs $\tilde V_k$ on $H_k$}

We now define a weight $\tilde V_k:\R\to\bddlin(H_k)=\bddlin(E_k\oplus E_k)$ by
setting, for any $x\in\R$,
\begin{equation}\label{eq:tildeVk}
  \tilde V_k(x):=\begin{pmatrix} I & 0 \\ \tilde B_k(x) & I \end{pmatrix}.
\end{equation}
From \eqref{eq:tildeBkBMO} and Lemma \ref{lem:Ap<BMO}, we deduce that $\tilde V_k\in\A_p(\R;\bddlin(H_k))$ with
\begin{equation}\label{eq:tildeVkAp}
  [\tilde V_k]_{\A_p(\R;\bddlin(H_k))}\lesssim 1.
\end{equation}
Since $\tilde B_k(x)\equiv 0$ for $x\in\R\setminus 3J_k$, it follows that
\begin{equation*}
  \tilde V_k(x)\equiv I\quad\forall\ x\in\R\setminus 3J_k.
\end{equation*}

Recall that $\tilde B_k=B_k$ on $J_k$, while $\tilde B_k$ is obtained from the Whitney averages of $B_k$ on the two adjacent intervals $J_k\pm 1$. Hence, for each of these intervals $J\in \{J_k,J_k\pm 1\}$, we have
\begin{equation*}
  \ave{\tilde B_k}_J=\ave{B_k}_{J_k}=0,
\end{equation*}
where the last identity was part of the construction of $B_k$.

By Lemma \ref{lem:Ap<BMO} again, we have the following additional special bounds
\begin{equation*}
\begin{split}
  \frac{\rho_{\aveL^{p'}(J,\tilde V_k^{-*})}(\fy)}{\Norm{\fy}{H_k}}
  &\leq 1+\Norm{\tilde B_k^*}{\aveL^{p'}_{\so}(J;\bddlin(E_k))}  \\
  &= 1+\Norm{\tilde B_k^*}{\aveL^{p'}_{0,\so}(J;\bddlin(E_k))} \qquad\text{since }\ave{\tilde B_k^*}_{J}=0 \\
  &\lesssim 1+\Norm{\tilde B_k^*}{\BMO_{\so}(\R;\bddlin(E_k))} \lesssim 1
\end{split}
\end{equation*}
and
\begin{equation*}
\begin{split}
  \frac{\rho_{\aveL^p(J,\tilde V_k)}(\fy)}{\Norm{\fy}{H_k}}
  &\leq 1+\Norm{\tilde B_k}{\aveL^{p}_{\so}(J;\bddlin(E_k))}  \\
  &= 1+\Norm{\tilde B_k}{\aveL^{p}_{0,\so}(J;\bddlin(E_k))} \qquad\text{since }\ave{\tilde B_k}_{J}=0 \\
  &\lesssim1+\Norm{\tilde B_k}{\BMO_{\so}(\R;\bddlin(E_k))}
  \lesssim 1;
\end{split}
\end{equation*}
thus
\begin{equation}\label{eq:specVk}
  \rho_{\aveL^{p'}(J,\tilde V_k^{-*})}(\fy)
  +\rho_{\aveL^p(J,\tilde V_k)}(\fy)  \lesssim\Norm{\fy}{H_k}\quad
  \text{for }J\in \{J_k,J_k\pm 1\}.
\end{equation}

\subsubsection*{Lifted weighs $V_k$ on the original space $\fX$}

We now lift these weights to the original space $\fX$ by setting
\begin{equation}\label{eq:VkDef}
  V_k:=I_k\oplus \iota_k^{-1} \tilde V_k \iota_k\quad
  \text{with respect to}\quad\fX=(I-P_k)\fX\oplus P_k\fX,
\end{equation}
where $I_k$ is the identity on $(I-P_k)\fX$, and we recall that $P_k\fX=G_k$.

From Lemma \ref{lem:ApIsom}, the bound $\Norm{\iota_k^{-1}}{}\Norm{\iota_k}{}\leq 2$, and \eqref{eq:tildeVkAp}, it follows that
\begin{equation*}
  [\iota_k^{-1}\tilde V_k\iota_k]_{\A_p(\R;\bddlin(G_k))}
  \leq\Norm{\iota_k^{-1}}{}\Norm{\iota_k}{}[\tilde V_k]_{\A_p(\R;\bddlin(H_k))}\lesssim 1.
\end{equation*}
Then Lemma \ref{lem:ApDirSum} shows that
\begin{equation}\label{eq:VkAp}
    [V_k]_{\A_p(\R;\bddlin(\fX))}\approx[I_k]_{\A_p(\R;\bddlin((I-P_k)\fX))}+[\iota_k^{-1}\tilde V_k\iota_k]_{\A_p(\R;\bddlin(G_k))} \lesssim 1.
\end{equation}
Noting that
\begin{equation*}
  V_k^{-*}:=I_k^*\oplus \iota_k^{*} \tilde V_k^{-*} \iota_k^{-*}\quad
  \text{with respect to}\quad\fX^*=(I-P_k^*)\fX^*\oplus P_k^*\fX^*,
\end{equation*}
we also deduce, for $J\in\{J_k,J_k\pm 1\}$ and $\fx^*\in\fX^*$, that
\begin{equation*}
\begin{split}
  \rho_{\aveL^{p'}(J,V_k^{-*})}(\fx^*)
  &\lesssim\Norm{(I-P_k^*)\fx^*}{\fX^*}+\Norm{\iota_k^*}{}\rho_{\aveL^{p'}(\tilde V_k^{-*})}(\iota_k^{-*}P_k^*\fx^*) \\
  &\lesssim\Norm{\fx^*}{\fX^*}+\Norm{\iota_k^*}{}\Norm{\iota_k^{-*}P_k^*\fx^*}{H_k} \qquad\text{by \eqref{eq:specVk}} \\
  &\lesssim(1+\Norm{\iota_k^*}{}\Norm{\iota_k^{-*}}{})\Norm{\fx^*}{\fX^*}\lesssim \Norm{\fx^*}{\fX^*}.
\end{split}
\end{equation*}
Similarly,
\begin{equation*}
\begin{split}
  \rho_{\aveL^p(J,V_k)}(\fx)
  &\lesssim\Norm{(I-P_k)\fx}{\fX}+\Norm{\iota_k^{-1}}{}\rho_{\aveL^p(J,\tilde V_k)}(\iota_k P_k\fx) \\
  &\lesssim\Norm{\fx}{\fX}+\Norm{\iota_k^{-1}}{}\Norm{\iota_k P_k\fx}{H_k} \qquad\text{by \eqref{eq:specVk}}\\
  &\lesssim(1+\Norm{\iota_k^{-1}}{}\Norm{\iota_k}{})\Norm{\fx}{\fX}\lesssim \Norm{\fx}{\fX}.
\end{split}
\end{equation*}
Thus we have shown that
\begin{equation}\label{eq:Vk=I}
 \begin{cases} \rho_{\aveL^{p'}(J,V_k^{-*})}(\fx^*) \lesssim\Norm{\fx^*}{\fX^*} & \forall\ \fx^*\in\fX^*, \\
   \quad\  \rho_{\aveL^p(J,V_k)}(\fx) \lesssim\Norm{\fx}{\fX} & \forall\ \fx\in\fX
 \end{cases}
\end{equation}
for $J\in\{J_k,J_k\pm 1\}$. If $J=[h,h+1]$ with $h\in\Z$ in any other unit interval besides the previous three, then $V_k\equiv I$ on $J$, so that both the left and the right-hand sides of \eqref{eq:Vk=I} reduce to just $\Norm{\fx^*}{\fX^*}$, and \eqref{eq:Vk=I} holds trivially also in this case. Thus \eqref{eq:Vk=I} is valid for all $J=[h,h+1]$, $h\in\Z$.

\subsubsection*{The combined weight $V$}

We now define a single weight $V$ on $\R$ by setting
\begin{equation}\label{eq:V=Vk}
  V(x):=\begin{cases} V_k(x) & \text{if }x\in[4k-1,4k+2]\text{ for some }k\geq 1, \\
  I & \text{otherwise}.\end{cases}
\end{equation}
Note that we actually have $V(x)=V_k(x)$ if $x\in[4k-2,4k+3]$, since $V_k(x)=I$ outside $[4k-1,4k+2]$. To show that $V\in\A_p(\R;\bddlin(\fX))$, we need to verify that
\begin{equation}\label{eq:VAp2show}
  \rho_{\aveL^{p'}(J;V^{-*})}(\fx^*)\lesssim\rho_{\aveL^p(J;V)}^*(\fx^*)
\end{equation}
for all finite intervals $J\subset\R$.

Suppose first that $\abs{J}\leq 1$. With $m:=\floor{\inf J}\in\Z$, it follows that $J\subset[m,m+2]$. Now $m\in 4k+\{-2,-1,0,1\}$ for some $k\in\Z$, and hence $J\subset[m,m+2]\subset[4k-2,4k+3]$. Since $V\equiv V_k$ on this interval, \eqref{eq:VAp2show} follows from \eqref{eq:VkAp}.

Hence it only remains to consider $\abs{J}>1$. With $m:=\floor{\inf J}$ and $n:=\ceil{\sup J}$, we have $J\subseteq[m,n]=:\hat J$ and $\abs{\hat J}\leq 2+\abs{J}\leq 3\abs{J}$. Then
\begin{equation*}
\begin{split}
  \rho_{\aveL^{p'}(J,V^{-*})}(\fx^*)
  &=\Big(\frac{1}{\abs{J}}\int_J\Norm{V^{-*}(\cdot)\fx^*}{\fX^*}^{p'}\Big)^{\frac{1}{p'}} \\
  &\leq\Big(\frac{3}{\abs{\hat J}}\sum_{h=m}^{n-1}\int_{[h,h+1]}\Norm{V^{-*}(\cdot)\fx^*}{\fX^*}^{p'}\Big)^{\frac{1}{p'}} \\
  &\lesssim\Big(\frac{1}{\abs{\hat J}}\sum_{h=m}^{n-1}\Norm{\fx^*}{\fX^*}^{p'}\Big)^{\frac{1}{p'}}\qquad\text{by \eqref{eq:Vk=I}} \\
  &=\Big(\frac{1}{\abs{\hat J}}(n-m)\Big)^{\frac{1}{p'}}\Norm{\fx^*}{\fX^*}=\Norm{\fx^*}{\fX^*},
\end{split}
\end{equation*}
where, in the application of \eqref{eq:Vk=I}, we observed that $V\equiv V_k$ for some $k=k(h)$ on each interval $[h,h+1]$, while \eqref{eq:Vk=I} is valid for all pairs of $k$ and $h$, as observed in the discussion following \eqref{eq:Vk=I}.

Similarly,
\begin{equation*}
  \rho_{\aveL^{p}(J,V)}(\fx)\lesssim\Norm{\fx}{\fX}
\end{equation*}
by the other estimate in \eqref{eq:Vk=I}. The definition of the dual norm and combination with the previous estimate shows that
\begin{equation*}
  \rho_{\aveL^{p'}(J,V^{-*})}(\fx^*)
  \lesssim\Norm{\fx^*}{\fX^*}\lesssim\rho_{\aveL^{p}(J,V)}^*(\fx^*),
\end{equation*}
which is the required estimate \eqref{eq:VAp2show} for the large intervals $\abs{J}>1$ and hence
shows that $V\in\A_p(\R;\bddlin(\fX))$.

\subsubsection*{The unboundedness of $\HT$ on $L^p(V)$}

It remains to show that $\HT$ is not bounded on $L^p(V)$. For this purpose, we recall the functions $f_k\in L^p(J_k,E_k)$ and $g_k\in L^{p'}(J_k,E_k)$ from \eqref{eq:BkH>Nk}. Let us then consider the $H_k=E_k\oplus E_k$-valued functions $(f_k,0)$ and $(0,g_k)$, and finally
\begin{equation}\label{eq:fkgk}
\begin{split}
    \tilde f_k &:=j_k\iota_k^{-1}(f_k,0)\in L^p(J_k;\fX), \\
    \tilde g_k &:=P_k^*\iota_k^*(0,g_k)\in L^{p'}(J_k;\fX^*),
\end{split}
\end{equation}
recalling that
\begin{itemize}
  \item $P_k:\fX\to G_k$ are uniformly bounded projections,
  \item $j_k:G_k\to\fX$ is the injection $\fx\mapsto\fx$ of $G_k\subseteq\fX$ into $\fX$,
  \item $\iota_k:G_k\to H_k$ are isomorphisms with $\Norm{\iota_k}{}\Norm{\iota_k^{-1}}{}\leq 2$; thus
  \item $\iota_k^*:H_k^*(=H_k)\to G_k^*$ are the adjoint isomorphisms,
  \item $P_k^*:G_k^*\to\fX^*$ are uniformly bounded.
\end{itemize}
We also recall that, by \eqref{eq:V=Vk} and \eqref{eq:VkDef}, on the interval $J_k$, we have
\begin{equation*}
  V^{\pm 1}=V_k^{\pm 1}=(I-P_k)+j_k\iota_k^{-1}\tilde V_k^{\pm 1}\iota_k P_k.
\end{equation*}
Hence
\begin{equation}\label{eq:VHVvsTilde}
\begin{split}
  &\pair{V\HT V^{-1}\tilde f_k}{\tilde g_k} \\
  &\quad=\pair{\iota_k P_k V_k\HT V_k^{-1}j_k\iota_k^{-1}(f_k,0)}{(0,g_k)} \\
  &\quad=\pair{\iota_k P_k j_k\iota_k^{-1}\tilde V_k\iota_k P_k \HT j_k\iota_k^{-1} \tilde V_k^{-1}\iota_k P_k j_k\iota_k^{-1}(f_k,0)}{(0,g_k)} \\
  &\quad=\pair{\tilde V_k  \HT  \tilde V_k^{-1}(f_k,0)}{(0,g_k)},
\end{split}
\end{equation}
where we used $\iota_k P_k j_k\iota_k^{-1}=\iota_k\iota_k^{-1}=I$ and the fact that the constant operators $\iota_k P_k$ or $j_k\iota_k^{-1}$ commute with $\HT$.

Next, from the definition of $\tilde V_k$ in \eqref{eq:tildeVk}, we obtain
\begin{equation}\label{eq:VHVtildeVsBH}
\begin{split}
  \pair{\tilde V_k  \HT  \tilde V_k^{-1}(f_k,0)}{(0,g_k)}
  &=\Bpair{\begin{pmatrix} I & 0 \\ \tilde B_k & I \end{pmatrix}\HT\begin{pmatrix} I & 0 \\ -\tilde B_k & I \end{pmatrix}
  \begin{pmatrix} f_k \\ 0 \end{pmatrix}}{\begin{pmatrix} 0 \\ g_k \end{pmatrix}} \\
  &=\Bpair{\begin{pmatrix} \HT & 0 \\ [\tilde B_k,\HT] & \HT \end{pmatrix}
  \begin{pmatrix} f_k \\ 0 \end{pmatrix}}{\begin{pmatrix} 0 \\ g_k \end{pmatrix}} \\
  &=\pair{[\tilde B_k,\HT]f_k}{g_k}=\pair{[B_k,\HT]f_k}{g_k},
\end{split}
\end{equation}
using in the last step the fact that $\tilde B_k=B_k$ on $J_k$, since $\tilde B_k$ was obtained from $\{\delta_{k,h}B_h\}_{h\in\Z}$ by the algorithm \eqref{eq:bCombo}.

A combination of \eqref{eq:VHVvsTilde}, \eqref{eq:VHVtildeVsBH}, and \eqref{eq:BkH>Nk} shows that
\begin{equation*}
  \abs{\pair{V\HT V^{-1}\tilde f_k}{\tilde g_k}}\gtrsim\log N_k.
\end{equation*}
On the other hand, from \eqref{eq:fkgk} and the normalisations of $f_k$ and $g_k$, we find that
\begin{equation*}
\begin{split}
  \Norm{\tilde f_k}{L^p(\R;\fX)}\Norm{\tilde g_k}{L^{p'}(\R;\fX^*)}
  &\leq\Norm{\iota_k^{-1}}{}\Norm{f_k}{L^p(\R;E_k)}\cdot\Norm{P_k^*}{}\Norm{\iota_k^*}{}\Norm{g_k}{L^{p'}(\R;E_k)} \\
  &=\Norm{\iota_k^{-1}}{}\Norm{\iota_k}{}\Norm{P_k}{}\leq 2C.
\end{split}
\end{equation*}
The last two displays clearly show the unboundedness of $V\HT V^{-1}$ on $L^p(\R;\fX)$, and hence of $\HT$ on $L^p(V;\fX)$. Since we already checked that $V\in\A_p(\R;\bddlin(\fX))$, the proof is complete.
\end{proof}

\begin{proof}[Proof of Theorem \ref{thm:Hunbd}]
If $\fX$ is not a UMD space, the Hilbert transform is already unbounded on $L^p(\R;\fX)$, so it suffices to take $V\equiv I$, the identity operator. If $\fX$ is $K$-convex, then the required weight is provided by Proposition \ref{prop:badV}. Since every UMD space is $K$-convex (see \cite[Proposition 4.3.10]{HNVW1}), these two cases cover all Banach spaces.
\end{proof}

\section{Another example of unboundedness}

In \cite[second half of page 22]{BHYY:II}, the following estimate is proved for {\em matrix} weights $V\in\A_p$ (written as $V=W^{\frac 1p}$ in \cite{BHYY:II}):
\begin{equation}\label{eq:p22}
  \fint_Q\Big(\fint_Q\Norm{V(x)f(y)}{}\,dy\Big)^p \,dx
  \lesssim\fint_Q\Norm{V(y)f(y)}{}^p \,dy.
\end{equation}
Indeed, after writing $V(x)f(y)=V(x)V(y)^{-1}V(y)f(y)$, the proof is an easy estimate using nothing but H\"older's inequality and one of the equivalent definitions of $V\in\A_p$ for matrix weights. Nevertheless, the said equivalent condition is the one that does not extend to the operator-valued setting. Despite the apparent simplicity of \eqref{eq:p22}, we will show the following:

\begin{example}\label{ex:p22}
For every $p\in(1,\infty)$, estimate \eqref{eq:p22} fails, in general, for operator-valued weights $V\in\A_p(\R;\bddlin(\ell^p))$.
\end{example}

The example uses an idea from \cite[paragraph after Corollary 4.7]{Lauzon}.

\begin{proof}[Proof of Example \ref{ex:p22}]
We consider a diagonal-valued weight $V(x)=\operatorname{diag}(v_i(x))_{i=1}^\infty$, where all $v_i\in\A_p(\R)$ are scalar weights with uniformly bounded $\A_p$ constants. Let $\fx=(\lambda_i)_{i=1}^\infty\in\ell^p$. Then
\begin{equation}\label{uniform-ap}
\begin{split}
  \rho_{\aveL^p(Q,V)}(\fx)^p
  &=\fint_Q\Norm{V(x)\fx}{\ell^p}^p \,dx
  =\sum_{i=1}^\infty\fint_Q\abs{v_i(x)}^p \,dx \abs{\lambda_i}^p \\
  &\lesssim\sum_{i=1}^\infty \Big(\fint_Q\abs{v_i(x)}^{-p'} \,dx\Big)^{-\frac{p}{p'}} \abs{\lambda_i}^p
  \quad\text{since $v_i\in\A_p$ uniformly} \\
  &=\sup\Big\{\Babs{\sum_{i=1}^\infty \Big(\fint_Q\abs{v_i(x)}^{-p'} \,dx\Big)^{-\frac{1}{p'}} \lambda_i\mu_i}^p:
  \sum_{i=1}^\infty\abs{\mu_i}^{p'}\leq 1\Big\} \\
  &=\sup\Big\{\Babs{\sum_{i=1}^\infty  \lambda_i\sigma_i}^p:
  \sum_{i=1}^\infty \Big(\fint_Q\abs{v_i(x)}^{-p'} \,dx\Big) \abs{\sigma_i}^{p'}\leq 1\Big\} \\
  &=\sup\Big\{\abs{\pair{\fx}{\fx^*}}^p:
  \fint_Q\Norm{V(x)^{-1}\fx^*}{\ell^{p'}}^{p'} \,dx \leq 1\Big\}
  =\rho_{\aveL^{p'}(Q,V^{-*})}^*(\fx)^p,
\end{split}
\end{equation}
introducing a dual vector $\fx^*=(\sigma_i)_{i=1}^\infty$ on the last line. (To make sure that $\fx^*\in\ell^{p'}$, the suprema in the previous computation may be taken over finitely non-zero sequences of $\mu_i$ and $\sigma_i$ only.) Thus $V\in\A_p$.

For our example, we will make a more concrete choice as follows. Recall that $w(x)=\abs{x}^\gamma\in A_p(\R)$ if and only if $-1<\gamma<p-1$; hence $v(x)=\abs{x}^\beta=w(x)^{\frac 1p}\in\A_p(\R)$ if and only if $-\frac{1}{p}<\beta<\frac{1}{p'}$. If further $\phi(x)=\abs{x}^\alpha$, then $\phi\in L^p_{\loc}(v)$, i.e., $v\phi\in L^p_{\loc}(\R)$ if and only if $\alpha+\beta>-\frac{1}{p}$. For a small $\eps>0$, we choose $\beta=\frac{1}{p'}-\eps$ and $\alpha=-1+2\eps$.

For $j\in\N$ and $k=0,\ldots,2^j-1$, we define
\begin{equation*}
  x_{2^j+k}:=2^{-j}k\in[0,1),\qquad \lambda_{2^j+k}^{(N)}:=2^{-j(\frac{1}{p}+\eps)}\one_{[0,N]}(j),
\end{equation*}
and functions $v_i(x)=v(x-x_i)=\abs{x-x_i}^{\frac{1}{p'}-\eps}$ and
\begin{equation*}
  f_i(x)=\one_{[0,1)}(x)\phi(x-x_i)=\one_{[0,1)}(x)\abs{x-x_i}^{-1+2\eps}.
\end{equation*}
We then wish to show the failure of \eqref{eq:p22} for $V(x)=\operatorname{diag}(v_i(x))_{i=1}^\infty$ and $f^{(N)}(x)=(\lambda_i^{(N)} f_i(x))_{i=1}^\infty$ in place of $f$, in the sense that there is no uniform estimate with respect to the truncation parameter $N$. Since $f_i\in L^1(\R)$ for each $i$, the truncation parameter ensures that $f^{(N)}\in L^1(\R;\ell^p)$ is well defined.

We consider $Q=[0,1)$. Then
\begin{equation*}
\begin{split}
  \operatorname{RHS}\eqref{eq:p22}
  &=\sum_{i=1}^\infty (\lambda_i^{(N)})^{p}\int_0^1 \abs{v_i(x)f_i(x)}^p \,dx
  \lesssim\sum_{i=1}^\infty(\lambda_i^{(N)})^p\int_0^1 \abs{x}^{(-\frac1p+\eps)p}\,dx \\
  &\approx\sum_{i=1}^\infty(\lambda_i^{(N)})^p
  =\sum_{j=0}^N 2^j\cdot 2^{-j(\frac 1p+\eps)p}
  =\sum_{j=0}^N 2^{-j\eps p}
  \approx 1,
\end{split}
\end{equation*}
where we are allowing dependence of the implicit constants on both $p$ and~$\eps$, but not on $N$.
Here, and thereafter, $\operatorname{RHS}$\eqref{eq:p22}
(resp.\ $\operatorname{LHS}$\eqref{eq:p22})
means the right-hand (resp.\ left-hand) side
of \eqref{eq:p22}.

We then turn to the left-hand side of \eqref{eq:p22}. For fixed $x,y\in[0,1)$, we consider
\begin{equation}\label{eq:Vf-lp}
\begin{split}
  \Norm{V(x)f^{(N)}(y)}{\ell^p}^p
  &=\sum_{i=1}^\infty (\lambda_i^{(N)} v_i(x)f_i(y))^p \\
  &=\sum_{j=0}^N 2^{-j(\frac1p+\eps)p}
  \sum_{k=0}^{2^j-1}\abs{x-2^{-j}k}^{(\frac{1}{p'}-\eps)p} \abs{y-2^{-j}k}^{(-1+2\eps)p}.
\end{split}
\end{equation}
For each $j\geq 3$, we consider two cases:

\textit{Case 1: $\abs{x-y}\geq 2\cdot 2^{-j}$.} In this case, let $k$ be such that $k2^{-j}$ is closest to $y$. Hence $\abs{y-k2^{-j}}\leq 2^{-j}$ and
\begin{equation*}
  \abs{x-k2^{-j}}\geq\abs{x-y}-\abs{y-k2^{-j}}\geq\abs{x-y}-2^{-j}\geq\frac12\abs{x-y}.
\end{equation*}

\textit{Case 2: $\abs{x-y}< 2\cdot 2^{-j}$.} In this case, let $k$ be such that $3\cdot 2^{-j}\leq\abs{y-k2^{-j}}<4\cdot 2^{-j}$. Then
\begin{equation*}
  \abs{x-k2^{-j}}\geq\abs{y-k2^{-j}}-\abs{x-y}\geq 3\cdot 2^{-j}-\abs{x-y}>\Big(\frac 32-1\Big)\abs{x-y}.
\end{equation*}

Thus, in each case, we find a $k$ such that $\abs{x-2^{-j}k}\gtrsim\abs{x-y}$ and $\abs{y-2^{-j}k}\lesssim 2^{-j}$; hence
\begin{equation*}
   \abs{x-2^{-j}k}^{(\frac{1}{p'}-\eps)p} \abs{y-2^{-j}k}^{(-1+2\eps)p}
   \gtrsim \abs{x-y}^{(\frac{1}{p'}-\eps)p} 2^{j(1-2\eps)p}.
\end{equation*}
Obviously, this is also a lower bound for the whole sum over $k$.

Substituting this into \eqref{eq:Vf-lp}, we find that
\begin{equation*}
\begin{split}
  \Norm{V(x)f^{(N)}(y)}{\ell^p}^p
  &\gtrsim\sum_{j=3}^N 2^{-j(\frac1p+\eps)p}\abs{x-y}^{(\frac{1}{p'}-\eps)p} 2^{j(1-2\eps)p} \\
  &=\abs{x-y}^{p-1-\eps p}\sum_{j=3}^N 2^{j(p-1-3\eps p)} \\
  &\geq\abs{x-y}^{p-1-\eps p} 2^{N(p-1-3\eps p)}.
\end{split}
\end{equation*}
Thus
\begin{equation*}
  \operatorname{RHS}\eqref{eq:p22}
  \gtrsim 2^{N(p-1-3\eps p)}
  \int_0^1 \Big( \int_0^1\abs{x-y}^{\frac{1}{p'}-\eps} \,dy\Big)^p \,dx
  \approx 2^{N(p-1-3\eps p)}.
\end{equation*}
Choosing $\eps\in(0,\frac{1}{3p'})$, this contradicts \eqref{eq:p22} as $N\to\infty$, recalling that $\operatorname{LHS}\eqref{eq:p22}\lesssim 1$ uniformly in $N$.
\end{proof}

\part{Besov and Triebel--Lizorkin spaces}\label{part:BTL}

\section{Main definitions}

In this part, we develop the real-variable theory of
Besov and Triebel--Lizorkin spaces associated with
the $\A_p$ weights.
We begin with the following definition of
Littlewood--Paley functions.

\begin{definition}
Let $\sqrt{2}<\alpha<\beta<\pi$.
We say that $\varphi\in\mathscr S(\R^n)$ is an $(\alpha,\beta)$-Littlewood--Paley function if
\begin{equation*}
  1_{[\frac{1}{\alpha},\alpha]}(\abs{\xi})\lesssim\abs{\hat\varphi(\xi)}\lesssim 1_{[\frac{1}{\beta},\beta]}(\abs{\xi}).
\end{equation*}
We say that $\varphi$ is a Littlewood--Paley function if it is an $(\alpha,\beta)$-Littlewood--Paley function for some $\sqrt{2}<\alpha<\beta<\pi$.

For a Littlewood--Paley function $\varphi$ and $j\in\Z$, we denote
\begin{equation*}
 \varphi_j(x):=2^{jn}\varphi(2^j x).
\end{equation*}
\end{definition}
Most of the time, we could simply consider Littlewood--Paley functions with some fixed $(\alpha,\beta)$, say $(\frac53,2)$ as in \cite[Eqs.\ (3.2) and (3.3)]{BHYY:I} or \cite[Eq.\ (3.2)]{BHYY:SCM}. The advantage of the slightly more general definition is the following simple observation:

\begin{remark}
Whenever $\varphi$ exists a Littlewood--Paley function, there exists another Littlewood--Paley function $\chi$ such that $\chi_j*\varphi_j=\varphi_j$. Indeed, if $\varphi$ is an $(\alpha,\beta)$-Littlewood--Paley function, it suffices to take $\chi$ to be a $(\beta,\beta')$-Littlewood--Paley function for some $\beta'\in(\beta,\pi)$ with the additional property that $\hat\chi(\xi)=1$ for $\abs{\xi}\in[\frac{1}{\beta},\beta]$.
\end{remark}

Similarly to \cite{FR:04,FR:21,Rou:03}, we will also
introduced and study average-weighted spaces related to
a family of quasi-norms.
Let $\rho=\{\rho_Q\}_{Q\in\mathscr D}$ be a family of quasi-norms on a Banach space $\fX$.
For each $j\in\Z$, we define piecewise variable quasi-norms of a function $f:\R^n\to\fX$ by
\begin{equation*}
  \begin{cases}  (\rho_{\mathscr D_j}f)(x) :=\rho_{Q}(f(x)), \\
  \Norm{f}{\aveL^p(\mathscr D_j)}(x) :=\Norm{f}{\aveL^p(Q)},
  \end{cases}\qquad x\in Q\in\mathscr D_j.
\end{equation*}
Recall that the \emph{notation} $\mathscr{S}'_\infty(\mathbb{R}^n;\fX)$ denotes
the set of all bounded linear functionals $f:\mathscr{S}_\infty\to\fX$.
Now we can define the following
weighted Besov/Triebel--Lizorkin norms by pointwise and averaging weights.

\begin{definition}\label{def:space}
Let $p\in(0,\infty)$, $q\in(0,\infty]$, and
$s\in\mathbb{R}$. Assume that $V$ is a weight,
$\rho=\{\rho_Q\}_{Q\in\mathscr D}$ is a family of quasi-norms, and
\begin{equation*}
  A\in\{B,F\},\qquad (L^p\ell^q)_A:=\begin{cases} L^p\ell^q, & A=F, \\ \ell^q L^p, & A=B.\end{cases}
\end{equation*}
\begin{enumerate}[{\rm(i)}]
  \item Assume that $\varphi$ is an $(\alpha,\beta)$-Littlewood--Paley function,
$\sqrt{2}<\alpha<\beta<\pi$.
The \emph{homogeneous operator-weighted
Besov space $\dot{B}^s_{p,q}(V,\varphi)$} and the
\emph{homogeneous operator-weighted
Triebel--Lizorkin space $\dot{F}^s_{p,q}(V,\varphi)$}
are defined by setting
\begin{align*}
\dot{A}^s_{p,q}(V,\varphi)
:=\Big\{f\in\mathscr S_\infty'(\R^n;\fX):
\|f\|_{\dot{A}^s_{p,q}(V,\varphi)}<\infty\Big\},
\end{align*}
where, for any $f\in\mathscr{S}_\infty'(\mathbb{R}^n;\fX)$,
\begin{align*}
\Norm{f}{\dot A^s_{p,q}(V,\varphi)}
  &:=\BNorm{\Big\{2^{js}\Norm{V(\varphi_j*f)}{\fX}\Big\}_{j\in\Z}}{(L^p\ell^q)_A}.
\end{align*}
The \emph{homogeneous averaging operator-weighted
Besov space $\dot{B}^s_{p,q}(V,\rho)$} and the
\emph{homogeneous averaging operator-weighted
Triebel--Lizorkin space $\dot{F}^s_{p,q}(V,\rho)$}
are defined by setting
\begin{align*}
\dot{A}^s_{p,q}(\rho,\varphi)
:=\Big\{f\in\mathscr S_\infty'(\R^n;\fX):
\|f\|_{\dot{A}^s_{p,q}(\rho,\varphi)}<\infty\Big\},
\end{align*}
where, for any $f\in\mathscr{S}_\infty'(\mathbb{R}^n;\fX)$,
\begin{align*}
  \Norm{f}{\dot A^s_{p,q}(\rho,\varphi)}
  &:=\BNorm{\Big\{2^{js}\rho_{\mathscr D_j}(\varphi_j*f)\Big\}_{j\in\Z}}{(L^p\ell^q)_A}.
\end{align*}
  \item\label{def:space:it2}  For any sequence
  $t=\{t_Q\}_{Q\in\mathscr{D}}$ of vectors $t_Q\in\fX$,
 define
\begin{equation}\label{df-tj}
  t_j:=\sum_{Q\in\mathscr D_j} t_Q\frac{\one_Q}{\abs{Q}^{\frac12}},
  \quad j\in\mathbb{Z},
\end{equation}
and
\begin{equation*}
\Norm{t}{\dot a^s_{p,q}(V)}
=\BNorm{\Big\{2^{js}\Norm{Vt_j}{\fX}\Big\}_{j\in\mathbb{Z}}}
{(L^p\ell^q)_A},\quad
\Norm{t}{\dot a^s_{p,q}(\rho)}=
\BNorm{\Big\{2^{js}\rho_{\mathscr D_j}(t_j)\Big\}_{j\in\mathbb{Z}}}
{(L^p\ell^q)_A}.
\end{equation*}
The \emph{homogeneous operator-weighted
Besov sequence space $\dot{b}^s_{p,q}(V)$} and the
\emph{homogeneous operator-weighted
Triebel--Lizorkin sequence space $\dot{f}^s_{p,q}(V)$}
are defined by setting
$$\dot{a}^s_{p,q}(V):=\Big\{t=(t_Q):\Norm{t}{\dot{a}^s_{p,q}(V)}<\infty\Big\}.$$
The \emph{homogeneous averaging operator-weighted
Besov sequence space $\dot{b}^s_{p,q}(\rho)$} and the
\emph{homogeneous averaging operator-weighted
Triebel--Lizorkin sequence space $\dot{f}^s_{p,q}(\rho)$}
are defined by setting
$$\dot{a}^s_{p,q}(\rho):=
\Big\{t=(t_Q):\Norm{t}{\dot{a}^s_{p,q}(\rho)}<\infty\Big\}.$$
\end{enumerate}
\end{definition}

The main example of $\rho$
is $\{\rho_{\aveL^p(Q,V)}\}_{Q\in\D}$
and its relatives studied above,
but parts of the real-variable theory can be conveniently developed
in a more general framework (see Definition \ref{def:uabc} below).

\section{Fourier analysis preliminaries}

In this section, we establish some basic tools
in Fourier analysis, including the Calder\'on reproducing
formula and the sampling type inequality in infinite-dimensional
target spaces. They will play key roles in developing
the real-variable theory of operator-valued Besov
and Triebel--Lizorkin spaces.

We begin with the following Calder\'on reproducing formula.

\begin{lemma}
Let $\sqrt{2}<\alpha<\beta<\pi$.
Let $\varphi$ be an $(\alpha,\beta)$-Littlewood--Paley function, and let $\gamma\in\mathscr S(\R^n)$ be such that $\hat\gamma(\xi)=1$ for $\abs{\xi}\leq \beta$ and $\supp\hat\gamma\subset\{\xi\in\R^n:\abs{\xi}<\pi\}$. Then for all $f\in\mathscr S'_\infty(\R^n;\fX)$ and all $j\in\Z$ and $x,y\in\R^n$,
\begin{equation}\label{eq:316}
  \varphi_j*f(x)=\sum_{R\in\mathscr D_j}2^{-jn}\gamma_j(x-x_R-y)(\varphi_j*f)(x_R+y),
\end{equation}
where $\gamma_j(x):=2^{jn}\gamma(2^j x)$.
\end{lemma}

\begin{proof}
For scalar-valued $f$, this is \cite[Lemma 3.4]{BHYY:SCM}, applied to the distribution $\varphi_j*f$ with Fourier support in $2^j\supp\hat\varphi\subseteq\{\xi\in\R^n:\abs{\xi}\leq\beta\cdot 2^j\}$. The general case is obtained either by repeating the same proof or by the following reduction to the scalar case: For $f\in\mathscr S'_\infty(\R^n;\fX)$ and $\fx^*\in\fX^*$, we have $\pair{\fx^*}{f}:=\fx^*\circ f\in\mathscr S_\infty'(\R^n)$, and hence \eqref{eq:316} holds with $\pair{\fx^*}{f}$ in place of $f$ by the scalar-valued result \cite[Lemma 3.4]{BHYY:SCM}. This may be written as
\begin{equation*}
  \pair{\fx^*}{\varphi_j*f(x)}=\Bpair{\fx^*}{\sum_{R\in\mathscr D_j}2^{-jn}\gamma_j(x-x_R-y)(\varphi_j*f)(x_R+y)}
\end{equation*}
for all $j\in\Z$ and $x,y\in\R^n$ as in the statement and for all $\fx^*\in\fX^*$.
Since the dual elements $\fx^*$ separate the points of $\fX$, this implies \eqref{eq:316} as written.
\end{proof}

We now turn to the sampling type inequality.
Recall that a \emph{seminorm} is like a norm except that it is allowed to have a nontrivial null space.
Moreover, we need the following concept of
$u$-seminorms.

\begin{definition}\label{def:semi}
Let $u\in(0,1]$. A \emph{$u$-(semi)norm}
$\rho$ is like a (semi)norm except that the triangle inequality is replaced by $$\rho(\fx+\fz)^u\leq\rho(\fx)^u+\rho(\fz)^u
\quad\forall\,\fx,\fz\in\fX.$$
\end{definition}

It easily follows that any $u$-(semi)norm is also
an $s$-(semi)norm for all $s\in(0,u)$.
In particular, any norm is also a $u$-norm for all $u\in(0,1)$.
Based on this, we now derive the following sampling type inequality.

\begin{lemma}\label{lem:extraPhi}
Let $\varphi$ be a Littlewood--Paley function.
If $f\in\mathscr S_\infty'(\R^n;\fX)$ and $\rho$ is a $u$-seminorm on $\fX$, then
for any $j\in\Z$ and $x\in\mathbb{R}^n$,
\begin{equation*}
  \rho(\varphi_j*f(x))^u
  \lesssim\big(\phi_j*\rho(\varphi_j*f)^u\big)(x),
\end{equation*}
where $\phi(x):=(1+\abs{x})^{-M}$ with $M>0$ taken as large as we like and where
the implicit positive constant is independent of $j$, $f$, and $x$.

Under the same assumptions, if $P\in\mathscr D_j$, we even have
\begin{equation*}
   \sup_{x\in P}\rho(\varphi_j*f(x))^u
  \lesssim \inf_{y\in P}\big(\phi_j*\rho(\varphi_j*f)^u\big)(y),
\end{equation*}
where the implicit positive constant is independent of $j$, $f$, and $P$.
\end{lemma}

\begin{proof}
Let $j\in\Z$, $x\in\R^n$ be arbitrary, and $y\in[0,2^{-j})^n$.
We apply $\rho(\cdot)^u$ to both sides of \eqref{eq:316} and use the $u$-triangle inequality to find that
\begin{equation*}
\begin{split}
  \rho(\varphi_j*f(x))^u
  &\leq\sum_{R\in\mathscr D_j}(2^{-jn}\abs{\gamma_j(x-x_R-y)})^u\rho((\varphi_j*f)(x_R+y))^u.
\end{split}
\end{equation*}
Averaging over $y\in[0,2^{-j})^n$ and using the decay $\abs{\gamma(x)}\lesssim(1+\abs{x})^{-M/u}$, we obtain
\begin{equation*}
\begin{split}
  \rho(\varphi_j*f(x))^u
  &\lesssim\sum_{R\in\mathscr D_j}2^{jn}\int_R(1+2^j\abs{x-z})^{-M} \rho((\varphi_j*f)(z))^u \,dz \\
  &=\int_{\R^n}\phi_j(x-z)\rho((\varphi_j*f)(z))^u \,dz
  =(\phi_j*\rho((\varphi_j*f))^u)(x),
\end{split}
\end{equation*}
which is the first claimed inequality.

If $x,y\in P\in\mathscr D_j$ and $z\in\R^n$, then $\phi_j(x-z)\approx\phi_j(y-z)$, so that the right-hand side of the previous bound is comparable to $(\phi_j*\rho((\varphi_j*f))^u)(y)$, and we also obtain the second claim.
\end{proof}

\section{Relations between pointwise and average-weighted spaces}\label{sec:BTL}

In this section, we study the relationships between pointwise-weighted spaces and
average-weighted spaces with
$$\rho=\rho_{\aveL^r(\mathscr D,V)}=\{\rho_{\aveL^r(Q,V)}\}_{Q\in\mathscr D}.$$
Since we always work with the same Littlewood--Paley function
$\varphi$, to simplify the presentation,
we omit the $\varphi$ and denote
$\dot{A}^{s}_{p,q}(\cdots,\varphi)$ simply as $\dot{A}^{s}_{p,q}(\cdots)$
throughout this section.

\subsection{Besov spaces}

We first study this problem for Besov spaces.
For Besov sequence spaces, we have the following result.

\begin{theorem}\label{b-arho-av}
Let $p\in(0,\infty)$, $q\in(0,\infty]$, $s\in\mathbb{R}$,
and $V$ be a weight.
\begin{enumerate}[{\rm(i)}]
  \item\label{b-arho-av1} For any sequence $t=\{t_Q\}_{Q\in\mathscr{D}}$
  of vectors in $\fX$,
  \begin{align*}
\|t\|_{\dot{b}^s_{p,q}(V)}=
\|t\|_{\dot{b}^s_{p,q}(\rho_{\aveL^p(\mathscr{D},V)})}.
\end{align*}
  \item\label{b-arho-av2} Suppose $r\in(0,\infty)$. The equivalence
  \begin{align*}
\|t\|_{\dot{b}^s_{p,q}(V)}\approx
\|t\|_{\dot{b}^s_{p,q}(\rho_{\aveL^r(\mathscr{D},V)})}
\end{align*}
holds for all sequences $t=\{t_Q\}_{Q\in\mathscr{D}}$ of vectors
in $\fX$ if and only if, for all cubes $Q\in\mathscr{D}$ and all
$\fx\in\fX$,
\begin{align}\label{b-arho-ave1}
\rho_{\aveL^p(Q,V)}(\fx)\approx\rho_{\aveL^r(Q,V)}(\fx)
\end{align}
with the positive equivalence constants independent of $Q$ and $\fx$.
\end{enumerate}
\end{theorem}

\begin{proof}
By a direct calculation, we obtain, for any $j\in\mathbb{Z}$,
\begin{align*}
\BNorm{\sum_{Q\in\mathscr{D}_j}\Norm{Vt_Q}{\fX}
\frac{{\bf1}_Q}{|Q|^{\frac12}}}{L^p}^p
&=\sum_{Q\in\mathscr{D}_j}
|Q|^{-\frac p2}\int_{Q}\Norm{V(x)t_Q}{\fX}^p\,dx\\
&=\sum_{Q\in\mathscr{D}_j}|Q|^{-\frac p2}|Q|\rho_{\aveL^p(Q,V)}(t_Q)^p
=\BNorm{\sum_{Q\in\mathscr{D}_j}\rho_{\aveL^p(Q,V)}(t_Q)
\frac{{\bf1}_Q}{|Q|^{\frac12}}}{L^p}^p.
\end{align*}
This, together with the definitions of
both sequence spaces [namely Definition \ref{def:space}\eqref{def:space:it2}],
finishes the proof of \eqref{b-arho-av1}.

On the other hand, by testing the sequence
$t=\{t_Q\}_{Q\in\mathscr{D}}$ with
$t_Q\ne{\bf0}$ only when $Q=R$ for a given cube $R\in\mathscr{D}$,
we conclude that \eqref{b-arho-ave1} holds for all $Q$ and $\fX$ if and only if
\begin{align*}
\|t\|_{\dot{b}^s_{p,q}(\rho_{\aveL^p(\mathscr{D},V)})}\approx
\|t\|_{\dot{b}^s_{p,q}(\rho_{\aveL^r(\mathscr{D},V)})}.
\end{align*}
Hence, from \eqref{b-arho-av1}, it follows that
\eqref{b-arho-av2} holds. This then finishes the proof of
Theorem \ref{b-arho-av}.
\end{proof}

\begin{corollary}\label{cor:equib}
Let $p,r\in(0,\infty)$, $q\in(0,\infty]$, and $s\in\mathbb{R}$.
If $V\in\mathscr{A}_{\max\{p,r\}}$, then
\eqref{b-arho-ave1} holds for all
$Q\in\mathscr{D}$ and $\fx\in\fX$.
As a consequence, for any sequence $t=\{t_Q\}_{Q\in\mathscr{D}}$
  of vectors in $\fX$,
  \begin{align*}
\|t\|_{\dot{b}^s_{p,q}(V)}\approx
\|t\|_{\dot{b}^s_{p,q}(\rho_{\aveL^r(\mathscr{D},V)})},
\end{align*}
where the positive equivalence constants are independent of $t$.
\end{corollary}

\begin{proof}
Since $V\in\A_{\max\{p,r\}}$, we deduce
\begin{align*}
\max\Big\{\rho_{\aveL^p(Q,V)}(\fx),\rho_{\aveL^r(Q,V)}(\fx)\Big\}
&\le \rho_{\aveL^{\max\{p,r\}}(Q,V)}(\fx)
\quad\text{by H\"older's inequality}\\
&\lesssim\min\Big\{\rho_{\aveL^p(Q,V)}(\fx),\rho_{\aveL^r(Q,V)}(\fx)\Big\}
\quad\text{by Corollary \ref{cor:Ap1}}.
\end{align*}
This shows Corollary \ref{cor:equib}.
\end{proof}

Now, we turn to investigate the relationship between
Besov spaces $\dot{B}^s_{p,q}(V)$ and
$\dot{B}^s_{p,q}(\rho_{\aveL^r(\D,V)})$.

\begin{theorem}\label{thm:equiB}
Let $p,r\in(0,\infty)$, $q\in(0,\infty]$, and $s\in\mathbb{R}$.
If $V\in\mathscr{A}_{\max\{p,r\}}$, then,
for any $f\in\mathscr{S}'_\infty(\mathbb{R}^n;\fX)$,
\begin{equation*}
\Norm{f}{\dot B^s_{p,q}(V)}\approx
\Norm{f}{\dot B^s_{p,q}(\rho_{\aveL^r(\mathscr D,V)})},
\end{equation*}
where the positive equivalence constants are independent of $f$.
\end{theorem}

Indeed, for each side of the above equivalence,
we can prove some more general results,
which are also of independent interest.
On the one hand,
the following result about dominating the pointwise weighted norm by the averaging weighted norm is inspired by \cite[Eq.~(3.14)]{AC:12}, where a similar estimate for operator-weighted Bergman spaces of analytic functions in the unit disc was obtained. Note that the components $\varphi_j*f$ of the Littlewood--Paley decomposition, having compactly supported Fourier transforms, are also analytic. In the proof of \cite[Eq.~(3.14)]{AC:12}, the analyticity of $f$ was used via the resulting subharmonicity of $z\mapsto\Norm{W^{\frac12}(\zeta)f(z)}{}^2$, meaning that pointwise values of this function (in particular at $z=\zeta$) are dominated by its averages over discs. Our substitute for this is Lemma \ref{lem:extraPhi}, where the pointwise values of $\rho(\varphi*f)^u$ are dominated by its convolution averages.

\begin{proposition}\label{prop:BV<Brho}
Let $p\in(0,\infty)$,
$V$ be a weight, and $\rho=\{\rho_{Q}\}_{Q\in\mathscr{D}}$
be a sequence of quasi-norms such that
\begin{align}\label{prop:BV<Brhoe1}
\rho_{\aveL^p(Q,V)}(\fx)\lesssim
\rho_Q(\fx)\quad\text{for all}\ Q\in\mathscr{D}\ \text{and}\ \fx\in\fX
\end{align}
and
$\rho$ satisfies that
\begin{equation}\label{eq:weakDb}
  \frac{\rho_{Q}(\fx)}{\rho_{R}(\fx)}
  \lesssim\Big(1+\frac{\abs{x_Q-x_R}}{\ell(Q)}\Big)^N
\end{equation}
for some $N\geq 0$, all $\fx\in\fX\setminus\{0\}$, and all cubes $Q,R\subset\R^n$ of equal size $\ell(Q)=\ell(R)$,
where the implicit positive constants are independent of $\fx$, $Q$, and $R$.
Then for all $j\in\Z$, $Q\in\mathscr D_j$, and $y\in Q$, we have
\begin{equation}\label{eq:BVQ<BrhoQ}
   \Big(\fint_Q\Norm{V(x)(\varphi_j*f)(x)}{}^p \,dx\Big)^{\frac1p}
   \lesssim\Big(\phi_j*\rho_{\mathscr D_j}(\varphi_j*f)^u\Big)^{\frac1u}(y),
\end{equation}
where $\phi(x):=(1+\abs{x})^{-M}$, $M\in(0,\infty)$ can be chosen as large as we like,
and $u\in(0,1]$ can be chosen as small as we like. As a consequence,
\begin{equation*}
  \Norm{f}{\dot B^s_{p,q}(V)}
  \lesssim\Norm{f}{\dot B^s_{p,q}(\rho)},
\end{equation*}
where the implicit positive constant is independent of $f$.
\end{proposition}

\begin{proof}
With $\rho(\fx)=\Norm{V(x)\fx}{}$, Lemma \ref{lem:extraPhi} gives, for all 
$x\in Q\in\mathscr D_j$,
\begin{equation*}
\begin{split}
\Norm{V(x)(\varphi_j*f)(x)}{}
&\lesssim\big(\phi_j'*\Norm{V(x)(\varphi_j*f)(\cdot)}{\fX}^u)^{\frac1u}(x_Q) \\
  &=\Big(\int_{\R^n}\phi_j'(x_Q-z)\Norm{V(x)(\varphi_j*f)(z)}{\fX}^u \,dz\Big)^{\frac1u},
\end{split}
\end{equation*}
where $u\in(0,1]$, $\phi'(x):=(1+\abs{x})^{-M'}$, and $M'>0$ may be chosen as large as we like.

Choosing $u\in(0,\min(p,1)]$, taking $\aveL^p(Q)$ norms of both sides, and using Minkowski's inequality, we obtain
\begin{equation*}
\begin{split}
  \Big(\fint_Q \Norm{V(x)(\varphi_j*f)(x)}{\fX}^p \,dx\Big)^{\frac1p} &\lesssim\Big(\int_{\R^n}\phi_j'(x_Q-z)\Big[\fint_Q\Norm{V(x)(\varphi_j*f)(z)}{\fX}^p \,dx\Big]^{\frac{u}{p}} \,dz\Big)^{\frac1u} \\
  &=\Big(\int_{\R^n}\phi_j'(x_Q-z)\rho_{\aveL^p(Q,V)}(\varphi_j*f(z))^u \,dz\Big)^{\frac1u} \\
  &\lesssim \Big(\int_{\R^n}\phi_j'(x-z)\rho_Q(\varphi_j*f(z))^u \, dz\Big)^{\frac1u}
  \quad\text{by \eqref{prop:BV<Brhoe1}}.
\end{split}
\end{equation*}
Since $\rho$ satisfies \eqref{eq:weakDb}, we deduce that
\begin{equation}\label{eq:applyWeakDb}
\begin{split}
  \int_{\R^n} &\phi_j'(x-z)\rho_Q(\varphi_j*f(z))^u \,dz \\
  &\lesssim\sum_{R\in\mathscr D_j}\int_R\phi_j'(x-z)
  (1+2^j\abs{x_Q-x_R})^{Nu}\rho_R(\varphi_j*f(z))^u \,dy \\
  &\approx\int_{\R^n}\phi_j(y-z)\rho_{\mathscr D_j}(\varphi*f)(z)^u \,dz,
\end{split}
\end{equation}
where $\phi(z):=\phi'(z)(1+\abs{z})^{Nu}=(1+\abs{z})^{-M}$ and $M=M'-Nu$ can still be chosen as large as we like by taking $M'$ large enough. Thus we have proved \eqref{eq:BVQ<BrhoQ}.

To conclude the proof, we observe that
\begin{equation*}
\begin{split}
   \Norm{V(\varphi_j*f)}{L^p}^p
   &=\sum_{Q\in\mathscr D_j}\abs{Q}\fint_Q\Norm{V(x)(\varphi_j*f)(x)}{}^p \,dx
   \lesssim\sum_{Q\in\mathscr D_j}\abs{Q}\fint_Q \Big(\phi_j*
   \rho_{\mathscr D_j}(\varphi_j*f)^u\Big)^{\frac pu}(y)\,dy \\
   &=\Norm{\phi_j*\rho_{\mathscr D_j}(\varphi_j*f)^u}{L^{\frac pu}}^{\frac pu}
   \leq\Big(\Norm{\phi_j}{L^1} \Norm{\rho_{\mathscr D_j}(\varphi_j*f)^u }
   {L^{\frac pu}}\Big)^{\frac pu}\lesssim \Norm{\rho_{\mathscr D_j}(\varphi_j*f) }{L^{p}}^p.
\end{split}
\end{equation*}
Taking the $p$-th roots, multiplying by $2^{js}$ and taking the $\ell^q$ norms over $j\in\Z$, we obtain the claim of the proposition.
\end{proof}

On the other hand, the following converse to Proposition \ref{prop:BV<Brho} is inspired by the second part of \cite[Theorem 3.1]{AC:12} and the key ingredient in \cite[Lemma 3.2]{AC:12}. The idea borrowed from there is the following: In order to estimate $\rho_{\aveL^p(Q,V)}(\fx)$, in the presence of the $\A_p$ condition, it suffices to estimate the dual norm $\rho_{\aveL^{p'}(Q,V^{-*})}^*(\fx)$. Estimating the latter, by definition, is about estimating the pairings $\pair{\fx}{\fx^*}=\pair{V(\cdot)\fx}{V(\cdot)^{-*}\fx^*}$, where we can introduce the pointwise weight as just indicated. The remaining technical details of the argument differ from \cite{AC:12}, but this basic idea is the same. The estimation of the dual norms is handled in the following lemma, after which the proof of Theorem \ref{thm:equiB} is relatively easy.

\begin{lemma}\label{lem:dualNorm}
Let $p\in(0,\infty)$ and $V$ be a $p$-weight.
If $V^{-*}$ is a $p'$-weight and $V$ satisfies, for some $N\geq 0$,
\begin{equation}\label{eq:weakDbDual}
  \rho_{L^{p'}(R,V^{-*})}(\fx^*)
  \lesssim\Big(1+\frac{\abs{x_Q-x_R}}{\ell(Q)}\Big)^N\rho_{L^{p'}(Q,V^{-*})}(\fx^*)
\end{equation}
for all $\fx^*\in\fX^*$ and all cubes $Q,R$ of equal size
with the implicit positive constant independent of $\fx^*$, $Q$, and $R$,
then, for all $f\in\mathscr S_\infty'(\R^n;\fX)$, $j\in\Z$, and
$x,y\in Q\in\mathscr D_j$,
\begin{equation*}
  \rho_{\aveL^{p'}(Q,V^{-*})}^*(\varphi_j*f(x))
  \lesssim(\phi_j*\Norm{V(\varphi_j*f)}{}^u)^{\frac1u}(y)
\end{equation*}
with the implicit positive constant independent
of $j$, $Q$, and $f$,
where $\phi_j(z)=2^{jn}\phi(2^j z)$ with $\phi(z)=(1+\abs{z})^{-M}$
and $M>0$ can be chosen as large as we like, while $u>0$ can be chosen as small as we like.
\end{lemma}

\begin{proof}
We apply Lemma \ref{lem:extraPhi} with the seminorm $\rho(\fx)=\abs{\pair{\fx}{\fx^*}}$ for a fixed $\fx^*\in\fX^*$. For $x,y\in Q\in\mathscr D_j$, this gives
\begin{equation*}
\begin{split}
  \abs{\pair{\varphi_j*f(x)}{\fx^*}}
  &\lesssim\big(\phi_j'*\abs{\pair{\varphi_j*f(\cdot)}{\fx^*}}^r)^{\frac1r}(y) =\Big(\int_{\R^n}\phi_j'(y-z)\abs{\pair{\varphi_j*f(z)}{\fx^*}}^r \,dz\Big)^{\frac1r},
\end{split}
\end{equation*}
where $r\in(0,1]$, $\phi'(x):=(1+\abs{x})^{-M'}$, and $M'>0$ may be chosen as large as we like.

Next, we note that
\begin{equation*}
\begin{split}
   \abs{\pair{\varphi_j*f(z)}{\fx^*}}
   &=\abs{\pair{V(z)(\varphi_j*f)(z)}{V(z)^{-*}\fx^*}}
   \leq\Norm{V(z)(\varphi_j*f)(z)}{\fX}\Norm{V(z)^{-*}\fx^*}{\fX^*}.
\end{split}
\end{equation*}
Hence
\begin{equation*}
  \abs{\pair{\varphi_j*f(x)}{\fx^*}}
  \lesssim\Big(\sum_{R\in\mathscr D_j}\phi_j'(y-x_R)\abs{R}\fint_{R}\Norm{V(z)(\varphi_j*f)(z)}{\fX}^r
  \Norm{V(z)^{-*}\fx^*}{\fX^*}^r \,dz\Big)^{\frac1r}.
\end{equation*}

Let $\frac{1}{u}:=\frac{1}{r}-\frac{1}{p'}$ and $\frac{1}{v}:=\frac{1}{r}-\frac{1}{p}$, recalling that $p':=\infty$ for $p\in(0,1]$.
Since $r>0$ can be chosen as small as we like, hence $\frac1r$ as large as we like, we see that both $\frac1u$ and $\frac1v$ can be chosen as large as we like, and hence both $u,v>0$ as small as we like.
Then
\begin{equation*}
\begin{split}
  &\fint_{R} \Norm{V(z)(\varphi_j*f)(z)}{\fX}^r
  \Norm{V(z)^{-*}\fx^*}{\fX^*}^r \,dz \\ &\quad\leq\Norm{V(\varphi_j*f)}{\aveL^u(R;\fX)}^r\big(\rho_{\aveL^{p'}(R,V^{-*})}(\fx^*)\big)^r \\
  &\quad\lesssim \Norm{V(\varphi_j*f)}{\aveL^u(R;\fX)}^r\big(1+2^j\abs{x_Q-x_R})^{Nr}\big(\rho_{\aveL^{p'}(Q,V^{-*})}(\fx^*)\big)^r.
\end{split}
\end{equation*}

Choosing $M':=M+Nr$, we obtain
\begin{equation*}
  \phi_j'(y-x_R)\big(1+2^j\abs{x_Q-x_R}\big)^{Nr}\approx\phi_j(y-x_R)\approx\phi_j(y-x),
\end{equation*}
where $\phi(z):=(1+\abs{z})^{-M}$.
Substituting back, we have
\begin{equation*}
\begin{split}
  \frac{\abs{\pair{\varphi_j*f(x)}{\fx^*}}}{\rho_{L^{p'}(Q,V^{-*})}(\fx^*)}
  &\lesssim \Big(\sum_{R\in\mathscr D_j}\phi_j(y-x_R)\abs{R}\cdot\Norm{V(\varphi_j*f)}{\aveL^u(R;\fX)}^r\Big)^{\frac1r} \\
  &\leq \Big(\sum_{R\in\mathscr D_j}\phi_j(y-x_R)\abs{R}\cdot\Norm{V(\varphi_j*f)}{\aveL^u(R;\fX)}^u\Big)^{\frac1u} \Big(\sum_{R\in\mathscr D_j}\phi_j(y-x_R)\abs{R}\Big)^{\frac1r-\frac 1u} \\
  &=: I\times II, \\
\end{split}
\end{equation*}
where $II\approx\Norm{\phi_j}{L^1}^{\frac1r-\frac1u}\approx 1$, and
\begin{equation*}
  I^u =\sum_{R\in\mathscr D_j}\phi_j(y-x_R)\int_R \Norm{V(z)(\varphi_j*f)(z)}{}^u \,dz \approx\int_{\R^n}\phi_j(y-z)\Norm{V(z)(\varphi_j*f)(z)}{}^u \,dz.
\end{equation*}
Hence
\begin{equation*}
   \frac{\abs{\pair{\varphi_j*f(x)}{\fx^*}}}{\rho_{L^{p'}(Q,V^{-*})}(\fx^*)}
   \lesssim(\phi_j*\Norm{V(\varphi_j*f)}{}^u)^{\frac1u}(y).
\end{equation*}

Taking the supremum over $\fx^*\in\fX^*\setminus\{0\}$ and using the definition of the dual norm, we obtain
\begin{equation*}
  \rho_{\aveL^{p'}(Q,V^{-*})}^*(\varphi_j*f(x)) \\
  \lesssim(\phi_j*\Norm{V(\varphi_j*f)}{}^u)^{\frac1u}(y),
\end{equation*}
which completes the proof of Lemma \ref{lem:dualNorm}.
\end{proof}

We now show the converse of Proposition \ref{prop:BV<Brho}.
Indeed, we can prove this for both Besov and Triebel--Lizorkin spaces in a unified way.

\begin{proposition}\label{prop:Arho<AV}
Let $p,r\in(0,\infty)$, $q\in(0,\infty]$, and $s\in\mathbb{R}$.
If $V\in\mathscr{A}_{\max\{p,r\}}$, then,
for any $f\in\mathscr{S}'_\infty(\mathbb{R}^n;\fX)$,
\begin{equation*}
  \Norm{f}{\dot A^s_{p,q}(\rho_{\aveL^r(\mathscr D,V)})}
  \lesssim\Norm{f}{\dot A^s_{p,q}(V)},
\end{equation*}
where the implicit positive constant is independent of $f$.
\end{proposition}

\begin{proof}
Let $w:=\max\{p,r\}$ and fix $j\in\mathbb{Z}$.
By the proof of Corollary \ref{cor:equib}, we
find that the assumption $V\in\mathscr{A}_{w}$ implies
\begin{align}\label{thm:equiB-e1}
\rho_{\aveL^p(Q,V)}(\fx)\approx
\rho_{\aveL^{w}(Q,V)}(\fx)
\approx\rho_{\aveL^r(Q,V)}(\fx)
\end{align}
for all $Q\in\mathscr{D}$ and $\fx\in\fX$.
From Corollary \ref{cor:Ap} for $w>1$ and Proposition \ref{prop:Ap<1}\eqref{it:Ap<1,f(y)} for $w\in(0,1]$, we then conclude that, for a.e.\ $x\in Q\in\mathscr{D}_j$,
\begin{equation*}
\rho_{\aveL^r(Q,V)}(\varphi_j\ast f(x))
\approx
  \rho_{\aveL^w(Q,V)}(\varphi_j*f(x))
  \lesssim
  \begin{cases}\rho_{\aveL^{w'}(Q,V^{-*})}^*(\varphi_j*f(x)), & w>1, \\ \displaystyle
    \Norm{V(x)(\varphi_j*f)(x)}{\fX}, & w\leq 1.
    \end{cases}
\end{equation*}
Hence, for the case $w\le 1$,
we have, for a.e.\ $x\in\mathbb{R}^n$,
\begin{align*}
\rho_{\aveL^r(\mathscr{D}_j,V)}(\varphi_j*f(x))
\lesssim\|V(x)\varphi_j\ast f(x)\|_{\fX}.
\end{align*}
The claim of the proposition in this case then follows by multiplying both sides by $2^{js}$ and taking the $(L^p\ell^q)_A$ norms.

On the other hand,
we would like to apply Lemma \ref{lem:dualNorm} to deal with the case $w>1$.
By Corollary \ref{cor:QvsR}, the required assumption \eqref{eq:weakDbDual}, with
$p=w$ and $N=n$, follows from the present assumption that $V\in\A_w$.
Thus the said Lemma applies to show that, for any $x\in Q\in\mathscr{D}_j$,
\begin{equation*}
  \rho_{\aveL^{w'}(Q,V^{-*})}^*(\varphi_j*f(x))
  \lesssim(\phi_j*\Norm{V(\varphi_j*f)}{}^u)^{\frac1u}(x).
\end{equation*}
This further implies that
\begin{equation*}
\begin{split}
   \Norm{f}{\dot A^s_{p,q}(\rho_{\aveL^r(\mathscr D,V))}}
   &\lesssim\BNorm{\Big\{2^{jsu}\phi_j*\Norm{V(\varphi_j*f)}{}^u\Big\}  }{(L^{\frac{p}{u}}\ell^{\frac{q}{u}})_A}^{\frac1u} \\
   &\lesssim\BNorm{\Big\{2^{jsu}\Norm{V(\varphi_j*f)}{}^u\Big\}  }{(L^{\frac{p}{u}}\ell^{\frac{q}{u}})_A}^{\frac1u}
   =\Norm{f}{\dot A^s_{p,q}(V)},
\end{split}
\end{equation*}
where the second inequality holds provided that both $\frac pu\in(1,\infty)$ and $\frac qu\in(1,\infty]$. Recalling that $u>0$ in Lemma \ref{lem:dualNorm} can be chosen arbitrarily small, we can clearly satisfy the said requirement. Fixing such a choice, the proof is complete.
\end{proof}

The proof of Theorem \ref{thm:equiB} is now a matter of collecting the pieces:

\begin{proof}[Proof of Theorem \ref{thm:equiB}]
The ``$\gtrsim$'' part is simply the Besov case of Proposition \ref{prop:Arho<AV}.

To prove the ``$\lesssim$" part of the present theorem,
we only need to check all the assumptions of
Proposition \ref{prop:BV<Brho} with $\{\rho_Q\}_{Q\in\mathscr{D}}$
therein replaced by $\{\rho_{\aveL^r(Q,V)}\}_{Q\in\mathscr{D}}$.

Noting that the assumptions of Proposition \ref{prop:Arho<AV} are the same as those of Theorem \ref{thm:equiB}, we can use \eqref{thm:equiB-e1}, which implies both \eqref{prop:BV<Brhoe1} and, in combination with Lemma \ref{lem:QvsR}, also \eqref{eq:weakDb}.

Thus the assumptions, and hence the conclusions, of Proposition \ref{prop:BV<Brho} are satisfied, which completes the proof of the ``$\lesssim$" part holds.
\end{proof}

\subsection{Triebel--Lizorkin spaces}

In this subsection,
we extend Theorem \ref{thm:equiB}
to Triebel--Lizorkin spaces.
We still cannot do this for whole range $q\in(0,\infty]$.
Indeed, for $V\in\mathscr{A}_w$ with $w>0$,
we need the following
\emph{optimal reverse H\"older lifting index} $\varepsilon_V$:
\begin{align*}
\begin{split}
\varepsilon_V:=\sup
\big\{\varepsilon\in(0,\infty)&:
\rho_{\aveL^{w+\varepsilon}(Q,V)}(\fx)\lesssim
\rho_{\aveL^w(Q,V)}(\fx)\\
&\qquad\text{holds for all cubes $Q\subset\mathbb{R}^n$ and
all $\fx\in\fX$}\big\}.
\end{split}
\end{align*}

Note that, by Proposition \ref{prop:RHI}, the assumption that $V\in\A_w$ always guarantees
$\varepsilon_V>0$. Then we have the following equivalences for
Triebel--Lizorkin spaces.

\begin{theorem}\label{thm:equiF}
Let $p,r\in(0,\infty)$,
$V\in\A_{\max\{p,r\}}$ with optimal reverse H\"older lifting index
$\varepsilon_V>0$, and $s\in\mathbb{R}$.
If $q\in(0,\max\{p,r\}+\eps_V)$, then, for any
$f\in\mathscr{S}_\infty'(\mathbb{R}^n;\fX)$,
\begin{equation}\label{thm:equiF-e1}
  \Norm{f}{\dot F^s_{p,q}(V)}
  \approx  \Norm{f}{\dot F^s_{p,q}(\rho_{\aveL^r(\mathscr D,V)})}
\end{equation}
and, for any sequence $t=\{t_Q\}_{Q\in\mathscr{D}}$
of vectors in $\fX$,
\begin{equation}\label{thm:equiF-e2}
  \Norm{t}{\dot f^s_{p,q}(V)}
  \approx  \Norm{t}{\dot f^s_{p,q}(\rho_{\aveL^r(\mathscr D,V)})},
\end{equation}
where the positive equivalence constants are independent of $f$ and $t$.
\end{theorem}

\begin{remark}
If, in Theorem \ref{thm:equiF},
let $p=r$ and $\fX=\mathbb{C}^m$, then
the equivalences \eqref{thm:equiF-e1}
and \eqref{thm:equiF-e2}
hold for all $q\in(0,\infty]$, which were proved in
\cite[Theorem 3.1]{FR:21}.
However, when $\fX$ is infinite-dimensional,
we only show these under the restriction
$q\in(0,r+\varepsilon_V)$
in Theorem \ref{thm:equiF}. When $q=\infty$, we prove that the equivalence
\eqref{thm:equiF-e2} may not be true for general infinite-dimensional
case in Proposition \ref{prop:f}.
It is still unclear whether Theorem \ref{thm:equiF} holds
for the remaining case $q\in[r+\varepsilon_V,\infty)$.
\end{remark}

We first consider the ``$\gtrsim$" part of Theorem \ref{thm:equiF}.
Indeed, this has already been done for \eqref{thm:equiF-e1} in Proposition \ref{prop:Arho<AV}.
Hence, we now deal with the ``$\gtrsim$" part of \eqref{thm:equiF-e2} as follows.

\begin{proposition}\label{f-arho<av}
Let $p,r\in(0,\infty)$, $q\in(0,\infty]$,
and $s\in\mathbb{R}$. If $V\in\mathscr{A}_{\max\{p,r\}}$, then
\begin{align}\label{f-arho<ave1}
\|t\|_{\dot{f}^{s}_{p,q}(\rho_{\aveL^r(\mathscr{D},V)})}
\lesssim\|t\|_{\dot{f}^s_{p,q}(V)}.
\end{align}
\end{proposition}

\begin{proof}
Let $t:=\{t_Q\}_{Q\in\mathscr{D}}$ and $w:=\max\{p,r\}$.
If $w\in(0,1]$, then,
for any $j\in\mathbb{Z}$ and
almost every $x\in Q\in\mathscr D_j$, we have
\begin{equation*}
   \rho_{\aveL^r(Q,V)}(t_Q)
   \leq\rho_{\aveL^w(Q,V)}(t_Q)
   \lesssim\Norm{V(x)t_Q}{\fX}
\end{equation*}
by Proposition \ref{prop:Ap<1}\eqref{it:Ap<1inf}. Hence $\rho_{\aveL^r(\mathscr D_j,V)}(t_j(\cdot))\lesssim\Norm{V(\cdot)t_j(\cdot)}{\fX}$, which implies \eqref{f-arho<ave1} when $w\in(0,1]$.

Next, we consider $w\in(1,\infty)$.
Note that, for any given $u<\min\{p,q\}$ and any $Q\in\mathscr{D}$,
\begin{align}\label{f-arho<ave2}
\begin{split}
\rho_{\aveL^r(Q,V)}(t_Q)
&\le\rho_{\aveL^w(Q,V)}(t_Q)
\quad\text{by H\"older's inequality}\\
&\lesssim\rho_{\aveL^u(Q,V)}(t_Q)
\quad\text{by Corollary \ref{cor:Ap1}}.
\end{split}
\end{align}
Since $t_j$ is a constant on every $Q\in\mathscr{D}_j$, that is,
$t_j(x)=\frac{t_Q}{|Q|^{\frac12}}$ for all $x\in Q$, it follows that
\begin{align*}
\frac{\rho_{\aveL^u(Q,V)}(t_Q)}{|Q|^{\frac12}}
&=\Bigg(\fint_Q\BNorm{V(y)\frac{t_Q}{|Q|^{\frac12}}}{\fX}^u\,dy\Bigg)^{\frac1u}
=\Big(\fint_Q\Norm{V(y)t_j(y)}{\fX}^u\,dy\Big)^{\frac1u}\\
&\le\inf_{x\in Q}M_{u}(\Norm{Vt_j}{\fX})(x),
\end{align*}
where
$M_u g(z):=\sup_{Q\owns z}\Norm{g}{\aveL^u(Q)}$ is the rescaled maximal operator.

Using this, we further obtain, for all $j\in\mathbb{Z}$,
\begin{align*}
\rho_{\aveL^r(\mathscr{D}_j,V)}(t_j)
&\lesssim\sum_{Q\in\mathscr{D}_j}
\rho_{\aveL^u(Q,V)}(t_Q)\frac{{\bf1}_Q}{|Q|^{\frac12}}
\quad\text{by \eqref{f-arho<ave2}}\\
&\le M_{u}(\Norm{Vt_j}{\fX}).
\end{align*}
Hence, from $u<\min\{p,q\}$ and the Fefferman--Stein
vector-valued maximal inequality, we infer that
\begin{align*}
\Norm{t}{\dot{f}^s_{p,q}(\rho_{\aveL^r(\mathscr{D},V)})}
&=\BNorm{\Big\{2^{js}\rho_{\aveL^p(\mathscr{D}_j,V)}(t_j)
\Big\}_{j\in\mathbb{Z}}}{L^p\ell^q}
\lesssim\BNorm{\Big\{2^{js}M_{u}
(\Norm{Vt_j}{\fX})\Big\}_{j\in\mathbb{Z}}}{L^p\ell^q}\\
&\lesssim\BNorm{\Big\{2^{js}
\Norm{Vt_j}{\fX}\Big\}_{j\in\mathbb{Z}}}{L^p\ell^q}
=\Norm{t}{\dot{f}^s_{p,q}(V)},
\end{align*}
which shows \eqref{f-arho<ave1} when $p\in[1,\infty)$.
Hence, we complete the proof of Proposition \ref{f-arho<av}.
\end{proof}

Next, we deal with the ``$\lesssim$" part of Theorem \ref{thm:equiF} as follows.

\begin{proposition}\label{prop:FV<Frho,q<p}
Let $p,r\in(0,\infty)$,
$V\in\A_{\max\{p,r\}}$ with the optimal reverse H\"older lifting index
$\varepsilon_V>0$, and $s\in\mathbb{R}$.
If $q\in(0,\max\{p,r\}+\eps_V)$, then, for any
$f\in\mathscr{S}_\infty'(\mathbb{R}^n;\fX)$,
\begin{equation}\label{prop:FV<Frhoe1}
  \Norm{f}{\dot F^s_{p,q}(V)}
  \lesssim  \Norm{f}{\dot F^s_{p,q}(\rho_{\aveL^r(\mathscr D,V)})}
\end{equation}
and, for any sequence $t=\{t_Q\}_{Q\in\mathscr{D}}$
of vectors in $\fX$,
\begin{equation}\label{prop:FV<Frhoe2}
  \Norm{t}{\dot f^s_{p,q}(V)}
  \lesssim  \Norm{t}{\dot f^s_{p,q}(\rho_{\aveL^r(\mathscr D,V)})},
\end{equation}
where the implicit positive constants are independent of $f$ and $t$.
\end{proposition}

\begin{proof}
We first show \eqref{prop:FV<Frhoe1}.
Fix $u\in(0,\min\{p,q\})$.
Then $\frac pu,\frac qu\in(1,\infty)$. By the duality
$(L^{\frac pu}\ell^{\frac qu})'
=L^{(\frac pu)'}\ell^{(\frac qu)'}$,
we obtain
\begin{equation*}
\begin{split}
  \Norm{f}{\dot F^s_{p,q}(V)}
  &=\BNorm{\Big\{2^{js}\|V(\varphi_j*f)\|_{\fX}\Big\}_{j\in\mathbb{Z}} }{L^p\ell^q}
  =\BNorm{\Big\{2^{jsu}\|V(\varphi_j*f)\|_{\fX}^u\Big\}_{j\in\mathbb{Z}} }
  {L^{\frac pu}\ell^{\frac qu}}^{\frac1u} \\
  &=\sup\Big(\int_{\R^n}
  \sum_{j\in\Z}2^{jsu}\Norm{V(x)(\varphi_j*f)(x)}{\fX}^u g_j(x)\,dx\Big)^{\frac1u},
\end{split}
\end{equation*}
where the supremum is taken over $\Norm{\{g_j\}_{j\in\mathbb{Z}}}
{L^{(\frac pu)'}\ell^{(\frac qu)'}}\leq 1$.

Define $w:=\max\{p,r\}$ and choose $\varepsilon\in(0,\varepsilon_V)$ such that,
for any $Q\in\mathscr{D}$ and $\fx\in\fX$,
\begin{align}\label{eq:RHIepse1}
\rho_{\aveL^{w+\varepsilon}(Q,V)}(\fx)\lesssim
\rho_{\aveL^w(Q,V)}(\fx).
\end{align}
For each $j\in\Z$,
\begin{equation*}
   \int_{\R^n} \Norm{V(x)(\varphi_j*f)(x)}{\fX}^u g_j(x)\,dx
   =\sum_{Q\in\mathscr D_j}\abs{Q}\fint_Q \Norm{V(x)(\varphi_j*f)(x)}{\fX}^u g_j(x)\,dx,
\end{equation*}
where
\begin{equation*}
\begin{split}
  \fint_Q &\Norm{V(x)(\varphi_j*f)(x)}{\fX}^u g_j(x)\,dx \\
  &\leq\Big(\fint_Q\Norm{V(x)(\varphi_j*f)(x)}{\fX}^{w+\eps}\,dx\Big)^{\frac{u}{w+\eps}}
  \Big(\fint_Q g_j(x)^{(\frac{w+\eps}{u})'}\,dx\Big)^{1-\frac{u}{w+\eps}}=:I^u\times II.
\end{split}
\end{equation*}
Denoting by $M_s g(z):=\sup_{Q\owns z}\Norm{g}{\aveL^s(Q)}$ the \emph{rescaled maximal function}, it holds
\begin{equation*}
  II\leq\inf_{y\in Q} M_{(\frac{w+\eps}{u})'}g(y).
\end{equation*}

To the factor $I$, we wish to apply Proposition \ref{prop:BV<Brho} and more precisely its conclusion \eqref{eq:BVQ<BrhoQ}, but with $w+\varepsilon$ and $\{\rho_{\aveL^r(Q,V)}\}_{Q\in\mathscr{D}}$
in place of $p$ and $\{\rho_Q\}_{Q\in\mathscr{D}}$,
respectively. To this end, we need to verify the assumptions
\eqref{prop:BV<Brhoe1} and \eqref{eq:weakDb}, similarly modified.
Indeed, from $V\in\mathscr{A}_w$, we infer that,
for every $Q\in\mathscr{D}$ and $\fx\in\fX$,
\begin{align*}
\rho_{\aveL^{w+\varepsilon}(Q,V)}(\fx)
&\lesssim\rho_{\aveL^w(Q,V)}(\fx)\quad\text{by \eqref{eq:RHIepse1}}\\
&\approx\rho_{\aveL^r(Q,V)}(\fx)\quad\text{by \eqref{thm:equiB-e1}}
\end{align*}
and, for every $R\in\mathscr{D}$ such that $\ell(R)=\ell(Q)$,
\begin{equation*}
\begin{split}
  \rho_{\aveL^{r}(Q,V)}(\fx)
  &\approx \rho_{\aveL^{w}(Q,V)}(\fx)\quad\text{by \eqref{thm:equiB-e1}}\\
  &\lesssim \Big(1+\frac{\abs{x_Q-x_R}}{\ell(Q)}\Big)^n\rho_{\aveL^{w}(R,V)}(\fx)\quad\text{by Lemma \ref{lem:QvsR}} \\
  &\approx \Big(1+\frac{\abs{x_Q-x_R}}{\ell(Q)}\Big)^n\rho_{\aveL^{r}(R,V)}(\fx)
  \quad\text{by \eqref{thm:equiB-e1}}.
\end{split}
\end{equation*}
This shows that the assumptions \eqref{prop:BV<Brhoe1} and \eqref{eq:weakDb}
are satisfied with $p$ and $\{\rho_Q\}_{Q\in\mathscr{D}}$ replaced,
respectively, by $w+\varepsilon$ and $\{\rho_{\aveL^r(Q,V)}\}_{Q\in\mathscr{D}}$.
Thus, applying \eqref{eq:BVQ<BrhoQ}, we obtain, for all $y\in Q$,
\begin{equation*}
\begin{split}
  I &\lesssim
  \Big(\phi_j*\rho_{\aveL^{r}(\mathscr D_j,V)}(\varphi_j*f)^v\Big)^{\frac1v}(y),
\end{split}
\end{equation*}
where $v$ can be chosen such that $v<\min\{p,q\}$.
Substituting back, it follows that
\begin{equation*}
  \sum_{Q\in\mathscr D_j}\abs{Q}\, I^u\times II
  \lesssim\int_{\R^n}\Big(\phi_j*\rho_{L^{r}(\mathscr D_j,V)}
  (\varphi_j*f)^v\Big)^{\frac uv}(y) \times M_{(\frac{w+\eps}{u})'}g_j(y)\,dy,
\end{equation*}
and hence
\begin{equation*}
\begin{split}
  &\sum_{j\in\Z} 2^{jsu} \sum_{Q\in\mathscr D_j}\abs{Q}\, I^u\times II \\
  &\quad\lesssim\BNorm{\Big\{2^{jsu}
  \Big(\phi_j*\rho_{L^{r}(\mathscr D_j,V)}(\varphi_j*f)^v\Big)^{\frac uv}\Big\}_{j\in\mathbb{Z}}}
  {L^{\frac pu}\ell^{\frac qu}}
  \BNorm{\Big\{M_{(\frac{w+\eps}{u})'}g_j\Big\}_{j\in\mathbb{Z}}}
  {L^{(\frac pu)'}\ell^{(\frac qu)'}} \\
  &\quad=:{\rm A}\cdot\BNorm{\Big\{M_{(\frac{w+\eps}{u})'}g_j\Big\}_{j\in\mathbb{Z}}}
  {L^{(\frac pu)'}\ell^{(\frac qu)'}}
  \lesssim {\rm A}\cdot\Norm{\{g_j\}_{j\in\mathbb{Z}}}{L^{(\frac pu)'}
  \ell^{(\frac qu)'}}\leq {\rm A}
\end{split}
\end{equation*}
by the Fefferman--Stein vector-valued maximal inequality,
since $(\frac{w+\eps}{u})'<\min\{(\frac{p}{u})',(\frac{q}{u})'\}$
or equivalently $w+\eps>\max\{p,q\}$.

Moreover, by the boundedness of $\{f_j\}\mapsto\{\phi_j*f_j\}$
(note that $\abs{\phi_j*f_j}\leq Mf_j$) on $L^{\frac pv}\ell^{\frac qv}$
with $\frac pv,\frac qv\in(1,\infty)$,
\begin{equation*}
\begin{split}
  {\rm A} &=\BNorm{\Big\{2^{jsv}\phi_j*\rho_{L^{r}(\mathscr D_j,V)}(\varphi_j*f)^v\Big\}
  _{j\in\mathbb{Z}}}{L^{\frac pv}\ell^{\frac qv}}^{\frac uv}
  \lesssim\BNorm{\Big\{2^{jsv}\rho_{L^{r}(\mathscr D_j,V)}(\varphi_j*f)^v\Big\}_{j\in\mathbb{Z}}}
  {L^{\frac pv}\ell^{\frac qv}}^{\frac uv} \\
  &=\BNorm{\Big\{2^{js}\rho_{L^{r}(\mathscr D_j,V)}(\varphi_j*f)\Big\}_{j\in\mathbb{Z}}}
  {L^{p}\ell^{q}}^u
  =\Norm{f}{\dot F^s_{p,q}(\rho_{L^r(\mathscr D,V)})}^u.
\end{split}
\end{equation*}
Combining everything, we have proved \eqref{prop:FV<Frhoe1}.

The proof of \eqref{prop:FV<Frhoe2} is almost the same as
above. The only difference is that, for any $j\in\mathbb{Z}$,
the function $t_j$ is constant on every $Q\in\mathscr{D}_j$ (more precisely,
$t_j(x)=\frac{t_Q}{|Q|^{\frac12}}$ for all $x\in Q$).
Therefore, we have, for any $x\in Q\in\mathscr{D}_j$,
\begin{align*}
\Big(\fint_Q\Norm{V(y)t_j(y)}{\fX}^{w+\varepsilon}\,dy
\Big)^{\frac{1}{w+\varepsilon}}
&=\frac{\rho_{\aveL^{w+\varepsilon}(Q,V)}(t_Q)}{|Q|^{\frac12}}\\
&\lesssim\frac{\rho_{\aveL^{w}(Q,V)}(t_Q)}{|Q|^{\frac12}}
\quad\text{by \eqref{eq:RHIepse1}}\\
&\approx\frac{\rho_{\aveL^{r}(Q,V)}(t_Q)}{|Q|^{\frac12}}
\quad\text{by $V\in\A_{w}$ and \eqref{thm:equiB-e1}}\\
&=\rho_{\aveL^r(Q,V)}(t_j(x)).
\end{align*}
Using this to replace the above estimation for $I$, we can
obtain \eqref{prop:FV<Frhoe2}.
Hence, we complete the proof of Proposition \ref{prop:FV<Frho,q<p}.
\end{proof}

\begin{proof}[Proof of Theorem \ref{thm:equiF}]
By the proofs of Theorem \ref{thm:equiB} and
Proposition \ref{f-arho<av}, we prove the ``$\gtrsim$'' part;
from Proposition \ref{prop:FV<Frho,q<p}, we deduce
the ``$\lesssim$'' part.
Thus, we complete the proof of Theorem \ref{thm:equiF}.
\end{proof}

Finally, we turn to show that \eqref{thm:equiF-e2}
may not be true $q=\infty$. We first
show the following equivalent characterization
of the pointwise quasi-norm being dominated by the average quasi-norm,
which is of independent interest.

\begin{proposition}\label{prop:fequi}
Let $p,r\in(0,\infty)$, $V$ be a weight, and
$s\in\mathbb{R}$.
Then
\begin{align}\label{prop:fequi-e1}
\|t\|_{\dot{f}^s_{p,\infty}(V)}
\lesssim\|t\|_{\dot{f}^s_{p,\infty}(\rho_{\aveL^r(\mathscr{D},V)})}
\end{align}
holds for all sequence $t=\{t_Q\}_{Q\in\mathscr{D}}$ of vectors in $\fX$
with the implicit positive constant independent of $f$
if and only if
\begin{align}\label{prop:fequi-e2}
\fint_Q\sup_{R\in\mathscr{D}:x\in R\subset Q}
\Norm{V(x)t_R}{\fX}^p\,dx
\lesssim\fint_Q\sup_{R\in\mathscr{D}:x\in R\subset Q}
\rho_{\aveL^r(R,V)}(t_R)^p\,dx
\end{align}
holds for all $Q\in\mathscr{D}$ and $\{t_R\}_{R\subset Q}$
in $\fX$
with the implicit positive constant independent of $Q$ and $t$.
\end{proposition}

\begin{remark}
Concerning \eqref{prop:fequi-e2},
there are two remarks.
\begin{enumerate}[{\rm(i)}]
  \item Suppose $\fX=\mathbb{C}^m$ and $r=p$.
  Let $\{A_Q\}_{Q\in\D}$ denote the reducing operators of $V$ (see, for instance, \cite[p.\,490]{FR:21});
that is, $\{A_Q\}_{Q\in\D}$ is a family of positive $m\times m$ matrices such that,
for any $Q\in\D$ and $\fx\in\fX$,
$\rho_{\aveL^p(Q,V)}(\fx)\approx\Norm{A_Q\fx}{\fX}$
with the positive equivalence constants independent of $Q$ and $\fx$.
Then it is easy to show that \eqref{prop:fequi-e2} is
equivalent to
\begin{align}\label{eq:sup}
\sup_{Q\in\D}\fint_Q\sup_{R\in\D:x\in R\subset Q}
\Norm{V(x)A_Q^{-1}}{\mathscr{L}(\fX)}^p\,dx<\infty,
\end{align}
which is \cite[(3.3)]{FR:21} with $r=p$.
Hence, Proposition \ref{prop:f} below
indeed proves that, in the infinite-dimensional case,
\cite[(3.3)]{FR:21} may not be true.
  \item Note that \eqref{eq:sup} and its improvement
  play important roles in the study
  of the boundedness of Calder\'on--Zygmund operators
  in the matrix-weighted case (see \cite{Gold:03}).
  Therefore, the failure of this key estimate
  reveals another essential difficulty in
  the operator-weighted theory.
\end{enumerate}
\end{remark}

\begin{proof}[Proof of Proposition \ref{prop:fequi}]
\eqref{prop:fequi-e2} follows immediately from \eqref{prop:fequi-e1}.
Next, we consider the converse direction
\eqref{prop:fequi-e2} $\Rightarrow$ \eqref{prop:fequi-e1}.
By monotone convergence,  to prove \eqref{prop:fequi-e1}, we only
need to show, for any given $M\in\mathbb{N}$,
\begin{align}\label{prop:fequi-e3}
\BNorm{\Big\{2^{js}\Norm{Vt_j}{\fX}\Big\}_{j\ge-M}}{L^p\ell^\infty}^p
\lesssim\Norm{t}{\dot{f}^s_{p,\infty}(\rho_{\aveL^r(\D,V)})}^p.
\end{align}
Indeed, we have
\begin{align*}
\operatorname{LHS}\eqref{prop:fequi-e3}
&=\int_{\mathbb{R}^n}
\sup_{j\ge-N}\Big|2^{js}\sum_{R\in\mathscr{D}_j}
\Norm{V(x)t_R}{\fX}\frac{{\bf1}_R(x)}{|R|^{\frac12}}\Big|^p\,dx\\
&=\int_{\mathbb{R}^n}\sup_{R\in\mathscr{D}:\ell(R)\le 2^M,x\in R}
\ell(R)^{-(s+\frac n2)p}\Norm{V(x)t_R}{\fX}^p\,dx\\
&=\sum_{Q\in\mathscr{D}_{-M}}
\int_{Q}\sup_{R\in\mathscr{D}:\ell(R)\le 2^M,x\in R}
\ell(R)^{-(s+\frac n2)p}\Norm{V(x)t_R}{\fX}^p\,dx\\
&\lesssim\sum_{Q\in\mathscr{D}_{-M}}
\int_{Q}\sup_{R\in\mathscr{D}:\ell(R)\le 2^M,x\in R}
\ell(R)^{-(s+\frac n2)p}\rho_{\aveL^r(R,V)}(t_R)^p\,dx
\quad \text{by \eqref{prop:fequi-e2}}\\
&=\int_{\mathbb{R}^n}
\sup_{j\ge-N}\Big|2^{js}\sum_{R\in\mathscr{D}_j}
\rho_{\aveL^r(R,V)}(t_R)\frac{{\bf1}_R(x)}{|R|^{\frac12}}\Big|^p\,dx\\
&\le\BNorm{\Big\{2^{js}\rho_{\aveL^r(\mathscr{D}_j,V)}(t_j)
\Big\}_{j\in\mathbb{Z}}}{L^p\ell^\infty}^p
=\operatorname{RHS}\eqref{prop:fequi-e3}.
\end{align*}
This implies \eqref{prop:fequi-e2} $\Rightarrow$ \eqref{prop:fequi-e1}
and hence finishes the proof of Proposition \ref{prop:fequi}.
\end{proof}

\begin{proposition}\label{prop:f}
If $p\in(0,\infty)$ and $r\in[1,\infty)$ such that
$\max\{p,r\}>1$ and if $\fX:=\ell^u(\mathbb{Z})$ for $u\in[1,\infty)$,
then there exists a weight $V\in\mathscr{A}_{\max\{p,r\}}(\mathbb{R};\mathscr{L}(\fX))$
such that \eqref{prop:fequi-e1} does not hold.
\end{proposition}

\begin{proof}
Define $w:=\max\{p,r\}$.
From Proposition \ref{prop:fequi},
we deduce that it is sufficient to find  $V\in\mathscr{A}_w(\mathbb{R};\mathscr{L}(\fX))$
such that \eqref{prop:fequi-e2} does not hold.
Denote by $\fX_1:=\ell^{u}(\mathbb{N})$ and
$\fX_2:=\ell^u(\mathbb{Z}\setminus\mathbb{N})$.
Then $\fX=\fX_1\oplus\fX_2$.
For any given $B\in L^w_{\loc,\so}(\mathbb{R};\mathscr{L}(\fX_1,\fX_2))$,
define
\begin{align*}
V:=\begin{pmatrix} I & 0 \\ B & I \end{pmatrix}.
\end{align*}
By Lemma \ref{lem:Ap<BMO}, we find that, if
$B\in\BMO_{\so}(\mathbb{R};\mathscr{L}(\fX_1,\fX_2))$
and $B^*\in\BMO_{\so}(\mathbb{R};\mathscr{L}(\fX_2^*,\fX_1^*))$,
then $V\in\mathscr{A}_w(\mathbb{R};\mathscr{L}(\fX))$.
For given $Q\in\mathscr{D}$ and $\{s_R\}_{R\subset Q}$ in $\fX_1$,
define
$$t_R:=
\begin{pmatrix}
  s_R\\
  -\ave{B}_R s_R
\end{pmatrix}\in\fX\quad\forall\,R\subset Q.
$$
Then we have, for any $R\subset Q$ and $x\in\mathbb{R}$,
\begin{align}\label{prop:fe1}
V(x)t_R
=\begin{pmatrix}
s_R\\
(B(x)-\ave{B}_R)s_R
 \end{pmatrix}
\end{align}
and hence, if $B\in\BMO_{\so}(\mathbb{R};\mathscr{L}(\fX_1,\fX_2))$,
\begin{align}\label{prop:fe2}
\begin{split}
\rho_{\aveL^r(R,V)}(t_R)^p
&=\Bigg\{\fint_R\Big[\Norm{s_R}{\fX_1}^u
+\Norm{(B(x)-\ave{B}_R)s_R}{\fX_2}^u\Big]^{\frac ru}\,dx\Bigg\}^{\frac pr}\\
&\approx\Bigg\{\fint_R\Big[\Norm{s_R}{\fX_1}^r
+\Norm{(B(x)-\ave{B}_R)s_R}{\fX_2}^r\Big]\,dx\Bigg\}^{\frac pr}\\
&\lesssim\Norm{B}{\BMO_{\so}(\mathbb{R};\mathscr{L}(\fX_1,\fX_2))}^p
\Norm{s_R}{\fX_1}^p\quad\text{by the John--Nirenberg inequality}.
\end{split}
\end{align}
Using this, we further obtain, if \eqref{prop:fequi-e2} holds,
then, for all $\{s_R\}_{R\subset Q}$
such that $\|s_R\|_{\fX_1}=1$\ $\forall\, R\subset Q$,
\begin{align*}
\begin{split}
&\fint_Q\sup_{R\in\mathscr{D}:x\in R\subset Q}
\Norm{(B(x)-\ave{B}_R)s_R}{\fX_2}^p\,dx\\
&\quad\le \fint_Q\sup_{R\in\mathscr{D}:x\in R\subset Q}
\Norm{V(x)t_R}{\fX}^p\,dx\quad\text{by (\ref{prop:fe1})}\\
&\quad\lesssim\fint_Q\sup_{R\in\mathscr{D}:x\in R\subset Q}
\rho_{\aveL^r(R,V)}(t_R)^p\,dx\quad\text{by (\ref{prop:fequi-e2})}\\
&\quad\lesssim\|B\|_{\BMO_{\so}(\mathbb{R};\mathscr{L}(\fX_1,\fX_2))}^p
\fint_Q\sup_{R\in\mathscr{D}:x\in R\subset Q}
\Norm{s_R}{\fX_1}^p\,dx\quad\text{by \eqref{prop:fe2}}\\
&\quad=\|B\|_{\BMO_{\so}(\mathbb{R};\mathscr{L}(\fX_1,\fX_2))}^p.
\end{split}
\end{align*}

Since $\fX_1$ and $\fX_2$ are isometrically isomorphic, the above
argument implies that, to disprove \eqref{prop:fequi-e2},
we only need to find $B\in\BMO_{\so}(\mathbb{R};\mathscr{L}(\fX_1))$
satisfying $B^*\in\BMO_{\so}(\mathbb{R};\mathscr{L}(\fX_1^*))$, but
there exist a cube $Q\in\mathscr{D}$ and a unit sequence $\{s_R\}_{R\subset Q}$
in $\fX_1$ such that
\begin{align}\label{prop:fe3}
\fint_Q\sup_{R\in\mathscr{D}:x\in R\subset Q}
\Norm{(B(x)-\ave{B}_R)s_R}{\fX_1}^p\,dx\lesssim
\Norm{B}{\BMO_{\so}(\mathbb{R};\mathscr{L}(\fX_1))}^p
\end{align}
does not hold. We do this by borrowing some ideas from \cite[Remark 5.2]{Pott:07}.
For any $j\in\mathbb{Z}_+$ and $k=0,\ldots,2^{j}-1$, define
\begin{align*}
b_{2^j+k}(\cdot):=\log|x-2^{-j}k|
\quad\text{and}\quad B:=\operatorname{diag}(b_i)_{i=1}^\infty.
\end{align*}
Note that $b_i\in \BMO(\mathbb{R})$ with a uniform norm with respect to $i$.
Thus, for any $\fx:=\{e_i\}_{i=1}^\infty\in\fX_1$ and any cube $Q$,
it holds that
\begin{align*}
\fint_Q\Norm{(B(x)-\ave{B}_Q)\fx}{\fX_1}^u\,dx
=\sum_{i=1}^{\infty}|e_i|^u\fint_Q
|b_i(x)-\ave{b_i}_Q|^u\,dx
\lesssim\sup_{i\in\mathbb{N}}
\Norm{b_i}{\BMO(\R)}^u\Norm{\fx}{\fX_1}^u.
\end{align*}
This, together with the John--Nirenberg inequality,
shows $B\in\BMO_{\so}(\mathbb{R};\mathscr{L}(\fX_1))$.
Similarly, $B^*\in\BMO_{\so}(\mathbb{R};\mathscr{L}(\fX_1^*))$.
Now, let $Q:=[0,1)$ and, if $R=Q_{j,k}:=2^{-j}[k,k+1)$ for some
$j\in\mathbb{Z}_+$ and $k=0,\ldots,2^j-1$,
let $s_R:=\fx_{2^j+k}$
(that is, the vector whose $(2^j+k)$-th element is 1 and all the other elements are 0).
Hence
\begin{align}\label{prop:fe4}
\begin{split}
&\fint_Q\sup_{R\in\mathscr{D}:x\in R\subset Q}
\Norm{(B(x)-\ave{B}_R)s_R}{\fX_1}^p\,dx\\
&\quad\ge\int_{0}^{1}\sup_{j\in\mathbb{Z}_+,k=0,\ldots2^j-1:x\in Q_{j,k}}
|b_{2^j+k}(x)-\ave{b_{2^j+k}}_{Q_{j,k}}|^p\,dx.
\end{split}
\end{align}
Observe that
\begin{align*}
b_{2^j+k}(x)-\ave{b_{2^j+k}}_{Q_{j,k}}
=\log|x-2^{-j}k|+j\log 2+1
=\log\abs{2^j x-k}+1.
\end{align*}
Note that $x\in Q_{j,k}$ if and only if $2^j\in[k,k+1)$, in which case $2^j-k=\{2^j x\}\in[0,1)$ is the fractional part of $2^j x$.
By this and \eqref{prop:fe4}, we conclude that
\begin{align}\label{prop:fe5}
\fint_Q\sup_{R\in\mathscr{D}:x\in R\subset Q}
\Norm{(B(x)-\ave{B}_R)s_R}{\fX_1}^p\,dx
\geq\int_{0}^{1}\sup_{j\in\mathbb{Z}_+}\abs{\log\{2^jx\}+1}^p\,dx.
\end{align}

We now claim that
$\inf_{j\in\mathbb{Z}_+}\{2^jx\}=0$ for almost every $x\in[0,1)$.
Indeed, if this claim holds, then the supremum on the right of \eqref{prop:fe5} is $\infty$ for a.e. $x$, thus the integral diverges, which disproves \eqref{prop:fe3}.

To check this claim, consider the binary expansion $x=\sum_{k=1}^{\infty}\frac{a_k}{2^k}$ with $a_k\in\{0,1\}$.
Hence $\{2^jx\}=\sum_{k=1}^{\infty}\frac{a_{k+j}}{2^k}$ for each $j\in\Z_+$, and we see that $\inf_{j\in\mathbb{Z}_+}\{2^jx\}=0$ if and only if
the binary expansion $(a_k)_{k=1}^\infty$ of $x$ contains arbitrarily long strings of consecutive zeros. This is a special case of the property that $x$ is {\em normal in base $2$}; the fact that a.e. $x\in[0,1)$ has this property is a classical result of Borel \cite[\S 11--12]{Borel}.
Therefore, we disprove \eqref{prop:fe3}
and hence \eqref{prop:fequi-e1}.
This completes the proof of Proposition \ref{prop:f}.
\end{proof}

\section{The $\varphi$ transform of Frazier and Jawerth}

For given Littlewood--Paley functions $\varphi$ and $\psi$, Frazier and Jawerth \cite{FJ:90} introduced the following of transformations, mapping distributions $f\in\mathscr S_\infty'(\R^n)$ to discrete sequences $t=\{t_Q\}_{Q\in\mathscr D}$ of numbers indexed by the dyadic cubes, and vice versa. The same formulas make perfect sense for Banach space valued distributions $f\in\mathscr S_\infty'(\R^n;\fX)$ and sequences $t=\{t_Q\}_{Q\in\mathscr D}$ of vectors in $\fX$:
\begin{equation}\label{eq:SphiTpsi}
   S_\varphi f:=\{\pair{f}{\varphi_Q}\}_{Q\in\mathscr D},\qquad
   T_\psi t=T_\psi\big(\{t_Q\}_{Q\in\mathscr D}\big):=\sum_{Q\in\mathscr D}t_Q\psi_Q,
\end{equation}
where, for any $Q\in\mathscr{D}$ and $x\in\mathbb{R}^n$,
\begin{align*}
\varphi_{Q}(x):=|Q|^{-\frac12}\varphi\Big(\frac{x-x_Q}{\ell(Q)}\Big),
\qquad \psi_{Q}(x):=|Q|^{-\frac12}\varphi\Big(\frac{x-x_Q}{\ell(Q)}\Big).
\qquad
\end{align*}

The goal of this section is to use the transforms \eqref{eq:SphiTpsi} to establish a correspondence between the function space and the sequence space version
of the Besov and Triebel--Lizorkin spaces
introduced in Definition \ref{def:space}.
We first consider the average-weighted spaces.
Indeed, we can prove a $\varphi$-transform characterization and
many other properties of $\dot{A}^s_{p,q}(\rho)$ under the following mild assumptions on $\rho$.

\begin{definition}\label{def:uabc}
Let $u\in(0,1]$ and $a,b,c\in[0,\infty)$.
A family $\rho=\{\rho_Q\}_{Q\in\mathscr D}$ of quasi-norms
is called a \emph{family of $(u,a,b,c)$-norms} if
\begin{enumerate}[(i)]
  \item for each $Q\in\mathscr{D}$, $\rho_Q$ is a $u$-norm
  in the sense of Definition \ref{def:semi},
  \item\label{def:uabc-it2} for each $Q\in\mathscr{D}$, $\rho_Q$ is equivalent to $\Norm{\cdot}{\fX}$,
  where the positive equivalence constants may depend on $Q$,
  \item\label{def:uabc-it3} for all $Q,R\in\mathscr{D}$ and $\fx\in\fX$,
  \begin{equation}\label{eq:strDb}
  \rho_Q(\fx)\lesssim  B_{a,b,c}(Q,R)\rho_R(\fx),
\end{equation}
where $B(a,b,c)$ is as in \eqref{eq:BabD}.
\end{enumerate}
\end{definition}

\begin{remark}\label{rem:uabc}
Let $r\in(0,\infty)$ and $V$ be a weight.
Under the following assumptions,
the family of quasi-norms $\{\rho_{\aveL^r(Q,V)}\}_{Q\in\mathscr{D}}$
is a family of $(u,a,b,c)$-norms.
\begin{enumerate}[{\rm(i)}]
  \item\label{rem:uabc-it1} If $V$ is an $(r,\beta)$-doubling weight
  for $\beta\in[n,\infty)$ in the sense of
Definition \ref{def:dbDim}, then, combining
Lemmas \ref{lem:equifx} and \ref{lem:QvsRdb}, we conclude that
$\{\rho_{\aveL^r(Q,V)}\}_{Q\in\mathscr{D}}$
is a family of $(\min\{1,r\},\frac nr,\frac{\beta-n}{r},\frac \beta r)$-norms.

 \item\label{rem:uabc-it2} If $V\in \mathscr{A}_w$ with $w\ge r$ has doubling
 dimension $\beta\in[n,\infty)$, then $\{\rho_{\aveL^r(Q,V)}\}_{Q\in\mathscr{D}}$
 is a family of $(\min\{1,r\},\frac nw,\frac{\beta-n}{w},\frac \beta w)$-norms.
 Indeed, from H\"older inequality and Corollary \ref{cor:Ap1}, we infer that,
 for any $Q\in\mathscr{D}$ and $\fx\in\fX$,
 \begin{align*}
 \rho_{\aveL^r(Q,V)}(\fx)\le\rho_{\aveL^w(Q,V)}(\fx)
 \lesssim\rho_{\aveL^r(Q,V)}(\fx).
 \end{align*}
 Hence, $\rho_{\aveL^r(Q,V)}(\fx)\approx \rho_{\aveL^w(Q,V)}(\fx)$.
Together with Lemmas \ref{lem:equifx} and \ref{lem:QvsR},
this implies the desired claim.
\end{enumerate}
\end{remark}

\begin{theorem}\label{thm:phi}
Let $p\in(0,\infty)$, $q\in(0,\infty]$,
and $s\in\mathbb{R}$.
Suppose that $\rho=\{\rho_Q\}_{Q\in\mathscr D}$ is a family
of $(u,a,b,c)$-norms for some $u\in(0,1]$ and
$a,b,c\in[0,\infty)$. Then
\begin{enumerate}[\rm(i)]
\item\label{eq:AindepPhi}  $\dot A^s_{p,q}(\rho,\eta)$
is independent of the choice of the Littlewood--Paley function $\eta$.
  \item $S_{\varphi}:\dot A^s_{p,q}(\rho,\tilde{\varphi})\to \dot a^s_{p,q}(\rho)$ is bounded,
  where $\tilde{\varphi}(\cdot):=\varphi(-\cdot)$ denotes the reflection of $\varphi$.
  \item $T_{\psi}:\dot a^s_{p,q}(\rho)\to \dot A^s_{p,q}(\rho,\varphi)$ is bounded.
  \item If, in addition, $\varphi$ and $\psi$ satisfy
\begin{equation}\label{eq:LPpair}
  \sum_{j\in\Z}\hat\varphi(-2^j\xi)\hat\psi(2^j\xi)\equiv 1,\quad\forall\ \xi\in\R^n\setminus\{0\},
\end{equation}
then $T_\psi\circ S_{\varphi}$ acts as the identity on $\dot A^s_{p,q}(\rho,\varphi)$.
\end{enumerate}
The assumptions, and hence the conclusions, hold in particular if
$\rho_Q=\rho_{\aveL^r(Q;V)}$, where
$V$ is an $(r,\beta)$-doubling weight.
\end{theorem}

Note that the last claim of Theorem \ref{thm:phi} is immediate from Remark \ref{rem:uabc}.
Before going to the proof of Theorem \ref{thm:phi},
we now apply this theorem and the relationships between
average and pointwise-weighted spaces studied in Section \ref{sec:BTL}
to obtain the $\varphi$-transform characterization of
pointwise-weighted spaces.

\begin{theorem}\label{thm:phiV}
Let $p\in(0,\infty)$, $V\in\bigcup_{r\ge p}\mathscr{A}_r$, and $s\in\mathbb{R}$.
\begin{enumerate}[{\rm(i)}]
  \item\label{thm:phiVB} If $q\in(0,\infty]$, then
 $\dot B^s_{p,q}(V,\eta)$
is independent of the choice of the Littlewood--Paley function $\eta$ and the operators
  $$S_{\varphi}:\dot B^s_{p,q}(V,\tilde{\varphi})\to \dot b^s_{p,q}(V)
  \quad\text{and}\quad T_{\psi}:\dot b^s_{p,q}(V)\to \dot B^s_{p,q}(V,\varphi)$$
  are bounded. Furthermore, if $\varphi$ and $\psi$ satisfy \eqref{eq:LPpair},
  then $T_\psi\circ S_{\varphi}$ acts as the identity on $\dot B^s_{p,q}(V,\varphi)$.

  \item If $V\in \mathscr{A}_r$ with some $r\in[p,\infty)$
  has optimal reverse H\"older lifting index $\varepsilon_V$ and if $q\in(0,r+\varepsilon_V)$,
  then all the conclusions of \eqref{thm:phiVB} hold with $B$ and $b$
  replaced by $F$ and $f$.
\end{enumerate}
\end{theorem}

\begin{proof}
Assume that $V\in A_r$ with $r\ge p$ and the doubling dimension $\beta\in[n,\infty)$.
From Remark \ref{rem:uabc}\eqref{rem:uabc-it2},
we infer that $\{\rho_{\aveL^r(Q,V)}\}_{Q\in\mathscr{D}}$
is a family of $(\min\{1,r\},\frac nr,\frac{\beta-n}{r},\frac{\beta}{r})$-norms.
Hence, all the conclusions of Theorem \ref{thm:phi} hold
for $\rho=\{\rho_{\aveL^r(Q,V)}\}_{Q\in\mathscr{D}}$.
On the other hand, under the assumptions of the present theorem, we have
\begin{align*}
\begin{split}
\|\cdot\|_{\dot{A}^{s}_{p,q}(V,\varphi)}
&\approx\|\cdot\|_{\dot{A}^s_{p,q}(\rho_{\aveL^r(\mathscr{D},V)},\varphi)}
\quad\text{by Theorems \ref{thm:equiB} and \ref{thm:equiF}},\\
\|\cdot\|_{\dot{a}^{s}_{p,q}(V)}
&\approx\|\cdot\|_{\dot{a}^s_{p,q}(\rho_{\aveL^r(\mathscr{D},V)})}
\quad\text{by Corollary \ref{cor:equib} and Theorem \ref{thm:equiF}}.
\end{split}
\end{align*}
This proves the present theorem.
\end{proof}

Now, we provide the necessary ingredients for the proof of Theorem \ref{thm:phi}.
We first record an easy but useful estimate:

\begin{lemma}\label{lem:tReasy}
Let $p\in(0,\infty)$, $q\in(0,\infty]$,
and $s\in\mathbb{R}$. If  $\rho=\{\rho_Q\}_{Q\in\mathscr{D}}$
is a family
of quasi-norms, then for any $R\in\mathscr{D}$ and any sequence $t=\{t_Q\}_{Q\in\mathscr{D}}$
of vectors in $\fX$,
\begin{equation*}
  \rho_R(t_R)\leq\ell(R)^s\abs{R}^{\frac12-\frac1p}\Norm{t}{\dot a^s_{p,q}(\rho)}.
\end{equation*}
\end{lemma}

\begin{proof}
By dropping all but one summand in the definition, we find that
\begin{equation*}
  \Norm{t}{\dot a^s_{p,q}(\rho)}\geq\BNorm{\ell(R)^{-s}\rho_R(t_R)
  \frac{\one_R}{\abs{R}^{\frac12}} }{L^p}=\ell(R)^{-s}\abs{R}^{\frac1p-\frac12}\rho_R(t_R),
\end{equation*}
which completes the proof of the lemma.
\end{proof}

Next, we need another quasi-norm of average-weighted spaces.
For any $f\in\mathscr{S}'_\infty(\mathbb{R}^n;\fX)$, define
\begin{align*}
\Norm{f}{\dot A^s_{p,q}(\rho,\sup,\varphi)}
  &:=\BNorm{\Big\{2^{js}\Norm{\rho_{\mathscr D_j}(\varphi_j*f)}{\aveL^\infty(\mathscr D_j)}\Big\}_{j\in\Z}}{(L^p\ell^q)_A}
\end{align*}
and
\begin{align*}
\dot{A}^s_{p,q}(\rho,\sup,\varphi)
:=\Big\{f\in\mathscr{S}'_\infty(\mathbb{R}^n;\fX):
\Norm{f}{\dot A^s_{p,q}(\rho,\sup,\varphi)}<\infty\Big\}.
\end{align*}

\begin{proposition}\label{prop:SphiBd}
For any family of quasi-norms $\rho$
and any $f\in\mathscr{S}'_\infty(\mathbb{R}^n;\fX)$, we have
\begin{equation*}
  \Norm{S_\varphi f}{\dot a^s_{p,q}(\rho)}\leq\Norm{f}{\dot A^s_{p,q}(\rho,\sup,\tilde\varphi)}.
\end{equation*}
\end{proposition}

Note that the bound here is just ``$\leq$'' instead of the typical ``$\lesssim$'', reflecting the relatively elementary nature of this estimate.

\begin{proof}[Proof of Proposition \ref{prop:SphiBd}]
\begin{equation*}
  \pair{f}{\varphi_Q}
  =\int f(y)\abs{Q}^{-\frac12}\varphi\Big(\frac{y-x_Q}{\ell(Q)}\Big)\,dy
  =\abs{Q}^{\frac12} f*\tilde\varphi_j(x_Q),
\end{equation*}
where $\tilde\varphi(z):=\varphi(-z)$ is the reflection. Thus
\begin{equation*}
\begin{split}
  \frac{\rho_Q(\pair{f}{\varphi_Q})}{\abs{Q}^{\frac12}}
  =\rho_Q(f*\tilde\varphi_j(x_Q))
\leq\sup_{y\in Q}\rho_Q(f*\tilde\varphi_j(y))
  =\Norm{\rho_{\mathscr D_j}(f*\tilde\varphi_j)}{\aveL^\infty(\mathscr D_j)}(x)
\end{split}
\end{equation*}
for all $x\in Q\in\mathscr D_j$. It follows that
\begin{equation*}
   \sum_{Q\in\mathscr D_j}\rho_Q(t_Q)
  \frac{\one_Q}{\abs{Q}^{\frac12}}
  \leq\Norm{\rho_{\mathscr D_j}(f*\tilde\varphi_j)}{\aveL^\infty(\mathscr D_j)}
\end{equation*}
pointwise, and hence
\begin{equation*}
   \Norm{S_\varphi f}{\dot a^s_{p,q}(\rho)}
   \leq\BNorm{\Big\{2^{js}\Norm{\rho_{\mathscr D_j}(f*\tilde\varphi_j)}{\aveL^\infty(\mathscr D_j)}\Big\}}{(L^p\ell^q)_A}
   =\Norm{f}{\dot A^s_{p,q}(\rho,\sup,\tilde\varphi)}.
\end{equation*}
This finishes the proof of the proposition.
\end{proof}

\begin{proposition}\label{prop:Asup<A}
Let $\rho=\{\rho_Q\}_{Q\in\mathscr D}$ be a family of
$u$-norms with some fixed $u\in(0,1]$ independent of $Q$.
Suppose that $\rho$ satisfies
\eqref{eq:weakDb} for some $N\geq 0$, all $\fx\in\fX\setminus\{0\}$,
and all cubes $Q,R\subset\R^n$ of equal size $\ell(Q)=\ell(R)$.
Then
\begin{equation}\label{eq:Asup<A-pointwise}
  \Norm{\rho_{\mathscr D_j}(\varphi_j*f)}{\aveL^\infty(\mathscr D_j)}\lesssim\big(\phi_j*\rho_{\mathscr D_j}(\varphi_j*f)^u\big)^{\frac1u},
\end{equation}
where $\phi(z)=(1+\abs{z})^{-M}$, $M>0$ can be taken as large as we like,
and the implicit positive constant is independent of $f$. Moreover,
\begin{equation*}
  \Norm{f}{\dot A^s_{p,q}(\rho,\sup,\varphi)}\lesssim\Norm{f}{\dot A^s_{p,q}(\rho,\varphi)},
\end{equation*}
where the implicit positive constant is independent of $f$.
\end{proposition}

\begin{remark}\label{rem:A<Asup}
The converse inequality
\begin{equation*}
  \Norm{f}{\dot A^s_{p,q}(\rho,\varphi)}\leq\Norm{f}{\dot A^s_{p,q}(\rho,\sup,\varphi)}
\end{equation*}
is trivial without any assumptions on $\rho$.
\end{remark}

\begin{proof}[Proof of Proposition \ref{prop:Asup<A}]
Let $x,y\in Q\in\mathscr D_j$. We apply Lemma \ref{lem:extraPhi}
with the $u$-norm $\rho_Q$. Since every $u$-norm is also a $\tilde u$-norm for all $\tilde u\in(0,u)$, we assume that $u\in(0,\min\{p,q\})$ to begin with. (This is only needed later in the proof.) This gives
\begin{equation*}
  \rho_Q(\varphi_j*f(x))^u
  \lesssim\int_{\R^n}\phi_j'(x-z)\rho_Q(\varphi_j*f(z))^u \,dz,
\end{equation*}
where $\phi_j'(z)=(1+\abs{z})^{-M'}$, and $M'>0$ can be taken as large as we like. Moreover,
by \eqref{eq:applyWeakDb}, we obtain
\begin{equation*}
  \int_{\R^n} \phi_j'(x-z)\rho_Q(\varphi_j*f(z))^u \,dz
  \lesssim\int_{\R^n}\phi_j(y-z)\rho_{\mathscr D_j}(\varphi*f)(z)^u \,dz,
\end{equation*}
where $\phi(z):=\phi'(z)(1+\abs{z})^{Nu}=(1+\abs{z})^{-M}$, and $M=M'-Nu$ can still be chosen as large as we like by taking $M'$ large enough. Since the previous bound holds for all $x,y\in Q\in\mathscr D_j$, the claim \eqref{eq:Asup<A-pointwise} follows.

For the second claim, we multiply both sides of \eqref{eq:Asup<A-pointwise} by $2^{js}$ and take the $(L^p\ell^q)_A$ quasi-norms of both sides. This gives
\begin{equation*}
\begin{split}
  \Norm{f}{\dot A^s_{p,q}(\rho,\sup,\varphi)}
  &=\BNorm{\Big\{2^{js}\Norm{\rho_{\mathscr D_j}(\varphi_j*f)}{\aveL^\infty(\mathscr D_j)}\Big\}}{(L^p\ell^q)_A} \lesssim\BNorm{\Big\{2^{js}\big(\phi_j*\rho_{\mathscr D_j}(\varphi_j*f)^u\big)^{\frac1u}\Big\}}{(L^p\ell^q)_A} \\
  &=\BNorm{\Big\{2^{js u}\phi_j*\rho_{\mathscr D_j}(\varphi_j*f)^u\Big\}}{(L^{\frac pu}\ell^{\frac qu})_A}^{\frac1u}
  \lesssim\BNorm{\Big\{2^{js u}\rho_{\mathscr D_j}(\varphi_j*f)^u\Big\}}{(L^{\frac pu}\ell^{\frac qu})_A}^{\frac1u} \\
  &=\BNorm{\Big\{2^{js}\rho_{\mathscr D_j}(\varphi_j*f)\Big\}}{(L^{p}\ell^{q})_A}
  =\Norm{f}{\dot A^s_{p,q}(\rho,\varphi)},
\end{split}
\end{equation*}
where the last inequality follows from the Fefferman--Stein vector-valued maximal inequality, noting that $|\phi_j*f|\lesssim Mf$ and $\frac pu,\frac qu>1$.
\end{proof}

\begin{lemma}\label{lem:phiPsi}
For $\varphi,\psi\in\mathscr S_\infty$, we have the estimates
\begin{equation*}
  \abs{\varphi_i*\psi_j(x)}\lesssim 2^{-\abs{i-j}M}\phi_{\min(i,j)}(x),
\end{equation*}
where $\phi(x):=(1+\abs{x})^{-M}$,  $\phi_i(x)=2^{in}\phi(2^i x)$,
and the implicit positive constant is independent of $i$, $j$, and $x$, and
\begin{equation*}
  \abs{\pair{\varphi_Q}{\psi_R}}
  \lesssim B_{-M}(Q,R):=B_{-M,-M,-M}(Q,R),
\end{equation*}
where the right hand side is defined in \eqref{eq:BabD}
and the implicit positive constant is independent of $Q$ and $R$.
In each case, $M>0$ can be taken as large as we like.
\end{lemma}

\begin{proof}
The first estimate is \cite[Lemma 2.2]{YY:08}, which is also reproduced as \cite[Lemma 3.31]{BHYY:I}, and the second one is \cite[Corollary 3.32]{BHYY:I}. In the cited sources, more precise forms of the upper bound are given; we have omitted the details that are irrelevant to us.
\end{proof}

\begin{lemma}\label{lem:TpsiWellDef}
Suppose that $\rho=\{\rho_Q\}_{Q\in\mathscr D}$ is a family of quasi-norms on $\fX$, each qualitatively equivalent to $\Norm{\ }{\fX}$, and that \eqref{eq:strDb} holds for some $a,b,c\geq 0$. Then for all $t\in\dot a^s_{p,q}(\rho)$, the series $T_\psi t$ as in \eqref{eq:SphiTpsi} converges in $\mathscr S_\infty'(\R^n;\fX)$.
\end{lemma}

\begin{proof}
The claim means that the series
\begin{equation*}
  \sum_{R\in\mathscr D}\Norm{\pair{\varphi}{\psi_R}t_R}{\fX}
\end{equation*}
converges for all $\varphi\in\mathscr S_\infty$. If $Q_0:=[0,1)^n$, we note that $\varphi=\varphi_{Q_0}$. Then
\begin{equation*}
\begin{split}
   \Norm{\pair{\varphi}{\psi_R}t_R}{\fX}
   &=\abs{\pair{\varphi}{\psi_R}}\sup_{\Norm{\fx^*}{\fX^*}\leq 1}\abs{\pair{t_R}{\fx^*}}
   \leq B_{-M}(Q_0,R)\sup_{\Norm{\fx^*}{\fX^*}\leq 1}\rho_R(t_R)\rho_R^*(\fx^*),
\end{split}
\end{equation*}
where
\begin{equation*}
\begin{split}
  \rho_R(t_R)
  &\leq\ell(R)^s\abs{R}^{\frac12-\frac1p}\Norm{t}{\dot a^s_{p,q}(\rho)}
  \approx  B_{\sigma,-\sigma,0}(Q_0,R)\Norm{t}{\dot a^s_{p,q}(\rho)},
  \qquad\sigma:=s+n\Big(\frac12-\frac1p\Big),
\end{split}
\end{equation*}
by Lemma \ref{lem:tReasy}, and
\begin{equation*}
  \rho_R^*(\fx^*)\lesssim B_{a,b,c}(Q_0,R)\rho_{Q_0}^*(\fx^*)
\end{equation*}
by assumption \eqref{eq:strDb} and the definition of the dual norms. Moreover,
\begin{equation*}
  \rho_{Q_0}^*(\fx^*)=\sup_{\fx\in\fX\setminus\{0\}}\frac{\abs{\pair{\fx^*}{\fx}}}{\rho_{Q_0}(\fx)}
  \approx\sup_{\fx\in\fX\setminus\{0\}}\frac{\abs{\pair{\fx^*}{\fx}}}{\Norm{\fx}{\fX}}
  =\Norm{\fx^*}{\fX^*}\leq 1
\end{equation*}
by the assumed equivalence of $\rho_{Q_0}$ and $\Norm{\ }{\fX}$.

Hence,
\begin{equation*}
\begin{split}
  \sum_{R\in\mathscr D}&\Norm{\pair{\varphi}{\psi_R}t_R}{\fX}
  \lesssim\sum_{R\in\mathscr D}B_{-M}(Q_0,R)B_{\sigma,-\sigma,0}(Q_0,R)
  B_{a,b,c}(Q_0,R)\Norm{t}{\dot a^s_{p,q}(\rho)},
\end{split}
\end{equation*}
and this series converges as soon as $M>0$ is chosen large enough.
\end{proof}

\begin{lemma}\label{lem:A<a}
Suppose that $\rho=\{\rho_Q\}_{Q\in\D}$ is
a family of $(u,a,b,c)$-norms for some
$u\in(0,1]$ and $a,b,c\in(0,\infty)$.
For $t\in\dot a^s_{p,q}(\rho)$ and two Littlewood--Paley functions $\varphi,\psi$, we have
\begin{equation*}
  \Norm{\rho_{\mathscr D_i}(\varphi_i*T_\psi t)}{\aveL^\infty(\mathscr D_i)}
  \lesssim\sum_{j:\abs{j-i}\leq 3}\big(\phi_j*\rho_{\mathscr D_j}(t_j)^u\big)^{\frac1u},
\end{equation*}
where $\phi_j(x):=2^{jn}\phi(2^j x)$, $\phi(x):=(1+\abs{x})^{-M}$ with $M>0$ as large as we like, and
the implicit positive constant is independent of
$i$ and $t$.
\end{lemma}

\begin{proof}
By Lemma \ref{lem:TpsiWellDef}, we can meaningfully consider
\begin{equation*}
  \varphi_i*T_{\psi}t(x)=\sum_{R\in\mathscr D}\varphi_i*\psi_R(x) t_R,
\end{equation*}
where
\begin{equation*}
    \psi_R(x)=\frac{1}{\abs{R}^{\frac12}}\psi\Big(\frac{x-x_R}{\ell(R)}\Big)=\abs{R}^{\frac12}\psi_{j_R}(x-x_R)
\end{equation*}
and $j_R:=-\log_2\ell(R)$.
The convolution $\varphi_i*\psi_R$ can be non-zero only if $\supp\hat\varphi_i$ and $\supp\hat\psi_R=\supp\hat\psi_{j_R}$ intersect,
which requires $\pi^{-2}<2^i/2^{j_R}<\pi^2$, i.e.,
\begin{equation*}
  \abs{i-j_R}<\log_2\pi^2<4.
\end{equation*}

Then, for $x\in Q\in\mathscr D_i$,
\begin{equation*}
\begin{split}
  \rho_Q( \varphi_i* T_{\psi}t(x) )^u &\leq\sum_{\genfrac{}{}{0pt}{}{R\in\mathscr D} {\abs{j_R-i}<4}}\big(\abs{R}^{\frac12}\abs{\varphi_i*\psi_{j_R}(x-x_R)}\cdot\rho_Q(t_R)\big)^u  \\
  &\lesssim\sum_{\genfrac{}{}{0pt}{}{R\in\mathscr D} {\abs{j_R-i}<4}}\big(\abs{R}^{\frac12}\phi_i'(x-x_R)\cdot (1+2^i\abs{x_Q-x_R})^c \rho_R(t_R)\big)^u,
\end{split}
\end{equation*}
using Lemma \ref{lem:phiPsi} to produce the factor with $\phi'(z)=(1+\abs{z})^{-M'}$, where $M'>0$
can be as large as we like, and assumption \eqref{eq:strDb} to produce the factor $(1+2^i\abs{x_Q-x_R})^c$. The fact that $\abs{j_R-i}<4$ allows us to ignore several factors involving $\ell(Q)/\ell(R)$ or $\abs{i-j_R}$, and use the numbers $i$ and $j_R$ interchangeably in the upper bound. We can then continue the estimate with
\begin{equation*}
\begin{split}
  \rho_Q( \varphi_i* T_{\psi}t(x) )^u
  &\lesssim\sum_{\genfrac{}{}{0pt}{}{R\in\mathscr D}{\abs{j_R-i}<4}} \Big(1+2^i\abs{x_Q-x_R}\Big)^{(c-M')u}
    \Big(\frac{\rho_R(t_R)}{\abs{R}^{\frac12}}\Big)^u \\
  &\approx\sum_{\abs{j-i}<4}\sum_{R\in\mathscr D_j} \int_R 2^{jn}\Big(1+2^j\abs{x-y}\Big)^{(c-M')u}
    (\rho_{\mathscr D_j}t_j)(y)^u \,dy \\
   &=\sum_{\abs{j-i}<4}\phi_j*(\rho_{\mathscr D_j}t_j)^u(x),
\end{split}
\end{equation*}
where $\phi(x):=(1+\abs{x})^{-M}$ and $M:=(M'-c)u$ may still be as large as we like, by taking $M'$ sufficiently large.
\end{proof}

\begin{proposition}\label{prop:TpsiBd}
Suppose that $\rho=\{\rho_Q\}_{Q\in\D}$ is
a family of $(u,a,b,c)$-norms for some
$u\in(0,1]$ and $a,b,c\in(0,\infty)$.
Then for all $t\in\dot a^s_{p,q}(\rho)$, the series $T_\psi t$ converges in
$\mathscr S_\infty'(\R^n;\fX)$, and
\begin{equation*}
  \Norm{T_\psi t}{\dot A^s_{p,q}(\rho,\sup,\varphi)}
  \lesssim\Norm{t}{\dot a^s_{p,q}(\rho)}.
\end{equation*}
\end{proposition}

\begin{proof}
Without loss of generality, we may take $u\in(0,\min\{p,q\})$.
The convergence of the series was already established in Lemma \ref{lem:TpsiWellDef}, and Lemma \ref{lem:A<a} gives the pointwise estimate
\begin{equation*}
  \Norm{\rho_{\mathscr D_i}(\varphi_i*T_\psi t)}{\aveL^\infty(\mathscr D_i)}
  \lesssim\sum_{k=-3}^3\big(\phi_{i+k}*\rho_{\mathscr D_{i+k}}(t_{i+k})^u\big)^{\frac1u}.
\end{equation*}
Multiplying by $2^{is}$, taking the $(L^p\ell^q)_A$ quasi-norms of both sides, and proceeding with the usual estimates, we obtain
\begin{equation*}
\begin{split}
  \Norm{T_\psi t}{\dot A^s_{p,q}(\rho,\sup,\varphi)}
  &=\BNorm{\Big\{2^{is}\Norm{\rho_{\mathscr D_i}(\varphi_i*T_\psi t)}{\aveL^\infty(\mathscr D_i)}\Big\}}{(L^p\ell^q)_A} \\
  &\lesssim\sum_{k=-3}^3\BNorm{\Big\{2^{is}\big(\phi_{i+k}*\rho_{\mathscr D_{i+k}}(t_{i+k})^u\big)^{\frac1u}\Big\}}{(L^p\ell^q)_A} \\
  &\approx\BNorm{ \Big\{ 2^{is u}\phi_{i}*\rho_{\mathscr D_{i}}(t_{i})^u\Big\}}{(L^{\frac pu}\ell^{\frac qu})_A}^{\frac 1u}
  \lesssim\BNorm{ \Big\{ 2^{is u}\rho_{\mathscr D_{i}}(t_{i})^u\Big\}}{(L^{\frac pu}\ell^{\frac qu})_A}^{\frac 1u} \\
  &\approx\BNorm{ \Big\{ 2^{is}\rho_{\mathscr D_{i}}(t_{i})\Big\}}{(L^{p}\ell^{q})_A}
  =\Norm{t}{\dot a^s_{p,q}(\rho)};
\end{split}
\end{equation*}
in the last estimate removing the convolutions $\phi_i$, we note that $|\phi_i*f|\lesssim Mf$, so that this can be deduced from the Fefferman--Stein vector-valued maximal inequality, given that $\frac pu,\frac qu>1$.
\end{proof}

\begin{lemma}\label{lem:id}
Let $\varphi,\psi$ be two Littlewood--Paley functions that satisfy \eqref{eq:LPpair}.
Then for all $f\in\mathscr S_\infty'(\R^n)$, we have
\begin{equation*}
  f=\sum_{Q\in\mathscr D}\pair{f}{\varphi_Q}\psi_Q.
\end{equation*}
\end{lemma}

\begin{proof}
This is \cite[Lemma 2.1]{FJ:85} (replacing $\varphi(z)$ by $\varphi(-z)$).
\end{proof}

We now have everything to give:

\begin{proof}[Proof of Theorem \ref{thm:phi}]
We already know from Proposition \ref{prop:Asup<A} and Remark \ref{rem:A<Asup} that
$$\dot A^s_{p,q}(\rho,\sup,\varphi)=
\dot{A}^s_{p,q}(\rho,\varphi),$$
so we will use these spaces interchangeably and omit the ``$\sup$''.
By Propositions \ref{prop:SphiBd} and \ref{prop:TpsiBd}, we know that
\begin{equation*}
  S_{\varphi}:\dot A^s_{p,q}(\rho,\tilde\varphi)\to\dot a^s_{p,q}(\rho)\quad
  \text{and}\quad
  T_\psi:\dot a^s_{p,q}(\rho)\to \dot A^s_{p,q}(\rho,\chi)
\end{equation*}
are bounded for all Littlewood--Paley functions $\varphi,\psi,\chi$. Hence
\begin{equation}\label{eq:Aphi2Achi}
  T_\psi\circ S_{\varphi}:\dot A^s_{p,q}(\rho,\tilde\varphi)
  \to \dot A^s_{p,q}(\rho,\chi)
\end{equation}
is bounded, and Lemma \ref{lem:TpsiWellDef} informs us that
\begin{equation*}
  T_\psi(S_\varphi f)
  =\sum_{Q\in\mathscr D}\pair{f}{\varphi_Q}\psi_Q
\end{equation*}
converges in $\mathscr S_\infty'(\R^n;\fX)$. Hence
\begin{equation}\label{eq:TS=Id}
  \pair{T_\psi(S_\varphi f)}{\fx^*}
  =\sum_{Q\in\mathscr D}\pair{\pair{f}{\fx^*}}{\varphi_Q}\psi_Q
  =T_\psi(S_\varphi \pair{f}{\fx^*})
\end{equation}
converges in $\mathscr S_\infty'(\R^n)$, where $\pair{f}{\fx^*}\in\mathscr S_\infty'(\R^n)$ is the distribution defined by $$\pair{\pair{f}{\fx^*}}{\phi}:=\pair{\pair{f}{\phi}}{\fx^*}$$ for all $\phi\in\mathscr S_\infty(\R^n)$.

Given a Littlewood--Paley function $\varphi$, it is well known how to construct another Littlewood--Paley function $\psi$ so that \eqref{eq:LPpair} holds: indeed, we can take
\begin{equation*}
  \hat\psi(\xi):=
  \frac{\overline{\hat\varphi(-\xi)}}
  {\sum\limits_{k\in\Z}
  \abs{\hat\varphi(2^k\xi)}^2},
\end{equation*}
where the denominator is strictly positive since the $k$-th term is strictly positive for $\abs{\xi}\in[\alpha^{-1},\alpha]2^{-k}$ and the condition that $\alpha>\sqrt{2}$ guarantees that the consecutive intervals overlap. For such $\psi$, Lemma \ref{lem:id} guarantees that the right-hand side of \eqref{eq:TS=Id} is simply $\pair{f}{\fx^*}$. Since this is true for all $\fx^*\in\fX^*$, \eqref{eq:TS=Id} implies that $T_\psi(S_\varphi f)=f$ for all $f\in\dot A^s_{p,q}(\rho,\tilde\varphi)$, and the boundedness \eqref{eq:Aphi2Achi} shows that
\begin{equation*}
  \Norm{f}{\dot A^s_{p,q}(\rho,\chi)}\lesssim\Norm{f}{\dot A^s_{p,q}(\rho,\tilde\varphi)}.
\end{equation*}
Since this is true for any two $\varphi,\chi$, we have proved claim \eqref{eq:AindepPhi}. The other claims then follow since we already proved them for some particular choice of $\dot A^s_{p,q}(\rho,\chi)$, and by \eqref{eq:AindepPhi}, they are all equal.
\end{proof}

Since, under the assumptions of
Theorems \ref{thm:phi} and \ref{thm:phiV}
respectively, the spaces $\dot{A}^s_{p,q}(\rho,\varphi)$
and $\dot{A}^s_{p,q}(V,\varphi)$
do not depend on the choice of
$\varphi$, in what follows we denote them simply by
$\dot{A}^s_{p,q}(\rho)$ and $\dot{A}^s_{p,q}(V)$,
respectively, under these assumptions.

At the end of this section, we provide some useful consequences
of the $\varphi$-transform characterization.
We first consider the embeddings between
Besov--Triebel--Lizorkin spaces
and Schwartz (distribution) spaces.
For any linear space $D$ of functions on $\mathbb{R}^n$, we define
\begin{align}\label{eq:sx}
D\otimes\fX:=
\Big\{\sum_{k=1}^{N}f_k\fx_k:N\in\mathbb{N},\
f_k\in D,\ \fx_k\in\fX\Big\}.
\end{align}

\begin{corollary}
Let $p\in(0,\infty)$, $q\in(0,\infty]$,
and $s\in\mathbb{R}$.
\begin{enumerate}[{\rm(i)}]
  \item\label{it:emb-rho} If $\rho=\{\rho_Q\}_{Q\in\mathscr D}$ is a family
of $(u,a,b,c)$-norms for some $u\in(0,1]$ and
$a,b,c\in[0,\infty)$, then $\mathscr{S}_\infty\otimes\fX\hookrightarrow
\dot{A}^s_{p,q}(\rho)\hookrightarrow \mathscr{S}'_\infty(\mathbb{R}^n;\fX)$.
  \item If $V\in\A_r$ with $r\in[p,\infty)$
  and if $q\in(0,r+\varepsilon_V)$
  when $A=F$, where $\varepsilon_V$ is
  the optimal reverse H\"older lifting index of $V$, then
  $\mathscr{S}_\infty\otimes\fX\hookrightarrow
\dot{A}^s_{p,q}(V)\hookrightarrow \mathscr{S}'_\infty(\mathbb{R}^n;\fX)$.
\end{enumerate}
\end{corollary}

\begin{proof}
By Theorems \ref{thm:equiB} and \ref{thm:equiF},
we only need to prove \eqref{it:emb-rho}.
Indeed, this argument is completely analogous to the proof in the
scalar or the matrix-weighted case: the embedding
$\dot{A}^s_{p,q}(\rho)\hookrightarrow \mathscr{S}'_\infty(\mathbb{R}^n;\fX)$
follows from repeating the proof of \cite[Proposition 3.39]{BHYY:I} with
Lemma 3.33 and Theorem 3.29 therein replaced, respectively,
by Lemma \ref{lem:TpsiWellDef} and Theorem \ref{thm:phi};
the embedding $\mathscr{S}_\infty\otimes\fX\hookrightarrow
\dot{A}^s_{p,q}(\rho)$
follows from repeating the proof of \cite[Proposition 3.2(viii)]{YY:08}
and using Definition \ref{def:uabc}\eqref{def:uabc-it3}.
\end{proof}

Finally, we give the following boundedness of the Hilbert transform and the Riesz potential,
which is in sharp contrast with a related result for $L^p(V)$ spaces given in Theorem \ref{thm:Hunbd}.

\begin{corollary}\label{cor:Hilbert}
Let $p\in(0,\infty)$, $q\in(0,\infty]$,
and $s\in\mathbb{R}$.
\begin{enumerate}[\rm(i)]
  \item\label{cor:Hilbert-it1} If $\rho=\{\rho_Q\}_{Q\in\mathscr D}$ is a family
of $(u,a,b,c)$-norms for some $u\in(0,1]$ and
$a,b,c\in[0,\infty)$, then,
  \begin{enumerate}[\rm(a)]
   \item for $n=1$, the Hilbert transform is an isomorphism of $\dot A^s_{p,q}(\rho)$ onto itself, and
   \item for all $n\geq 1$, the operator $(-\Delta)^{\frac t2}$ is an isomorphism of
$\dot A^{s+t}_{p,q}(\rho)$ onto $\dot A^{s}_{p,q}(\rho)$.
  \end{enumerate}
  \item If $V\in\A_r$ with $r\in[p,\infty)$
  and if $q\in(0,r+\varepsilon_V)$
  when $A=F$, where $\varepsilon_V$ is
  the optimal reverse H\"older lifting index of $V$, then
  the conclusions of \eqref{cor:Hilbert-it1} hold
  with $\rho$ replaced by $V$.
\end{enumerate}
\end{corollary}

\begin{proof}
By Theorems \ref{thm:equiB} and \ref{thm:equiF}, we only need to
prove \eqref{cor:Hilbert-it1}.
We note that both operators are Fourier multipliers $T_m f=(m\hat f)^{\vee}$ with symbols $m(\xi)=-i\operatorname{sgn}(\xi)$ and $m(\xi)=\abs{\xi}^t$, respectively, that are smooth on $\R^n\setminus\{0\}$, with at most moderate blow-up as $\abs{\xi}\to 0$ or $\infty$. Hence they act on $\mathscr S_\infty(\R^n)$, and thus on $\mathscr S_\infty'(\R^n;\fX)$. For the quantities featuring in the $\dot A^s_{p,q}$ norm, we have
\begin{equation*}
  (\varphi_j*T_m f)^{\wedge}
  =\hat\varphi(2^{-j}\cdot)m\hat f
  =(m(2^j\cdot)\hat\varphi)(2^{-j})\hat f.
\end{equation*}
In the case of $(-\Delta)^{\frac t2}$ resp. the Hilbert transform, we have $m(2^j\cdot)=2^{jt}m$ resp. $m(2^j\cdot) =m=2^{j\cdot 0}m$; hence
\begin{equation*}
  (\varphi_j*T_m f)^{\wedge}=2^{jt}(m\hat\varphi)(2^{-j}\cdot)\hat f=2^{jt}(\psi_j*f)^{\wedge},
\end{equation*}
where $\psi:=(m\hat\varphi)^{\vee}=T_m\varphi$, and we take $t=0$ in the case of the Hilbert transform.

In either case, since $m$ is smooth and bounded away from zero on compact subsets of $\R^n\setminus\{0\}$, it follows that $\psi=T_m\varphi$ is another Littlewood--Paley function. Thus
\begin{equation*}
\begin{split}
  \Norm{T_m f}{\dot A^s_{p,q}(\rho,\varphi)}
  &=\bNorm{\big\{2^{js}\varphi_j*T_m f\big\}_{j\in\mathbb{Z}}}
  {(L^p\ell^q)_A} \\
  &=\bNorm{\big\{ 2^{j(s+t)}\psi_j*f\big\}_{j\in\mathbb{Z}}}
  {(L^p\ell^q)_A}
  =\Norm{f}{\dot A^{s+t}_{p,q}(\rho,\psi)},
\end{split}
\end{equation*}
where $t=0$ in the case of the Hilbert transform. Since the norms on the left and the right are independent (up to constants) of the Littlewood--Paley functions, the claims follow.
\end{proof}

\section{Relation of Triebel--Lizorkin spaces to $L^p(V)$}

In the classical theory of function spaces, one of the key features of Triebel--Lizorkin spaces is that they contain the usual $L^p$ spaces as the special case $\dot F^0_{p,2}(\R^n)=L^p(\R^n)$ for $p\in(1,\infty)$. Under the $\A_p$ condition, this remains true even for matrix-weighted spaces, as shown in slightly different versions in \cite{FR:21,Isra:21,Volberg:97}.
However, this cannot possibly extend to the level of generality that we consider: By Corollary \ref{cor:Hilbert},
the Hilbert transform is bounded on $\dot F^0_{p,q}(\rho_{\aveL^p(\mathscr D,V)})$ for all $p\in(0,\infty)$ and $q\in(0,\infty]$, and on $\dot F^0_{p,q}(V)$ for all $0<q\leq p<\infty$,  whenever $V\in\A_p$.
On the other hand, Theorem \ref{thm:Hunbd} shows that this condition is insufficient for the boundedness on $L^p(V)$. Thus $V\in\A_p$ does not guarantee the coincidence of $L^p(V)$ with any of the Triebel--Lizorkin spaces $\dot F^0_{p,q}(\rho_{\aveL^p(\mathscr D,V)})$ or $\dot F^0_{p,q}(V)$ with $q$ in the indicated ranges.

On the other hand, for functions taking values in a Banach space, even in the unweighted case, the best that one can say in general is the ``sandwich'' result
\begin{equation}\label{eq:HNVW}
  \dot F^0_{p,1}(\R^n;\fX)\subset L^p(\R^n;\fX)\subset \dot F^0_{p,\infty}(\R^n;\fX),\qquad p\in(1,\infty);
\end{equation}
see \cite[Propositions 14.6.13 and 14.7.8]{HNVW3} for the positive result and its sharpness, respectively. We will here prove that \eqref{eq:HNVW} remains valid for the operator-weighted versions if the weight satisfies $V\in\A_p$.

\begin{proposition}\label{prop:Lp<F}
Let $p,r\in(1,\infty)$
 and $V\in\A_{\max\{p,r\}}$. If
 $f\in \dot F^0_{p,1}(\rho_{\aveL^r(\mathscr D,V)})=\dot F^0_{p,1}(V)$, then $f\in L^p(V)$ and
\begin{equation*}
  \Norm{f}{L^p(V)}\lesssim \Norm{f}{\dot F^0_{p,1}(V)}
  \approx \Norm{f}{\dot F^0_{p,1}(\rho_{\aveL^r(\mathscr D,V)})}
\end{equation*}
with the implicit positive constants independent of $f$.
\end{proposition}

\begin{proof}
The equality of the two $\dot F^0_{p,1}$ spaces and the related norm equivalence is a special case of Theorem \ref{thm:equiF} with $q=1\in(0,\max\{p,r\})$. To show that this space embeds into $L^p(V)$, we first note that
\begin{equation*}
  \Norm{f}{\dot F^0_{p,1}(V)}=\BNorm{\sum_{j\in\Z}\Norm{V(\varphi_j*f)}{}}{L^p(\R^n)}
\end{equation*}
In particular, the partial sums of $\sum_{j\in\Z}V(\varphi_j*f)$ form a Cauchy sequence in $L^p(\R^n;\fX)$, which converges to some $g\in L^p(\R^n;\fX)$. Equivalently, $\sum_{j\in\Z}\varphi_j*f$ converges to $h:=V^{-1}g\in L^p(V;\fX)$, where
\begin{equation*}
  \Norm{h}{L^p(V;\fX)}
  =\Norm{g}{L^p(\R^n;\fX)}
  \leq\BNorm{\sum_{j\in\Z}\Norm{V(\varphi_j*f)}{}}{L^p(\R^n)}
  =\Norm{f}{\dot F^0_{p,1}(V)}.
\end{equation*}
Recalling from Theorem \ref{thm:phi} or \ref{thm:phiV} that $\dot F^0_{p,1}(\rho_{\aveL^r(\mathscr D,V)})=\dot F^0_{p,1}(V)$ is independent of the choice of the Littlewood--Paley function implicit in its definition,  we can choose a special Littlewood--Paley function with $\sum_{j\in\Z}\hat\varphi(2^j\xi)\equiv 1$ for $\xi\in\R^n\setminus\{0\}$. Then $\sum_{j\in\Z}\varphi_j*f=f$ in $\mathscr S_\infty'(\R^n;\fX)$, and hence $f=h$ in the sense of $\mathscr S_\infty'(\R^n;\fX)$. Thus the previous display is the claimed embedding.
\end{proof}

\begin{proposition}\label{prop:F<Lp}
Let $p\in(1,\infty)$,
$r\in(0,p]$, and $V\in\A_p$. If $f\in L^p(V)$, then $f\in \dot F^0_{p,\infty}(\rho_{\aveL^r(\mathscr D,V)})$ and
\begin{equation*}
  \Norm{f}{\dot F^0_{p,\infty}(\rho_{\aveL^r(\mathscr D,V)})}\lesssim\Norm{f}{L^p(V)}
\end{equation*}
with the implicit positive constant independent of $f$.
\end{proposition}

\begin{proof}
Let $x\in Q\in\mathscr D_j$ and $\fx^*\in\fX^*$. Then
\begin{equation*}
\begin{split}
  \abs{\pair{\varphi_j*f(x) }{ \fx^*}}
  &=\Babs{\int_{\R^n}\varphi_j(x-y)\pair{f(y)}{\fx^*}\,dy} \\
  &\lesssim\sum_{R\in\mathscr D_j}\abs{R}\phi_j'(x-x_R) \fint_R\abs{\pair{V(y)f(y)}{V^{-*}(y)\fx^*}}\,dy \\
  &\leq\sum_{R\in\mathscr D_j}\abs{R}\phi_j'(x-x_R)\Norm{Vf}{\aveL^{(p'+\eta)'}(R:\fX)}
     \Norm{V^{-*}\fx^*}{\aveL^{p'+\eta}(R;\fX^*)},
\end{split}
\end{equation*}
where $\phi_j'(z)=2^j\phi(2^j z)$ with $\phi(z):=(1+\abs{z})^{-M'}$ dominates $\varphi_j\in\mathscr S_\infty(\R^n)$ for any $M'>0$ and where $\eta>0$ is a number guaranteed by the reverse H\"older property of Proposition \ref{prop:RHI}. Thus
\begin{equation*}
\begin{split}
  \Norm{V^{-*}\fx^*}{\aveL^{p'+\eta}(R;\fX^*)}
  &=\rho_{\aveL^{p'+\eta}(R,V^{-*})}(\fx^*) \\
  &\lesssim\rho_{\aveL^{p'}(R,V^{-*})}(\fx^*)\qquad\text{by Proposition \ref{prop:RHI}} \\
  &\lesssim(1+2^j\abs{x_Q-x_R})^c \rho_{\aveL^{p'}(Q,V^{-*})}(\fx^*).
\end{split}
\end{equation*}
Substituting back and denoting $\phi(z):=\phi(z)(1+\abs{z})^c=(1+\abs{z})^{-M}$, where $M:=M'-c$ is still as large as we like, we obtain
\begin{equation*}
\begin{split}
   \frac{\abs{\pair{\varphi_j*f(x) }{ \fx^*}}}{\rho_{\aveL^{p'}(Q,V^{-*})}(\fx^*)}
   &\lesssim\sum_{R\in\mathscr D_j}\abs{R}\phi_j(x-x_R)\Norm{Vf}{\aveL^{(p'+\eta)'}(R;\fX)} \\
   &\lesssim \Big(\sum_{R\in\mathscr D_j}\abs{R}\phi_j(x-x_R)
   \Norm{Vf}{\aveL^{(p'+\eta)'}(R;\fX)}^{(p'+\eta)'}\Big)^{\frac{1}{(p'+\eta)'}} \\
   &\approx(\phi_j*\Norm{Vf(\cdot)}{}^{(p'+\eta)'})^{\frac{1}{(p'+\eta)'}}(x).
\end{split}
\end{equation*}
Thus, by H\"older's inequality, the $\A_p$ condition, and the definition of the dual norm,
\begin{equation*}
\begin{split}
\rho_{\aveL^r(Q,V)}(\varphi_j\ast f(x))
&\le    \rho_{\aveL^p(Q,V)}(\varphi_j*f(x))
\lesssim\rho_{\aveL^{p'}(Q,V^{-*})}^*(\varphi_j*f(x)) \\
    &\lesssim (\phi_j*\Norm{Vf(\cdot)}{}^{(p'+\eta)'})^{\frac{1}{(p'+\eta)'}}(x)
    \lesssim M_{(p'+\eta)'}(Vf)(x)
\end{split}
\end{equation*}
for all $x\in Q\in\mathscr D_j$. Therefore
\begin{equation*}
\begin{split}
  \Norm{f}{\dot F^0_{p,\infty}(\rho_{\aveL^r(\mathscr D,V)})}
  &=\bNorm{\big\{\rho_{\aveL^r(\mathscr D_j,V)}(\varphi_j*f)\big\}_{j\in\mathbb{Z}}}
  {L^p\ell^\infty} \\
  &\lesssim\Norm{ M_{(p'+\eta)'}(Vf)}{L^p}
  \lesssim\Norm{Vf}{L^p}=\Norm{f}{L^p(V)},
\end{split}
\end{equation*}
where the last estimate was based on the maximal inequality and the fact that $p'+\eta>p'$ and hence $(p'+\eta)'<(p')'=p$ for $p\in(1,\infty)$.
\end{proof}

\section{Almost diagonal operators}

In this section, we study the boundedness of the following class of
almost diagonal operators on Besov and Triebel--Lizorkin spaces.
This is a cornerstone for developing further properties of these spaces
in the subsequent sections.

\begin{definition}\label{def:AD}
A matrix $B=\{b_{QR}\}_{Q,R\in\mathscr D}$ is said to be
\emph{$(D,E,F)$-almost diagonal} if
\begin{equation*}
  \abs{b_{QR}}\lesssim B_{-E,-F,-D}(Q,R)
\end{equation*}
where the right-hand side is defined in \eqref{eq:BabD}.
\end{definition}

\begin{remark}
Definition \ref{def:AD} is equivalent to \cite[Definition 4.1]{BHYY:II}, although it was there formulated slightly differently. The said \cite[Definition 4.1]{BHYY:II} seems to be the first place where this condition was formulated with generic parameters $D,E,F\in\R$. With specific choices of parameters, the definition goes back to \cite[p.~53]{FJ:90}, and reappears in \cite[Definition 8.1]{Rou:03} and \cite[Definition 2.5]{FR:21}, each time with different specific parameters for the particular context.
\end{remark}

Before proceeding to our main results,
we first recall the classical result.
Recall that the scalar sequence space $\dot{a}^s_{p,q}$ is defined to be
the set of all $t=\{t_Q\}_{Q\in\D}$ with each $t_Q\in\mathbb{C}$ such that
\begin{align*}
\|t\|_{\dot{a}^s_{p,q}}
:=\BNorm{\Big\{2^{js}t_j\Big\}_{j\in\mathbb{Z}}}
{(L^p\ell^q)_A}<\infty,
\end{align*}
where $t_j$ is as in \eqref{df-tj}.

\begin{theorem}[{\cite[Theorem 3.3]{FJ:90}}]\label{thm:FJ}
Let $B$ be $(D,E,F)$-almost diagonal with
\begin{equation}\label{eq:FJad}
   D>J,\qquad E>\frac{n}{2}+s,\qquad F>J-\frac{n}{2}-s,
\end{equation}
where
\begin{equation}\label{eq:J}
  J:=\begin{cases} n/\min(1,p,q) & \textup{(Triebel--Lizorkin case)}, \\ n/\min(1,p) & \textup{(Besov case)}.\end{cases}
\end{equation}
Then $B$ is bounded on $\dot a^s_{p,q}$.
\end{theorem}

\begin{proof}
Strictly speaking, \cite[Theorem 3.3]{FJ:90} covers only the Triebel--Lizorkin case, but the Besov case is similar and easier.
\end{proof}

\begin{remark}\label{rem:sharp}
As showed in \cite[Sections 7 and 9]{BHYY:II},
the assumption \eqref{eq:FJad} on $(D,E,F)$
is optimal.
More precisely,
if $a=b$ or $a=f$ and $q\ge\min\{1,p\}$, and if
every $(D,E,F)$-almost diagonal operator $B$ is bounded
on $\dot{a}^s_{p,q}$, then, by
 \cite[Sections 7 and 9]{BHYY:II}, the triple $(D,E,F)$ satisfies
\eqref{eq:FJad}.
\end{remark}

In this section, we deal with the boundedness of $B$
on average and pointwise-weighted spaces separately.
First, in Subsection \ref{sec:ADrho}, we establish the boundedness of
almost diagonal operators on average-weighted spaces $\dot{A}^s_{p,q}(\rho)$
under very mild assumptions on $\rho$.
But, in this setting,
the optimal assumption on $(D,E,F)$ is still unknown
(see Remark \ref{rem:ADrho}).
Next, in Subsection \ref{sec:ADV},
we improve the assumption on $(D,E,F)$ for pointwise-weighted spaces
$\dot{A}^s_{p,q}(V)$ under $V\in\bigcup_{r\ge p}\A_r$,
and we prove that this assumption is optimal in some sense
(see Remark \ref{rem:ADV}).

\subsection{Average-weighted spaces with abstract seminorms}\label{sec:ADrho}

The main result of this subsection is the following
boundedness result.

\begin{theorem}\label{prop:AD}
Let $p\in(0,\infty)$, $q\in(0,\infty]$,
and $s\in\mathbb{R}$.
Suppose $\rho=\{\rho_Q\}_{Q\in\mathscr D}$ is a family of
$(u,a,b,c)$-norms with some $u\in(0,1]$ and $a,b,c\in[0,\infty)$.
If $B$ is $(D,E,F)$-almost diagonal with
\begin{equation*}
  D>J^{(u)}+c,\qquad
  E>\frac{n}{2}+s+a,\qquad
  F>J^{(u)}-\frac{n}{2}-s+b,
\end{equation*}
where
\begin{equation}\label{eq:Jr}
  J^{(u)}:=\begin{cases} n/\min\{p,q,u\} & \textup{(Triebel--Lizorkin case)}, \\
  n/\min\{p,u\} & \textup{(Besov case)},\end{cases}
\end{equation}
then $B$ is bounded on $\dot a^s_{p,q}(\rho)$.
\end{theorem}

Before going to the proof of Theorem \ref{prop:AD},
we give some applications to concrete examples of $\rho$ as follows.

\begin{corollary}\label{cor:AD}
Let $p,r\in(0,\infty)$, $q\in(0,\infty]$, and $s\in\mathbb{R}$.
Suppose that $B$ is a $(D,E,F)$-almost diagonal operator.
If one of the following conditions holds:
\begin{enumerate}[{\rm(i)}]
  \item\label{cor:AD-it1} $V$ is an $(r,\beta)$-doubling weight and
\begin{equation*}
  D>J^{(\min\{1,r\})}+\frac{\beta}{r},
  \qquad E>\frac n2+s+\frac nr,\qquad F>J^{(\min\{1,r\})}
  -\frac n2-s+\frac{\beta-n}{r};\qquad\text{or}
\end{equation*}
  \item\label{cor:AD-it2}  $V\in \mathscr{A}_w$ with $w\ge r$ has doubling
 dimension $\beta\in[n,\infty)$, and
\begin{equation*}
  D>J^{(\min\{1,r\})}+\frac{\beta}{w},
  \qquad E>\frac n2+s+\frac nw,\qquad F>J^{(\min\{1,r\})}
  -\frac n2-s+\frac{\beta-n}{w},
\end{equation*}
\end{enumerate}
then $B$ is bounded on $\dot a^s_{p,q}(\rho_{\aveL^r(\mathscr D,V)})$.
\end{corollary}

\begin{proof}
By Remark \ref{rem:uabc}, we find that
\eqref{cor:AD-it1} or \eqref{cor:AD-it2}
implies that all the assumptions of
Theorem \ref{prop:AD} are satisfied.
Thus, the claim of the present corollary is precisely the conclusion of Theorem \ref{prop:AD}.
\end{proof}

\begin{remark}\label{rem:ADrho}
Now we give some comments
on the assumptions to $(D,E,F)$ in Theorem \ref{prop:AD}
and Corollary \ref{cor:AD}.
Note that
\begin{equation*}
    J^{(\min\{1,r\})}=J\quad\text{if}\quad r=p,
\end{equation*}
where $J^{(\min\{1,r\})}$ is defined in \eqref{eq:Jr} and appears in Corollary \ref{cor:AD}, while $J$ is defined in \eqref{eq:J}.
In particular, specializing Corollary \ref{cor:AD}\eqref{cor:AD-it2}
to $w=r=p$ and the case of Triebel--Lizorkin spaces
$$\dot a^s_{p,q}(\rho_{\aveL^p(V;\mathscr D)})
=\dot f^s_{p,q}(\rho_{\aveL^p(V;\mathscr D)}),$$
Corollary \ref{cor:AD}\eqref{cor:AD-it2} has the same assumptions as
\cite[Theorem 2.6]{FR:21} in the matrix-weighted case,
and hence provides its natural extension to the operator-weighted case.
However, a sharper condition is known for matrix-$\A_p$-weights
from \cite[Theorem 5.1]{BHYY:II}.
In Subsection \ref{sec:ADV} below, we will
obtain a corresponding result in the operator-weighted setting.
For a general family $\rho$ of quasi-norms,
the optimal assumptions on $(D,E,F)$
for Theorem \ref{prop:AD} to hold are still unknown.
\end{remark}

Now, we turn to show Theorem \ref{prop:AD}.
In analogy with the familiar rescaling property of the $L^p$ norms,
$\Norm{f}{L^p}=\Norm{\abs{f}^r}{L^{\frac pr}}^{\frac1r}$,
we first record the following version for the $\dot a^s_{p,q}$ norms:

\begin{lemma}\label{lem:rescale}
Let $p\in(0,\infty)$, $q\in(0,\infty]$, and  $s\in\mathbb{R}$.
If $\rho=\{\rho_Q\}_{Q\in\mathscr{D}}$
is a family of quasi-norms and $u\in(0,\infty)$, then,
for any $t=\{t_Q\}_{Q\in\mathscr{D}}$ of vectors
in $\fX$,
\begin{equation*}
  \Norm{t}{\dot a^s_{p,q}(\rho)}
  =\bNorm{\big\{ \rho_Q(t_Q)^u \ell(Q)^{\frac n2(1-u)} \big\}_{Q\in\mathscr{D}}}
  {\dot a^{su}_{\frac pu,\frac qu}}^{\frac1u}.
\end{equation*}
\end{lemma}

\begin{proof} By the definitions of the norms under consideration,
\begin{equation*}
\begin{split}
   \Norm{t}{\dot a^s_{p,q}(\rho)}
   &=\BNorm{\Big\{\sum_{Q\in\mathscr D_j}\frac{\rho_Q(t_Q)}
   {\ell(Q)^{s+\frac n2}}\one_Q\Big\}_{j\in\mathbb{Z}}}{(L^p\ell^q)_A}
   =\BNorm{ \Big\{ \sum_{Q\in\mathscr D_j} \frac{\rho_Q(t_Q)^u }
   { \ell(Q)^{(s+\frac n2)u}} \one_Q\Big\}_{j\in\mathbb{Z}}  }
   {(L^{\frac pu}\ell^{\frac qu})_A}^{\frac1u} \\
  &=\BNorm{ \Big\{ \sum_{Q\in\mathscr D_j} \frac{\rho_Q(t_Q)^u \ell(Q)^{\frac n2(1-u)} }
  { \ell(Q)^{(su+\frac n2)} } \one_Q\Big\}_{j\in\mathbb{Z}}}{(L^{\frac pu}\ell^{\frac qu})_A}^{\frac1u}
  =\bNorm{\big\{ \rho_Q(t_Q)^u \ell(Q)^{\frac n2(1-u)} \big\}_{Q\in\D}}{\dot a^{su}_{\frac pu,\frac qu}}^{\frac1u},
\end{split}
\end{equation*}
which completes the proof of the lemma.
\end{proof}

\begin{lemma}\label{lem:Brho<newB}
Let an infinite matrix $B=\{b_{QR}\}_{Q,R\in\mathscr D}$ of scalars act on sequences $t=\{t_Q\}_{Q\in\mathscr D}$ by
\begin{equation*}
  (Bt)_Q:=\sum_{R\in\mathscr D}b_{QR}t_R.
\end{equation*}
If $\rho=\{\rho_Q\}_{Q\in\mathscr D}$ is a family of $u$-norms, then
\begin{equation*}
  \Norm{B}{\bddlin(\dot a^s_{p,q}(\rho))}
  \leq\Norm{\tilde B}{\bddlin(\dot a^{su}_{\frac pu,\frac qu})}^{\frac1u},
\end{equation*}
where
\begin{equation}\label{eq:tildeB}
  \tilde B:=\{\tilde b_{QR}\}_{Q,R\in\mathscr D},\quad
  \tilde b_{QR}:=\abs{b_{QR}}^u \bNorm{\frac{\rho_Q}{\rho_R}}{}^u
  \Big(\frac{\abs{Q}}{\abs{R}}\Big)^{\frac12(1-u)},
\end{equation}
and
\begin{equation*}
  \bNorm{\frac{\rho_Q}{\rho_R}}{}:=\sup_{\fx\in\fX\setminus\{0\}}\frac{\rho_Q(\fx)}{\rho_R(\fx)}
\end{equation*}
\end{lemma}

\begin{proof}
By Lemma \ref{lem:rescale} and a direct calculation, we obtain
\begin{equation}\label{eq:tran}
\begin{split}
  \Norm{Bt}{\dot a^s_{p,q}(\rho)}
  &=\bNorm{\big\{\rho_Q((Bt)_Q)^u \abs{Q}^{\frac12(1-u)} \big\}_{Q\in\D}}
  {\dot a^{su}_{\frac pu,\frac qu}}^{\frac 1u} \\
  &\leq\BNorm{\Big\{\sum_{R\in\mathscr D}\rho_Q(b_{QR}t_R)^u \abs{Q}^{\frac12(1-u)}
  \Big\}_{Q\in\D}}{\dot a^{su}_{\frac pu,\frac qu}}^{\frac 1u} \\
  &\leq\BNorm{\Big\{\sum_{R\in\mathscr D}\abs{b_{QR}}^u\bNorm{\frac{\rho_Q}{\rho_R} }{}^u\Big(\frac{\abs{Q}}{\abs{R}}\Big)^{\frac12(1-u)}
  \rho_R(t_R)^u \abs{R}^{\frac12(1-u)} \Big\}_{Q\in\D}}
  {\dot a^{su}_{\frac pu,\frac qu}}^{\frac 1u}. \\
\end{split}
\end{equation}
Denoting
\begin{equation*}
   \tilde t_R:= \rho_R(t_R)^u \abs{R}^{\frac12(1-u)}
\end{equation*}
and recalling the notation \eqref{eq:tildeB}, it follows that
\begin{equation}\label{eq:tran-2}
   \Norm{Bt}{\dot a^s_{p,q}(\rho)}
   \leq\Norm{\tilde B\tilde t}{\dot a^{su}_{\frac pu,\frac qu}}^{\frac1u}
   \leq\Norm{\tilde B}{\bddlin(\dot a^{su}_{\frac pu,\frac qu})}^{\frac1u}
   \Norm{\tilde t}{\dot a^{su}_{\frac pu,\frac qu}}^{\frac1u}
   =\Norm{\tilde B}{\bddlin(\dot a^{su}_{\frac pu,\frac qu})}
   ^{\frac1u}\Norm{t}{\dot a^{s}_{p,q}(\rho)}
\end{equation}
by Lemma \ref{lem:rescale} again in the last step.
\end{proof}

\begin{proof}[Proof of Theorem \ref{prop:AD}]
Let
\begin{equation*}
  \tilde{u}:=\begin{cases} \min\{p,q,u\} & \text{(Triebel--Lizorkin case)}, \\
  \min\{p,u\} & \text{(Besov case)}.\end{cases}
\end{equation*}
Then $\rho$ is also a family of $\tilde{u}$-norms, and Lemma \ref{lem:Brho<newB} shows that
\begin{equation}\label{eq:B<tildeB}
  \Norm{B}{\bddlin(\dot a^s_{p,q}(\rho))}\leq\Norm{\tilde B}{\bddlin(\dot
  a^{s\tilde{u}}_{\frac p{\tilde u},\frac q{\tilde u}})}^{\frac1{\tilde u}},
\end{equation}
where $\tilde B$ is defined in \eqref{eq:tildeB}. The parameter $J$ in \eqref{eq:J}, but corresponding to the space on the right of \eqref{eq:B<tildeB}, is then
\begin{equation*}
  \tilde J=\left.\begin{cases} n/\min\{1,\frac p{\tilde u},\frac q{\tilde u}\}=
  n{\tilde u}/\min\{p,q,\tilde u\}=n{\tilde u}/{\tilde u} \\
  n/\min(1,\frac p{\tilde u})=n{\tilde u}/\min\{p,\tilde u\}
  =n{\tilde u}/{\tilde u} \end{cases}\right\}=n
\end{equation*}
in the Triebel--Lizorkin and Besov cases, respectively, but both of them actually result in the same value in the end.

Thus, by Theorem \ref{thm:FJ}, the right-hand side of \eqref{eq:B<tildeB} is finite provided that $\tilde B$ is $(\tilde D,\tilde E,\tilde F)$-almost diagonal with
\begin{equation*}
  \tilde D>n,\qquad \tilde E>\frac{n}{2}+s{\tilde u},
  \qquad\tilde F>n-\frac{n}{2}-s{\tilde u}=\frac{n}{2}-s{\tilde u}.
\end{equation*}
Recalling the definition of $\tilde B$ from \eqref{eq:tildeB}, this means that we need
\begin{equation*}
  \abs{b_{QR}}^{\tilde u}\bNorm{\frac{\rho_Q}{\rho_R}}{}^{\tilde u}
  \Big(\frac{\abs{Q}}{\abs{R}}\Big)^{\frac12(1-{\tilde u})}
  \lesssim B_{-\frac n2-s{\tilde u}-\eps,-\frac n2+s{\tilde u}-\eps,-n-\eps}(Q,R),
\end{equation*}
or, taking $\tilde{u}$-th roots of both sides, and denoting
$\frac{\eps}{\tilde u}$ again simply by $\eps$,
\begin{equation}\label{eq:ad2show}
  \abs{b_{QR}}\bNorm{\frac{\rho_Q}{\rho_R}}{}
  \Big(\frac{\abs{Q}}{\abs{R}}\Big)^{\frac12(\frac 1{\tilde u}-1)}
  \lesssim B_{-\frac{n}{2{\tilde u}}-s-\eps,-\frac{n}{2{\tilde u}}+s-\eps,
  -\frac{n}{\tilde u}-\eps}(Q,R).
\end{equation}

Using the assumption that
\begin{equation*}
   \bNorm{\frac{\rho_Q}{\rho_R}}{}
   \lesssim B_{a,b,c}(Q,R),
\end{equation*}
and noting that
\begin{equation*}
  \frac{\abs{Q}}{\abs{R}}= \Big(\frac{\ell(Q)}{\ell(R)}\Big)^n
  \approx\Big(1+\frac{\ell(R)}{\ell(Q)}\Big)^{-n}\Big(1+\frac{\ell(Q)}{\ell(R)}\Big)^n
  \approx B_{-n,n,0}(Q,R),
\end{equation*}
we find that
\begin{equation*}
  \operatorname{LHS}\eqref{eq:ad2show}
  \lesssim\abs{b_{QR}} B_{a-\frac n2(\frac1{\tilde u}-1),b+
  \frac n2(\frac1{\tilde u}-1),c}(Q,R).
\end{equation*}
This is dominated by the right-hand side of \eqref{eq:ad2show} provided that
\begin{equation*}
  \abs{b_{QR}}
  \lesssim B_{-\frac n2-s-a-\eps,-\frac{n}{\tilde u}+
  \frac{n}{2}+s-b-\eps,-\frac n{\tilde u}-c-\eps}(Q,R).
\end{equation*}
Observing that $\frac n{\tilde u}=J^{(u)}$,
this is precisely the almost diagonality condition that we assumed.
\end{proof}

\begin{remark}\label{rem:abso-rho}
The proofs of Lemma \ref{lem:Brho<newB} and
Theorem \ref{prop:AD} tell us a stronger conclusion:
under the assumptions of Theorem \ref{prop:AD},
the series $\sum_{R\in\mathscr{D}}b_{QR}t_R$
is absolutely convergent in $\fX$ for all
$Q\in\mathscr{D}$ and all sequences
$t=\{t_R\}_{R\in\mathscr{D}}\in\dot{a}^s_{p,q}(\rho)$.

Indeed, suppose that the notation is as in the proof of Theorem \ref{prop:AD}.
Then, from the proof of \eqref{eq:B<tildeB}
(more precisely, from \eqref{eq:tran} and
\eqref{eq:tran-2} in the proof of Lemma \ref{lem:Brho<newB}),
we deduce that
\begin{align*}
\BNorm{\Big\{\sum_{R\in\mathscr D}\rho_Q(b_{QR}t_R)^{\tilde u}
\abs{Q}^{\frac12(1-{\tilde u})}
  \Big\}_{Q\in\D}}{\dot a^{s{\tilde u}}_{\frac p{\tilde u},
  \frac q{\tilde u}}}^{\frac 1{\tilde u}}
  \leq\Norm{\tilde B}{\bddlin(\dot a^{s{\tilde u}}_{\frac p{\tilde u},\frac q{\tilde u}})}
   ^{\frac1{\tilde u}}\Norm{t}{\dot a^{s}_{p,q}(\rho)}.
\end{align*}
Therefore, from the assumption that $\rho$ is a family of $(\tilde u,a,b,c)$-norms (see in particular Definition \ref{def:uabc}\eqref{def:uabc-it2}),
we obtain
\begin{align*}
\Big(\sum_{R\in\mathscr{D}}|b_{QR}|\Norm{t_R}{\fX}\Big)^{\tilde{u}}
&\le \sum_{R\in\mathscr{D}}\Norm{b_{QR}t_R}{\fX}^{\tilde{u}}
\sim_Q\sum_{R\in\mathscr{D}}\rho_Q(b_{QR}t_R)^{\tilde{u}}\\
&\le|Q|^{\frac12(\tilde{u}-1)}
\Norm{\tilde B}{\bddlin(\dot a^{s{\tilde u}}_{\frac p{\tilde u},\frac q{\tilde u}})}
   ^{\frac1{\tilde u}}\Norm{t}{\dot a^{s}_{p,q}(\rho)}<\infty.
\end{align*}
This shows the absolute convergence of $\sum_{R\in\mathscr{D}}b_{QR}t_R$. This conclusion will be used later. (Since absolute convergence is a qualitative property, the fact that the implicit constant above, inherited from Definition \ref{def:uabc}\eqref{def:uabc-it2}, is allowed to depend on $Q$, is of no concern.)
\end{remark}

\subsection{Muckenhoupt-weighted spaces}\label{sec:ADV}
Now, we aim to improve the assumptions on $(D,E,F)$
for weighted spaces under the operator $\A_p$ condition.
To this end, we need the following technical lemma
about the almost diagonal operators and ball average operators,
which was proved in the case $\fX=\mathbb{C}^m$ in
\cite[Lemma 5.4]{BHYY:II}.
But the proof of \cite[Lemma 5.4]{BHYY:II} does not
use any special property of finite-dimensional spaces
and hence this estimate still holds for general
$\fX$.

\begin{lemma}\label{ad prelim}
Let $p\in(0,\infty)$, $q\in(0,\infty]$,
and $B$ be $(D,E,F)$-almost diagonal for some $D,E,F\in\mathbb R$.
Then there exists a positive constant $C$ such that,
for any $t:=\{t_Q\}_{Q\in\mathscr D}$ in $\fX$
satisfying that $Bt$ is well defined, any $u\in(0,1]$,
and any sequence $\{H_j\}_{j\in\mathbb{Z}}$
in $L^p_{\loc,\so}(\R^n;\bddlin(\fX))$,
\begin{align*}
\begin{split}
\BNorm{\Big\{\|H_j\left(Bt\right)_j\|_{\fX}\Big\}_{j\in\mathbb Z}}{(L^p\ell^q)_A}^v
&\leq C\sum_{k\in\mathbb{Z}}\sum_{l=0}^\infty
\Bigg[2^{-(E-\frac{n}{2})k_-}2^{-k_+(F+\frac{n}{2}-\frac{n}{u})}2^{-(D-\frac{n}{u})l} \\
&\quad\times\Bigg\|\Bigg\{
\Bigg[\fint_{B(\cdot,2^{l+k_+-i})}\| H_{i-k}(\cdot)t_{i}(y)\|_\fX^u
\,dy\Bigg]^{\frac{1}{u}}\Bigg\}_{i\in\mathbb Z}\Bigg\|_{(L^p\ell^q)_A}\Bigg]^v,
\end{split}
\end{align*}
where $t_i$ and $(Bt)_j$ are defined as in \eqref{df-tj}
and $v:=\min\{1,p,q\}$.
\end{lemma}

\begin{theorem}\label{bdd-AD-ple1}
Let $p\in(0,\infty)$, $V\in\mathscr{A}_r$ with $r\in[p,\infty)$
have doubling dimension $\beta$, and $s\in\mathbb{R}$.
Suppose that $B$ is a $(D,E,F)$-almost diagonal operator
satisfying
\begin{equation}\label{eq:DEF-V}
   D>J+\frac{\beta-n}{r},\qquad E>\frac{n}{2}+s,\qquad F>J-\frac{n}{2}-s
   +\frac{\beta-n}{r},
\end{equation}
where $J$ is as in \eqref{eq:J}.
\begin{enumerate}[\rm (i)]
  \item\label{bdd-AD-ple1-b} If $q\in(0,\infty]$,
  then $B$ is bounded on $\dot{b}^s_{p,q}(V)=\dot{b}^s_{p,q}(\rho_{\aveL^r(\mathscr D,V)})$.
  \item\label{bdd-AD-ple1-f} If $q\in(0,r+\varepsilon_V)$,
  where $\varepsilon_V$ is
the optimal reverse H\"older lifting index of $V$, then
$B$ is bounded on $\dot{f}^s_{p,q}(V)=\dot{f}^s_{p,q}(\rho_{\aveL^r(\mathscr D,V)})$.
\end{enumerate}
\end{theorem}

\begin{remark}\label{rem:ADV}
The assumptions on $(D,E,F)$ in Theorem \ref{bdd-AD-ple1}
improve those in Corollary \ref{cor:AD}\eqref{cor:AD-it2} for $w=r$.
Moreover, when $V\in\A_p$,
the assumptions on $(D,E,F)$ in Theorem \ref{bdd-AD-ple1} are optimal
in the following sense.
\begin{enumerate}[{\rm(i)}]
  \item If $p\in(0,1]$, then, by Corollary \ref{cor:ApDb},
  one can choose the doubling dimension $\beta=n$.
In this case, \eqref{eq:DEF-V} reduces to
  \eqref{eq:FJad} and hence, by Remark \ref{rem:sharp},
  is optimal even in the scalar case.
  \item If $p\in(1,\infty)$, then, compared with \eqref{eq:FJad},
  the assumptions on $D$ and $F$ in \eqref{eq:DEF-V} contain an additional
  term ``$\frac{\beta-n}{p}$". This is mainly due
  to our use of \eqref{bdd-AD-ple1-e5} in the proof of Theorem \ref{bdd-AD-ple1}.
  Indeed, we later show that the estimate
  \eqref{bdd-AD-ple1-e5} cannot be improved in the
  infinite-dimensional case; see Proposition \ref{sharp-ad}.
\end{enumerate}
\end{remark}

\begin{proof}[Proof of Theorem \ref{bdd-AD-ple1}]
The equalities of $\dot a^s_{p,q}(V)=\dot a^s_{p,q}(\rho_{\aveL^r(\mathscr D,V)})$ follow from Corollary \ref{cor:equib}  in the Besov case and Theorem \ref{thm:equiF} in the Triebel--Lizorkin case; note that we now assume $r\in[p,\infty)$, and hence $\max\{p,r\}=r$. Indeed, the coincidence of the spaces is important, since our proof will estimate the pointwise-weighted norm $\Norm{Bt}{\dot a^s_{p,q}(V)}$ by the average-weighted norm  $\Norm{t}{\dot a^s_{p,q}(\rho_{\aveL^r(\mathscr D,V)})}$.

Similarly to \cite[Lemma 5.3]{BHYY:II},
we only need to consider the case $s=0$.
More precisely, if we define
$(J_st)_Q:=\ell(Q)^{-s}t_Q$ for any $t:=\{t_Q\}_{Q\in\mathscr{D}}$
and any $Q\in\mathscr{D}$,
and define
$$\widetilde{B}
:=\Big\{\Big[\frac{\ell(R)}{\ell(Q)}\Big]^sb_{QR}\Big\}_{Q,R\in\mathscr{D}}$$
for $B:=\{b_{QR}\}_{Q,R\in\mathscr{D}}$,
then, from definitions, it immediately follows
that $J_s:\dot{a}^s_{p,q}(V)\to\dot{a}^0_{p,q}(V)$
is an isometric isomorphism, $\widetilde{B}$
is a $(D,E-s,F+s)$-almost diagonal operator, and
$J_s\circ B=\widetilde{B}\circ J_s$. Therefore,
once we show that, for a certain triple $(D,E,F)$,
every $(D,E,F)$-almost diagonal operator
is bounded on $\dot{a}^0_{p,q}(V)$,
we can use the above
transformations to further obtain the
boundedness of all $(D,E+s,F-s)$-almost diagonal operators
on  $\dot{a}^s_{p,q}(V)$.

In the remainder of the proof, we always assume $s=0$.
We deal with \eqref{bdd-AD-ple1-b} and \eqref{bdd-AD-ple1-f} in
a unified way. By the assumptions on $(D,E,F)$, we can choose
$u\in(0,\frac{n}{J})$ satisfying
$D>\frac nu+\frac{\beta-n}{r}$, $E>\frac n2$, and $F>\frac nu-\frac n2
+\frac{\beta-n}{r}$.
Applying Lemma \ref{ad prelim} with this $u$ and with $H_j:=V$ for all $j\in\mathbb{Z}$, we find that
\begin{align}\label{bdd-AD-ple1-e4}
\begin{split}
\Norm{B t}{\dot{a}^0_{p,q}(V)}^v
&\lesssim\sum_{k\in\mathbb{Z}}\sum_{l=0}^\infty
\Bigg[2^{-(E-\frac{n}{2})k_-}2^{-(F+\frac{n}{2}-\frac{n}{u})k_+}2^{-(D-\frac{n}{u})l} \\
&\quad\times\Bigg\|\Bigg\{
\Bigg[\fint_{B(\cdot,2^{l+k_+-i})}\| V(\cdot)t_{i}(y)\|_\fX^u
\,dy\Bigg]^{\frac{1}{u}}\Bigg\}_{i\in\mathbb Z}\Bigg\|_{(L^p\ell^q)_A}\Bigg]^v.
\end{split}
\end{align}
Hence, if we can prove that, for any $M\in\mathbb{Z}_+$,
\begin{align}\label{bdd-AD-ple1-e5}
\Bigg\|\Bigg\{
\Bigg[\fint_{B(\cdot,2^{M-i})}\| V(\cdot)t_{i}(y)\|_\fX^u
\,dy\Bigg]^{\frac{1}{u}}\Bigg\}_{i\in\mathbb Z}\Bigg\|_{(L^p\ell^q)_A}
\lesssim2^{\frac{M(\beta-n)}{r}}\|t\|_{\dot{a}^0_{p,q}(\rho_{\aveL^r(\mathscr D,V)})},
\end{align}
then, by \eqref{bdd-AD-ple1-e4} and the assumptions
on $a$ and $(D,E,F)$, we further obtain the boundedness of
$B$ on $\dot{a}^0_{p,q}(V)=\dot{a}^0_{p,q}(\rho_{\aveL^r(\mathscr D,V)})$, recalling that we already observed the coincidence of the two spaces under the assumptions of the theorem that we are proving.

For an arbitrary fixed $i_0\in\Z$, let us denote
\begin{equation*}
  \Norm{t}{(L^p\ell^q)_{\hat A}}=\Norm{\{t_i\}_{i\in\Z}}{(L^p\ell^q)_{\hat A}}:=
  \begin{cases}\Norm{t_{i_0}}{L^p}& \text{if }A=B, \\ \Norm{\{t_i\}_{i\in\Z}}{L^p\ell^q} & \text{if }A=F.\end{cases}
\end{equation*}
Thus $(L^p\ell^q)_{\hat B}=L^p$, we simply ignore all $i$ except $i=i_0$.

To prove \eqref{bdd-AD-ple1-e5}, it suffices to show that
\begin{align}\label{bdd-AD-ple1-e5b}
\Bigg\|\Bigg\{
\Bigg[\fint_{B(\cdot,2^{M-i})}\| V(\cdot)t_{i}(y)\|_\fX^u
\,dy\Bigg]^{\frac{1}{u}}\Bigg\}_{i\in\mathbb Z}\Bigg\|_{(L^p\ell^q)_{\hat A}}
\lesssim2^{\frac{M(\beta-n)}{r}}
\BNorm{\Big\{\rho_{\aveL^r(\mathscr{D}_i,V)}(t_i)\Big\}_{i\in\mathbb{Z}}}{(L^p\ell^{q})_{\hat A}}.
\end{align}
Indeed, this is precisely a restatement of \eqref{bdd-AD-ple1-e5} when $A=F$, while for $A=B$ we obtain \eqref{bdd-AD-ple1-e5} by summing \eqref{bdd-AD-ple1-e5b} over $i_0\in\Z$.

We now show \eqref{bdd-AD-ple1-e5b}. Since
\begin{equation*}
u<\frac{n}{J}=
\begin{cases}
\min\{1,p\}&\text{if $A=B$},\\
\min\{1,p,q\}&\text{if $A=F$},
\end{cases}
\end{equation*}
we have the duality $(L^{(\frac pu)'}\ell^{(\frac qu)'})_{\hat A}'=(L^{\frac pu}\ell^{\frac qu})_{\hat A}$.

Hence, from the Hahn--Banach theorem, we deduce
\begin{equation}\label{bdd-AD-ple1-e5sup}
\begin{split}
\operatorname{LHS}\eqref{bdd-AD-ple1-e5b}
&=\Bigg\|\Bigg\{
\fint_{B(\cdot,2^{M-i})}\| V(\cdot)t_{i}(y)\|_\fX^u
\,dy\Bigg\}_{i\in\mathbb Z}\Bigg\|_{(L^{\frac pu}\ell^{\frac qu})_A}^{\frac1u}\\
&=\sup\Bigg(
\int_{\mathbb{R}^n}\sum_{i\in\mathbb{Z}}
\fint_{B(x,2^{M-i})}\Norm{V(x)t_i(y)}{\fX}^u\,dy g_i(x)\,dx
\Bigg)^{\frac1u},
\end{split}
\end{equation}
where the supremum is taken over all $\{g_i\}_{i\in\mathbb{Z}}$
such that $\Norm{\{g_i\}_{i\in\mathbb{Z}}}{(L^{(\frac pu)'}\ell^{(\frac qu)'})_A}\le 1$. (When $A=B$, we can simply replace the sum over $i\in\Z$ by the single term $i=i_0$ and take the supremum over $\Norm{g_i}{L^{(\frac pu)'}}\leq 1$.)

We now choose some $\eps\in(0,\eps_V)$, where we also demand that
\begin{equation}\label{eq:q-r<eps}
  q-r<\eps<\eps_V\quad\text{if}\quad A=F,
\end{equation}
which is possible, since $\eps_V>0$ and $q<r+\eps_V$ in that case.

Then
\begin{align}\label{eq:RHIepse2}
\rho_{\aveL^{r+\varepsilon}(Q,V)}(\fx)
\lesssim\rho_{\aveL^{r}(Q,V)}(\fx)
\quad\forall\,Q\in\D, \forall\,\fx\in\fX.
\end{align}
Note that, for any $i\in\mathbb{Z}$, $Q\in\mathscr{D}_{i-M}$, and $x\in Q$,
we have $B(x,2^{M-i})\subset 3Q$.
Thus, for any $i\in\mathbb{Z}$,
\begin{align}\label{bdd-AD-ple1-e7}
\begin{split}
&\int_{\mathbb{R}^n}\fint_{B(x,2^{M-i})}\Norm{V(x)t_i(y)}{\fX}^u\,dy g_i(x)\,dx\\
&\quad\lesssim\sum_{Q\in\mathscr{D}_{i-M}}\int_{3Q}
\fint_{3Q}\Norm{V(x)t_i(y)}{\fX}^ug_i(x)\,dx\,dy
\quad \text{by Tonelli's Theorem}\\
&\quad\le\sum_{Q\in\mathscr{D}_{i-M}}\int_{3Q}
\rho_{\aveL^{r+\varepsilon}(3Q,V)}\big(t_i(y)\big)^u
\Bigg[\fint_{3Q}g_i(x)^{(\frac{r+\varepsilon}{u})'}\,dx
\Bigg]^{1-\frac{u}{r+\varepsilon}}\,dy
\end{split}
\end{align}
by H\"older's inequality.
For any $R\in\mathscr{D}_i$ and $Q\in\mathscr{D}_{i-M}$ satisfying $R\cap 3Q\ne \emptyset$,
it holds that $R\subset 5Q$, which further implies
\begin{align*}
\begin{split}
\rho_{\aveL^r(3Q,V)}(t_i)
&\lesssim\rho_{\aveL^r(5Q,V)}(t_i)\\
&\lesssim\Big(\frac{\ell(5Q)}{\ell(R)}\Big)^{\frac{\beta-n}{r}}
\rho_{\aveL^r(R,V)}(t_i)\quad\text{by Definition \ref{def:dbDim}}\\
&\approx2^{\frac{M(\beta-n)}{r}}\rho_{\aveL^r(R,V)}(t_i).
\end{split}
\end{align*}
Hence, for any $Q\in\mathscr{D}_{i-M}$,
\begin{align*}
\begin{split}
\rho_{\aveL^{r+\varepsilon}(3Q,V)}(t_i){\bf 1}_{3Q}
&\lesssim\rho_{\aveL^r(3Q,V)}(t_i){\bf1}_{3Q}\quad\text{by \eqref{eq:RHIepse2}}\\
&=\sum_{R\in\mathscr{D}_i,R\cap 3Q\ne\emptyset}
\rho_{\aveL^r(3Q,V)}(t_i){\bf 1}_{R\cap 3Q}\\
&\lesssim2^{\frac{M(\beta-n)}{r}}
\rho_{\aveL^r(\mathscr{D}_i,V)}(t_i){\bf 1}_{3Q}.
\end{split}
\end{align*}
From the previous two displays, we find that
\begin{equation}\label{bdd-AD-ple1-e8}
\begin{split}
\operatorname{RHS}\eqref{bdd-AD-ple1-e7}
&\lesssim 2^{\frac{M(\beta-n)u}{r}}
\sum_{Q\in\mathscr{D}_{i-M}}\int_{3Q}
\rho_{\aveL^r(\mathscr{D}_i,V)}(t_i(y))^u
M_{(\frac{r+\varepsilon}{u})'}g_i(y)\,dy\\
&\lesssim2^{\frac{M(\beta-n)u}{r}}
\int_{\mathbb{R}^n}\rho_{\aveL^r(\mathscr{D}_i,V)}(t_i(y))^u
M_{(\frac{r+\varepsilon}{u})'}g_i(y)
\quad\text{since }\sum_{Q\in\mathscr{D}_{i-M}}{\bf 1}_{3Q}\lesssim1,
\end{split}
\end{equation}
where $M_sg_i(y):=\sup_{Q\owns y}\|g_i\|_{\aveL^s(Q)}$.
Therefore, from $r+\eps>r\geq p$ and \eqref{eq:q-r<eps}, it follows that
\begin{align*}
\left\{
\begin{aligned}
&\displaystyle{\Big(\frac{r+\varepsilon}{u}\Big)'<\Big(\frac{p}{u}\Big)'}
\qquad\qquad\qquad\text{if $A=B$},\\
&\displaystyle{\Big(\frac{r+\varepsilon}{u}\Big)'<\min\Big\{\Big(\frac{p}{u}\Big)',
\Big(\frac{q}{u}\Big)'\Big\}}\quad\hspace{.17cm}\text{if $A=F$}.
\end{aligned}
\right.
\end{align*}
Thus, from the boundedness of the maximal operator (for the case $A=B$)
or the Fefferman--Stein vector-valued maximal inequality (for the case $A=F$),
we further obtain
\begin{align*}
\operatorname{LHS}\eqref{bdd-AD-ple1-e5b}
&\lesssim\sup 2^{\frac{M(\beta-n)}{r}}
\Bigg(\int_{\mathbb{R}^n}\sum_{i\in\mathbb{Z}}
\rho_{\aveL^r(\mathscr{D}_i,V)}(t_i(y))^u
M_{(\frac{r+\varepsilon}{u})'}g_i(y)\,dy\Bigg)^{\frac1u}\\
&\qquad\qquad\text{by \eqref{bdd-AD-ple1-e5sup}, \eqref{bdd-AD-ple1-e7}, and \eqref{bdd-AD-ple1-e8}} \\
&\le \sup 2^{\frac{M(\beta-n)}{r}}
\BNorm{\Big\{\rho_{\aveL^r(\mathscr{D}_i,V)}(t_i)^u\Big\}
_{i\in\mathbb{Z}}}{(L^{\frac pu}\ell^{\frac qu})_{\hat A}}^{\frac1u}
\BNorm{\Big\{M_{(\frac{r+\varepsilon}{u})'}g_i
\Big\}_{i\in\mathbb{Z}}}{(L^{(\frac pu)'}\ell^{(\frac qu)'})_{\hat A}}^{\frac1u}\\
&\lesssim\sup 2^{\frac{M(\beta-n)}{r}}
\BNorm{\Big\{\rho_{\aveL^r(\mathscr{D}_i,V)}(t_i)\Big\}_{i\in\mathbb{Z}}}{(L^p\ell^{q})_{\hat A}}
\Norm{\{g_i\}_{i\in\mathbb{Z}}}{(L^{(\frac pu)'}\ell^{(\frac qu)'})_{\hat A}}^{\frac1u} \\
&\qquad\qquad\text{by the relevant maximal inequality in each case} \\
&=2^{\frac{M(\beta-n)}{r}}
\BNorm{\Big\{\rho_{\aveL^r(\mathscr{D}_i,V)}(t_i)\Big\}_{i\in\mathbb{Z}}}{(L^p\ell^{q})_{\hat A}},
\end{align*}
since the supremum coming from \eqref{bdd-AD-ple1-e5sup} is over $\Norm{\{g_i\}_{i\in\mathbb{Z}}}{(L^{(\frac pu)'}\ell^{(\frac qu)'})_{\hat A}}\leq 1$. The previous inequality is precisely \eqref{bdd-AD-ple1-e5b}, which completes the proof of Theorem \ref{bdd-AD-ple1}.
\end{proof}

\begin{remark}\label{rem:abso-V}
Combining \cite[Remark 5.5]{BHYY:II} and
the proof of Theorem \ref{bdd-AD-ple1}, we can actually obtain
the following stronger conclusion:
under the assumptions of Theorem \ref{bdd-AD-ple1},
the series $\sum_{R\in\mathscr{D}}b_{QR}t_R$
is absolutely convergent in $\fX$ for all
$Q\in\mathscr{D}$ and all sequences
$t=\{t_R\}_{R\in\mathscr{D}}\in\dot{a}^s_{p,q}(V)$.

To prove this, suppose that the notation
is as in the proof of Theorem \ref{bdd-AD-ple1}.
\cite[Remark 5.5]{BHYY:II} claims that, when $\fX=\mathbb{C}^m$,
$\eqref{bdd-AD-ple1-e4}$ can be improved with the
left-hand side replaced by
\begin{align}\label{rem:abso-V-e1}
\BNorm{\Bigg\{\sum_{Q\in\mathscr{D}_j}
\sum_{R\in\mathscr{D}}
\Norm{V b_{QR}t_R}{\fX}
\frac{{\bf1}_Q}{|Q|^{\frac12}}\Bigg\}_{j\in\mathbb{Z}}}{(L^p\ell^q)_A}^\nu.
\end{align}
But the proof of this claim does not
use any special property of finite-dimensional spaces
and hence this estimate still holds for general
$\fX$. Hence, by the proof of Theorem \ref{bdd-AD-ple1},
we find that \eqref{rem:abso-V-e1} is estimated by
$\|t\|_{\dot{a}^s_{p,q}(V)}^v<\infty.$
This means that, for almost every $x\in Q$,
$\sum_{R\in\D}\Norm{V(x)b_{QR}t_R}{\fX}<\infty$.
From this, we further infer that, for almost every $x\in Q$,
\begin{align*}
\sum_{R\in\mathscr{D}}\Norm{b_{QR}t_R}{\fX}
\le \Norm{V^{-1}(x)}{\mathscr{L}(\fX)}
\sum_{R\in\mathscr{D}}\Norm{V(x)b_{QR}t_R}{\fX}<\infty,
\end{align*}
which shows the absolute convergence of $\sum_{R\in\mathscr{D}}b_{QR}t_R$.
We will use this conclusion later.
\end{remark}

\subsection{Sharpness of the assumptions}

We now claim that the assumptions on
$(D,E,F)$ in Theorem \ref{bdd-AD-ple1} are optimal in some sense.
Note that the key to show Theorem \ref{bdd-AD-ple1} is
\eqref{bdd-AD-ple1-e5}.
Motivated by Example \ref{ex:p22}, we now
prove that the estimate \eqref{bdd-AD-ple1-e5}
cannot be improved in the infinite-dimensional case
and hence, in this sense,
the assumptions in Theorem \ref{bdd-AD-ple1} are optimal.

\begin{proposition}\label{sharp-ad}
Let $p\in(1,\infty)$ and $u\in(0,1]$. If
there exists $a\in\mathbb{R}$ such that, for any
$V\in\mathscr{A}_p(\mathbb{R};\mathscr{L}(\ell^p))$
with doubling dimension $\beta\in(1,p)$ and for any $t\in\dot{b}^0_{p,p}(V,\mathbb{R})$ and $M\in\mathbb{Z}_+$,
\begin{align}\label{sharp-ade1}
\Bigg\|\Bigg[\fint_{B(\cdot,2^{M-i})}
\Norm{V(\cdot)t_i(y)}{\ell^p}^u\,dy
\Bigg]^{\frac1u}\Bigg\|_{L^p\ell^p}
\lesssim2^{Ma}\|t\|_{\dot{b}^0_{p,p}(V,\mathbb{R})},
\end{align}
then $a\ge\frac{\beta-1}{p}$.
\end{proposition}

\begin{proof}
Let $v(x):=|x|^{\frac{\beta-1}{p}}$ for any $x\in\mathbb{R}$.
Then, since $\beta\in(1,p)$, we infer $v\in\mathscr{A}_p(\mathbb{R})$.
For any $j\in\mathbb{N}$ and $k=0,\ldots,2^{j}-1$, define
\begin{align*}
x_{2^j+k}:=2^{-j}k\quad\text{and}\quad
\lambda_{2^j+k}:=2^{\frac{j(\beta-1)}{p}}j^{-1}.
\end{align*}
Let $v_i(\cdot):=v(\cdot-x_i)$ and $V:=\operatorname{diag}
(v_i)_{i=1}^\infty$.
Then, by \eqref{uniform-ap}, we find that
$V\in\mathscr{A}_p(\mathbb{R};\mathscr{L}(\ell^p))$.
Moreover, using \cite[Lemma 2.40]{BHYY:I}, we conclude that,
for any $i\in\mathbb{N}$, $x\in\mathbb{R}$, and $r\in(0,\infty)$,
\begin{align*}
\int_{B(x,r)}v_i(y)^p\,dy\approx r(r+|x-x_i|)^{\beta-1}.
\end{align*}
Since $\beta>1$, this
implies that, for any $\lambda\in[1,\infty)$,
\begin{align*}
\int_{B(x,\lambda r)}v_i(y)^p\,dy&\approx \lambda
r(\lambda r+|x-x_i|)^{\beta-1}
\le \lambda^{\beta}r(r+|x-x_i|)^{\beta-1}\\
&\approx\lambda^\beta \int_{B(x,r)}v_i(y)^p\,dy.
\end{align*}
Hence, by \eqref{eq:dbRou}, $v_i$ has doubling dimension $\beta$.
Then it follows that, for any cubes $Q\subset S$ and any $\fx=\{e_i\}_{i\in\mathbb{N}}\in\ell^p$
\begin{align*}
\int_{S}\Norm{V(y)\fx}{\ell^p}^p\,dy
&=\sum_{i\in\mathbb{N}}|e_i|^p\int_{S}v_i(y)^p\,dy
\lesssim\Big(\frac{\ell(S)}{\ell(Q)}\Big)^\beta
\sum_{i\in\mathbb{N}}|e_i|^p\int_{Q}v_i(y)^p\,dy\\
&=\Big(\frac{\ell(S)}{\ell(Q)}\Big)^\beta\int_{Q}\Norm{V(y)\fx}{\ell^p}^p\,dy,
\end{align*}
which implies that $V$ also has doubling dimension $\beta$.

For any $i\in\mathbb{N}$, let $\fx_i:=(0,\ldots,0,1,0,\ldots)$,
where the $i$-th element is 1 and all the others are 0.
For any $Q\in\mathscr{D}$, if $Q=Q_{j,k}:=2^{-j}[k,k+1)$ for some $j\ge 3$ and $k=0,\ldots,2^{j}-1$,
define
\begin{align*}
t_Q:=2^{-\frac{j}{2}}\lambda_{2^j+k}
\sum_{\ell=0,\ldots,2^{j}-1,|\ell-k|\le 4}\fx_{2^j+\ell};
\end{align*}
otherwise, define $t_Q:={\bf0}$ (the zero element of $\ell^p$).
Then $t:=\{t_Q\}_{Q\in\mathscr{D}}\in\dot{b}^0_{p,p}(V,\mathbb{R})$.
Indeed,
\begin{align}\label{sharp-ade2}
\begin{split}
\|t\|_{\dot{b}^0_{p,p}(V,\mathbb{R})}^p
&=\sum_{j\in\mathbb{Z}}\sum_{Q\in\mathscr{D}_j}
|Q|^{-\frac p2}
\int_{Q}\|V(x)t_Q\|_{\ell^p}^p\,dx\\
&=\sum_{j=3}^{\infty}
\sum_{k=0}^{2^j-1}\lambda_{2^j+k}^p
\sum_{\ell=0,\ldots,2^{j}-1,|\ell-k|\le 4}
\int_{2^{-j}k}^{2^{-j}(k+1)}v(x-2^{-j}\ell)^p\,dx\\
&=\sum_{j=3}^{\infty}2^{j(\beta-1)}j^{-p}
\sum_{k=0}^{2^j-1}
\sum_{\ell=0,\ldots,2^{j}-1,|\ell-k|\le 4}
\int_{2^{-j}k}^{2^{-j}(k+1)}|x-2^{-j}\ell|^p\,dx.
\end{split}
\end{align}
Observe that, for any $k,\ell=0,\ldots,2^j-1$ with $|k-\ell|\le 4$,
\begin{align*}
\int_{2^{-j}k}^{2^{-j}(k+1)}|x-2^{-j}\ell|^p\,dx
\lesssim\int_{2^{-j}k}^{2^{-j}(k+1)}
[|x-2^{-j}k|^{\beta-1}+
2^{-j(\beta-1)}|k-\ell|^{\beta-1}]\,dx
\lesssim2^{-j\beta}.
\end{align*}
Thus, by \eqref{sharp-ade2}, we obtain
\begin{align*}
\|t\|_{\dot{b}^0_{p,p}(V,\mathbb{R})}^p
\lesssim\sum_{j=3}^{\infty}2^{j(\beta-1)}j^{-p}
2^j2^{-j\beta}
=\sum_{j=3}^{\infty}j^{-p}<\infty,
\end{align*}
which implies $t\in\dot{b}^{0}_{p,p}(V,\mathbb{R})$.

On the other hand, for every $M\in\mathbb{N}\cap[3,\infty)$,
\begin{equation}\label{sharp-ade3}
\operatorname{LHS}(\ref{sharp-ade1})
\gtrsim\Bigg\{\int_{0}^{1}
\Bigg[\int_{0}^{1}\Norm{V(x)t_M(y)}{\ell^p}^u\,dy\Bigg]^{\frac pu}\,dx
\Bigg\}^{\frac1p}.
\end{equation}
We now estimate $\Norm{V(x)t_M(y)}{\ell^p}$.
Note that, for any $x,y\in[0,1)$,
\begin{align}\label{sharp-ade4}
\begin{split}
\Norm{V(x)t_M(y)}{\ell^p}^p
&=\sum_{k=0}^{2^M-1}\lambda_{2^M+k}^p
\BNorm{V(x)\sum_{\ell=0,\ldots,2^M-1,|\ell-k|\le 4}\fx_{2^M+\ell}}{\ell^p}^p
{\bf 1}_{Q_{j,k}}(y)\\
&=2^{M(\beta-1)}M^{-p}
\sum_{k=0}^{2^M-1}\Big[\sum_{\ell=0,\ldots,2^M-1,|\ell-k|\le 4}
|x-2^{-M}\ell|^{p(\beta-1)}\Big]{\bf1}_{Q_{j,k}}(y).
\end{split}
\end{align}
Fix $y\in Q_{j,k}$,
if $|x-y|\ge 2^{-M+1}$, then
\begin{align*}
|x-2^{-M}k|\ge|x-y|-|y-2^{-M}k|
\ge|x-y|-2^{-M}\ge\frac12|x-y|;
\end{align*}
if $|x-y|<2^{-M+1}$, then choose $\ell=0,\ldots,2^{M}-1$
satisfying $|\ell-k|=4$, which further implies
\begin{align*}
|x-2^{-M}\ell|&\ge |2^{-M}\ell-2^{-M}k|
-|2^{-M}k-y|-|y-x|\\
&\ge 2^{-M}(|\ell-k|-1)-|x-y|=3\cdot2^{-M}-|x-y|
>\frac{1}{2}|x-y|.
\end{align*}
Therefore, by \eqref{sharp-ade4}, we obtain, for any $x,y\in[0,1)$,
\begin{align*}
\Norm{V(x)t_M(y)}{\ell^p}^p\gtrsim
2^{M(\beta-1)}M^{-p}|x-y|^{p(\beta-1)}.
\end{align*}
Then we have
\begin{align*}
2^{Ma}&\approx
\operatorname{RHS}\eqref{sharp-ade1}
\gtrsim
\operatorname{LHS}\eqref{sharp-ade1}\\
&\gtrsim2^{M\frac{(\beta-1)}{p}}M^{-1}
\Bigg\{\int_{0}^{1}\Bigg[
\int_{0}^{1}|x-y|^{u(\beta-1)}\,dy\Bigg]^{\frac pu}\,dx\Bigg\}^{\frac1p}
\quad\text{by \eqref{sharp-ade3}}\\
&\approx2^{M\frac{(\beta-1)}{p}}M^{-1}.
\end{align*}
Thus, letting $M\to\infty$, we find $a\ge\frac{\beta-1}{p}$,
which completes the proof of Proposition \ref{sharp-ad}.
\end{proof}

\part{Abstract discretizable spaces}\label{part:discretizable}

\section{An axiomatic framework and examples}

In this remaining part of this work, we extend various further results from the theory of (unweighted, weighted, or matrix-weighted) Besov and Triebel--Lizorkin spaces to the operator-weighted spaces that we have considered.
Having obtained the $\varphi$-transform characterization
and the boundedness of almost diagonal operators on a given function space,
many further real-variable properties can be established in an axiomatic way.
For this purpose, we introduce the following framework.

\begin{definition}\label{def:seq}
Let $d$, $e$, and $f$ be three real numbers.
A quasi-normed space $\dot{\mathbf{A}}(\mathbb{R}^n;\fX)\subset\mathscr{S}_\infty'(\mathbb{R}^n;\fX)$ is
called a \emph{homogeneous $(d,e,f)$-discretizable space} if
the following three conditions hold:
\begin{enumerate}[{\rm(i)}]
  \item $\mathscr{S}_\infty\otimes\fX \subset\dot{\mathbf{A}}(\mathbb{R}^n;\fX)$,
  where $\mathscr{S}_\infty\otimes\fX$ is as in \eqref{eq:sx}.
  \item\label{def:seq-it1} There exists a quasi-norm space $\dot{\mathbf{a}}(\mathbb{R}^n;\fX)$
  consisting of sequences $t = \{t_Q\}_{Q \in \mathcal{D}}$, where each $t_Q \in \mathfrak{X}$,
  such that, for any Littlewood--Paley
functions $\varphi,\psi$ that satisfy \eqref{eq:LPpair},
$$S_{\varphi}:\dot{\mathbf{A}}(\mathbb{R}^n;\fX)\to \dot{\mathbf{a}}(\mathbb{R}^n;\fX)
  \quad\text{and}\quad
  T_{\psi}:\dot{\mathbf{a}}(\mathbb{R}^n;\fX)\to \dot{\mathbf{A}}(\mathbb{R}^n;\fX)$$
  are bounded and $T_\psi\circ S_{\varphi}$ acts as the identity on $\dot{\mathbf{A}}(\mathbb{R}^n;\fX)$.
  \item\label{def:seq-it2} For any triple $(D,E,F)$ that satisfies $D>d$, $E>e$, and $F>f$ and for any $(D,E,F)$-almost diagonal operator $B$,
  the series $\sum_{R\in\mathscr{D}}b_{QR}t_R$
is absolutely convergent in $\fX$ for all
$Q\in\mathscr{D}$ and all sequences
$t=\{t_R\}_{R\in\mathscr{D}}\in\dot{\mathbf{a}}(\mathbb{R}^n;\fX)$,
and $B$ is bounded on $\dot{\mathbf{a}}(\mathbb{R}^n;\fX)$.
\end{enumerate}
\end{definition}

\begin{lemma}\label{lem:seq}
Let $p\in(0,\infty)$, $q\in(0,\infty]$,  and $s\in\mathbb{R}$.
\begin{enumerate}[{\rm(1)}]
  \item\label{ex:seq-rho} If $\rho=\{\rho_{Q}\}_{Q\in\mathscr{D}}$
  is a family of $(u,a,b,c)$-norms (Definition \ref{def:uabc}) with
  $u\in(0,1]$ and $a,b,c\in[0,\infty)$,
  then
$\dot{A}^s_{p,q}(\rho)$ is $(d,e,f)$-discretizable with
\begin{equation*}
  d=J^{(u)}+c,\quad e=\frac{n}{2}+s+a,\quad f=J^{(u)}-\frac{n}{2}-s+b,
\end{equation*}
  where $J^{(u)}$ is as in \eqref{eq:Jr}.

  \item\label{ex:seq-V} If $V\in\A_r$ with $r\in[p,\infty)$ has doubling dimension $\beta\in[n,\infty)$ and optimal reverse H\"older lifting index $\eps_V>0$, then
\begin{equation*}
  \dot{A}^s_{p,q}(V)=\dot{A}^s_{p,q}(\rho_{\aveL^r(\mathscr D,V)}),\quad\text{where}\quad q\in
  \begin{cases} (0,\infty] & \text{if }A=B,  \\ (0,r+\varepsilon_V) & \text{if }A=F,\end{cases}
\end{equation*}
is $(d,e,f)$-discretizable with
\begin{equation*}
  d=J+\frac{\beta-n}{r},\quad e=\frac{n}{2}+s,\quad f=J-\frac{n}{2}-s+\frac{\beta-n}{r},
\end{equation*}
 where $\varepsilon_V$ is the optimal reverse H\"older lifting index of $V$ and $J$ is as in \eqref{eq:J}.
\end{enumerate}
\end{lemma}

\begin{proof}
Part \eqref{ex:seq-rho} follows from Theorems \ref{thm:phi}  and \ref{prop:AD} and Remark \ref{rem:abso-rho}, and part \eqref{ex:seq-V} from Theorems \ref{thm:phiV} and
  \ref{bdd-AD-ple1} and Remark \ref{rem:abso-V}.
\end{proof}

In the next Section \ref{sec:molecule}, we will prove a characterization of $(d,e,f)$-discretizable spaces in terms of so-called molecules. We already introduce the relevant concepts here, to show how the various parameters of molecules and discretizable spaces interact with each other.

We first recall some notation related to integer and fractional parts of numbers.
For $r\in\mathbb R$, let
\begin{equation}\label{ceil}
\begin{cases}
\lfloor r\rfloor:=\max\{k\in\mathbb Z:\ k\leq r\},\\
\lfloor\!\lfloor r\rfloor\!\rfloor:=\max\{k\in\mathbb Z:\ k< r\},
\end{cases}
\begin{cases}
\lceil r\rceil:=\min\{k\in\mathbb Z:\ k\geq r\},\\
\lceil\!\lceil r\rceil\!\rceil:=\min\{k\in\mathbb Z:\ k>r\},
\end{cases}
\end{equation}
and
\begin{equation}\label{r**}
\begin{cases}
r^*:=r-\lfloor r\rfloor\in[0,1),\\
r^{**}:=r-\lfloor\!\lfloor r\rfloor\!\rfloor\in(0,1].
\end{cases}
\end{equation}

We introduce the concept of molecules as follows; see, for instance, \cite[Definition 3.4]{BHYY:III}.
Since we are dealing with space of $\fX$-valued distributions, it might appear natural to also make the molecules $\fX$-valued. However, we will only need scalar-valued molecules; with $\fX$-valued coefficients, they will serve as building blocks of $\fX$-valued distributions in the molecular characterization further below.

\begin{definition}\label{def-mole}
Let $K,M\in[0,\infty)$ and $L,N\in\mathbb{R}$.
A function $m_Q$ is called a
{\em $(K,L,M,N)$-molecule on a cube $Q$} if
\begin{enumerate}[\rm(i)]
  \item\label{it:mole-size}
  (size)
  for any $x\in\mathbb{R}^n$,
  $$|m_Q(x)|\le|Q|^{-\frac12}\Big[1+\frac{|x-x_Q|}{\ell(Q)}\Big]^{-K},$$

  \item\label{it:mole-cancel}
  (cancellation)
  for any $\gamma\in\mathbb{Z}_+^n$
  satisfying $|\gamma|\le L$,
\begin{equation}\label{eq:mole-cancel}
  \int_{\mathbb{R}^n}m_Q(x)x^\gamma\,dx=0,
\end{equation}

  \item\label{it:mole-smooth}
  (smoothness)
  for any $\gamma\in\mathbb{Z}_+^n$ satisfying
 $|\gamma|\le\lfloor\!\lfloor N\rfloor\!\rfloor$
 and for any $x\in\mathbb{R}^n$,
  \begin{align*}
  |\partial^{\gamma}m_Q(x)|\le |Q|^{-\frac12-\frac{|\gamma|}{n}}
  \Big[1+\frac{|x-x_Q|}{\ell(Q)}\Big]^{-M}
  \end{align*}
  and, for any $\gamma\in\mathbb{Z}_+^n$ satisfying $|\gamma|=\lfloor\!\lfloor N\rfloor\!\rfloor$
  and for any $x,y\in\mathbb{R}^n$,
  \begin{align*}
  |\partial^{\gamma}m_Q(x)-\partial^\gamma m_Q(y)|&\le |Q|^{-\frac12-\frac{|\gamma|}{n}}
  \Big[\frac{|x-y|}{\ell(Q)}\Big]^{N^{**}}\sup_{|z|\le |x-y|}
  \Big[1+\frac{|x+z-x_Q|}{\ell(Q)}\Big]^{-M}.
  \end{align*}
\end{enumerate}
\end{definition}

\begin{definition}\label{df-a-syan}
Let $\dot{\mathbf{A}}(\mathbb{R}^n;\fX)$
be a homogeneous $(d,e,f)$-discretizable space with parameters $(d,e,f)$.
Assume that $m_Q$ is
a $(K,L,M,N)$-molecule on a cube $Q$.
Then $m_Q$ is called
\begin{enumerate}[\rm(i)]
  \item\label{defit:a-an} an \emph{$\dot{\mathbf{a}}$-analysis
molecule} on $Q$ if
\begin{equation*}
  K>\max\big\{n,d,e+\frac n2\big\},\quad
  L>e-\frac n2,\quad
  M>\max\{d,n\},\quad
  N>f-\frac n2;
\end{equation*}
  \item\label{defit:a-sy} an \emph{$\dot{\mathbf{a}}$-synthesis
molecule} on $Q$ if
\begin{equation*}
  K>\max\big\{n,d,f+\frac n2\big\},\quad
  L>f-\frac n2,\quad
  M>\max\{d,n\},\quad
  N>e-\frac n2.
\end{equation*}
\end{enumerate}

Moreover, when we talk about a family of
$\dot{\mathbf{a}}$-analysis or synthesis molecules, we
understand that they should be the $(K,L,,M,N)$-molecules for some fixed
$(K,L,M,N)$ that satisfies the condition above.
\end{definition}

\begin{example}\label{ex:mole}
Let $p\in(0,\infty)$, $q\in(0,\infty]$,
 and $s\in\mathbb{R}$.
By Lemma \ref{lem:seq}, we conclude that,
under some assumptions, the average-weighted space $\dot{A}^s_{p,q}(\rho)$
and the pointwise-weighted space
$\dot{A}^s_{p,q}(V)$ are $(d,e,f)$-discretizable spaces
for some $d$, $e$, and $f$. Hence, we obtain the following special cases of Definition \ref{df-a-syan}:
\begin{enumerate}[{\rm(i)}]
  \item Let $\rho=\{\rho_{Q}\}_{Q\in\mathscr{D}}$
  be a family of $(u,a,b,c)$-norms with
  $u\in(0,1]$ and $a,b,c\in[0,\infty)$.
  From Lemma \ref{lem:seq}\eqref{ex:seq-rho},
  one can choose
  \begin{align}\label{eq:def-rho}
  d=J^{(u)}+c, \quad e=\frac{n}{2}+s+a,\quad
  f=J^{(u)}-\frac{n}{2}-s+b,
  \end{align}
  where $J^{(u)}$ is as in \eqref{eq:Jr}.
Then a $(K,L,M,N)$-molecule on a cube $Q$ is
\begin{itemize}
\item an \emph{$\dot{a}^s_{p,q}(\rho)$-analysis molecule} if $(K,L,M,N)$ satisfies Definition
\ref{df-a-syan}\eqref{defit:a-an} for $(d,e,f)$ as in \eqref{eq:def-rho}, i.e.,
\begin{equation*}
  \begin{cases}
  K >\max(n,J^{(u)}+c,n+s+a),&
  L>s+a,\\
  M>\max(n,J^{(u)}+c),&
  N>J^{(u)}-n-s+b;
  \end{cases}
\end{equation*}

\item an \emph{$\dot{a}^s_{p,q}(\rho)$-synthesis molecule} if $(K,L,M,N)$ satisfies Definition
\ref{df-a-syan}\eqref{defit:a-sy} for $(d,e,f)$ as in \eqref{eq:def-rho}, i.e.,
\begin{equation*}
  \begin{cases}
  K >\max(n,J^{(u)}+c,J^{(u)}-s+b),&
  L>J^{(u)}-n-s+b,\\
  M>\max(n,J^{(u)}+c),&
  N>s+a.
  \end{cases}
\end{equation*}
\end{itemize}

  \item Let $V\in\A_r$ with $r\in[p,\infty)$ and the doubling
  dimension $\beta\in[n,\infty)$, and let $q\in(0,\varepsilon_V)$
  when $A=F$, where $\varepsilon_V$ is
  the optimal reverse H\"older lifting index of $V$.
  Lemma \ref{lem:seq}\eqref{ex:seq-V} shows that
  one can choose
  \begin{align}\label{eq:def-V}
  d=J+\frac{\beta-n}{r}, \quad e=\frac{n}{2}+s,\quad
  f=J-\frac{n}{2}-s+\frac{\beta-n}{r},
  \end{align}
  where $J$ is as in \eqref{eq:J}.
  Then a $(K,L,M,N)$-molecule on a cube $Q$ is
  \begin{itemize}
  \item   an \emph{$\dot{a}^s_{p,q}(V)=\dot{a}^s_{p,q}(\rho_{\aveL^r(\mathscr D,V)})$-analysis molecule} if $(K,L,M,N)$ satisfies Definition
\ref{df-a-syan}\eqref{defit:a-an} for $(d,e,f)$ as in \eqref{eq:def-V}, i.e.,
\begin{equation*}
  \begin{cases}
  K >\max(n,J+\frac{\beta-n}{r},n+s),&
  L>s,\\
  M>\max(n,J+\frac{\beta-n}{r}),&
  N>(J-n)-s+\frac{\beta-n}{r};
  \end{cases}
\end{equation*}

  \item   an \emph{$\dot{a}^s_{p,q}(V)=\dot{a}^s_{p,q}(\rho_{\aveL^r(\mathscr D,V)})$-analysis molecule} if $(K,L,M,N)$ satisfies Definition
\ref{df-a-syan}\eqref{defit:a-sy} for $(d,e,f)$ as in \eqref{eq:def-V}, i.e.,
\begin{equation*}
  \begin{cases}
  K >\max(n,J+\frac{\beta-n}{r},J-s+\frac{\beta-n}{r}),&
  L>(J-n)-s+\frac{\beta-n}{r},\\
  M>\max(n,J+\frac{\beta-n}{r}),&
  N>s.
  \end{cases}
\end{equation*}
\end{itemize}
\end{enumerate}
\end{example}

For later convenience, we record the following simple observation:

\begin{remark}\label{rem:noCanc}
If $f-\frac n2<0$, then $\dot{\mathbf{a}}$-synthesis molecules are
not required to have any cancellation conditions \eqref{eq:mole-cancel}. Indeed, in this case, Definition \ref{df-a-syan}\eqref{defit:a-sy} allows us to take $L\in(f-\frac n2,0)$ so that the range of multi-indices in Definition \ref{def-mole}\eqref{it:mole-cancel} is empty.
By Example \ref{ex:mole},
for the average-weighted space $\dot{A}^s_{p,q}(\rho)$
and the pointwise-weighted space $\dot{A}^s_{p,q}(V)$ introduced in
Definition \ref{def:space},
the condition $f-\frac n2<0$ requires that the smoothness parameter $s$ is sufficiently large,
i.e.,
\begin{equation}\label{eq:sLarge}
  \begin{cases}
  s>J^{(u)}-n+b &\text{for}\quad \dot{A}^s_{p,q}(\rho),\\
  s>J-n+\frac{\beta-n}{r} &\text{for}\quad \dot{A}^s_{p,q}(V),
  \end{cases}
\end{equation}
where all the assumptions are as in
Example \ref{ex:mole}.
\end{remark}

\section{Molecular characterization}\label{sec:molecule}

The goal of this section is to establish the molecular characterization of operator-weighted Besov and Triebel--Lizorkin spaces within the abstract framework of $(d,e,f)$-discretizable spaces from Definition \ref{def:seq}.
Before we state the molecular characterization,
we need to show that certain dual pairings $\langle f,\Phi\rangle$ are well defined.

\begin{proposition}\label{le-ana}
Let $\dot{\mathbf{A}}(\mathbb{R}^n;\fX)$
be a homogeneous $(d,e,f)$-discretizable space with some real
parameters $(d,e,f)$. If
$f\in\dot{\mathbf{A}}(\mathbb{R}^n;\fX)$ and
$\Phi$ is an $\dot{\mathbf{a}}$-analysis
molecule on $P\in\mathscr{D}$, then, for any Littlewood--Paley
functions $\varphi,\psi$ that satisfy \eqref{eq:LPpair}, the expression
\begin{align*}
\langle f,\Phi\rangle:=\sum_{R\in\mathscr{D}}
\left\langle f,\varphi_R\right\rangle
\langle\psi_R,\Phi\rangle
\end{align*}
is well defined; that is,
the series above converges absolutely and its value is independent of the
choices of $\varphi$ and $\psi$.
\end{proposition}

To prove this proposition, we need some preparations. We recall the following result from \cite{BHYY:III}:

\begin{lemma}[\cite{BHYY:III}, Lemma 3.7]\label{lem:BHYY:III-3.7}
Let $m_Q$ be a $(K_m,L_m, M_m,N_m)$-molecule on a cube $Q$ and let $b_R$ be
a $(K_b,L_b,M_b,N_b)$-molecule on a cube $R$, where $K_m,M_m,K_b,M_b>n$.
Then the matrix $\{\pair{m_Q}{b_R}\}_{Q,R\in\mathscr D}$ is $(D,E,F)$-almost diagonal with
\begin{equation*}
\begin{split}
  D &=\min\{K_m,M_m,K_b,M_b\}, \\
  E &=\tfrac n2+\min\{N_b,L_m,(K_m-n-\alpha)\}_+, \\
  F &=\tfrac n2+\min\{N_m,L_b,(K_b-n-\alpha)\}_+,
\end{split}
\end{equation*}
where $\alpha>0$ can be chosen as we wish.
\end{lemma}

\begin{lemma}\label{le-ad}
Let $\dot{\mathbf{A}}(\mathbb{R}^n;\fX)$
be a homogeneous $(d,e,f)$-discretizable space with some real
parameters $(d,e,f)$ and let $\varphi,\psi$ be Littlewood--Paley functions.
Let $\{m_Q\}_{Q\in\mathscr{D}}$ and $\{b_Q\}_{Q\in\mathscr{D}}$ be families of $\dot{\mathbf{a}}$-analysis and $\dot{\mathbf{a}}$-synthesis molecules, respectively. Then both matrices
\begin{enumerate}[\rm(i)]
   \item\label{it:mQ} $\{\langle m_Q,\psi_R\rangle\}_{Q,R\in\mathscr{D}}$
   \item\label{it:bQ} $\{\langle\varphi_Q,b_R\rangle\}_{Q,R\in\mathscr{D}}$
\end{enumerate}
are $(D,E,F)$-almost diagonal for some
\begin{equation*}
  D>\max\{d,n\},\qquad E>e,\qquad F>f.
\end{equation*}
In particular, they are bounded on $\dot{\mathbf{a}}(\mathbb{R}^n;\fX)$.
\end{lemma}

\begin{proof}
Once we prove the almost diagonality, the boundedness is immediate from Definition \ref{def:seq}\eqref{def:seq-it2}.

We first consider \eqref{it:mQ}. Assume that
$\{m_Q\}_{Q\in\mathscr{D}}$ are $(K,L,M,N)$-molecules
where the parameters satisfy Definition \ref{df-a-syan}\eqref{defit:a-an}, i.e.,
\begin{equation*}
  K>\max\{n,d,e+\tfrac{n}{2}\},\quad
  L>e-\tfrac{n}{2},\quad
  M>\max\{d,n\},\quad
  N>f-\tfrac{n}{2}.
\end{equation*}
On the other hand, each
$\psi_R$ is a $(\widehat{K},\widehat{L},\widehat{M},\widehat{N})$-molecule
on $R$ for arbitrarily large $\widehat{K},\widehat{L},\widehat{M},\widehat{N}$.
Using Lemma \ref{lem:BHYY:III-3.7} with $\psi_R$ in place of $b_R$‚ it follows that
$\{\pair{m_Q}{\psi_R}\}_{Q,R\in\mathscr D}$ is $(D,E,F)$-almost diagonal with
\begin{equation*}
  D =\min\{K,M\}>\max\{d,n\},\qquad
  F=N+\tfrac{n}{2}>f,
\end{equation*}
and
\begin{equation*}
  E=\min\{L+\tfrac{n}{2},K-\tfrac{n}{2}-\alpha\}_+>e,
\end{equation*}
provided that $\alpha>0$ is chosen small enough.

Turning to \eqref{it:bQ}, we now assume that $\{b_R\}_{R\in\mathscr D}$ are $(K,L,M,N)$-molecules with parameters as in Definition \ref{df-a-syan}\eqref{defit:a-sy}, i.e.,
\begin{equation*}
  K>\max\{n,d,f+\tfrac{n}{2}\},\quad
  L>f-\tfrac{n}{2},\quad
  M>\max\{d,n\},\quad
  N>e-\tfrac{n}{2},
\end{equation*}
while each $\varphi_Q$ is a $(\widehat{K},\widehat{L},\widehat{M},\widehat{N})$-molecule for arbitrarily large parameters. Then, by Lemma \ref{lem:BHYY:III-3.7} with $\varphi_Q$ in place of $m_Q$‚ it follows that
$\{\pair{\varphi_Q}{b_R}\}_{Q,R\in\mathscr D}$ is $(D,E,F)$-almost diagonal with
\begin{equation*}
  D=\min\{K,M\}>\max\{d,n\},\qquad E=N+\tfrac{n}{2}>e,
\end{equation*}
and
\begin{equation*}
  F=\min\{L+\tfrac n2,K-\tfrac n2-\alpha\}_+>f
\end{equation*}
provided that $\alpha>0$ is chosen small enough.

This finishes the proof of Lemma \ref{le-ad} in both cases.
\end{proof}

We also need the following
estimate, whose proof is similar to that of
\cite[Theorem D.2]{FJ:90}.

\begin{lemma}\label{ad-com}
If $D_1,D_2\in(n,\infty)$ and
$E_1,E_2,F_1,F_2\in\mathbb{R}$
satisfy $E_1\ne E_2$, $F_1\ne F_2$,
and $E_1+F_2,E_2+F_1\in(\min\{D_1,D_2\},\infty)$, then,
for any $Q,R\in\mathscr{D}$,
\begin{align*}
\sum_{P\in\mathscr{D}}B_{-E_1,-F_1,-D_1}(Q,P)
B_{-E_2,-F_2,-D_2}(P,R)
\lesssim B_{-E,-F,-D}(Q,R),
\end{align*}
where $G:=\min\{G_1,G_2\}$ for any $G\in\{D,E,F\}$ and the
implicit positive constant is independent of $Q$ and $R$.
\end{lemma}

\begin{proof}
By the symmetry, we only consider the case $\ell(Q)\le\ell(R)$ here.
Therefore, we have
\begin{align*}
&\sum_{P\in\mathscr{D}}B_{-E_1,-F_1,-D_1}(Q,P)
B_{-E_2,-F_2,-D_2}(P,R)\\
&\quad=\sum_{\{P:\ell(P)<\ell(Q)\}}\cdots
+\sum_{\{P:\ell(Q)\le\ell(P)<\ell(R)\}}\cdots
+\sum_{\{P:\ell(R)\le\ell(P)\}}\cdots\notag\\
&\quad=:{\rm I}+{\rm II}+{\rm III}.\notag
\end{align*}

First, we deal with ${\rm I}$. From the definition of $B$, we deduce that
\begin{align}\label{ad-come2}
\begin{split}
{\rm I}&\lesssim\sum_{2^{-j}<\ell(Q)}
\sum_{P\in\mathscr{D}_j}\left[\frac{\ell(P)}{\ell(Q)}\right]^{F_1}
\left[\frac{\ell(P)}{\ell(R)}\right]^{E_2}\\
&\qquad\times\left[1+\frac{|x_Q-x_P|}{\max\{\ell(Q),\ell(P)\}}
\right]^{-D_1}
\left[1+\frac{|x_R-x_P|}{\max\{\ell(R),\ell(P)\}}
\right]^{-D_2}.
\end{split}
\end{align}
Note that, by \cite[Lemma D.1]{FJ:90}, we find that, for any $j\in\mathbb{Z}$,
\begin{align}\label{ad-come3}
\begin{split}
&\sum_{P\in\mathscr{D}_j}\left[1+\frac{|x_Q-x_P|}{\max\{\ell(Q),\ell(P)\}}
\right]^{-D_1}
\left[1+\frac{|x_R-x_P|}{\max\{\ell(R),\ell(P)\}}
\right]^{-D_2}\\
&\quad\lesssim\left[1+\frac{|x_Q-x_R|}{\max\{2^{-j},\ell(R)\}}
\right]^{-D}\max\{1,2^j\ell(Q)\}^n,
\end{split}
\end{align}
where the maximum is $2^j\ell(Q)>1$ by the summation condition in term I.

Substituting \eqref{ad-come3} back to \eqref{ad-come2}, we obtain
\begin{align*}
{\rm I}&\lesssim\left[1+\frac{|x_Q-x_R|}{\max\{\ell(Q),\ell(R)\}}
\right]^{-D}
\sum_{2^{-j}<\ell(Q)}\left[\frac{2^{-j}}{\ell(Q)}\right]^{F_1}
\left[\frac{2^{-j}}{\ell(R)}\right]^{E_2}
\left[2^j\ell(Q)\right]^n\\
&\lesssim\left[\frac{\ell(Q)}{\ell(R)}\right]^{E}
\left[1+\frac{|x_Q-x_R|}{\max\{\ell(Q),\ell(R)\}}
\right]^{-D}\qquad\text{since }F_1+E_2>D>n\\
&\approx B_{-E,-F,-D}(Q,R)\qquad\text{since }\ell(Q)\leq\ell(R).
\end{align*}
This is the desired estimate of ${\rm I}$.

Next, we focus on ${\rm II}$. Applying
the definition of $B$ and \eqref{ad-come3} again and the
assumption that $E_1\ne E_2$, we obtain
\begin{align*}
{\rm II}&\lesssim\sum_{\ell(Q)\le2^{-j}<\ell(R)}
\sum_{P\in\mathscr{D}_j}\left[\frac{\ell(Q)}{\ell(P)}\right]^{E_1}
\left[\frac{\ell(P)}{\ell(R)}\right]^{E_2}\\
&\qquad\times\left[1+\frac{|x_Q-x_P|}{\max\{\ell(Q),\ell(P)\}}
\right]^{-D_1}
\left[1+\frac{|x_R-x_P|}{\max\{\ell(R),\ell(P)\}}
\right]^{-D_2}\\
&\lesssim\sum_{\ell(Q)\le2^{-j}<\ell(R)}
\left[2^j\ell(Q)\right]^{E_1}\left[\frac{2^{-j}}{\ell(R)}\right]^{E_2}
\left[1+\frac{|x_Q-x_R|}{\max\{2^{-j},\ell(R)\}}
\right]^{-D}\\
&\lesssim\left[\frac{\ell(Q)}{\ell(R)}\right]^{E}
\left[1+\frac{|x_Q-x_R|}{\max\{\ell(Q),\ell(R)\}}
\right]^{-D}\approx B_{-E,-F,-D}(Q,R),
\end{align*}
which completes the estimation of ${\rm II}$.

Finally, we estimate ${\rm III}$.
Similarly to the above argument, we have
\begin{align}\label{ad-come4}
{\rm III}\lesssim
\sum_{2^{-j}\ge\ell(R)}
\left[2^j\ell(Q)\right]^{E_1}\left[2^j\ell(R)\right]^{F_2}
\left[1+2^j|x_Q-x_R|\right]^{-D}.
\end{align}
Observe that
\begin{align*}
1+2^j|x_Q-x_R|\ge 2^j\ell(R)
\left[1+\frac{|x_Q-x_R|}{\max\{\ell(Q),\ell(R)\}}\right].
\end{align*}
Substituting this into \eqref{ad-come4}, we conclude that
\begin{align*}
{\rm III}&\lesssim
\left[1+\frac{\abs{x_Q-x_R}}{\max\{\ell(Q),\ell(R)\}}\right]^{-D}
\sum_{2^{-j}\ge\ell(R)}
\left[2^j\ell(Q)\right]^{E_1}\left[2^j\ell(R)\right]^{F_2-D}\\
&\lesssim\left[\frac{\ell(Q)}{\ell(R)}\right]^{E}
\left[1+\frac{|x_Q-x_R|}{\max\{\ell(Q),\ell(R)\}}
\right]^{-D}\qquad\text{since }E_1+F_2>D \\
&\approx B_{-E,-F,-D}(Q,R).
\end{align*}
This finishes the estimation of ${\rm III}$ and hence Lemma \ref{ad-com}.
\end{proof}

\begin{proof}[Proof of Proposition \ref{le-ana}]
Let $(\varphi_1,\psi_1)$ and $(\varphi_2,\psi_2)$
be two pairs of Littlewood--Paley
functions that satisfy \eqref{eq:LPpair}.
Define $\Phi_P:=\Phi$ and $\Phi_{S}:=0$ if $S\ne P$.
Then $\{\Phi_S\}_{S\in\mathscr{D}}$
is a family of $\dot{\mathbf{a}}$-analysis molecules.
Hence, by Lemma \ref{le-ad}, we can
choose $D_1>\max\{n,d\}$, $E_1>e$, and $F_1>f$ such that
$\{\langle \Phi_S,\psi_{Q}^{(2)}\rangle\}_{S,Q\in\mathscr{D}}$
is $(D_1,E_1,F_1)$-almost diagonal. Moreover,
since $\psi^{(1)},\varphi^{(2)}\in\mathscr{S}_\infty(\mathbb{R}^n)$,
we can choose sufficiently large $(D_2,E_2,F_2)$ such that the pairs
$\{(D_i,E_i,F_i)\}_{i=1}^2$ satisfy all the assumptions of
Lemma \ref{ad-com} and
$\{\langle\varphi_Q^{(2)},\psi_R^{(1)}\rangle\}_{Q,R\in\mathscr{D}}$
is $(D_2,E_2,F_2)$-almost diagonal. Thus, from
Lemma \ref{ad-com}, we infer that, for any $S,R\in\mathscr{D}$,
\begin{align*}
t_{SR}:=
\sum_{Q\in\mathscr{D}}\abs{\langle \Phi_S,\psi_{Q}^{(2)}\rangle
\langle\varphi_Q^{(2)},\psi_R^{(1)}\rangle}
\lesssim B_{-E_1,-F_1,-D_1}(S,R).
\end{align*}
This shows that $\{t_{SR}\}_{S,R\in\mathscr{D}}$
is $(D_1,E_1,F_1)$-almost diagonal.
Combining this and Definition \ref{def:seq}, we
further conclude that,
for any $S\in\mathscr{D}$ and $f\in\dot{\mathbf{A}}(\mathbb{R}^n;\fX)$,
the series $\sum_{R\in\mathscr{D}}t_{SR}\langle f,\varphi_R^{(1)}\rangle$
is absolutely convergent in $\fX$; in particular,
\begin{align*}
\sum_{Q,R\in\mathscr{D}}\langle f,\varphi_R^{(1)}\rangle
\langle\psi_R^{(1)},\varphi_Q^{(2)}\rangle
\langle\psi_Q^{(2)},\Phi\rangle
\end{align*}
is absolutely convergent.
Thus,
\begin{align*}
\sum_{R\in\mathscr{D}}\langle f,\varphi_R^{(1)}\rangle
\langle\psi_R^{(1)},\Phi\rangle
&=\sum_{R\in\mathscr{D}}\langle f,\varphi_R^{(1)}\rangle
\Big[\sum_{Q\in\mathscr{D}}\langle\psi_R^{(1)},
\varphi_Q^{(2)}\rangle
\langle\psi_Q^{(2)},\Phi\rangle\Big]\\
&=\sum_{Q\in\mathscr{D}}
\Big[\sum_{R\in\mathscr{D}}
\langle f,\varphi_R^{(1)}\rangle\langle\psi_R^{(1)},
\varphi_Q^{(2)}\rangle\Big]\langle\psi_Q^{(2)},\Phi\rangle \\
&=\sum_{Q\in\mathscr{D}}\langle f,\varphi_Q^{(2)}\rangle
\langle\psi_Q^{(2)},\Phi\rangle.
\end{align*}
This implies that the present lemma holds.
\end{proof}

Now, we are able to prove the
molecular characterization of
$\dot{\mathbf{A}}(\mathbb{R}^n;\fX)$.

\begin{theorem}\label{thm-mole}
Let $\dot{\mathbf{A}}(\mathbb{R}^n;\fX)$
be a homogeneous $(d,e,f)$-discretizable space with some real
parameters $(d,e,f)$.
\begin{enumerate}[\rm(i)]
  \item\label{thmit:an} If $\{m_Q\}_{Q\in\mathscr{D}}$ are
  $\dot{\mathbf{a}}$-analysis molecules on their indicated cubes,
  then there exists a positive constant $C$
  such that, for any $f\in\dot{\mathbf{A}}(\mathbb{R}^n;\fX)$,
  \begin{align*}
  \bNorm{\big\{\langle f,m_Q\rangle\big\}_{Q\in\mathscr{D}}}
  {\dot{\mathbf{a}}(\mathbb{R}^n;\fX)}
  \le C\|f\|_{\dot{\mathbf{A}}(\mathbb{R}^n;\fX)}.
  \end{align*}
  \item\label{thmit:sy} If
  $\{b_Q\}_{Q\in\mathscr{D}}$ are $\dot{\mathbf{a}}$-synthesis molecules on their indicated cubes
  and if $t:=\{t_Q\}_{Q\in\mathscr{D}}\in
  \dot{\mathbf{a}}(\mathbb{R}^n;\fX)$,
  then $$f:=\sum_{Q\in\mathscr{D}}t_Qb_Q
  \in\dot{\mathbf{A}}(\mathbb{R}^n;\fX)$$ and
  there exists a positive constant $C$
  such that
\begin{equation*}
  \|f\|_{\dot{\mathbf{A}}(\mathbb{R}^n;\fX)}
  \le C\|t\|_{\dot{\mathbf{a}}(\mathbb{R}^n;\fX)}.
\end{equation*}
\end{enumerate}
\end{theorem}

\begin{proof}
Let $\varphi,\psi$ be two Littlewood--Paley functions that satisfy \eqref{eq:LPpair}.
We first show \eqref{thmit:an}. From Proposition \ref{le-ana}, we infer that
$\{\langle f,m_Q\rangle\}_{Q\in\mathscr{D}}$
is well defined. Using the definition, we obtain,
for any $Q\in\mathscr{D}$,
\begin{align*}
\langle f,m_Q\rangle=\sum_{R\in\mathscr{D}}
\langle f,\varphi_R\rangle\langle \psi_R,m_Q\rangle
=\sum_{R\in\mathscr{D}}\langle m_Q,\psi_R\rangle
(S_\varphi f)_R.
\end{align*}
Moreover, by Lemma \ref{le-ad}\eqref{it:mQ}, we find that
$B:=\{\langle m_Q,\psi_R\rangle\}_{Q,R\in\mathscr{D}}$
is bounded on $\dot{\mathbf{a}}(\mathbb{R}^n;\fX)$.
Consequently,
it follows that
\begin{align*}
\bNorm{\big\{\langle f,m_Q\rangle\big\}_{Q\in\mathscr{D}}}
  {\dot{\mathbf{a}}(\mathbb{R}^n;\fX)}
&=\Norm{B\circ S_\varphi f}{\dot{\mathbf{a}}(\mathbb{R}^n;\fX)}
\lesssim\Norm{S_\varphi f}{\dot{\mathbf{a}}(\mathbb{R}^n;\fX)}\\
&\lesssim\|f\|_{\dot{\mathbf{A}}(\mathbb{R}^n;\fX)}\qquad
\text{by Definition \ref{def:seq}\eqref{def:seq-it1}},
\end{align*}
which confirms \eqref{thmit:an}.

Next, we prove \eqref{thmit:sy}. From Lemma \ref{le-ad}\eqref{it:bQ},
we deduce that the matrix
$\widetilde{B}:=\{\langle\varphi_Q,b_R\rangle\}_{Q,R\in\mathscr{D}}$
is $(D,E,F)$-almost diagonal for some $D>d$, $E>e$, and $F>f$
and hence, by Definition \ref{def:seq}\eqref{def:seq-it2},
for any $Q\in\mathscr{D}$,
the series $\sum_{R\in\mathscr{D}}\langle\varphi_Q,b_R\rangle t_R$
is absolutely convergent.
Therefore, for any $Q\in\mathscr{D}$,
\begin{align*}
(S_\varphi f)_Q
=\langle f,\varphi_Q\rangle=(\widetilde{B}t)_Q.
\end{align*}
Combining this and Definition \ref{def:seq} again,
we further obtain
\begin{align*}
\|f\|_{\dot{\mathbf{A}}(\mathbb{R}^n;\fX)}
&\lesssim\Norm{S_\varphi f}{\dot{\mathbf{a}}(\mathbb{R}^n;\fX)}
=\Norm{\widetilde{B}t}{\dot{\mathbf{a}}(\mathbb{R}^n;\fX)}
\lesssim\|t\|_{\dot{\mathbf{a}}(\mathbb{R}^n;\fX)}.
\end{align*}
This finishes the proof of \eqref{thmit:sy} and hence Theorem \ref{thm-mole}.
\end{proof}

By Lemma \ref{lem:seq} and Example \ref{ex:mole},
we immediately obtain the following applications of Theorem \ref{thm-mole},
which provide the molecular characterizations of
the Besov and the Triebel--Lizorkin spaces introduced in Definition \ref{def:space};
we omit the straightforward details.

\begin{corollary}
Let $p\in(0,\infty)$, $q\in(0,\infty]$,
and  $s\in\mathbb{R}$.
\begin{enumerate}[{\rm(i)}]
  \item If $\rho=\{\rho_{Q}\}_{Q\in\mathscr{D}}$
  is a family of $(u,a,b,c)$-norms with
  $u\in(0,1]$ and $a,b,c\in[0,\infty)$, then
  Theorem \ref{thm-mole} applies to the average-weighted space $\dot{A}^s_{p,q}(\rho)$
  and the $\dot{a}^s_{p,q}(\rho)$-analysis and synthesis molecules.
  \item If $V\in\A_r$ with $r\in[p,\infty)$ has doubling
  dimension $\beta\in[n,\infty)$ and if $q\in(0,r+\varepsilon_V)$
  when $A=F$, where $\varepsilon_V$ is
  the optimal reverse H\"older lifting index of $V$,
  then Theorem \ref{thm-mole} applies to the space $\dot{A}^s_{p,q}(V)=\dot A^s_{p,q}(\rho_{\aveL^r(\mathscr D,V)})$
  and the $\dot{a}^s_{p,q}(V)=\dot a^s_{p,q}(\rho_{\aveL^r(\mathscr D,V)})$-analysis and synthesis molecules.
\end{enumerate}
\end{corollary}

\section{Wavelet characterization}

In this section, we first establish the wavelet
characterization of the abstract space $\dot{\mathbf{A}}(\mathbb{R}^n;\fX)$
and, as an application, we give the
wavelet characterization of Besov and Triebel--Lizorkin spaces
introduced in Definition \ref{def:space}.

To begin with, we recall some notation
on the Daubechies wavelet system (see, for example,
\cite{mey92,woj97}). Take a univariate scaling function
$\varphi\in C_{\rm c}^N(\mathbb{R})$ and an associated univariate
wavelet function $\psi\in C_{\rm c}^{N}(\mathbb{R})$ for some
$N\in\mathbb{N}$ and define
$\psi^{(0)}:=\varphi$ and $\psi^{(1)}:=\psi$. For any $\lambda\in \{0,1\}^n$
and $x:=(x_1,\ldots,x_n)\in\mathbb{R}^n$,
let $\Psi^{(\lambda)}(x):=\prod_{i=1}^{n}\psi^{(\lambda_i)}(x_i)$.
Also, assume that, for any $\lambda\in \Lambda_n:=\{0,1\}^n\setminus\{{\bf0}\}$
and $\alpha:=(\alpha_1,\ldots,\alpha_n)\in \mathbb{Z}_+^n$ with $|\alpha|
:=|\alpha_1|+\cdots+|\alpha_n|\le N$,
\begin{align}\label{e1937}
\int_{\mathbb{R}^n}x^\alpha \Psi^{(\lambda)}(x)\,dx=0.
\end{align}
For any $\lambda\in\{0,1\}^n$, $Q\in\mathscr{D}$, and $x\in\mathbb{R}^n$,
define
\begin{align*}
\Psi^{(\lambda)}_Q(x):=|Q|^{-\frac12}\Psi^{(\lambda)}
\Big(\frac{x-x_Q}{\ell(Q)}\Big).
\end{align*}
In many considerations, we only deal with the cancellative wavelets corresponding to $\lambda\in\Lambda_n$, but it is convenient to have the non-cancellative wavelets with $\lambda={\bf 0}$ also included in the definition.

For brevity, let $\Omega_n:=\Lambda_n\times\mathscr{D}$ and,
for any $\omega=(\lambda,Q)\in\Omega_n$,
define $\lambda_\omega:=\lambda$, $Q_\omega:=Q$, and
$\Psi_{\omega}:=\Psi^{(\lambda)}_{Q}$.
It is well known that these functions can be constructed so that, in addition to the properties stated,
$\{\Psi_\omega\}_{\omega\in\Omega_n}$
forms a complete orthonormal basis of $L^2(\mathbb{R}^n)$.
The number $N$ is called the \emph{order} of this wavelet system.

The following proposition is a
consequence of the definitions
of molecules (namely Definition \ref{df-a-syan}) and wavelets and Proposition \ref{le-ana}.

\begin{proposition}\label{prop-wavelet}
Let $\dot{\mathbf{A}}(\mathbb{R}^n;\fX)$
be a $(d,e,f)$-discretizable space for some real parameters
$(d,e,f)$.
Let $N\in\mathbb{N}$ satisfy
\begin{equation}\label{NDau}
N>\max\{e,f\}-\frac n2,
\end{equation}
and let $\Psi^{(\lambda)}_Q$ be Daubechies wavelets of order $N$.
\begin{enumerate}[\rm(i)]
\item\label{it:wave=mole}
The wavelet $\Psi_Q^{(\lambda)}$
is a constant multiple of both an $\dot{\mathbf{a}}$-analysis molecule and
an $\dot{\mathbf{a}}$-synthesis molecule on $Q$
for all $\lambda\in\Lambda_n$ and $Q\in\mathscr D$.

\item\label{it:wave=symole}
If $f-\frac n2<0$, then
$\Psi_Q^{\lambda}$ is a constant multiple of an $\dot{\mathbf a}$-synthesis molecule on~$Q$
for all $\lambda\in\{0,1\}^n$ and $Q\in\mathscr D$.

\item\label{it:f,wave} If $\varphi,\psi$ are Littlewood--Paley
functions satisfying \eqref{eq:LPpair},
then, for any $f\in\dot{\mathbf{A}}(\mathbb{R}^n;\fX)$
and $\omega\in\Omega_n$,
$$
\left\langle f,\Psi_\omega\right\rangle
:=\sum_{R\in\mathscr{Q}}
\left\langle f,\varphi_R\right\rangle
\left\langle\psi_R,\Psi_\omega\right\rangle
$$
is well defined; that is, the series converges absolutely
and its value is independent of the choice of $\varphi$ and $\psi$.
\end{enumerate}
\end{proposition}

Note the difference between cases \eqref{it:wave=mole} and \eqref{it:wave=symole}: The latter makes an additional assumption \eqref{eq:sLarge} and only applies to synthesis molecules, but then allows the full range of $\lambda\in\{0,1\}^n$, including the non-cancellative case $\lambda={\bf 0}$, which is excluded by the assumption that $\lambda\in\Lambda_n$ in case \eqref{it:wave=mole}.

\begin{proof}[Proof of Proposition \ref{prop-wavelet}]
We first consider \eqref{it:wave=mole}.
By Definitions \ref{df-a-syan} and \ref{def-mole}, $\dot{\mathbf{a}}$-analysis and
$\dot{\mathbf{a}}$-synthesis molecules are
$(K,L,M,N)$-molecules with sufficiently large values of the parameters,
as specified in Definition \ref{df-a-syan}.

From $\varphi,\psi \in C_{\rm c}^N(\R)$ and the way that $\Psi_Q^{(\lambda)}$ is formed by scaling and translation, the molecular size condition (Definition \ref{def-mole}\eqref{it:mole-size}) holds with $K$ as large as we want, and the molecular smoothness condition (Definition \ref{def-mole}\eqref{it:mole-smooth} holds with the same $N$ as the order of the wavelets and arbitrarily large $M$. Hence, both $K$ and $M$ (trivially) and $N$ (by assumption \eqref{NDau}) are sufficiently large for both types of molecules.

For $\lambda\in\Lambda_n$, assumption \eqref{e1937} shows that the molecular cancellation condition (Definition \ref{def-mole}\eqref{it:mole-cancel}) also holds with $L=N$, which is sufficiently large for analysis/synthesis molecules (Definition \ref{df-a-syan}) by assumption \eqref{NDau}.

For \eqref{it:wave=symole},
under the additional assumption $f-\frac n2<0$, Remark \ref{rem:noCanc} shows that we can take $L<0$, meaning that no cancellation is needed for synthesis molecules. (This is needed for $\lambda'={\bf 0}$, in which case there is no cancellation.) The fact that $\Psi_Q^{(\lambda)}$ satisfies the other molecular conditions (size and smoothness) with relevant $K,M,N$ is the same as in case \eqref{it:wave=mole}.

Finally, we deal with \eqref{it:f,wave}.
By part \eqref{it:wave=mole}, we know that $\Psi_\omega=\Psi^{(\lambda)}_Q$ (with $\lambda\in\Lambda_n$) is an $\dot{\mathbf{a}}$-analysis molecule. Then the claim of \eqref{it:f,wave} is a direct application of Proposition \ref{le-ana} with $\Phi=\Psi_\omega$.
\end{proof}

Now, we give the wavelet characterization of
$\dot{\mathbf{A}}(\mathbb{R}^n;\fX)$ as follows.

\begin{theorem}\label{thm-wavelet}
Let $\dot{\mathbf{A}}(\mathbb{R}^n;\fX)$
be a $(d,e,f)$-discretizable space for some real parameters
$(d,e,f)$.
If $N\in\mathbb{N}$ satisfies \eqref{NDau}
and $\{\Psi_\omega\}_{\omega\in\Omega_n}$
is a Daubechies wavelet system of order $N$,
then every $f\in\dot{\mathbf{A}}(\mathbb{R}^n;\fX)$ has a series representation
\begin{equation}\label{thm-wavelet-e1}
f=\sum_{\omega\in\Omega_n}
\left\langle f,\Psi_\omega\right\rangle\Psi_\omega
\end{equation}
convergent in $\mathscr{S}_\infty'(\mathbb{R}^n;\fX)$ and
\begin{align}\label{thm-wavelet-e2}
\|f\|_{\dot{\mathbf{A}}(\mathbb{R}^n;\fX)}
&\approx\|f\|_{\dot{\mathbf{A}}(\mathbb{R}^n;\fX)_\mathrm{w}}
:=\sum_{\lambda\in\Lambda_n}
\bNorm{\big\{\langle f,\Psi^{(\lambda)}_Q \rangle
\big\}_{Q\in\mathscr{D}}}{\dot{\mathbf{a}}(\mathbb{R}^n;\fX)},
\end{align}
where the positive equivalence constants are independent of $f$.
\end{theorem}

\begin{proof}
Fix $f\in\dot{\mathbf{A}}(\mathbb{R}^n;\fX)$
and two Littlewood--Paley
functions $\varphi$ and $\psi$
satisfying \eqref{eq:LPpair}.
For any $\lambda\in\Lambda_n$ and $Q\in\mathscr{D}$,
define $t_Q^{(\lambda)}:=\langle f,\Psi_Q^{(\lambda)}\rangle$.
Proposition \ref{prop-wavelet}\eqref{it:f,wave}
shows that $t_Q^{(\lambda)}$ is well defined and
\begin{align*}
t_Q^{(\lambda)}
=\sum_{R\in\mathscr{D}}\langle f,\varphi_R\rangle
\langle \psi_R,\Psi_Q^{(\lambda)}\rangle
=(B^{(\lambda)}\circ S_\varphi f)_Q,
\end{align*}
where $B^{(\lambda)}:=\{\langle\Psi_Q^{(\lambda)},\psi_R\rangle
\}_{Q,R\in\mathscr{D}}$.
From Proposition \ref{prop-wavelet}\eqref{it:wave=mole}
and Lemma \ref{le-ad}\eqref{it:mQ}, we deduce that
$B^{(\lambda)}$ is bounded on $\dot{\mathbf{a}}(\mathbb{R}^n;\fX)$.
Combined with Definition \ref{def:seq}\eqref{def:seq-it1}, this implies that
\begin{align*}
\bNorm{\big\{t_Q^{(\lambda)}\big\}_{Q\in\mathscr{D}}}
{\dot{\mathbf{a}}(\mathbb{R}^n;\fX)}
&=\bNorm{B^{(\lambda)}\circ S_\varphi f}{\dot{\mathbf{a}}(\mathbb{R}^n;\fX)}
\lesssim\|f\|_{\dot{\mathbf{A}}(\mathbb{R}^n;\fX)}.
\end{align*}
Summing over the finitely many $\lambda\in\Lambda_n$, this implies ``$\gtrsim$'' in \eqref{thm-wavelet-e2}.

Conversely, fix $\lambda\in\Lambda_n$, and suppose that
$\{t_Q^{(\lambda)}\}_{Q\in\mathscr{D}}\in\dot{\mathbf{a}}(\mathbb{R}^n;\fX)$.
Therefore, from Proposition \ref{prop-wavelet}\eqref{it:wave=mole} and Theorem
\ref{thm-mole}\eqref{thmit:sy},
it follows that $\sum_{Q\in\mathscr{D}}t_Q^{(\lambda)}\Psi_Q^{(\lambda)}$
converges in $\mathscr{S}_\infty'(\mathbb{R}^n;\fX)$
and
\begin{align}\label{thm-wavelet-e3}
\BNorm{\sum_{Q\in\mathscr{D}}t_Q^{(\lambda)}
\Psi_Q^{(\lambda)}}{\dot{\mathbf{A}}(\mathbb{R}^n;\fX)}
\lesssim\bNorm{\big\{t_Q^{(\lambda)}\big\}_{Q\in\mathscr{D}}}
{\dot{\mathbf{a}}(\mathbb{R}^n;\fX)}.
\end{align}
Moreover, similarly to the proof of Proposition \ref{le-ana},
we conclude that, for any $\phi\in\mathscr{S}_\infty(\mathbb{R}^n)$, the series
\begin{align*}
\sum_{Q,R\in\mathscr{D}}\langle f,\varphi_R\rangle
\langle\psi_R,\Psi_Q^{(\lambda)}\rangle
\langle\Psi_Q^{(\lambda)},\phi\rangle
\end{align*}
is absolutely convergent.
Together with the fact that $\{\Psi_\omega\}_{\omega\in\Omega_n}$
is an orthonormal basis of $L^2(\mathbb{R}^n)$ and
the assumed Littlewood--Paley property \eqref{eq:LPpair},
this further implies
\begin{align*}
\sum_{\omega\in\Omega_n}\langle f,\Psi_\omega
\rangle\langle\Psi_\omega,\phi\rangle
&=\sum_{\lambda\in\Lambda_n}
\sum_{Q,R\in\mathscr{D}}\langle f,\varphi_R\rangle
\langle\psi_R,\Psi_Q^{(\lambda)}\rangle
\langle\Psi_Q^{(\lambda)},\phi\rangle\\
&=\sum_{R\in\mathscr{D}}\langle f,\varphi_R\rangle
\sum_{\omega\in\Omega_n}\langle \psi_R,\Psi_\omega\rangle
\langle \Psi_\omega,\phi\rangle\\
&=\sum_{R\in\mathscr{D}}\langle f,\varphi_R\rangle
\langle\psi_R,\phi\rangle=\langle f,\phi\rangle.
\end{align*}
That is, \eqref{thm-wavelet-e1} holds in $\mathscr S_\infty'(\R^n;\fX)$ and
\begin{align*}
\|f\|_{\dot{\mathbf{A}}(\mathbb{R}^n;\fX)}
&=\BNorm{\sum_{\omega\in\Omega_n}
\langle f,\varphi_R\rangle
\langle\psi_R,\phi\rangle}
{\dot{\mathbf{A}}(\mathbb{R}^n;\fX)}
\lesssim\sum_{\lambda\in\Lambda_n}
\BNorm{\sum_{Q\in\mathscr{D}}t_Q^{(\lambda)}
\Psi_Q^{(\lambda)}}{\dot{\mathbf{A}}(\mathbb{R}^n;\fX)}\\
&\lesssim\sum_{\lambda\in\Lambda_n}
\bNorm{\big\{t_Q^{(\lambda)}\big\}_{Q\in\mathscr{D}}}
{\dot{\mathbf{a}}(\mathbb{R}^n;\fX)}\quad\text{by \eqref{thm-wavelet-e3}}\\
&=\|f\|_{\dot{\mathbf{A}}(\mathbb{R}^n;\fX)_{\rm w}}.
\end{align*}
This is the desired upper estimate of \eqref{thm-wavelet-e2}
and hence finishes the proof of Theorem \ref{thm-wavelet}.
\end{proof}

Combining Lemma \ref{lem:seq} and Example \ref{ex:mole},
Theorem \ref{thm-wavelet} yields the
following wavelet characterizations of
the function spaces introduced in Definition \ref{def:space};
we omit the details.

\begin{corollary}\label{cor:wavelet}
Let $p\in(0,\infty)$, $q\in(0,\infty]$,
 and $s\in\mathbb{R}$.
\begin{enumerate}[{\rm(i)}]
  \item Let $\rho=\{\rho_{Q}\}_{Q\in\mathscr{D}}$
  be a family of $(u,a,b,c)$-norms with
  $u\in(0,1]$ and $a,b,c\in[0,\infty)$.
  If $N\in\mathbb{N}$ satisfies
  \begin{align}\label{eq:N-wavelet}
  N>\max\big\{s+a,J^{(u)}-n-s+b\big\}
  \end{align}
  with $J^{(u)}$ as in \eqref{eq:Jr},
  then Proposition \ref{prop-wavelet} and Theorem \ref{thm-wavelet} apply to $\dot{A}^s_{p,q}(\rho)$ in place of $\dot{\mathbf{A}}(\R^n;\fX)$ with condition $f-\frac{n}{2}<0$ of Proposition \ref{prop-wavelet}\eqref{it:wave=symole} replaced by the $\dot A^s_{p,q}(\rho)$ case of \eqref{eq:sLarge}.

  \item Let $V\in\A_r$ with $r\in[p,\infty)$ and the doubling
  dimension $\beta\in[n,\infty)$, and let $q\in(0,r+\varepsilon_V)$
  when $A=F$, where $\varepsilon_V$ is
  the optimal reverse H\"older lifting index of $V$. If
  \begin{align}\label{eq:N-wavelet-V}
  N>\max\big\{s,J-n-s+\frac{\beta-n}{r}\big\}
  \end{align}
  with $J$ as in \eqref{eq:J},
  then Proposition \ref{prop-wavelet} and Theorem \ref{thm-wavelet} apply to the space $\dot{A}^s_{p,q}(V)=\dot A^s_{p,q}(\rho_{\aveL^r(\mathscr D,V)})$ in place of $\dot{\mathbf{A}}(\R^n;\fX)$ with condition $f-\frac{n}{2}<0$ of Proposition \ref{prop-wavelet}\eqref{it:wave=symole} replaced by the $\dot A^s_{p,q}(V)$ case of \eqref{eq:sLarge}.
\end{enumerate}
\end{corollary}

\section{Calder\'on--Zygmund operators}

In this section, we study the boundedness of
Calder\'on--Zygmund operators and the related $T(1)$ theorem.
First, we give a general extension
criterion of a given operator $T\in\mathcal{L}(\mathscr{S},\mathscr{S}')$
in Subsection \ref{sec:CZO-1}.
Then we establish the $T(1)$ theorem
of Calder\'on--Zygmund operators in Subsection \ref{sec:CZO-2}.
Finally, in Subsection \ref{sec:CZO-3},
we present some concrete examples of Calder\'on--Zygmund operators
and corresponding $T(1)$ theorems.

\subsection{Extension of weakly defined operators}\label{sec:CZO-1}

In what follows, we use the \emph{symbol} $C_{\mathrm{c}}^\infty$ to denote
the set of all functions $f\in C^\infty$ on $\mathbb R^n$ with compact support.
We need the following definition of atoms. As in the case of molecules, despite dealing with $\fX$-valued distributions, our atoms will be scalar-valued, and they be equipped with $\fX$-valued coefficients when decomposing $\fX$-valued distributions.

\begin{definition}
Let $L,N\in(0,\infty)$.
A function $a_Q\in C_{\mathrm{c}}^\infty$ is called an \emph{$(L,N)$-atom} on a cube $Q$ if
\begin{enumerate}[\rm(i)]
\item $\operatorname{supp}a_Q\subset3Q$,
\item $\int_{\mathbb{R}^n}x^\gamma a_Q(x)\,dx=0$
for any $\gamma\in\mathbb{Z}_+^n$ with $|\gamma|\leq L$,
\item $|D^\gamma a_Q(x)|\leq|Q|^{-\frac12-\frac{|\gamma|}{n}}$
for any $x\in\mathbb{R}^n$ and $\gamma\in\mathbb{Z}_+^n$ with $|\gamma|\leq N$.
\end{enumerate}
\end{definition}

Furthermore, denote by $I_\fX$ the identity operator
on $\fX$. Then, for any linear operator $T$ on a \emph{space} of scalar-valued functions $D$ and any
$f:=\sum_{k=1}^{N}f_k\fx_k\in D\otimes\fX$,
we define
\begin{align*}
(T\otimes I_\fX)(f):=\sum_{k=1}^{N}T(f_k)\fx_k.
\end{align*}

Then we have the following extension criterion
for general linear operators.

\begin{proposition}\label{ext}
Let $\dot{\mathbf{A}}(\mathbb{R}^n;\fX)$
be a homogeneous $(d,e,f)$-discretizable space with some real
parameters $(d,e,f)$, and let $L,N\in(0,\infty)$.
\begin{enumerate}[\rm(i)]
\item\label{it:ext1} If $T\in\mathcal L(\mathscr S,\mathscr S')$
maps $(L,N)$-atoms to $\dot{\mathbf{a}}$-synthesis molecules,
then there exists an operator
$\widetilde T\in\mathcal L(\dot{\mathbf{A}}(\mathbb{R}^n;\fX))$
that agrees with $T\otimes I_{\fX}$ on $\mathscr{S}_\infty\otimes\fX$.
\item\label{it:ext2} If, in addition, there exists an operator $T_2\in\mathcal L(L^2)$
that agrees with $T$ on $\mathscr S$,
then $\widetilde T$ agrees with $T\otimes I_\fX$ on
$(\mathscr S\otimes\fX)\cap\dot{\mathbf{A}}(\mathbb{R}^n;\fX)$.
\end{enumerate}
\end{proposition}

\begin{proof}
First, we prove \eqref{it:ext1}.
Let $\varphi,\psi$ be two Littlewood--Paley functions that satisfy \eqref{eq:LPpair}. Then, from the proof of
\cite[Proposition 6.5]{BHYY:III}, we deduce that,
for any $R\in\mathscr{D}$, there exists
a sequence $\{a_P^R\}_{\{P\in\mathscr{D}:\ell(P)=\ell(R)\}}$,
where each $a_P^R$ is an $(L,N)$-atom on $P$, such that, for any $Q\in\mathscr{D}$,
\begin{equation}\label{eq:phiTpsi}
\left\langle\varphi_Q,T\psi_R\right\rangle
=\sum_{P\in\mathscr{D}}\big\langle\varphi_Q,T\big(a_P^R\big)\big\rangle b_{PR},
\end{equation}
where, for any $P,R\in\mathscr{Q}$,
$$
b_{PR}:=
\begin{cases}
\displaystyle{\Big[1+\frac{|x_R-x_P|}{\ell(R)}\Big]^{-D'}}&\text{if}\ \ell(P)=\ell(R),\\
0&\text{otherwise}
\end{cases}
$$
and $D'$ can be arbitrarily large. In particular, $(b_{PR})_{PR\in\mathscr D}$ is $(D',E',F')$-almost diagonal for any $D',E',F'$.
Moreover,
\begin{itemize}
  \item each $a^R_P$ is an $(L,N)$-atom on $P$, as we already observed,
  \item hence each $T(a^P_R)$ is an $\dot{\mathbf{a}}$-synthesis molecules by our assumption on the operator $T$,
  \item thus, by Lemma \ref{le-ad}\eqref{it:bQ},
\begin{equation*}
  \abs{\pair{\varphi_Q}{T\big(a_P^R\big)}}\lesssim B_{-E,-F,-D}(Q,P),
\end{equation*}
where the parameters $(D,E,F)$ satisfy
$D>\max\{n,d\}$, $E>e$, and $F>f$.
\end{itemize}
Substituting into \eqref{eq:phiTpsi}, we find that
\begin{equation*}
\begin{split}
  \abs{\left\langle\varphi_Q,T\psi_R\right\rangle}
  &\lesssim\sum_{P\in\mathscr D} B_{-E,-F,-D}(Q,P)B_{-E',-F',-D'}(P,R) \\
  &\lesssim B_{-E,-F,-D}(Q,R)\qquad\text{by Lemma \ref{ad-com}},
\end{split}
\end{equation*}
where the assumptions of Lemma \ref{ad-com} are easily satisfied, since $D>n$, and $D',E',F'$ can be chosen as large as needed.
Thus $B_T:=\{\langle\varphi_Q,T\psi_R\rangle
\}_{Q,R\in\mathscr{D}}$ is also $(D,E,F)$-almost diagonal,
and hence, by Definition \ref{def:seq}\eqref{def:seq-it2},
it is bounded on $\dot{\mathbf{a}}(\mathbb{R}^n;\fX)$.

Together with Definition \ref{def:seq}\eqref{def:seq-it1}, this
 implies that
 \begin{equation}\label{eq:tildeT=TBS}
  \widetilde{T}:=T_\psi\circ B_T\circ S_\varphi\in\mathcal{L}(\dot{\mathbf{A}}(\mathbb{R}^n;\fX)).
\end{equation}
Moreover,
\begin{equation}\label{eq:BS=ST}
\begin{split}
  B_T\circ S_\varphi=S_\varphi\circ T\quad &\text{on}\quad\mathscr{S}_\infty
  \quad\text{by \cite[Lemma 6.1]{BHYY:III}, and} \\
  &\text{on}\quad\mathscr{S}_\infty\otimes\fX
  \quad\text{by linearity}.
\end{split}
\end{equation}
On the other hand, we know from \cite[Lemma 2.1]{YY:10} that
\begin{equation}\label{eq:TS=I}
  T_\psi\circ S_\varphi=I\quad\text{on}\quad \mathscr S_\infty'.
\end{equation}
Combining these identities, it follows that, on $\mathscr{S}_\infty\otimes\fX$,
\begin{equation}\label{ext-e1}
\begin{split}
  \widetilde{T}
  &=T_\psi\circ B_T\circ S_\varphi\quad\text{by \eqref{eq:tildeT=TBS}} \\
  &=T_\psi\circ S_\varphi\circ T\quad\text{by \eqref{eq:BS=ST}} \\
  &=T\quad\text{by \eqref{eq:TS=I}}.
\end{split}
\end{equation}
This finishes the proof of \eqref{it:ext1}.

Moreover, applying the assumption $T_2\in\mathcal{L}(L^2)$
and \cite[(6.3)]{BHYY:III}, we conclude that
$S_\varphi\circ T=B_T\circ S_\varphi$ on $\mathscr{S}$,
which implies $S_\varphi\circ T
=B_T\otimes S_\varphi$ on $\mathscr{S}\otimes\fX$.
Consequently, from Definition \ref{def:seq}\eqref{def:seq-it1}, we further infer that
\eqref{ext-e1} also holds
on $(\mathscr{S}\otimes\fX)\cap\dot{\mathbf{A}}(\mathbb{R}^n;\fX)$,
which then completes the proof of \eqref{it:ext2} and hence Proposition
\ref{ext}.
\end{proof}

\subsection{$T(1)$ theorem for Calder\'on--Zygmund operators}\label{sec:CZO-2}

We now turn to the actual discussion of Calder\'on--Zygmund operators.
The following classical definitions are standard.
Let $\mathcal D :=C_{\mathrm{c}}^\infty $ with the usual inductive limit topology.
We denote by $\mathcal D' $ the space of
all continuous linear functionals on $\mathcal D$,
equipped with the weak-$*$ topology.
If $T\in\mathcal L(\mathscr S,\mathscr S')$,
then, by the well-known Schwartz kernel theorem,
we find that there exists $\mathcal K\in\mathscr S'(\mathbb R^n\times\mathbb R^n)$
such that
$\langle T\varphi,\psi\rangle
=\langle\mathcal K,\varphi\otimes\psi\rangle$ for all $\varphi,\psi\in\mathscr S$.
This $\mathcal K$ is called the \emph{Schwartz kernel} of $T$.

The \emph{adjoint operator} $T^*\in\mathcal L(\mathscr S,\mathscr S')$
  of $T$ is defined by setting
$\langle T\varphi,\psi\rangle=\langle T^*\psi,\varphi\rangle$
for all $\varphi,\psi\in\mathscr S$.
Then, from the definition, we easily infer that,
if $\mathcal K\in\mathscr S'(\mathbb R^n\times\mathbb R^n)$
is the Schwartz kernel of $T$, then
the Schwartz kernel $\mathcal{K}^*$ of $T^*$
satisfies $\langle K,\varphi\otimes\psi\rangle=\langle K^*,\psi\otimes\varphi\rangle$
for all $\varphi,\psi\in\mathscr{S}$.

Inspired by \cite{ftw88,Torres:91},
to establish the boundedness of
Calder\'on--Zygmund operators on Besov--Triebel--Lizorkin spaces,
we use the following kernel estimates from \cite[Definition 6.11]{BHYY:III}.

\begin{definition}\label{CZK}
Let $ E,F\in\mathbb{R}$, $T\in\mathcal L(\mathscr S,\mathscr S')$,
and $\mathcal K\in\mathscr S'(\mathbb R^n\times\mathbb R^n)$ be its Schwartz kernel.
We say that $T\in\operatorname{CZK}^0(E;F)$
if the restriction of $\mathcal K$ to
$\{(x,y)\in\mathbb R^n\times\mathbb R^n:\ x\neq y\}$
is a continuous function such that all derivatives below
exist as continuous functions and there exists a positive constant $C$ such that,
for any $\alpha\in\mathbb{Z}_+^n$ with $|\alpha|\leq\lfloor\!\lfloor E\rfloor\!\rfloor_+$
and for any $x,y\in\mathbb R^n$ with $x\neq y$,
\begin{equation}\label{CZK0}
|\partial_x^\alpha\mathcal K(x,y)|
\leq C|x-y|^{-n-|\alpha|},
\end{equation}
that, for any $\alpha\in\mathbb{Z}_+^n$ with $|\alpha|=\lfloor\!\lfloor E\rfloor\!\rfloor$
and for any $x,y,u\in\mathbb R^n$ with $|u|<\frac12|x-y|$,
\begin{equation}\label{CZKx}
|\partial_x^\alpha\mathcal K(x,y)-\partial_x^\alpha\mathcal K(x+u,y)|
\leq C|u|^{E^{**}}|x-y|^{-n-E},
\end{equation}
and that, for any $\alpha,\beta\in\mathbb{Z}_+^n$ with $|\alpha|\leq\lfloor\!\lfloor E\rfloor\!\rfloor_+$
and $|\beta|=\lfloor\!\lfloor F-|\alpha|\rfloor\!\rfloor$
and for any $x,y,v\in\mathbb R^n$ with $|v|<\frac12|x-y|$,
\begin{align}\label{CZKy}
&\left|\partial_x^\alpha\partial_y^\beta\mathcal K(x,y)
-\partial_x^\alpha\partial_y^\beta\mathcal K(x,y+v)\right|
\leq C|v|^{(F-|\alpha|)^{**}}|x-y|^{-n-|\alpha|-(F-|\alpha|)},
\end{align}
where $\lfloor\!\lfloor E\rfloor\!\rfloor$ and $E^{**}$ are as in \eqref{ceil} and \eqref{r**}.

We say that $T\in\operatorname{CZK}^1(E;F)$
if $T\in\operatorname{CZK}^0(E;F)$ and, in addition,
there exists a positive constant $C$ such that,
for any $\alpha,\beta\in\mathbb{Z}_+^n$ with
$|\alpha|=\lfloor\!\lfloor E\rfloor\!\rfloor$ and $|\beta|=\lfloor\!\lfloor F-E\rfloor\!\rfloor$
and for any $x,y,u,v\in\mathbb R^n$ with $|u|+|v|<\frac12|x-y|$,
\begin{align}\label{CZKxy}
\begin{split}
&\left|\partial_x^\alpha\partial_y^\beta\mathcal K(x,y)
-\partial_x^\alpha\partial_y^\beta\mathcal K(x+u,y)
-\partial_x^\alpha\partial_y^\beta\mathcal K(x,y+v)
+\partial_x^\alpha\partial_y^\beta\mathcal K(x+u,y+v)\right|\\
&\quad\leq C|u|^{E^{**}}|v|^{(F-E)^{**}}
|x-y|^{-n-E-(F-E)}.
\end{split}
\end{align}
\end{definition}

\begin{remark}
Let $T\in\mathcal L(\mathscr S,\mathscr S')$.
We have the following observations.
\begin{enumerate}[{\rm(i)}]
  \item If $F\leq 0<E$, then the conditions
  \eqref{CZKy} and \eqref{CZKxy} are both void and hence,
  for any $\sigma\in\{0,1\}$,
  $\operatorname{CZK}^\sigma(E;F)=\operatorname{CZK}^0(E;0)$.
  Hence, the set $\operatorname{CZK}^0(E;0)$
  means the set of operators with Calder\'on--Zygmund estimates of order
$E$ in the $x$ variable, with no assumptions on the $y$ variable.
Furthermore, if $E\in(0,1]$, then the operator
$T\in\operatorname{CZK}^0(E;0)$ satisfying
$T^*\in\operatorname{CZO}^0(F;0)$ is precisely the classical
Calder\'on--Zygmund operator presented in \cite[Definition 5.11]{d01}.
  \item For more examples of operators in $\operatorname{CZK}^\sigma(E;F)$,
  we refer to \cite[Remark 6.13]{BHYY:III}.
\end{enumerate}
\end{remark}

To study the boundedness
of Calder\'on--Zygmund type operators
on $\dot{\mathbf{A}}(\mathbb{R}^n;\fX)$,
we extend the domain of operators
introduced in Definition \ref{CZK}
to include polynomials (which are not in $\mathscr{S})$
in the following way:

\begin{remark}\label{rem:Tx-gamma-def}
Let $T\in\operatorname{CZK}^\sigma(E;F)$ with
$\sigma\in\{0,1\}$ and let $E,F\in\mathbb{R}$.
By \cite[Lemma 6.9]{BHYY:III}
(see also \cite[Lemma 2.2.12]{Torres:91}),
we find that, if $E\in(0,\infty)$, then
$T(x^\gamma)$ can be defined for any
$\gamma\in\mathbb{Z}_+^n$ such that
$|\gamma|\le\lfloor\!\lfloor E\rfloor\!\rfloor$;
if $F\in(0,\infty)$, then
$T^*(x^\theta)$ can be defined for any
$\theta\in\mathbb{Z}_+^n$ such that
$|\gamma|\le\lfloor\!\lfloor F\rfloor\!\rfloor$.
\end{remark}

Then we introduce the
Calder\'on--Zygmund conditions
following \cite[Definition 6.17]{BHYY:III}.

\begin{definition}\label{def:cz}
Let $\sigma\in\{0, 1\}$, $ E,F\in\mathbb{R}$, and $G,H\in\mathbb{Z}$
satisfy $G\le \lfloor\!\lfloor E\rfloor\!\rfloor$ and
$H\le \lfloor\!\lfloor F\rfloor\!\rfloor$.
We say that $T\in\operatorname{CZO}^\sigma(E,F,G,H)$
if $T\in\mathcal L(\mathscr S,\mathscr S')$ and
\begin{enumerate}[\rm(i)]
\item $T$ satisfies the \emph{weak boundedness property}
(which is denoted by $T\in\operatorname{WBP}$); that is,
for any bounded subset $\mathcal B$ of $\mathcal D$,
there exists a positive constant $C=C(\mathcal B)$ such that,
for any $\varphi,\eta\in\mathcal B$, $h\in\mathbb R^n$, and $r\in(0,\infty)$,
\begin{equation*}
\Big|\Big\langle T \Big[\varphi\Big(\frac{\cdot-h}{r}\Big)\Big],
\eta \Big(\frac{\cdot-h}{r}\Big)\Big\rangle\Big|\leq Cr^n,
\end{equation*}
\item $T\in\operatorname{CZK}^\sigma(E;F)$,
\item\label{it:Ty}
$T(y^\gamma)=0$ for all $\gamma\in\mathbb{Z}_+^n$ with $|\gamma|\leq G$, (This is void if $G<0$.)
\item\label{it:Tsy} $T^*(x^\theta)=0$ for all $\theta\in\mathbb{Z}_+^n$ with $|\theta|\leq H$. (This is void if $H<0$.)
\end{enumerate}
\end{definition}

\begin{remark}
Definition \ref{def:cz} is identical to \cite[Definition 6.17]{BHYY:III} except for the fact that we have added the parameter restrictions $G\le \lfloor\!\lfloor E\rfloor\!\rfloor$ and $H\le \lfloor\!\lfloor F\rfloor\!\rfloor$, which are needed to ensure that $T(y^\gamma)$ and $T^*(x^\theta)$ in \eqref{it:Ty} and \eqref{it:Tsy} are well defined for all $\gamma,\theta$ appearing there according to Remark \ref{rem:Tx-gamma-def}. We point out that the same restriction should have been made in \cite[Definition 6.17]{BHYY:III} as well. However, its omission is a relatively harmless oversight, since this additional restriction is valid in all concrete instances where \cite[Definition 6.17]{BHYY:III} is used in \cite{BHYY:III}. Notably, this applies to \cite[Proposition 6.19]{BHYY:III}, which is restated below as Lemma \ref{T1EFGH}.
\end{remark}

We now state the $T(1)$ theorem as follows.

\begin{theorem}\label{T1 BF}
Let $\dot{\mathbf{A}}(\mathbb{R}^n;\fX)$
be a homogeneous $(d,e,f)$-discretizable space with some real
parameters $(d,e,f)$.
Suppose $T\in\operatorname{CZO}^\sigma(E,F,G,H)$
in the sense of Definition \ref{def:cz}, where
$\sigma\in\{0,1\}$ and the parameters $(E,F,G,H)$ satisfy
\begin{equation*}
\begin{split}
 &\sigma\geq\mathbf{1}_{[0,\infty)}\big(e-\frac{n}{2}\big),\quad
E>\big(e-\frac n2\big)_+,\quad
F>\max\big\{n,d,f+\frac n2\big\}-n,\\
&G\geq\big\lfloor e-\frac n2\big\rfloor_+,\quad \text{and}\quad
H\geq\big\lfloor f-\frac n2\big\rfloor.
\end{split}
\end{equation*}

Then the following two statements hold:
\begin{enumerate}[\rm(i)]
\item\label{it:T1 BF1} There exists an operator
$\widetilde T\in\mathcal L(\dot{\mathbf{A}}(\mathbb{R}^n;\fX))$
that agrees with $T\otimes I_\fX$ on $\mathscr{S}_\infty\otimes\fX$.
\item\label{it:T1 BF2} If further $H\geq0$,
then there exists $\widetilde T
\in\mathcal L(\dot{\mathbf{A}}(\mathbb{R}^n;\fX))$
that agrees with $T\otimes I_\fX$ on
$(\mathscr S\otimes\fX)\cap\dot{\mathbf{A}}(\mathbb{R}^n;\fX)$.
\end{enumerate}
\end{theorem}

We need the following estimate from \cite{BHYY:III}.

\begin{lemma}[{\cite[Proposition 6.19]{BHYY:III}}]\label{T1EFGH}
Suppose that $T\in\operatorname{CZO}^\sigma(E,F,G,H)$
in the sense of Definition \ref{def:cz}.
Then $T$ maps sufficiently regular atoms to $(K,L,M,N)$-molecules
provided that the following conditions are satisfied:
\begin{align*}
\sigma&\geq\mathbf{1}_{(0,\infty)}(N),\quad
\begin{cases}
E\geq N,\\
E>\lfloor N\rfloor_+,
\end{cases}\quad
\begin{cases}
F\geq (K\vee M)-n,\\
F>\lfloor L\rfloor,
\end{cases}\quad
G\geq\lfloor N\rfloor_+,\quad\text{and}\quad
H\geq\lfloor L\rfloor.
\end{align*}
\end{lemma}

\begin{proof}[Proof of Theorem \ref{T1 BF}]
By the assumptions on $\sigma$, $E$, $F$, $G$, and $H$,
we can choose $(K,L,M,N)$ satisfying:
\begin{align*}
K\in\big(\max\big\{n,d,f+\frac n2\big\},F+n\big),
\quad L\in\big(f-\frac n2,
\big\lceil\!\big\lceil f-\frac n2\big\rceil\!\big\rceil\big),
\quad M\in(\max\{d,n\},F+n),
\end{align*}
and
\begin{align*}
\left\{
\begin{aligned}
&\displaystyle{N:=0\in\big(e-\frac n2,E\big)}
\hspace{1.47cm}\text{if}\
\displaystyle{e-\frac n2<0},\\
&\displaystyle{N\in\big(e-\frac n2,
\big\lceil\!\big\lceil e-\frac n2\big\rceil\!\big\rceil\wedge E\big)}
\hspace{.5cm}\text{if}\ \displaystyle{e-\frac n2\geq 0}.
\end{aligned}\right.
\end{align*}
Note that the lower bounds in the choice of these parameters ensure that $(K,L,M,N)$-molecules are $\dot{\mathbf{a}}$-synthesis molecules in the sense of Definition \ref{df-a-syan}\eqref{defit:a-sy}.

From the choice of the parameters and some basic calculations, we infer that
$\sigma$, $(E,F,G,H)$, and $(K,L,M,N)$ satisfy the
assumptions on Lemma \ref{T1EFGH}.
Therefore, using Lemma \ref{T1EFGH}, we find that
$T$ maps sufficiently regular atoms to
$(K,L,M,N)$-molecules, which are
$\dot{\mathbf{a}}$-synthesis molecules by what we observed.
Thus, Proposition \ref{ext}\eqref{it:ext1} completes the proof of Theorem \ref{T1 BF}\eqref{it:T1 BF1}.

Now, we prove case \eqref{it:T1 BF2} of the theorem.
From the assumptions on $\sigma$ and $(E,F,G,H)$,
we deduce that $E,F>0$ and $G,H\geq 0$,
which further implies
\begin{equation*}
  T\in\operatorname{WPB}\cap\operatorname{CZK}^0(E;0),\quad
  T^*\in\operatorname{CZK}^0(F;0),\quad \text{and}\quad T(1)=0=T^*(1)
\end{equation*}
for these $E$ and $F$.
Therefore, by the well-known $T(1)$ theorem of David and Journ\'e \cite{dj84},
we conclude that there exists $T_2\in\mathcal L(L^2)$
that agrees with $T$ on $\mathscr S$.
Thus, the assumptions of Proposition \ref{ext}\eqref{it:ext2} are satisfied, and the said proposition
 implies the existence of the operator $\widetilde T$ with the additional property as desired.
This finishes the proof of Theorem \ref{T1 BF}.
\end{proof}

By considering special cases of $\dot{\mathbf{A}}(\mathbb{R}^n;\fX)$,
we obtain the following $T(1)$ theorem for the Besov and Triebel--Lizorkin spaces of Definition \ref{def:space}.
This is a direct consequence of Lemma \ref{lem:seq} and Example \ref{ex:mole}, and we omit the straightforward details.

\begin{corollary}\label{cor:T1}
Let $p\in(0,\infty)$, $q\in(0,\infty]$,
and  $s\in\mathbb{R}$.
\begin{enumerate}[{\rm(i)}]
  \item Let $\rho=\{\rho_{Q}\}_{Q\in\mathscr{D}}$
  be a family of $(u,a,b,c)$-norms with
  $u\in(0,1]$ and $a,b,c\in[0,\infty)$.
  If $T\in\operatorname{CZO}^\sigma(E,F,G,H)$
in the sense of Definition \ref{def:cz}, where
$\sigma\in\{0,1\}$ and the parameters $(E,F,G,H)$ satisfy
\begin{equation*}
\begin{split}
  &\sigma\geq\mathbf{1}_{[0,\infty)}(s+a),\quad
  E>(s+a)_+,\quad
  F>J^{(u)}-n+\max\{c,b-s\},\\
  &G\geq\lfloor s+a\rfloor_+,\quad \text{and}\quad
  H\geq\big\lfloor J^{(u)}-n-s+b\big\rfloor,
\end{split}
\end{equation*}
where $J^{(u)}$ is as in \eqref{eq:Jr},
  then Theorem \ref{T1 BF} applies to $\dot{A}^s_{p,q}(\rho)$ in place of $\dot{\mathbf A}(\R^n;\fX)$.

  \item\label{cor:T1-V} Let $V\in\A_r$ with $r\in[p,\infty)$ has doubling
  dimension $\beta\in[n,\infty)$, and let $q\in(0,r+\varepsilon_V)$
  when $A=F$, where $\varepsilon_V$ is
  the optimal reverse H\"older lifting index of $V$.
  If $T\in\operatorname{CZO}^\sigma(E,F,G,H)$
in the sense of Definition \ref{def:cz}, where
$\sigma\in\{0,1\}$ and the parameters $(E,F,G,H)$ satisfy
\begin{equation}\label{eq:T1-V}
\begin{split}
&\sigma\geq\mathbf{1}_{[0,\infty)}(s),\quad
E>s_+,\quad
F>J-n+s_-+\frac{\beta-n}{r},\\
&G\geq\lfloor s\rfloor_+,\quad \text{and}\quad
H\geq\Big\lfloor J-n-s+\frac{\beta-n}{r}\Big\rfloor,
\end{split}
\end{equation}
where $J$ is as in \eqref{eq:J},
  then Theorem \ref{T1 BF} applies to $\dot{A}^s_{p,q}(V)=\dot A^s_{p,q}(\rho_{\aveL^r(\mathscr D,V)})$ in place of $\dot{\mathbf A}(\R^n;\fX)$.
\end{enumerate}
\end{corollary}

\begin{remark}
Even in the unweighted case $V\equiv I$, the result of Corollary \ref{cor:T1} on the boundedness of general Calder\'on--Zygmund operators on $\fX$-valued Besov and Triebel--Lizorkin spaces $\dot A^s_{p,q}(\R^n;\fX)$ seems to be largely new. In this direction, we are only aware the results of Kaiser \cite{Kaiser:09}, which deals with the Besov spaces $\dot B^s_{p,q}(\R^n;\fX)$ in the range $p,q\in[1,\infty]$. Noting that $J=\beta=n$ in this case, the conditions \eqref{eq:T1-V} reduce to
\begin{equation*}
  \sigma\geq\one_{[0,\infty)}(s),\quad E>s_+,\quad F>s_-,\quad G\geq\floor{s}_+,\quad H\geq\floor{-s},
\end{equation*}
and we note in particular that, for spaces of positive smoothness $s>0$, the condition on $H$ is trivial, i.e., there exist no conditions of the type $T^*(x^\theta)=0$ in this case. These assumptions essentially agree with those of \cite[Theorem 3.1]{Kaiser:09} in the case of $\fX$-valued extensions of scalar-valued operators $T\in\mathcal L(\mathscr S_\infty,\mathscr S_\infty')$ that we consider. However, \cite[Theorem 3.1]{Kaiser:09} deals with more general operators $T\in\mathcal L(\mathscr S_\infty(\R^n;\fX),\mathscr S_\infty'(\R^n;\fY))$ with $\mathcal L(\fX,\fY)$-valued kernels for two Banach spaces $\fX,\fY$. Hence, while Corollary \ref{cor:T1} goes beyond \cite[Theorem 3.1]{Kaiser:09} in several directions (a larger range of indices $p,q$, both Besov and Triebel--Lizorkin spaces, and operator-weighted versions), the two results are not strictly comparable.
\end{remark}

\subsection{Examples}\label{sec:CZO-3}

When specialized to $\fX=\C^m$ and matrix weights, Corollary \ref{cor:T1} falls slightly short of reproducing the best available result in this setting, which is the recent \cite[Theorem 6.18]{BHYY:III}. Namely, compared to the related condition \cite[(6.20)]{BHYY:III}, the present \eqref{eq:T1-V} has the additional nonnegative term $\frac{\beta-n}{r}$ in the lower bounds for $F$ and $H$. Nevertheless, Corollary \ref{cor:T1} is strong enough to recover the main results on Calder\'on--Zygmund operators on matrix-weighted Besov and Triebel--Lizorkin spaces known before \cite{BHYY:III}, including the still quite recent \cite{FR:21}. To facilitate this comparison, we recall the following \cite[Definition 5.6]{FR:21}:

\begin{definition}\label{def-Lsmooth}
Let $L\in\mathbb{N}$.
A operator $T: \mathscr{S}\to\mathscr{S}'$
is called an \emph{$L$-smooth Calder\'on--Zygmund operator}
if
\begin{enumerate}[{\rm(i)}]
  \item there exists $K\in\mathscr{S}'$
such that, for any $f\in
\mathcal{S}$, $T(f):=K\ast f$,
  \item $K$ coincides with a
  function belonging to $C^{L}(\mathbb{R}^n\setminus\{{\bf 0}\})$
  in the sense that, for any given $f\in\mathscr{S}$ with
  compact support and for any $x\notin\operatorname{supp}f$,
  $$
  T(f)(x)=\int_{\R^n}K(x-y)f(y)\,dy,
  $$
  \item there exists a positive
  constant $C$ such that,
  for any $\gamma\in\mathbb{Z}_+^n$ with $|\gamma|\leq L$
  and for any $x\in\mathbb{R}^n\setminus\{{\bf 0}\}$,
  \begin{equation}\label{con-T-size}
  \left|\partial^{\gamma}K(x)\right|
  \leq\frac{C}{|x|^{n+|\gamma|}},
  \end{equation}
  \item for any $0 < R_1 < R_2 < \infty$,
  $$\int_{R_1 < |x| < R_2} K(x) \, dx = 0.$$
\end{enumerate}
\end{definition}

In \cite{FR:21,Rou:03}, the authors first present results for general Calder\'on--Zygmund operators in the spirit of Corollary \ref{cor:T1}, and then give a separate treatment for convolution-type operators. However, the need for this latter seems to come from a slight misunderstanding in \cite[Remark 9.12]{Rou:03}: ``Note that the condition $T^*(y^\gamma) = 0$, $\abs{\gamma}\leq N$, can be very restrictive; for example, the Hilbert transform does not satisfy this condition for $\abs{\gamma}>0$.'' This is repeated in \cite[pp.~526--527]{FR:21}.

However, if one interprets $T^*(x^\gamma)$ in the usual way as a distribution modulo polynomials of order $\abs{\gamma}$, then we do have $T^*(x^\gamma)=0$ {\em for all multi-indices $\gamma$} for the Hilbert, the Riesz, and other classical convolution transforms with sufficiently regular kernels. Indeed, by \cite[Proposition 2.2.17]{Torres:91},
every $L$-smooth Calder\'on--Zygmund operator $T$ satisfies
$T(x^\gamma)=0$ (resp.\, $T^*(x^\theta)=0$)
as a distribution on $\mathscr D_{|\gamma|}$
(resp.\,$\mathscr D_{|\theta|}$), where
 \begin{equation*}
  \mathscr D_{\ell}=\Big\{\varphi\in\mathscr D:\int_{\mathbb{R}^n} x^{\gamma}\varphi(x)\,dx=0,\ \forall\ \abs{\gamma}\leq\ell\Big\},\quad \ell\in\N
\end{equation*}
for all $\gamma,\theta\in\mathbb{Z}_+^n$ with $|\gamma|,|\theta|\le L-1$.

This fact seems to have been partially overlooked in \cite{Rou:03,FR:21}. Instead of needing a separate treatment for operators of Definition \ref{def-Lsmooth}, these can be readily dealt with as a special case of the results above:

\begin{corollary}\label{cor:L-smooth}
Assume $L\in\mathbb{N}$ and $T$ is an
$L$-smooth Calder\'on--Zygmund operator.
Then $T\in\operatorname{CZO}^\sigma(L,L,L-1,L-1)$
for $\sigma\in\{0,1\}$. As a consequence,
for every $p\in(0,\infty)$, $q\in(0,\infty]$,
and $s\in\mathbb{R}$, the following two
results hold.
\begin{enumerate}[{\rm(i)}]
  \item\label{it:Lrho} Let $\rho=\{\rho_{Q}\}_{Q\in\mathscr{D}}$
  be a family of $(u,a,b,c)$-norms with
  $u\in(0,1]$ and $a,b,c\in[0,\infty)$.
  If
  \begin{align*}
  L>\max\big\{(s+a)_+,J^{(u)}-n+\max\{c,b-s\}\big\},
  \end{align*}
  where $J^{(u)}$ is as in \eqref{eq:Jr}, then
  there exists $\widetilde T
\in\mathcal L(\dot{A}^s_{p,q}(\rho))$
that agrees with $T\otimes I_\fX$ on
$(\mathscr S\otimes\fX)\cap\dot{A}^s_{p,q}(\rho)$.
  \item\label{it:LV} Let $V\in\A_r$ with $r\in[p,\infty)$ and the doubling
  dimension $\beta\in[n,\infty)$, and let $q\in(0,r+\varepsilon_V)$
  when $A=F$, where $\varepsilon_V$ is
  the optimal reverse H\"older lifting index of $V$.
  If
  \begin{align*}
  L>\max\Big\{s_+,J-n+s_-+\frac{\beta-n}{r}\Big\},
  \end{align*}
  where $J$ is as in \eqref{eq:J},
  then there exists $\widetilde T
\in\mathcal L(\dot{A}^s_{p,q}(V))$
that agrees with $T\otimes I_\fX$ on
$(\mathscr S\otimes\fX)\cap\dot{A}^s_{p,q}(V)$.
\end{enumerate}
\end{corollary}

\begin{proof}
If we show $T\in\operatorname{CZO}^\sigma(L,L,L-1,L-1)$
with $\sigma\in\{0,1\}$, then
the remaining results follow from Corollary \ref{cor:T1}
directly.
Hence, we now turn to prove  $T\in\operatorname{CZO}^\sigma(L,L,L-1,L-1)$.
Indeed, by the discussion preceding the present corollary,
the cancelation conditions \eqref{it:Ty}
and \eqref{it:Tsy} of Defintition\eqref{def:cz} hold with
$G=H=L-1$.
Suppose $K$ is the kernel in the sense of Definition \ref{def-Lsmooth}.
Applying \eqref{con-T-size} and the mean value theorem,
the size conditions \eqref{CZK0}, \eqref{CZKx}, and \eqref{CZKy}
hold with $\mathcal{K}(x,y):=K(x,y)$ and $E=F=L$ and, if we take $E=F$, then
condition \eqref{CZKxy} is void. This means that
$T\in\operatorname{CZK}^\sigma(L;L)$ for $\sigma\in\{0,1\}$.
Finally, as pointed out in \cite{m87} (see also \cite[p.\,35]{s70}),
$T\in \operatorname{WBP}$.
Therefore, $T\in\operatorname{CZO}^\sigma(L,L,L-1,L-1)$,
$\sigma\in\{0,1\}$, in the sense of Definition \ref{def:cz}.
This then finishes the proof of Corollary \ref{cor:L-smooth}.
\end{proof}

\begin{remark}
Corollary \ref{cor:L-smooth} has several applications as follows.
\begin{enumerate}[{\rm(i)}]
  \item Since the Hilbert transform $\mathcal{H}$
  and every Riesz transform $\mathcal{R}_j$, $j\in\{1,\ldots,n\}$,
  satisfy Definition \ref{def-Lsmooth} for arbitrarily large $L\in\mathbb{N}$,
the boundedness conclusion of Corollary \ref{cor:L-smooth} holds automatically for
$\mathcal{H}$ and $\{R_j\}_{j=1}^n$.
Different from the proof of Corollary \ref{cor:Hilbert},
we therefore provide an alternative proof of the boundedness of
the Hilbert transform on both $\dot{A}^s_{p,q}(\rho)$ and $\dot{A}^s_{p,q}(V)$.
  \item If $p\in(0,\infty)$, $q\in(0,\infty]$, $s\in\mathbb{R}$, and
  $V\in\mathscr{A}_p$ has doubling dimension $\beta\in[n,\infty)$,
  and if we let $\rho:=\{\rho_{\aveL^p(Q,V)}\}_{Q\in\D}$, then,
  from Remark \ref{rem:uabc}\eqref{rem:uabc-it1}, we infer that
  the ranges of $L$ in \eqref{it:Lrho} and \eqref{it:LV} of
  Corollary \ref{cor:L-smooth} become, respectively,
  \begin{align}\label{eq:Lrho}
  L>\max\Bigg\{\Big(s+\frac np\Big)_+,J-n+\max\Big\{\frac{\beta}{p},
  \frac{\beta-n}{p}-s\Big\}\Bigg\}
  \end{align}
  and
  \begin{align}\label{eq:LV}
  L>\max\Big\{s_+,J-n+s_-+\frac{\beta-n}{p}\Big\}.
  \end{align}
  These ranges coincide only if $s\le-\frac np$. When
  $s>-\frac{n}{p}$, \eqref{eq:LV} is strictly  larger than \eqref{eq:Lrho}.
  On the other hand, the range \eqref{eq:Lrho} coincides with that in \cite[Theorem 5.8]{FR:21}.
  Hence, in the Triebel--Lizorkin case, by the equivalence between $\dot{F}^s_{p,q}(\rho)$ and
  $\dot{F}^s_{p,q}(V)$ when $\fX=\mathbb{C}^m$, Corollary
  \ref{cor:L-smooth}\eqref{it:Lrho} coincides with \cite[Theorem 5.8]{FR:21};
  furthermore, if $q\in(0,p+\varepsilon_V)$ (the optimal
  reverse H\"older lifting index of $V$), then
  Corollary \ref{cor:L-smooth}\eqref{it:LV} improves \cite[Theorem 5.8]{FR:21}.
  Finally, in the Besov case, Corollary \ref{cor:L-smooth}\eqref{it:LV}
  extends \cite[Corollary 9.15]{Rou:03} and \cite[Theorem 4.2(2)]{FR:04}
  from $s\in[0,1)$ to arbitrary $s\in\mathbb{R}$.
\end{enumerate}
\end{remark}

Moreover, Frazier and Roudenko \cite{Rou:03,FR:04}
also consider a class of \emph{non-convolution Calder\'on--Zygmund operators}
$\operatorname{CZO}(N+\varepsilon)$ with $N\in\mathbb{Z}_+$ and $\varepsilon\in(0,1]$,
which is precisely the class $\operatorname{CZK}^0(N+\varepsilon;N+\varepsilon)$
in the sense of Definition \ref{CZK} (see \cite[p.\,303]{Rou:03} for details).
Hence, applying Corollary \ref{cor:T1}\eqref{cor:T1-V}
and the discussion preceding Corollary \ref{cor:L-smooth},
we immediately obtain the following
$T(1)$ theorem for $\operatorname{CZO}(N+\varepsilon)$
on $\dot{A}^s_{p,q}(V)$; we omit details.

\begin{corollary}\label{cor:N}
Let $N\in\mathbb{Z}_+$, $\varepsilon\in(0,1]$,
and $T\in\operatorname{CZO}(N+\varepsilon)$.
Suppose $p\in(0,\infty)$, $q\in(0,\infty]$,
and $s\in\mathbb{R}$.
Let $V\in\A_r$ with $r\in[p,\infty)$ and the doubling
  dimension $\beta\in[n,\infty)$, and let $q\in(0,r+\varepsilon_V)$
  when $A=F$, where $\varepsilon_V$ is
  the optimal reverse H\"older lifting index of $V$.
If $N$, $\varepsilon$, and $T$ satisfy
\begin{align*}
\begin{split}
&T\in\operatorname{WBP},\quad
N+\varepsilon>\max\Big\{s_+,J-n+s_-+\frac{\beta-n}{r}\Big\},\\
& N\ge\max\Big\{\lfloor s\rfloor_+,\Big\lfloor
J-n-s+\frac{\beta-n}{r}\Big\rfloor\Big\},
\quad T(x^\gamma)=0\quad\forall\, |\gamma|\le\lfloor s\rfloor_+,
\end{split}
\end{align*}
and $T^*(x^\theta)=0$ for all $|\theta|\le\lfloor
J-n-s+\frac{\beta-n}{r}\rfloor$ (this assumption is void if the upper bound is negative),
then
\begin{enumerate}[\rm(i)]
\item there exists an operator
$\widetilde T\in\mathcal L(\dot{A}^s_{p,q}(V))$
that agrees with $T\otimes I_\fX$ on $\mathscr{S}_\infty\otimes\fX$,

\item if further $\lfloor
J-n-s+\frac{\beta-n}{r}\rfloor\ge 0$,
then there exists $\widetilde T
\in\mathcal L(\dot{A}^s_{p,q}(V))$
that agrees with $T\otimes I_\fX$ on
$(\mathscr S\otimes\fX)\cap\dot{A}^s_{p,q}(V)$.
\end{enumerate}
\end{corollary}

\begin{remark}
If, in Corollary \ref{cor:N}, $r=p\in(0,1]$ (resp.\,$\in[1,\infty)$)
and $s\in[0,1)$, then
the assumptions on $N$, $\varepsilon$, and $T$
coincide with those of \cite[Theorem 4.2(1)]{FR:04}
(resp.\,\cite[Corollaries 9.10 and 9.11]{Rou:03});
hence, we extend their results from Besov spaces and
$\fX=\mathbb{C}^m$ to general
$\dot{A}^s_{p,q}(V)\subset\mathscr{S}_\infty'(\mathbb{R}^n;\fX)$.
\end{remark}

\section{Trace and extension operators}

In this section, we study the trace and the extension operators.
Due to the specific structure of these two classes of operators,
it seems difficult to obtain satisfactory results within the
general framework $\dot{\mathbf{A}}(\mathbb{R}^n;\fX)$.
Hence, we focus on the pointwise and the average-weighted spaces
in this section directly.

\subsection{Sequence spaces}

We first consider a toy model in the sequence spaces $\dot a^s_{p,q}$. The underlying key estimate is the following lemma in the unweighted, scalar-valued case. It is implicitly contained in the proofs in \cite[Section 5]{BHYY:III}, but we detail it here for convenience of reference.

\begin{lemma}\label{lem:bfTrace}
Given a sequence $u=\{u_I\}_{I\in\mathscr D(\R^{n-1})}$ and $k\in\Z$, let $\hat u^{(k)}=\{\hat u^{(k)}_Q\}_{Q\in\mathscr D(\R^{n})}$ be defined by
\begin{equation}\label{eq:hat-u-k}
  \hat u^{(k)}_Q:=\begin{cases} \ell(I)^{\frac12}u_I &
  \text{if}\quad Q=Q(I,k):=I\times[k\ell(I),(k+1)\ell(I)) \\ 0 & \text{else}.\end{cases}
\end{equation}
Then
\begin{equation}\label{eq:bfTrace}
  \Norm{\hat u^{(k)}}{\dot a^{s}_{p,q}(\R^n)}\approx\Norm{u}{\dot b^{s-\frac1p}_{p,r}(\R^{n-1})},
\end{equation}
where
\begin{equation}\label{eq:r=q,p}
  r=\begin{cases} q, & \text{in the Besov case }(a=b), \\
  p, & \text{in the Triebel--Lizorkin case }(a=f)\end{cases}
\end{equation}
and the positive equivalence constants are independent of $u$ and $k$.
\end{lemma}

\begin{proof}
We prove the Besov and the Triebel--Lizorkin cases separately. In the Besov case, we will obtain an exact identity
\begin{equation*}
  \Norm{\hat u^{(k)}}{\dot b^{s}_{p,q}(\R^n)}=\Norm{u}{\dot b^{s-\frac1p}_{p,q}(\R^{n-1})},
\end{equation*}
whose proof is a direct computation:
\begin{equation}\label{eq:bTrace-pf}
\begin{split}
  \Norm{u}{\dot b^{s-\frac1p}_{p,q}(\R^{n-1})}^q
  &=\sum_{j\in\Z}2^{j(s-\frac1p)q} 2^{j(n-1)(\frac12-\frac1p)q}
    \Big(\sum_{I\in\mathscr D_j(\R^{n-1})}\abs{u_I}^p\Big)^{\frac qp} \\
  &=\sum_{j\in\Z}2^{jsq} 2^{jn(\frac12-\frac1p)q}
    \Big(\sum_{I\in\mathscr D_j(\R^{n-1})}\abs{\ell(I)^{\frac12}u_I}^p\Big)^{\frac qp}     \\
  &=\sum_{j\in\Z}2^{jsq} 2^{jn(\frac12-\frac1p)q}
    \Big(\sum_{Q\in\mathscr D_j(\R^n)}\abs{\hat u_Q^{(k)}}^p\Big)^{\frac qp}
    =\Norm{\hat u^{(k)}}{\dot b^s_{p,q}(\R^n)}^q.
\end{split}
\end{equation}

We then turn to the Triebel--Lizorkin case of \eqref{eq:bfTrace}, restated for concreteness as
\begin{equation}\label{eq:fTrace}
  \Norm{\hat u^{(k)}}{\dot f^{s}_{p,q}(\R^n)}\approx
  \Norm{u}{\dot b^{s-\frac1p}_{p,p}(\R^{n-1})}.
\end{equation}
To prove this, we define
\begin{equation*}
  E(I,k):=I\times[(k+\tfrac13)\ell(I),(k+\tfrac23)\ell(I))\subset Q(I,k)\subset\R^n,\qquad
  I\in\mathscr D(\R^{n-1}),\quad k\in\Z,
\end{equation*}
which obviously satisfies $\abs{E(I,k)}=\frac13\abs{Q(I,k)}$.
We then apply case $p=q$ of the previous computation \eqref{eq:bTrace-pf} to find that
\begin{equation*}
\begin{split}
  \Norm{u}{\dot b^{s-\frac1p}_{p,p}(\R^{n-1})}^p
  &=\sum_{j\in\Z}2^{jsp} 2^{jn(\frac12-\frac1p)p}
    \sum_{I\in\mathscr D_j(\R^{n-1})}\abs{\ell(I)^{\frac12}u_I}^p \\
  &=3\int_{\R^n}\sum_{j\in\Z}2^{jsp} 2^{jn(\frac12-\frac1p)p}
    \sum_{I\in\mathscr D_j(\R^{n-1})}\abs{\ell(I)^{\frac12}u_I}^p \frac{\one_{E(I,k)}(x)}{\abs{Q(I,k)}}\,dx \\
  &=3\int_{\R^n}\sum_{j\in\Z}2^{jsp}
    \sum_{I\in\mathscr D_j(\R^{n-1})} \abs{\ell(I)^{\frac12}u_I}^p
    \frac{\one_{E(I,k)}(x)}{\abs{Q(I,k)}^{\frac p2}}\,dx.
\end{split}
\end{equation*}
As observed in \cite[(5.25) et seq.]{BHYY:III}, the sets $\{E(I,k)\}_{I\in\mathscr D(\R^{n-1})}$ are pairwise disjoint, so there is at most one non-zero term in the previous double sum for each fixed $x\in\R^n$. Thus
\begin{equation}\label{eq:fTrace-step}
\begin{split}
  \Norm{u}{\dot b^{s-\frac1p}_{p,p}(\R^{n-1})}^p
  &=3\int_{\R^n}\Big(\sum_{j\in\Z} \Babs{2^{js}\sum_{I\in\mathscr D_j(\R^{n-1})}
      \ell(I)^{\frac12}u_I \frac{\one_{E(I,k)}}{\abs{Q(I,k)}^{\frac12}}}^q\Big)^{\frac pq}\,dx \\
  &=3\int_{\R^n}\Big(\sum_{j\in\Z} \Babs{2^{js}\sum_{Q\in\mathscr D_j(\R^{n})}
      \hat u_Q^{(k)} \frac{\one_{E_Q}}{\abs{Q}^{\frac12}}}^q\Big)^{\frac pq}\,dx,
\end{split}
\end{equation}
where
\begin{equation*}
  E_Q:=Q'\times[(h+\frac13)\ell(Q),(h+\frac23)\ell(Q))\quad\text{for all}\quad
  Q=Q'\times[h\ell(Q),(h+1)\ell(Q))\in\mathscr D(\R^n).
\end{equation*}
Since $E_Q\subset Q$ and $\abs{E_Q}=\frac13\abs{Q}\approx\abs{Q}$ for all $Q\in\mathscr D(\R^n)$, it follows from \cite[Proposition 2.7]{FJ:90} that
\begin{equation*}
  \operatorname{RHS}\eqref{eq:fTrace-step}\approx\Norm{\hat u_Q^{(k)}}{\dot f^s_{p,q}(\R^n)}^p.
\end{equation*}
Substituting this into \eqref{eq:fTrace-step}, we complete the proof of \eqref{eq:fTrace} and hence Lemma \ref{lem:bfTrace}.
\end{proof}

\begin{lemma}\label{lem:bfTrace-rho}
Let $p\in(0,\infty)$, $q\in(0,\infty]$,
and $d=\{d_{I}\}_{I\in\mathscr{D}(\mathbb{R}^{n-1})}$
and $\rho=\{\rho_Q\}_{Q\in\mathscr{D}(\mathbb{R}^n)}$
be two families of quasi-norms on a Banach space $\fX$.
For each $k\in\Z$, the following statements are equivalent:
\begin{enumerate}[\rm(i)]
  \item\label{it:seq-test} For all $I\in\mathscr{D}(\mathbb{R}^{n-1})$ and $\fx\in\fX$,
\begin{equation}\label{eq:sec-test}
  d_I(\fx)\lesssim\rho_{Q(I,k)}(\fx)
\end{equation}
with the implicit positive constant independent of $I$ and $\fx$.
  \item\label{it:seq-trace} For all $t=\{t_I\}_{I\in\mathscr D(\R^{n-1})}\in\fX^{\mathscr D(\R^{n-1})}$ and $\hat t^{(k)}$ as in \eqref{eq:hat-u-k}
\begin{equation}\label{eq:sec-trace}
   \Norm{t}{\dot b^{s-\frac1p}_{p,r}(d,\R^{n-1})}\lesssim \Norm{\hat t^{(k)}}{\dot a^{s}_{p,q}(\rho,\R^n)}
\end{equation}
with $r$ as in \eqref{eq:r=q,p}
and the implicit positive constant independent of $t$.
\end{enumerate}
A similar equivalence holds with both instances of ``$\lesssim$'' replaced by ``$\gtrsim$''.
\end{lemma}

\begin{proof}
Let $u_I:=d_I(t_I)$ and $v_I:=\rho_{Q(I,k)}(t_I)$.
Then
\begin{equation}\label{eq:tuv}
  \Norm{t}{\dot b^{s-\frac1p}_{p,r}(d,\R^{n-1})}
  =\Norm{u}{\dot b^{s-\frac1p}_{p,r}(\R^{n-1})}, \qquad
  \Norm{\hat t^{(k)}}{\dot a^{s}_{p,q}(\rho,\R^n)}
  =\Norm{\hat v^{(k)}}{\dot a^{s}_{p,q}(\R^n)}.
\end{equation}
For \eqref{it:seq-test} $\Rightarrow$ \eqref{it:seq-trace}, we have
\begin{equation*}
\begin{split}
  \Norm{u}{\dot b^{s-\frac1p}_{p,r}(\R^{n-1})}
  &\lesssim\Norm{v}{\dot b^{s-\frac1p}_{p,r}(\R^{n-1})}\quad\text{by assumption \eqref{it:seq-test}} \\
  &\approx\Norm{\hat v^{(k)}}{\dot a^{s}_{p,q}(\R^n)}\quad\text{by Lemma \ref{lem:bfTrace}}.
\end{split}
\end{equation*}
Together with \eqref{eq:tuv}, this proves \eqref{it:seq-trace}.

For \eqref{it:seq-trace} $\Rightarrow$ \eqref{it:seq-test}, take $t_I=\delta_{I,J}\fx$ with arbitrary $J\in\mathscr D(\R^{n-1})$ and $\fx\in\fX$. Then
\begin{equation*}
\begin{split}
  \operatorname{LHS}\eqref{eq:sec-trace}
  &=\ell(J)^{-(s-\frac1p)}\ell(J)^{(n-1)(\frac1p-\frac12)}d_J(\fx)
  =\ell(J)^{-s}\ell(J)^{n(\frac1p-\frac12)}\ell(J)^{\frac12}d_J(\fx), \\
  \operatorname{RHS}\eqref{eq:sec-trace}
  &=\ell(J)^{-s}\ell(J)^{n(\frac1p-\frac12)}\rho_{Q(J,k)}(\ell(J)^{\frac12}\fx)
  =\ell(J)^{-s}\ell(J)^{n(\frac1p-\frac12)}\ell(J)^{\frac12}\rho_{Q(J,k)}(\fx).
\end{split}
\end{equation*}
Cancelling out the common factor, we see that \eqref{it:seq-trace} implies $d_J(\fx)\lesssim \rho_{Q(J,k)}(\fx)$, i.e., \eqref{it:seq-test}.

The proof of the similar equivalence with both ``$\lesssim$'' replaced by ``$\gtrsim$'' is exactly the same, making the same replacements in the proof.
\end{proof}

\begin{remark}\label{rem:sec-test-0}
If the family $\rho=\{\rho_Q\}_{Q\in\mathscr D(\R^n)}$ satisfies an estimate of the form
\begin{equation}\label{eq:weak-db}
  \frac{\rho_Q(\fx)}{\rho_R(\fx)}\lesssim\Big(1+\frac{\abs{c_Q-c_R}}{\ell(Q)}\Big)^c
\end{equation}
for all $\fx\in\fX$ and all pairs of cubes of equal size $\ell(Q)=\ell(R)$ with the implicit
positive constant independent of $\fx$,
$Q$, and $R$, then condition \eqref{eq:sec-test} is equivalent (up to the dependence of the implicit constant on $k$) to
\begin{equation}\label{eq:sec-test-0}
  d_I(\fx)\lesssim\rho_{Q(I,0)}(\fx)
\end{equation}
with the implicit positive constant independent of $\fx$ and $I$.
Indeed, \eqref{eq:sec-test-0} implies \eqref{eq:sec-test} with constant $(1+\abs{k})^c$ and vice versa. The same applies to ``$\gtrsim$'' in place of ``$\lesssim$" in both \eqref{eq:sec-test} and \eqref{eq:sec-test-0}.
\end{remark}

\subsection{Function spaces}

We now turn to the trace and the extension operators acting between function spaces over $\R^n$ and $\R^{n-1}$.

For a sufficiently nice function $\Psi$ on $\R^n$, its \emph{trace} on $\R^{n-1}$ is simply
\begin{equation}\label{eq:trace-naive}
  (\operatorname{Tr}\Psi)(x'):=\Psi(x',0),\qquad x'\in\R^{n-1}.
\end{equation}
For general elements $f\in\dot A^s_{p,q}(\rho,\R^n)$, the idea is to expand $f$ in terms sufficiently regular wavelets $\Psi_Q^{(\lambda)}$ and apply \eqref{eq:trace-naive} component-wise to each $\Psi_Q^{(\lambda)}=\Psi^{(\lambda',\lambda_n)}_{Q(I,k)}$.
Then it is useful to note that
\begin{equation}\label{eq:trace-PsiQ}
\begin{split}
  \operatorname{Tr}\Psi_{Q(I,k)}^{(\lambda',\lambda_n)}
  &:=\Psi_{Q(I,k)}^{(\lambda',\lambda_n)}(\cdot,0),\qquad
  \lambda'\in\{0,1\}^{n-1},\quad\lambda_n\in\{0,1\}, \\
  &=\psi_{[k,k+1)\ell(I)}^{(\lambda_n)}(0)\cdot \Psi_{I}^{(\lambda')}
  =\ell(I)^{-\frac12}\psi^{(\lambda_n)}(-k)\cdot\Psi_{I}^{(\lambda')}.
\end{split}
\end{equation}
is a constant multiple of a wavelet $\Psi_I^{(\lambda')}$ on $\R^{n-1}$.

The case of the extension is slightly more tricky. In principle, given a function $F$ on $\R^{n-1}$‚ we want to construct a function $\widetilde F=\operatorname{Ext}F$ on $\R^n$ such that $\widetilde F(x',0)=F(x')$. Obviously, there is a great deal of flexibility in the choice $\widetilde F$. Starting with the wavelets first, we make the following choice:

\begin{lemma}\mbox{}
\begin{enumerate}[\rm(i)]
  \item\label{it:k0} For any system of Daubechies wavelets, there is $k_0\in\Z$ such that $\psi^{(0)}(-k_0)=\varphi(-k_0)\neq 0$.
  \item\label{it:tr-ext} Using the number $k_0$ from \eqref{it:k0} and defining
\begin{equation}\label{eq:def-ext}
  \operatorname{Ext}\Psi_I^{(\lambda')}:=\frac{\ell(I)^{\frac12}}{\psi^{(0)}(-k_0)}
  \Psi^{(\lambda',0)}_{Q(I,k_0)},
  \qquad I\in\mathscr D(\R^{n-1}),\quad\lambda'\in\Lambda_{n-1},
\end{equation}
it follows that
\begin{align}\label{tr-ext}
  (\operatorname{Tr}\circ\operatorname{Ext})\Psi^{(\lambda')}_I=\Psi_I^{(\lambda')}.
\end{align}
\end{enumerate}
\end{lemma}

\begin{proof}
Part \eqref{it:k0} is \cite[Remark 5.2]{BHYY:III}. Part \eqref{it:tr-ext} then follows from the computation
\begin{equation*}
\begin{split}
  (\operatorname{Tr}\circ\operatorname{Ext})\Psi^{(\lambda')}_I
  &=\frac{\ell(I)^{\frac12}}{\psi^{(0)}(-k_0)}\operatorname{Tr}\Psi^{(\lambda',0)}_{Q(I,k_0)}
    \quad\text{by \eqref{eq:def-ext}} \\
   &=\frac{\ell(I)^{\frac12}}{\psi^{(0)}(-k_0)}\frac{\psi^{(0)}(-k_0)}
   {\ell(I)^{\frac12}}\Psi^{(\lambda')}_{I}
   =\Psi^{(\lambda')}_I     \quad\text{by \eqref{eq:trace-PsiQ}}.
\end{split}
\end{equation*}
This completes the proof.
\end{proof}

\begin{theorem}\label{thm:ext-trace}
Let $p\in(0,\infty)$, $q\in(0,\infty]$,
$d=\{d_{I}\}_{I\in\mathscr{D}(\mathbb{R}^{n-1})}$
be a family of $(u_d,a_d,b_d,c_d)$-norms,
and $\rho=\{\rho_Q\}_{Q\in\mathscr{D}(\mathbb{R}^n)}$
be a family of $(u_\rho,a_\rho,b_\rho\,c_\rho)$-norms,
where $u_{i}\in(0,1]$ and $a_i,b_i,c_i\in[0,\infty)$
for $i\in\{d,\rho\}$.
Let $N\in\mathbb{N}$ be sufficiently large
such that \eqref{eq:N-wavelet} holds for the parameters
corresponding to
both spaces
\begin{equation*}
    \dot{A}^{s}_{p,q}(\rho,\R^n),\qquad
    \dot{B}^{s-\frac1p}_{p,r}(d,\R^{n-1}),\quad\text{where}\quad
    r=\begin{cases} q & \text{if }A=B, \\ p & \text{if }A=F,\end{cases}
\end{equation*}
and let $\{\Psi_Q^{(\lambda)}\}_{(\lambda,Q)\in\Omega_n}$ be a Daubechies wavelet system of order $N$.
Then the following assertions hold:
\begin{enumerate}[\rm(i)]
  \item\label{it:ext-thm} With $\operatorname{Ext}\Psi_I^{(\lambda')}$ defined as in \eqref{eq:def-ext}, the rule
\begin{equation*}
  \operatorname{Ext}:f\mapsto
  \sum_{\lambda'\in\Lambda_{n-1}}
  \sum_{I\in\mathscr D(\R^{n-1})}\pair{f}{\Psi_I^{(\lambda')}}
     \operatorname{Ext}\Psi_{I}^{(\lambda')}
\end{equation*}
gives a well-defined bounded linear operator
\begin{equation*}
  \operatorname{Ext}:\
  \dot B^{s-\frac{1}{p}}_{p,r}(d,\R^{n-1})\to
  \dot A^{s}_{p,q}(\rho,\R^n)
\end{equation*}
if and only if
\begin{equation}\label{eq:ext-test}
d_I(\fx)\gtrsim \rho_{Q(I,0)}(\fx),
\qquad\forall\ I\in\mathscr{D}(\mathbb{R}^{n-1}),\ \fx\in\fX
\end{equation}
with the implicit positive constant independent of $I$ and $\fx$.

  \item\label{it:trace-thm} Suppose further that
\begin{equation}\label{eq:trace-s}
   s>\frac1p+(n-1)\Big(\frac{1}{\max\{p,u_d\}}-1\Big)   +b_d.
\end{equation}
  With $\operatorname{Tr}\Psi_Q^{(\lambda)}$ defined as in \eqref{eq:trace-PsiQ}, the rule
\begin{equation*}
  \operatorname{Tr}:f\mapsto
  \sum_{Q\in\mathscr D(\R^n)}\pair{f}{\Psi_Q^{(\lambda)}}\operatorname{Tr}\Psi_Q^{(\lambda)}
\end{equation*}
gives a well-defined bounded linear operator
\begin{equation*}
  \operatorname{Tr}:\ \dot A^{s}_{p,q}(\rho,\R^n)\to\dot B^{s-\frac{1}{p}}_{p,r}(d,\R^{n-1})
\end{equation*}
if and only if
\begin{equation}\label{eq:trace-test}
d_I(\fx)\lesssim \rho_{Q(I,0)}(\fx),\qquad
\forall\ I\in\mathscr{D}(\mathbb{R}^{n-1}),\ \fx\in\fX
\end{equation}
with the implicit positive constant independent of $I$ and $\fx$.

  \item\label{it:tr-ext-comp} Under all assumptions \eqref{eq:ext-test}, \eqref{eq:trace-s}, and \eqref{eq:trace-test}, the composition $\operatorname{Tr}\circ\operatorname{Ext}$ is the identity on $\dot B^{s-\frac1p}_{p,r}(\rho)$.
\end{enumerate}
\end{theorem}

\begin{proof}
We begin with some observations that are relevant to both parts \eqref{it:ext-thm} and \eqref{it:trace-thm}.

First, the assumption that $\rho$ is a family of $(u_\rho,a_\rho,b_\rho,c_\rho)$-norms (Definition \ref{def:uabc}) implies condition \eqref{eq:weak-db} of Remark \ref{rem:sec-test-0} with $c=c_\rho$. Thus, by the said remark, both \eqref{eq:ext-test} and \eqref{eq:trace-test} are equivalent to versions with $Q(I,k)$ in place of $Q(I,0)$, allowing the implicit positive constants to depend on $k\in\Z$. These are precisely assumption \eqref{eq:sec-test} of Lemma \ref{lem:bfTrace-rho} and its reverse version, giving us access to that lemma.

Second, by the assumption on $N$, Corollary \ref{cor:wavelet} (referring to Proposition \ref{prop-wavelet}\eqref{it:wave=mole}) guarantees that
\begin{equation}\label{eq:PsiQ=mole}
 \text{$\Psi_Q^{(\lambda)}$ ($Q\in\mathscr D(\R^n)$, $\lambda\in\Lambda_n$) are analysis and synthesis molecules for $\dot A^s_{p,q}(\rho,\R^n)$}
\end{equation}
and that
\begin{equation}\label{eq:PsiI=mole}
 \text{$\Psi_I^{(\lambda')}$ ($I\in\mathscr D(\R^{n-1})$, $\lambda\in\Lambda_{n-1}$) are analysis and synthesis molecules for $\dot B^{s-\frac1p}_{p,r}(d,\R^{n-1})$}.
\end{equation}
Moreover, the additional assumption \eqref{eq:trace-s} in part \eqref{it:trace-thm} is exactly the assumption \eqref{eq:sLarge} corresponding to $(u_d,a_d,b_d,c_d)$-norms in dimension $n-1$ in place of $(u,a,b,c)$-norms in dimension $n$. Thus, under this assumption, Corollary \ref{cor:wavelet} (referring to Proposition \ref{prop-wavelet}\eqref{it:wave=symole}) further guarantees that
\begin{equation}\label{eq:Psi0=mole}
 \text{$\Psi_I^{(0)}$ ($I\in\mathscr D(\R^{n-1})$) are synthesis molecules for $\dot B^{s-\frac1p}_{p,r}(d,\R^{n-1})$}.
\end{equation}
Below, we will repeatedly use the mapping properties of analysis and synthesis molecules established in Theorem \ref{thm-mole}. We will do this without quoting Theorem \ref{thm-mole}, only referring to the relevant property among \eqref{eq:PsiQ=mole}, \eqref{eq:PsiI=mole}, and \eqref{eq:Psi0=mole}.

The proofs of the ``only if'' claims in both parts \eqref{it:ext-thm} and \eqref{it:trace-thm} follow routinely by testing the boundedness of $\operatorname{Ext}$ on functions of the form $f=\Psi_I^{(\lambda')}\fx$ and the boundedness of $\operatorname{Tr}$ on functions of the form $f=\Psi_Q^{(\lambda)}\fx$, and we omit details. We then consider the ``if'' claims of each part in turn:

\eqref{it:ext-thm}:
From \eqref{eq:def-ext}, we see that the extension operator
\begin{equation*}
  \operatorname{Ext}f=\frac{1}{\psi^{(0)}(-k_0)}\sum_{\lambda'\in\Lambda_{n-1}}\sum_{I\in\mathscr D(\R^{n-1})}
    \ell(I)^{\frac12} \pair{f}{\Psi_I^{(\lambda')}} \Psi^{(\lambda',0)}_{Q(I,k_0)}
\end{equation*}
is a finite linear combination over compositions of the following maps:
\begin{enumerate}[\rm(a)]
  \item the wavelet analysis map $f\mapsto\{u_I:=\pair{f}{\Psi_I^{(\lambda')}}\}_{I\in\mathscr D(\R^{n-1})}$, which is bounded from $\dot B^{s-\frac1p}_{p,r}(d,\R^{n-1})$ to $\dot b^{s-\frac1p}_{p,q}(\rho,\R^n)$ by the analysis part \eqref{eq:PsiI=mole};
  \item the map $u=\{u_I\}_{I\in\mathscr D(\R^{n-1})}\mapsto\hat u^{(k_0)}=:\{t_Q\}_{Q\in\mathscr D(\R^n)}$, which is bounded from $\dot b^{s-\frac1p}_{p,r}(d,\R^{n-1})$ to $\dot a^{s}_{p,q}(\rho,\R^n)$ by Lemma \ref{lem:bfTrace-rho} and Remark \ref{rem:sec-test-0}, under the assumption \eqref{eq:trace-test};
  \item the wavelet synthesis map $\displaystyle\{t_Q\}_{Q\in\mathscr D(\R^{n})}\mapsto \sum_{Q\in\mathscr D(\R^{n})}t_Q\Psi_Q^{(\lambda',0)}$, which is bounded from $\dot a^s_{p,q}(\rho,\R^{n})$ to $\dot A^s_{p,q}(\rho,\R^{n})$ by the synthesis part of \eqref{eq:PsiQ=mole}.
\end{enumerate}
As a finite linear combination of compositions of these bounded maps, $\operatorname{Ext}$ defines a bounded linear operator from $\dot B^{s-\frac1p}_{p,r}(d,\R^{n-1})$ to $\dot A^s_{p,q}(\rho,\R^n)$.

\eqref{it:trace-thm}: We note from \eqref{eq:trace-PsiQ} that $\operatorname{Tr}\Psi_{Q(I,k)}^{(\lambda)}$ is non-zero only if $\psi^{(\lambda_n)}(-k)\neq 0$. Recalling that the Daubechies wavelets are compactly supported, this happens only if $\abs{k}\leq K$ for some finite $K$. Thus
\begin{equation*}
\begin{split}
  \operatorname{Tr}f
  &=\sum_{\abs{k}\leq K}\sum_{\lambda\in\Lambda_n}\sum_{I\in\mathscr D(\R^{n-1})}
  \pair{f}{\Psi_{Q(I,k)}^{(\lambda)}}\operatorname{Tr}\Psi_{Q(I,k)}^{(\lambda)} \\
  &=\sum_{\abs{k}\leq K}\sum_{\lambda\in\Lambda_n}
      \psi^{(\lambda_n)}(-k)\sum_{I\in\mathscr D(\R^{n-1})}\ell(I)^{-\frac12}\pair{f}{\Psi_{Q(I,k)}^{(\lambda)}}
      \Psi_I^{(\lambda')}\quad\text{by \eqref{eq:trace-PsiQ}}
\end{split}
\end{equation*}
is a finite linear combination over compositions of the following maps:
\begin{enumerate}[\rm(a)]
  \item the wavelet analysis map $f\mapsto\{t_Q:=\pair{f}{\Psi_Q^{(\lambda)}}\}_{Q\in\mathscr D(\R^n)}$, which is bounded from $\dot A^s_{p,q}(\rho,\R^n)$ to $\dot a^s_{p,q}(\rho,\R^n)$ by the analysis part of \eqref{eq:PsiQ=mole};
  \item the truncation map $\{t_Q\}_{Q\in\mathscr D(\R^n)}\mapsto \{ \one_{\mathscr Q_{k_0}}(Q)t_Q\}_{Q\in\mathscr D(\R^n)}$, where $\mathscr Q_{k_0}=\{Q(I,k):I\in\mathscr D(\R^{n-1})\}$, which is obviously bounded on $\dot a^s_{p,q}(\rho,\R^n)$;
  \item the map $\{ \one_{\mathscr Q_{k_0}}(Q)t_Q\}_{Q\in\mathscr D(\R^n)}\mapsto \{u_I:= \ell(I)^{-\frac12} t_{Q(I,k)} \}_{ I\in\mathscr D(\R^{n-1})}$, which is bounded from (a subspace of) $\dot a^s_{p,q}(\rho,\R^n)$ to $\dot b^{s-\frac1p}_{p,r}(d,\R^{n-1})$ by Lemma \ref{lem:bfTrace-rho} and Remark \ref{rem:sec-test-0}, under the assumption \eqref{eq:trace-test};
  \item the wavelet synthesis map
  $$\displaystyle\{u_I\}_{I\in\mathscr D(\R^{n-1})}\mapsto \sum_{I\in\mathscr D(\R^{n-1})}u_I\Psi_I^{(\lambda')},$$
  which is bounded from $\dot b^{s-\frac1p}_{p,r}(d,\R^{n-1})$ to $\dot B^{s-\frac1p}_{p,r}(d,\R^{n-1})$ by the synthesis part of \eqref{eq:PsiI=mole} and \eqref{eq:Psi0=mole}; we note that the latter is needed, since the partial vector $\lambda'\in\{0,1\}^{n-1}$ can be identically $0$ even if $\lambda\in\Lambda_n:=\{0,1\}^n\setminus\{\mathbf{0}\}$.
\end{enumerate}
As a finite linear combination of compositions of these bounded maps, $\operatorname{Tr}$ defines a bounded linear operator from $\dot A^s_{p,q}(\rho,\R^n)$ to $\dot B^{s-\frac1p}_{p,r}(d,\R^{n-1})$.

\eqref{it:tr-ext-comp}:
Let $f\in\dot{B}^{s-\frac1p}_{p,r}(d,\mathbb{R}^{n-1})$.
Noting that $$\{\operatorname{Ext}\Psi_I^{(\lambda')}
\}_{\lambda'\in\Lambda_{n-1},I\in\mathscr{D}(\mathbb{R}^{n-1})}$$
is a subset of the wavelet system $\{\Psi_Q^{\lambda}\}_{\lambda\in\Lambda_n,
Q\in\mathscr{D}(\mathbb{R}^n)}$
on $\mathbb{R}^n$, we infer that
\begin{align*}
(\operatorname{Tr}\circ\operatorname{Ext})
f&=\operatorname{Tr}
\Big(\sum_{\lambda'\in\Lambda_{n-1}}\sum_{I\in\mathscr{D}(\mathbb{R}^{n-1})}
\langle f,\Psi_I^{(\lambda')}\rangle
\operatorname{Ext}\Psi_I^{(\lambda')}\Big)\\
&=\sum_{\lambda'\in\Lambda_{n-1}}\sum_{I\in\mathscr{D}(\mathbb{R}^{n-1})}
\langle f,\Psi_I^{(\lambda')}\rangle
(\operatorname{Tr}\circ\operatorname{Ext})\Psi_I^{(\lambda')}\\
&=\sum_{\lambda'\in\Lambda_{n-1}}\sum_{I\in\mathscr{D}(\mathbb{R}^{n-1})}
\langle f,\Psi_I^{(\lambda')}\rangle\Psi_I^{(\lambda')}=f
\qquad\text{by \eqref{tr-ext}}
\end{align*}
holds in $\mathscr{S}'_\infty(\mathbb{R}^{n-1};\fX)$.
This completes the proof of Theorem \ref{thm:ext-trace}.
\end{proof}

Finally, we obtain the corresponding boundedness of the trace and the extension operators on pointwise-weighted spaces as follows.

\begin{corollary}\label{cor:ext-trace}
Let $p\in(0,\infty)$,
and let $V\in \A_v(\mathbb{R}^n)$ and $W\in\A_w(\mathbb{R}^{n-1})$ with $v,w\in[p,\infty)$
have doubling dimensions $\beta_V,\beta_W$.
Let $N\in\mathbb{N}$ be sufficiently large
such that \eqref{eq:N-wavelet-V} holds for the parameters
corresponding to
both spaces
\begin{equation*}
  \dot{A}^{s}_{p,q}(V,\R^n),\quad
  \dot{B}^{s-\frac1p}_{p,r}(W,\R^{n-1}),\quad\text{where}\quad
  \begin{cases} q\in(0,\infty]\quad\&\quad r=q & \text{if }A=B, \\
  q\in(0,v+\varepsilon_V)\quad\&\quad r=p & \text{if }A=F,
   \end{cases}
\end{equation*}
where $\varepsilon_V$ is the optimal reverse H\"older lifting index of $V$, and let $\{\Psi_Q^{(\lambda)}\}_{(\lambda,Q)\in\Omega_n}$ be a Daubechies wavelet system of order $N$.
For $\operatorname{Ext}$ and $\operatorname{Tr}$ as in Theorem \ref{thm:ext-trace}, the following results hold:
\begin{enumerate}[\rm(1)]
  \item $\operatorname{Ext}$ defines a bounded linear operator
\begin{equation*}
  \operatorname{Ext}:\ \dot B^{s-\frac{1}{p}}_{p,r}(W,\R^{n-1})
   \to\dot A^{s}_{p,q}(V,\R^n)
\end{equation*}
if and only if
\begin{equation}\label{eq:ext-test-V}
\rho_{\aveL^w(I,W)}(\fx)\gtrsim \rho_{\aveL^v(Q(I,0),V)}(\fx)\qquad
\forall\ I\in\mathscr{D}(\mathbb{R}^{n-1}),\ \fx\in\fX
\end{equation}
with the implicit positive constant independent of $I$ and $\fx$.

  \item Suppose further that
\begin{equation}\label{eq:traceVs}
   s>\frac1p+(n-1)\Big(\frac{1}{\max\{1,p\}}-1\Big)
   +\frac{\beta_W-n+1}{w}.
\end{equation}
  Then $\operatorname{Tr}$ defines a bounded linear operator
\begin{equation*}
  \operatorname{Tr}:\ \dot A^{s}_{p,q}(V,\R^n)\to\dot B^{s-\frac{1}{p}}_{p,r}(W,\R^{n-1})
\end{equation*}
if and only if
\begin{equation}\label{eq:trace-test-V}
  \rho_{\aveL^w(I,W)}(\fx)\lesssim \rho_{\aveL^v(Q(I,0),V)}(\fx)\qquad
  \forall\ I\in\mathscr{D}(\mathbb{R}^{n-1}),\ \fx\in\fX\end{equation}
  with the implicit positive constant independent of $I$ and $\fx$.

  \item Under all the assumptions \eqref{eq:ext-test-V}, \eqref{eq:traceVs}, and \eqref{eq:trace-test-V}, the composition $\operatorname{Tr}\circ\operatorname{Ext}$ is the identity on $\dot B^{s-\frac1p}_{p,r}(W,\R^{n-1})$.
\end{enumerate}
\end{corollary}

\begin{proof}
By Theorems \ref{thm:equiB} and \ref{thm:equiF}, we have the coincidence of spaces
\begin{equation*}
 \dot A^s_{p,q}(V,\R^n)=\dot A^s_{p,q}(\rho,\R^n),\quad\text{where}\quad
  \rho:=\{\rho_Q:=\rho_{\aveL^v(Q,V)}\}_{Q\in\mathscr{D}(\mathbb{R}^n)},
\end{equation*}
and
\begin{equation*}
 \dot B^{s-\frac1p}_{p,r}(W,\R^n)=\dot B^{s-\frac1p}_{p,r}(d,\R^{n-1}),\quad\text{where}\quad
  d:=\{d_I:=\rho_{\aveL^w(I,W)}\}_{I\in\mathscr{D}(\mathbb{R}^{n-1})},
\end{equation*}
By Remark \ref{rem:uabc}, these are families of $(u_i,a_i,b_i,c_i)$-norms, $i\in\{\rho,d\}$, where in particular
\begin{equation*}
  (u_d,a_d,b_d,c_d)
  =(\min\{1,w\},\frac{n-1}{w},\frac{\beta_W-(n-1)}{w},\frac{\beta_W}{w}).
\end{equation*}
Since $w\in[p,\infty)$ by assumption, it follows that $\min\{p,u_d\}=\min\{p,1,w\}=\min\{p,1\}$. Thus assumption \eqref{eq:traceVs} of the corollary coincides with assumption \eqref{eq:trace-s} of Theorem \ref{thm:ext-trace} for the special case of the quasi-norms now under consideration. Similarly, the assumption \eqref{eq:N-wavelet-V} on $N$ is the special case, for these parameters, of assumption \eqref{eq:N-wavelet} made in Theorem \ref{thm:ext-trace}. Thus Corollary \ref{cor:ext-trace} follows by specialising Theorem \ref{thm:ext-trace} to this case and using the equivalence of the pointwise-weighted and average-weighted spaces already mentioned.
\end{proof}


\begin{thebibliography}{99}

\bibitem{AC:12}
A. Aleman and O. Constantin,
The Bergman projection on vector-valued $L^2$-spaces with operator-valued weights,
J. Funct. Anal. 262 (2012), 2359--2378.

\bibitem{AB:book}
C. D. Aliprantis and K. C. Border,
Infinite-Dimensional Analysis. A Hitchhiker's Guide,
Stud. Econom. Theory 4,
Springer-Verlag, Berlin, 1994.



\bibitem{AH:11}
S. A. Argyros and R. G. Haydon,
A hereditarily indecomposable {$\mathscr{L}_\infty$}-space that
solves the scalar-plus-compact problem,
Acta Math. 206 (2011), 1--54

\bibitem{BX:study}
T. Bai and J. Xu, Non-regular pseudo-differential operators on matrix weighted
  {B}esov-{T}riebel-{L}izorkin spaces, J. Math. Study
  57 (2024), 84--100.

\bibitem{BX:Iran}
T. Bai and J. Xu,
Pseudo-differential operators on matrix weighted
  {B}esov-{T}riebel-{L}izorkin spaces,
Bull. Iranian Math. Soc. 50 (2024),
Paper No. 31, 26 pp.


\bibitem{BGM:86}
E. Berkson, T. A. Gillespie and P. S. Muhly,
Abstract spectral decompositions guaranteed by the {H}ilbert
  transform,
Proc. London Math. Soc. (3) 53 (1986), 489--517.

\bibitem{Borel}
E. Borel,
Les probabilit\'es d\'enombrables et leurs applications arithm\'etiques,
Palermo Rend. 27 (1909), 247--271.

\bibitem{Bou:83}
J. Bourgain, Some remarks on {B}anach spaces in which martingale difference
  sequences are unconditional, Ark. Mat. 21 (1983), 163--168.

\bibitem{BCHYY}
F. Bu, Y. Chen, T. Hyt\"onen, D. Yang and W. Yuan,
Real-variable theory of matrix-weighted multi-parameter Besov--Triebel--Lizorkin-type spaces,
preprint (2026), arXiv:2603.25137.

\bibitem{BHYY:I}
F. Bu, T. Hyt\"onen, D. Yang and W. Yuan,
Matrix-weighted {Besov}-type and {Triebel--Lizorkin}-type spaces {I}:
  {$A_p$}-dimensions of matrix weights and $\phi$-transform characterizations,
  Math. Ann. 391 (2025), 6105--6185.

\bibitem{BHYY:II}
F. Bu, T. Hyt\"onen, D. Yang and W. Yuan,
Matrix-weighted {B}esov-type and {T}riebel--{L}izorkin-type spaces
  {II}: {S}harp boundedness of almost diagonal operators,
J. Lond. Math. Soc. (2) 111 (2025),
Paper No. e70094, 59 pp.

\bibitem{BHYY:III}
F. Bu, T. Hyt\"onen, D. Yang and W. Yuan,
{Matrix-weighted {B}esov-type and {T}riebel-{L}izorkin-type spaces
  {III}: characterizations of molecules and wavelets, trace theorems, and
  boundedness of pseudo-differential operators and {C}alder\'on-{Z}ygmund
  operators}, Math. Z. 308 (2024), Paper No. 32, 67 pp.

\bibitem{BHYY:SCM}
F. Bu, T. Hyt\"onen, D. Yang and W. Yuan,
{B}esov--{T}riebel--{L}izorkin-type spaces with matrix {$A_\infty$}
  weights, Sci. China Math. 69 (2026), 383--460.


\bibitem{byyz26}
F. Bu, D. Yang, W. Yuan and M. Zhang,
Matrix-weighted Besov--Triebel--Lizorkin spaces of optimal scale:
real-variable characterizations,
invariance on integrable index, and Sobolev-type embedding,
J. Differential Equations 463 (2026), Paper No. 114140, 101 pp.

\bibitem{Bui:82}
H. Q. Bui, Weighted {B}esov and {T}riebel spaces: interpolation by the real
  method, Hiroshima Math. J. 12 (1982), 581--605.

\bibitem{Burk:83}
D. L. Burkholder,
A geometric condition that implies the existence of certain singular
  integrals of {B}anach-space-valued functions,
  in: Conference on Harmonic Analysis in Honor of {A}ntoni
  {Z}ygmund, {V}ol. {I}, {II} ({C}hicago, {I}ll., 1981),
  Wadsworth Math. Ser., pp. 270--286,
  Wadsworth, Belmont, CA, 1983.

\bibitem{CG:01}
M. Christ and M. Goldberg,
Vector {$A_2$} weights and a {H}ardy-{L}ittlewood maximal function,
Trans. Amer. Math. Soc. 353 (2001), 1995--2002.


\bibitem{CF:74}
R. R. Coifman and C. Fefferman,
Weighted norm inequalities for maximal functions and singular
  integrals, Studia Math. 51 (1974), 241--250.


\bibitem{CRW:76}
R. R. Coifman, R. Rochberg and G. Weiss,
Factorization theorems for {H}ardy spaces in several variables,
Ann. of Math. (2) 103 (1976), 611--635.

\bibitem{dj84}
G. David and J.-L. Journ\'e,
A boundedness criterion for generalized {C}alder\'on-{Z}ygmund
  operators, Ann. of Math. (2) 120 (1984), 371--397.

\bibitem{d01}
J. Duoandikoetxea, Fourier Analysis,
Grad. Stud. Math. 29,
American Mathematical Society, Providence, RI, 2001.

\bibitem{Fere:97}
V. Ferenczi, A uniformly convex hereditarily indecomposable {B}anach space,
Israel J. Math. 102 (1997), 199--225.

\bibitem{FTJ:79}
T. Figiel and N. Tomczak-Jaegermann,
Projections onto {H}ilbertian subspaces of {B}anach spaces,
Israel J. Math. 33 (1979), 155--171.


\bibitem{FJ:85}
M. Frazier and B. Jawerth,
Decomposition of {B}esov spaces,
Indiana Univ. Math. J. 34 (1985), 777--799.

\bibitem{FJ:90}
M. Frazier and B. Jawerth,
A discrete transform and decompositions of distribution spaces,
J. Funct. Anal. 93 (1990), 34--170.

\bibitem{FR:04}
M. Frazier and S. Roudenko,
Matrix-weighted {B}esov spaces and conditions of {$A_p$} type for
  {$0<p\leq1$},
Indiana Univ. Math. J. 53 (2004), 1225--1254.

\bibitem{FR:08}
M. Frazier and S. Roudenko,
Traces and extensions of matrix-weighted {B}esov spaces,
Bull. Lond. Math. Soc. 40 (2008), 181--192.

\bibitem{FR:21}
M. Frazier and S. Roudenko,
Littlewood--{P}aley theory for matrix-weighted function spaces,
Math. Ann. 380 (2021), 487--537.

\bibitem{ftw88}
M. Frazier, R. Torres and G. Weiss,
The boundedness of {C}alder\'on-{Z}ygmund operators on the spaces
  {$\dot F^{\alpha,q}_p$},
Rev. Mat. Iberoamericana 4 (1988), 41--72.

\bibitem{Gehring}
F. W. Gehring,
The {$L^{p}$}-integrability of the partial derivatives of a
  quasiconformal mapping, Acta Math. 130 (1973), 265--277.

\bibitem{GMS}
S. Geiss, S. Montgomery-Smith and E. Saksman,
On singular integral and martingale transforms,
Trans. Amer. Math. Soc. 362 (2010), 553--575.

\bibitem{GPTV:00}
T. A. Gillespie, S. Pott, S. Treil and A. Volberg,
The transfer method in estimates for vector {H}ankel operators,
Algebra i Analiz 12 (2000), 178--193.

\bibitem{GPTV:01}
T. A. Gillespie, S. Pott, S. Treil and A. Volberg,
Logarithmic growth for matrix martingale transforms,
J. London Math. Soc. (2) 64 (2001), 624--636.

\bibitem{GPTV:04}
T. A. Gillespie, S. Pott, S. Treil and A. Volberg,
Logarithmic growth for weighted {H}ilbert transforms and vector
  {H}ankel operators,
J. Operator Theory 52 (2004), 103--112.

\bibitem{Gold:03}
M. Goldberg,
Matrix {$A_p$} weights via maximal functions,
Pacific J. Math. 211 (2003), 201--220.

\bibitem{GM:93}
W. T. Gowers and B. Maurey,
The unconditional basic sequence problem,
J. Amer. Math. Soc. 6 (1993), 851--874.

\bibitem{Guerre:91}
S. Guerre-Delabri\`ere,
Some remarks on complex powers of {$(-\Delta)$} and {UMD} spaces,
Illinois J. Math. 35 (1991), 401--407.

\bibitem{HMW}
R. Hunt, B. Muckenhoupt and R. Wheeden,
Weighted norm inequalities for the conjugate function and {H}ilbert
  transform, Trans. Amer. Math. Soc. 176 (1973), 227--251.

\bibitem{HNVW1}
T. Hyt\"onen, J. van Neerven, M. Veraar and L. Weis,
Analysis in Banach Spaces. Vol. I. Martingales and Littlewood--Paley Theory,
Ergeb. Math. Grenzgeb. (3) 63, Springer, Cham, 2016.

\bibitem{HNVW2}
T. Hyt\"onen, J. van Neerven, M. Veraar and L. Weis,
Analysis in Banach Spaces. Vol. II. Probabilistic Methods and Operator Theory,
Ergeb. Math. Grenzgeb. (3) 67, Springer, Cham, 2017.

\bibitem{HNVW3}
T. Hyt\"onen, J. van Neerven, M. Veraar and L. Weis,
Analysis in Banach Spaces. Vol. III. Harmonic Analysis and Spectral Theory,
Ergeb. Math. Grenzgeb. (3) 76,
Springer, Cham, 2023.

\bibitem{Isra:21}
J. Isralowitz, Matrix weighted Triebel--Lizorkin bounds:
a simple stopping time proof, Proc. Amer. Math. Soc. 149 (2021), 4145--4158.

\bibitem{John:48}
F. John, Extremum problems with inequalities as subsidiary conditions, in:
Studies and Essays Presented to R. Courant on his 60th Birthday, January 8, 1948,
pp.\,187--204, Interscience Publishers, New York, 1948.

\bibitem{Kaiser:09}
C. Kaiser,
Calder\'on--Zygmund operators with operator-valued kernel on homogeneous Besov spaces,
Math. Nachr. 282 (2009), 69--85.

\bibitem{KT:book}
E. Klein and A. C. Thompson,
Theory of Correspondences.
Including Applications to Mathematical Economics,
Canad. Math. Soc. Ser. Monogr. Adv. Texts,
Wiley-Intersci. Publ.,
John Wiley \& Sons, Inc., New York, 1984.

\bibitem{Lauzon}
M. Lauzon, Maximal functions and {C}alderon-{Z}ygmund theory for vector-valued
  functions with operator weights,
  Indiana Univ. Math. J. 56 (2007), 1723--1748.

\bibitem{Lerner:simple}
A. K. Lerner, A simple proof of the {$A_2$} conjecture.
Int. Math. Res. Not. IMRN 2013 (2013), 3159--3170.

\bibitem{LYY:24}
Z. Li, D. Yang and W. Yuan,
Matrix-weighted {B}esov--{T}riebel--{L}izorkin spaces with logarithmic
  smoothness, Bull. Sci. Math. 193 (2024), Paper No. 103445, 54 pp.

\bibitem{LP:22}
A. Limani and S. Pott,
Sparse {L}erner operators in infinite dimensions, preprint (2021), arXiv:2103.17005.

\bibitem{m87}
Y. Meyer,
Principe d'incertitude, bases hilbertiennes et alg\`ebres d'op\'erateurs,
in: S\'eminaire Bourbaki, Vol. 1985/86,
Ast\'erisque, No. 145--146 (1987), 4, 209--223.

\bibitem{mey92}
Y. Meyer.
Wavelets and Operators,
Cambridge Stud. Adv. Math. 37,
Cambridge University Press, Cambridge, 1992.

\bibitem{Muck:72}
B. Muckenhoupt,
Weighted norm inequalities for the {H}ardy maximal function,
Trans. Amer. Math. Soc. 165 (1972), 207--226.

\bibitem{NPTV:17}
F. Nazarov, S. Petermichl, S. Treil and A. Volberg,
Convex body domination and weighted estimates with matrix weights,
Adv. Math. 318 (2017), 279--306.

\bibitem{NPTV:02}
F. Nazarov, G. Pisier, S. Treil and A. Volberg,
Sharp estimates in vector {C}arleson imbedding theorem and for vector
  paraproducts, J. Reine Angew. Math. 542 (2002), 147--171.

\bibitem{NT:hunt}
F. L. Nazarov and S. R. Treil,
The hunt for a Bellman function: Applications to estimates for
singular integral operators and to other
classical problems of harmonic analysis,
Algebra i Analiz 8 (1996), 32--162;
translation in St. Petersburg Math. J. 8 (1997), 721--824.

\bibitem{Page:70}
L. B. Page,
Bounded and compact vectorial {H}ankel operators,
Trans. Amer. Math. Soc. 150 (1970), 529--539.

\bibitem{Pisier:book}
G. Pisier, Martingales in {B}anach Spaces,
Cambridge Stud. Adv. Math. 155,
Cambridge University Press, Cambridge, 2016.

\bibitem{Pott:07}
S. Pott, A sufficient condition for the boundedness of operator-weighted
  martingale transforms and {H}ilbert transform,
Studia Math. 182 (2007), 99--111.

\bibitem{Rou:03}
S. Roudenko,
Matrix-weighted {B}esov spaces.
Trans. Amer. Math. Soc. 355 (2003), 273--314.

\bibitem{Rou:04}
S. Roudenko,
Duality of matrix-weighted {B}esov spaces,
Studia Math. 160 (2004), 129--156.

\bibitem{Rudin:FA}
W. Rudin,
Functional Analysis, Second edition,
Internat. Ser. Pure Appl. Math.,
McGraw-Hill, Inc., New York, 1991.

\bibitem{Sarason:67}
D. Sarason.
Generalized interpolation in {$H\sp{\infty }$},
Trans. Amer. Math. Soc. 127 (1967), 179--203.

\bibitem{s70}
E. M. Stein,
Singular Integrals and Differentiability Properties of Functions,
Princeton Math. Ser. No. 30,
Princeton University Press, Princeton, NJ, 1970.

\bibitem{Torres:91}
R. H. Torres,
Boundedness results for operators with singular kernels on
  distribution spaces,
Mem. Amer. Math. Soc. 90 (1991), no. 442, viii+172.

\bibitem{TV:angle}
S. Treil and A. Volberg,
Wavelets and the angle between past and future,
J. Funct. Anal. 143 (1997), 269--308.

\bibitem{Volberg:97}
A. Volberg, Matrix {$A_p$} weights via {$S$}-functions,
J. Amer. Math. Soc. 10 (1997), 445--466.

\bibitem{WYZ:24}
Q. Wang, D. Yang and Y. Zhang,
Real-variable characterizations and their applications of
  matrix-weighted {T}riebel--{L}izorkin spaces,
J. Math. Anal. Appl. 529 (2024), Paper No. 127629, 37 pp.

\bibitem{woj97}
P. Wojtaszczyk,
A Mathematical Introduction to Wavelets,
London Math. Soc. Stud. Texts 37,
Cambridge University Press, Cambridge, 1997.

\bibitem{YY:08}
D. Yang and W. Yuan,
A new class of function spaces connecting {T}riebel--{L}izorkin spaces
  and {$Q$} spaces, J. Funct. Anal. 255 (2008), 2760--2809.

\bibitem{YY:10}
D. Yang and W. Yuan,
New Besov-type spaces and Triebel-Lizorkin-type spaces including $Q$ spaces.
Math. Z. 265 (2010), 451--480.

\end{thebibliography}
\end{document}